\documentclass[11pt,reqno]{amsart}

\usepackage{BLS_Macros}

\usepackage{xcolor}

\title{The wave maps equation and Brownian paths}
\author{Bjoern Bringmann, Jonas  L\"uhrmann, and Gigliola Staffilani}
\date{\today}

\address
{Bjoern Bringmann\\
School of Mathematics\\
Institute for Advanced Study, Princeton, NJ 08540 \& Department of Mathematics, Princeton University, Princeton, NJ 08544}
\email{bjoern@ias.edu}

\address
{Jonas L\"{u}hrmann\\
Department of Mathematics\\
Texas A\&M University, College Station, TX 77843}
\email{luhrmann@math.tamu.edu}

\address
{Gigliola Staffilani\\
Department of Mathematics\\
Massachusetts Institute of Technology, Cambridge, MA 02139}
\email{gigliola@math.mit.edu}

\begin{document}

\maketitle 

\begin{abstract}
We discuss the $(1+1)$-dimensional wave maps equation with values in a compact Riemannian manifold $\mathcal{M}$. Motivated by the Gibbs measure problem, we consider Brownian paths on the manifold $\mathcal{M}$ as initial data. Our main theorem is the probabilistic local well-posedness of the associated initial value problem. The analysis in this setting combines analytic, geometric, and probabilistic methods.
\end{abstract}

%%%%%%%%%%%%%%%%%%%%%%%%%%% TABLE OF CONTENTS %%%%%%%%%%%%%%%%%%%%%%%%%%%%%%%%%%%%%%%%%

\tableofcontents

%%%%%%%%%%%%%%%%%%%%%%%%% INTRO %%%%%%%%%%%%%%%%%%%%%%%%%%%
\section{Introduction}

\subsection{Motivation}\label{section:motivation}

The wave maps equation is a geometric generalization of the linear wave equation for scalar-valued fields to fields that take values in a Riemannian manifold. Formally, a wave map $\phi \colon \mathbb{R}^{1+1} \to \M$ from $(1+1)$-dimensional Minkowski space $\R^{1+1}$ into a Riemannian manifold $(\M, g)$ is a critical point of the Lagrangian action functional
\begin{equation*}
 \mathcal{L}(\phi) := \iint_{\mathbb{R}^{1+1}} \bigl( - |\partial_t \phi|_g^2 + |\partial_x \phi|_g^2 \bigr) \, dx \, dt.
\end{equation*}
Throughout this paper we work with the extrinsic formulation of the wave maps equation and we consider smooth, compact Riemannian target manifolds $(\M,g)$ without boundary. By Nash's embedding theorem, we may regard $\M$ as an isometrically embedded submanifold of some Euclidean space $\R^\dimA$. We write $\phi = (\phi^1, \ldots, \phi^\dimA) \colon \mathbb{R}^{1+1} \to \M \hookrightarrow \mathbb{R}^\dimA$ for the corresponding extrinsic representation of the wave map and we denote by $\Second \colon T\M \times T\M \rightarrow N\M$ the second fundamental form of the embedding $\M \hookrightarrow \mathbb{R}^\dimA$. Then the wave maps equation from $(1+1)$-dimensional Minkowski space to $\M$ takes the form
\begin{equation}\tag{WM}\label{intro:eq-WM}
 \partial_\mu \partial^\mu \phi^k = - \Second^k_{ij}(\phi) \partial_\mu \phi^i \partial^\mu \phi^j, \quad 1 \leq i,j, k \leq D, \quad 0 \leq \mu \leq 1.
\end{equation}
We use the standard conventions of raising or lowering indices with respect to the Minkowski metric with signature $(-+)$ on $\mathbb{R}^{1+1}$, and of summing over repeated upper and lower indices.

The equation (WM) is invariant under the scaling
\begin{equation} \label{intro:eq-scaling}
 \phi(t,x) \mapsto \phi(\lambda t, \lambda x), \quad \lambda > 0,
\end{equation}
and solutions formally conserve the Hamiltonian
\begin{equation} \label{intro:eq-energy}
    H[\phi,\phi_t] := \frac12 \int_{\mathbb{R}} \bigl(   |\partial_x \phi|^2 + |\partial_t \phi|^2\bigr) \, dx.
\end{equation}

Our interest lies in the Cauchy problem for~\eqref{intro:eq-WM} for initial data 
\begin{equation*}
 (\phi, \partial_t \phi)|_{t=0} = (\phi_0, \phi_1) 
\end{equation*} 
given by maps 
\begin{equation}
 \phi_0 \colon \mathbb{R} \to \M, \quad \phi_1 \colon \mathbb{R} \to \phi_0^\ast T\M, \quad \phi_1(x) \in T_{\phi_0(x)}\M.
\end{equation}

Motivated by the problem of constructing and proving the invariance of the Gibbs measure for the wave maps equation, we choose the initial data $\phi_0$ as a Brownian path on the manifold and the initial velocity $\phi_1$ as white noise on the  pullback bundle $\phi_0^\ast T\mathcal{M}$. A Brownian path $B\colon \R \rightarrow \M$ is the natural generalization of Euclidean Brownian motion $W\colon \R\rightarrow \R^\dimA$ from the vector-valued to the manifold-valued case. As we will describe below (Section \ref{section:intro-main} and Section \ref{section:Brownian}), the Brownian path $B\colon \R\rightarrow \M$ can be approximated by smooth paths $(B^\varepsilon)_{\varepsilon}\colon \R \rightarrow \M$, which are illustrated in Figure \ref{figure:Brownian}. In the literature, Brownian paths are more commonly referred to as Brownian motion on a manifold, but our terminology better distinguishes the argument of $B=B(x)$ from the time-coordinate in the wave maps equation \eqref{intro:eq-WM}. Brownian paths are natural geometric objects and have important applications to heat kernel estimates and index theorems. For a detailed introduction, we refer the reader to the textbook \cite{Hsu02}. The initial velocity $V\in B^\ast T\mathcal{M}$, which is described in more detail below, is chosen as the natural generalization of vector-valued white noise. 

The Gibbs measure of the wave maps equation is formally given by
\begin{equation}\label{intro:eq-gibbs}
\mathrm{d}\mu(\phi_0,\phi_1) = \text{``} \mathcal{Z}^{-1}
\exp\Big(- \frac{1}{2} \int_\R  |\partial_x \phi_0|_g^2 +|\phi_1|_g^2 \, \mathrm{d}x \Big) \, \mathrm{d}\phi_0 \mathrm{d}\phi_1 \text{"},
\end{equation}
where $\mathcal{Z}$ is a normalization constant. While the Gibbs measure in \eqref{intro:eq-gibbs} has not yet been constructed rigorously, we believe that the pair $(B,V)$ accurately describes its samples. This belief is based on the Wiener measure, which is the analogue of the Gibbs measure in the parabolic setting (cf. \cite{H16-2,BGHZ21}). The Wiener measure is formally given by 
\begin{equation*}
\mathrm{d}\nu(\phi_0) = \text{``} \mathcal{Z}^{-1}
\exp\Big(- \frac{1}{2} \int_\R  |\partial_x \phi_0|_g^2 \, \mathrm{d}x \Big) \, \mathrm{d}\phi_0 \text{"}
\end{equation*}
and has been constructed\footnote{To be precise, \cite{AD99} constructed the Wiener measure on the compact domain $[0,1]$ instead of $\R$. } rigorously by Andersson and Driver in \cite{AD99}. Furthermore, \cite{AD99} proves that the Wiener measure is absolutely continuous with respect to the law of Brownian paths, where the Radon-Nikodym derivative only depends on the scalar curvature of the manifold.

We now state an informal version of our main result.

\begin{figure}
     \centering
     \begin{subfigure}[b]{0.49\textwidth}
        \hspace{-6ex}
         \includegraphics[trim=400 0 100 0,clip,width=1.4\textwidth]{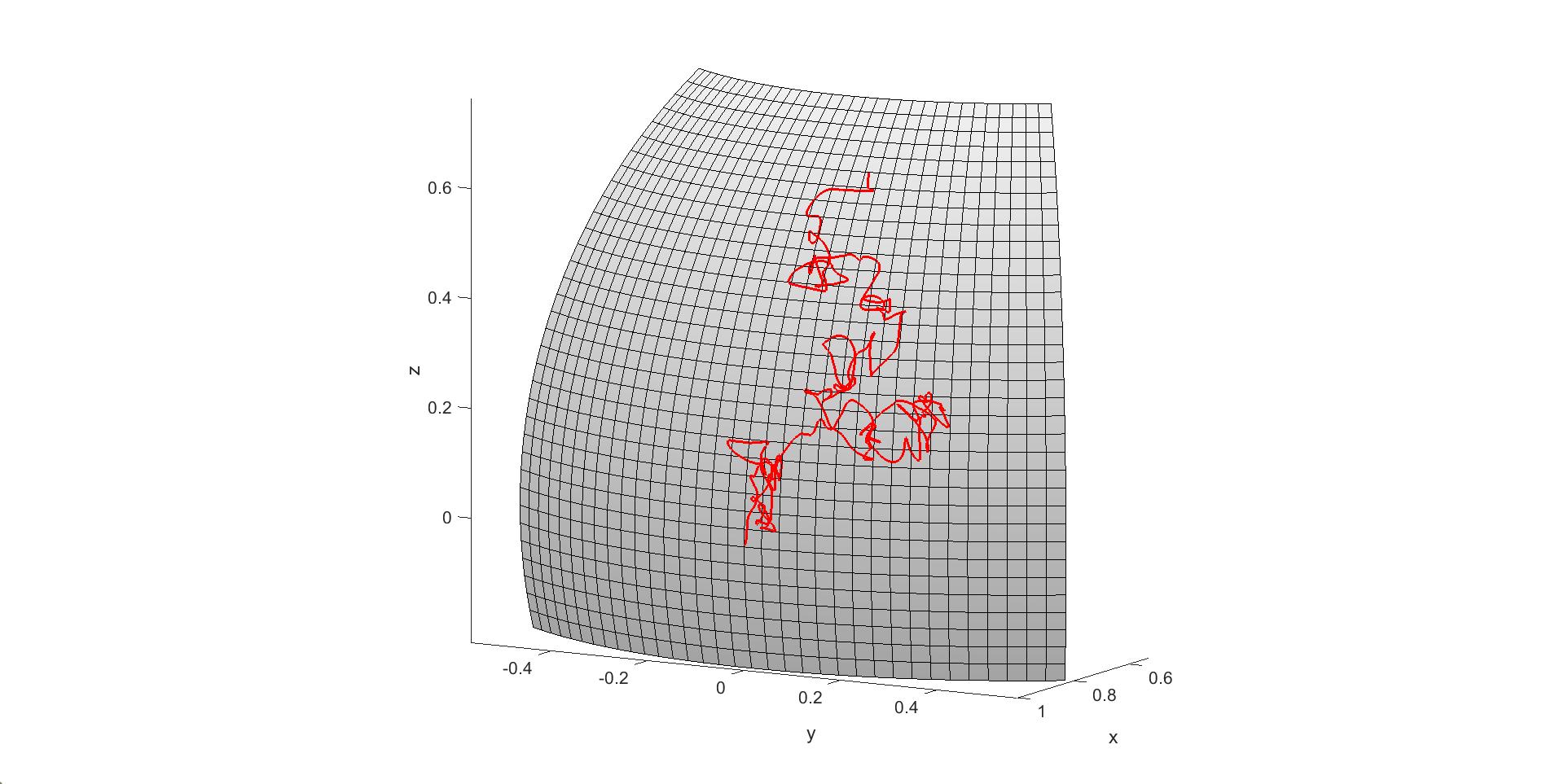}
         \caption{$\epsilon=10^{-3}$}
     \end{subfigure}
     \hfill
     \begin{subfigure}[b]{0.49\textwidth}
     \hspace{-6ex}
      \includegraphics[trim=400 0 100 0,clip,width=1.4\textwidth]{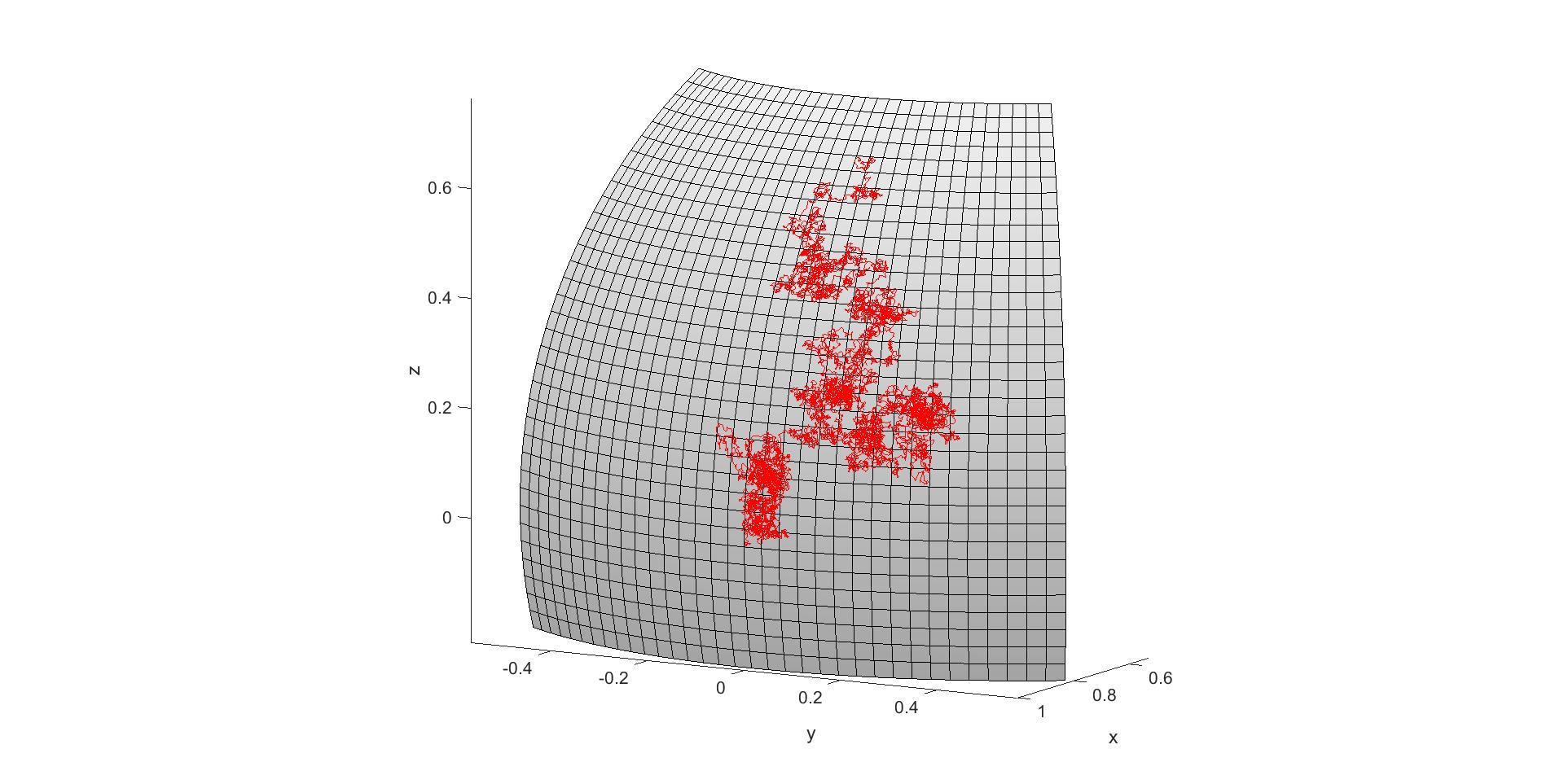}      \caption{$\epsilon=10^{-5}$}
     \end{subfigure}
     \caption{We display the approximations $B^\varepsilon$ of the Brownian path $B$ on $\mathbb{S}^2$ for two different values of $\varepsilon$. For $\varepsilon=10^{-3}$, the path still looks relatively smooth. For $\varepsilon=10^{-5}$, the path already exhibits heavy oscillations on small scales. In spite of these oscillations, the behavior of the two approximation agrees on large scales.}
     \label{figure:Brownian}
\end{figure}

\begin{theorem}[Informal version] \label{intro:thm}
 The wave maps equation \eqref{intro:eq-WM} with initial data $(\phi_0,\phi_1)=(B,V)$ is almost surely locally well-posed. 
\end{theorem}

The statement and proof of Theorem \ref{intro:thm} involve analytic, geometric, and probabilistic aspects, which is our motivation for proving this theorem. As mentioned above and further detailed below (see Section \ref{section:open}), Theorem \ref{intro:thm} can also be seen as partial progress towards proving the invariance of the Gibbs measure for the wave maps equation.

\subsection{Related works}\label{section:related}

The study of the $(1+1)$-dimensional wave maps evolution of Brownian paths on Riemannian manifolds features intriguing connections to several research areas in partial differential equations. 
In this subsection we describe some of these connections and we highlight related results on the study of the initial value problem for dispersive and hyperbolic equations with random initial data, on the deterministic well-posedness of (semi-linear) geometric wave equations, and on singular parabolic stochastic PDEs. 
In view of the rich and vast literature on these subjects, we do not attempt to be exhaustive.

\subsubsection{Random dispersive equations}\label{section:literature-dispersive} In recent years, there has been tremendous interest in random dispersive equations. In this introduction, we only discuss selected works in this field and refer the reader to the surveys \cite{BOP19, NAST}, the introduction of \cite{B20_2}, and the original works \cite{BOP15_2,B94,B96,B18_3,B20_2,BDNY22,BT08_1,BT08_2,CG19,DH19,DH21,DLM17,DNY19,DNY20,DNY21,GKO18-2,KLS20,KM19,KMV20,LM14,NORBS,OOT20,OOT21,OST21,ST21,T15}. 

The main motivation of this work is the Gibbs measure problem for Hamiltonian PDEs. At a formal level, we consider a symplectic manifold $(\mathcal{N},\omega)$ of dimension $\operatorname{dim}(\mathcal{N})=2n$, where $n\in \mathbb{N}\,\raisebox{0.5pt}{$\medcup$} \, \{ \infty\}$, and a Hamiltonian $H\colon \mathcal{N}\rightarrow \R$. Then, the Hamiltonian equation and Gibbs measure are formally given by 
\begin{equation}\label{intro:eq-Hamiltonian-eq}
\dot{\phi} = \nabla_\omega H(\phi) 
\end{equation}
and 
\begin{equation}\label{intro:eq-Gibbs-formal}
\mu = \mathcal{Z}^{-1} \exp(-H) \, \omega^n,
\end{equation}
where $\mathcal{Z}$ is a normalization constant and $\omega^n$ is the $n$-fold wedge product of $\omega$. 

\begin{gibbs}
Prove the existence of the Gibbs measure and its invariance under the Hamiltonian flow.
\end{gibbs}

In the following, we only discuss proofs of invariance, and refer to aspects regarding the existence of the Gibbs measure to \cite{ADC21,GJ87,GH21}. To fix ideas, we further restrict our discussion to the (renormalized) periodic defocusing nonlinear Schrödinger equation (NLS). In dimension $d=2$, it is given by\footnote{In dimension $d=1$, the renormalization in \eqref{intro:eq-NLS} is not necessary.}
\begin{equation}\label{intro:eq-NLS}
i \partial_t \phi + \Delta \phi = \lcol |\phi|^{p-1} \phi \rcol \qquad (t,x) \in \R \times \mathbb{T}^2. 
\end{equation}
Here, $\lcol \hspace{-0.05ex}|\phi|^{p-1} \phi  \hspace{-0.05ex} \rcol$ denotes the Wick-ordering of $ |\phi|^{p-1} \phi$. Inspired by Lebowitz-Rose-Speer~\cite{LRS88}, the seminal work of Bourgain \cite{B94} solved the Gibbs measure problem for the (NLS) in dimension $d=1$. In later work \cite{B96}, Bourgain also solved this problem for (NLS) with $(d,p)=(2,3)$, i.e., the cubic nonlinear Schrödinger equation in two dimensions. The main difficulties in \cite{B96} stem from the low regularity of the initial data. Let $\phi_0$ be a random sample from the corresponding Gibbs measure. In particular, it holds almost surely that  $\phi_0 \in H^s_x(\mathbb{T}^2)\backslash L_x^2(\mathbb{T}^2)$ for all $s<0$. Since the deterministic critical regularity of (NLS) with $(d,p)=(2,3)$ is given by $s=0$, the initial value problem with $\phi(0)=\phi_0$ cannot be solved using only deterministic arguments. In order to utilize probabilistic cancellations, Bourgain wrote the solution as 
\begin{equation}\label{intro:eq-Bourgain-trick}
\phi= e^{it\Delta} \phi_0 + \psi,
\end{equation}
where the nonlinear remainder $\psi$ solves the forced nonlinear Schrödinger equation
\begin{equation}\label{intro:eq-forced-NLS}
\begin{cases}
i \partial_t \psi + \Delta \psi = \lcol |e^{it\Delta} \phi_0 + \psi|^{2} \big( e^{it\Delta} \phi_0 + \psi \big) \rcol \qquad (t,x) \in \R \times \mathbb{T}^2, \\
\psi|_{t=0}=0.
\end{cases}
\end{equation}
In the dispersive PDE community, the decomposition \eqref{intro:eq-Bourgain-trick} is known as Bourgain's trick. In the parabolic SPDE community, a similar decomposition is often called Da Prato-Debussche trick  \cite{DD02}. Using a combination of dispersive and probabilistic estimates, Bourgain proved probabilistic nonlinear smoothing for
\eqref{intro:eq-forced-NLS}, which yields control over $\psi$ at the higher\footnote{In \cite{B96}, the contraction argument is actually performed at a regularity strictly between $0$ and $1/2$, but a minor variant yields the optimal regularity $1/2-$ for $\psi$ (cf. \cite{CLS21,DNY19}).} regularity $s=1/2-$.

The Gibbs measure problem for the two-dimensional (NLS) with higher-order nonlinearities, i.e., $p \geq 5$, was solved only recently by Deng-Nahmod-Yue~\cite{DNY19}. In this case, the deterministic threshold for local well-posedness is $s=1-2/(p-1)$, which is higher than the regularity of the nonlinear remainder in Bourgain's trick. The main contribution of \cite{DNY19} is a more detailed random expansion, which is written as
\begin{equation}\label{intro:eq-DNY}
\phi = e^{it\Delta} \phi_0 + \sum_N h^N(t) \big( e^{it\Delta} P_N \phi_0 \big) + z.
\end{equation}
Here,  $(h^N(t))_N$ is a sequence of random averaging operators, $(P_N)_N$ is the sequence of Littlewood-Paley projections, and $z$ is a nonlinear remainder. The term $h^N(t)\big( e^{it\Delta} P_N \phi_0  \big)$ incorporates certain low$\times\hdots\times$low$\times$high-interactions between the solution $\phi$ and the linear evolution of $P_N \phi_0$, which are the worst interactions in \eqref{intro:eq-NLS}. An essential feature of the argument in \cite{DNY19}, which was first observed in a different context by the first author~\cite{B18_2}, is that the random averaging operator $h^N(t)$ can be chosen as probabilistically independent from the high-frequency data $P_N \phi_0$. In comparison with Bourgain's trick \eqref{intro:eq-Bourgain-trick}, the advantage of the random expansion \eqref{intro:eq-DNY} is that the nonlinear remainder $z$ lives at regularity $1-$, which is above the deterministic threshold for local well-posedness. More recently, the random averaging operators were generalized to random tensors in \cite{DNY20}.\\ 

In dimension $d=3$, the Gibbs measure problem for the (NLS) is still open. However, the Gibbs measure problem has been solved for other three-dimensional dispersive equations in \cite{B20_2,BDNY22,DNY21,OOT20,OOT21}.  Since a more detailed discussion of this open problem is beyond the scope of this introduction, we refer the interested reader to \cite[Section 1.2.1]{DNY21}.

\subsubsection{The wave maps equation}

The wave maps equation is a prime example of a semi-linear geometric wave equation. It is the natural generalization of the linear wave equation on $(1+d)$-dimensional Minkowski space for scalar-valued fields to fields that map into a Riemannian manifold. Its nonlinearity arises from the geometric constraints imposed by the target manifold and it contains the null form $\partial_\mu \phi^i \partial^\mu \phi^j$. The latter has a favorable algebraic structure, which can be thought of as a cancellation property for the worst interactions of parallel waves.

For sufficiently regular initial data the local well-posedness of the wave maps equation in any space dimension $d \geq 1$ can be established using just energy estimates. 
At lower regularities the null structure of the wave maps nonlinearity plays a key role for the local existence theory.
The optimal sub-critical local well-posedness in $(H^s_x \times H^{s-1}_x)(\mathbb{R}^d)$, $s > \frac{d}{2}$, for $d \geq 2$ was obtained by Klainerman-Machedon~\cite{KM93, KM95, KM97} and by Klainerman-Selberg~\cite{KS97, KS02}. 
In one space dimension $d=1$ the analogous optimal sub-critical local well-posedness in $(H^s_x \times  H^{s-1}_x)(\mathbb{R}^d)$, $s > \frac12$, was established by  Machihara-Nakanishi-Tsugawa~\cite{MNT10}. The initial value problem for the wave maps equation is expected to be ill-posed in $(H^s_x \times H^{s-1}_x)(\mathbb{R}^d)$ for $s \leq \frac{d}{2}$, see Tao~\cite{Tao00} and references therein.

In one space dimension the global existence of finite energy solutions to the wave maps equation just follows from the local well-posedness in $(H^1_x \times L^2_x)(\mathbb{R})$ and from energy conservation.
In contrast, the global regularity question, i.e., the global existence of smooth solutions, becomes much more subtle in higher space dimensions $d \geq 2$, even for small data.
A key difficulty is that the wave maps nonlinearity is no longer perturbative at the critical regularity $(\dot{H}^{\frac{d}{2}}_x \times \dot{H}^{\frac{d}{2}-1}_x)(\mathbb{R}^d)$, more precisely this difficulty stems from certain low-high interactions. 
An influential idea of Tao~\cite{Tao01_1, Tao01_2} to overcome this issue is to exploit the gauge freedom of the wave maps problem, specifically the freedom in the choice of coordinates on the target manifold or rather the freedom in the choice of frames on the tangent space of the target manifold, to define a physical space gauge transform that recasts the wave maps nonlinearity into a perturbative form.
Using this insight, Tao established global regularity for wave maps into the unit sphere for smooth initial data that is small in the critical Sobolev space, first in dimensions $d \geq 5$ in~\cite{Tao01_1} and then for all dimensions $d \geq 2$ in~\cite{Tao01_2}. These works also incorporated a functional framework introduced in earlier global regularity results of Tataru~\cite{Tat98, Tat01} for smooth initial data that are small in the critical homogeneous Besov space. Other target manifolds were considered in~\cite{KR01, NSU03, SS02, K03, K04, Tat05}.

In the energy-critical case of $d=2$ space dimensions, the dynamics of wave maps with large energies is quite well-understood by now, while in the energy super-critical case $d \geq 3$ little is known about the long-time behavior for large initial data.
It turns out that the geometry of the target manifold is decisive for the long-time dynamics of wave maps with large energies in $d=2$ dimensions. Indeed, the blowup analysis of Struwe~\cite{Str03} in the equivariant case uncovered that singularity formation must be tied to the existence of a non-trivial harmonic map into the target manifold. Krieger-Schlag-Tataru~\cite{KST08}, Rodnianski-Sterbenz~\cite{RS10}, and Rapha{\"e}l-Rodnianski~\cite{RR12} later constructed examples of (equivariant) wave maps into the unit sphere that blow up in finite time by concentration of a non-trivial harmonic map. 
Finally, the threshold conjecture for energy-critical wave maps asserts that global regularity is expected for initial data with energy below the energy of any non-trivial finite energy harmonic map into the target manifold. 
This conjecture was proved independently by Krieger-Schlag~\cite{KS12} for the hyperbolic plane as the target, by Tao~\cite{Tao_large_WM} for all hyperbolic spaces, and by Sterbenz-Tataru~\cite{ST10_1, ST10_2} for all target manifolds that can be isometrically embedded into Euclidean space. 

\medskip 

\subsubsection{Singular parabolic SPDEs} \label{section:literature-parabolic}
While the subject of this paper is the wave maps equation \eqref{intro:eq-WM}, which is \emph{hyperbolic}, our methods are influenced by recent advances on singular \emph{parabolic} SPDEs. Since a complete discussion of the literature is beyond the scope of this introduction, we only provide a broad overview and start with the \emph{scalar-valued} setting. While the methods are more general, the reader may think of the parabolic $\Phi^4_3$-model, which is given by
\begin{equation}\label{intro:eq-phi43}
\partial_t \phi - \Delta \phi = - \lcol \phi^3 \rcol + \infty \cdot \phi + \xi \qquad (t,x)\in \R \times \mathbb{T}^3. 
\end{equation}
Here, the term ``$\infty\cdot \phi$" denotes a further renormalization, which goes beyond the Wick-ordering $\lcol \phi^3 \rcol$, and $\xi$ denotes space-time white noise. The local well-posedness of \eqref{intro:eq-phi43} was first proven by Hairer in his seminal work on the \emph{theory of regularity structures} \cite{H14}. Alternative approaches to singular parabolic SPDEs, such as the parabolic $\Phi^4_3$-model, are given by the para-controlled calculus of Gubinelli, Imkeller, and Perkowski \cite{GIP15}, the renormalization group approach of Kupiainen \cite{K16}, and an approach of Otto and Weber \cite{OW16}. Out of these four different approaches, the para-controlled calculus of \cite{GIP15} is closest to both the methods in this paper and the dispersive PDE literature. In fact, it served as an inspiration for the random averaging operators in \cite{DNY19}, which were previously discussed in Section \ref{section:literature-dispersive}. \\

We now leave the scalar-valued setting and consider \emph{geometric} equations. In \cite{BGHZ21,H16-2}, Bruned, Gabriel, Hairer, and Zambotti studied the geometric stochastic heat equation, which is given by 
\begin{equation}\label{intro:eq-stochastic-heat-approx}
\partial_t \phi_\varepsilon^k = \partial_x^2 \phi_\varepsilon^k + \Gamma^k_{ij}(\phi_\varepsilon) \partial_x \phi_\varepsilon^i \partial_x \phi_\varepsilon^j + h^k(\phi_\varepsilon) + \sigma^k_i(\phi_\varepsilon) \xi_\varepsilon^i + V_{\rho,\sigma}^k(\phi_\varepsilon) \qquad (t,x) \in (0,\infty)\times \mathbb{T}. 
\end{equation}
Here, $(\Gamma^k_{ij})$ are the Christoffel symbols, $h$ is a given vector field, and $(\sigma_i)$ are vector fields which are chosen depending on the Riemannian metric $g$. The stochastic forcing $\xi_\varepsilon^i$ is a smooth approximation of a space-time white noise $\xi^i$ obtained through convolution with $\rho_\varepsilon(t,x):= \varepsilon^{-3} \rho(t/\varepsilon^2,x/\varepsilon)$. Finally, $V^k_{\rho,\sigma}$ is a vector field which serves as a renormalization and can depend on the vector fields $\sigma_i$ and the convolution kernel $\rho$.  The main theorem \cite[Theorem 1.6]{BGHZ21} proves that the limit $\phi:= \lim_{\varepsilon\rightarrow 0} \phi_\varepsilon$ exists on a small time-interval. Furthermore, it proves that there exists a choice $V_{\rho,\sigma}=V^{\textup{canon}}_{\rho,\sigma}$ such that the limit $\phi$  does not depend on the choice of $\sigma$ or $\rho$. It is natural to call the corresponding limit the solution of 
\begin{equation}\label{intro:eq-stochastic-heat}
\partial_t \phi^k = \partial_x^2 \phi^k + \Gamma^k_{ij}(\phi) \partial_x \phi^i \partial_x \phi^j + h^k(\phi) + \sigma^k_i(\phi) \xi^i  \qquad (t,x) \in (0,\infty)\times \mathbb{T}. 
\end{equation}
The invariance of Brownian loops under the dynamics of \eqref{intro:eq-stochastic-heat} is still open and we refer the reader to \cite[Section 4.3]{BGHZ21} for a detailed discussion. 

In related works, Shen \cite{Shen21} and Chandra, Chevyrev, Hairer, Shen \cite{CCHS20,CCHS22} obtained similar results for the stochastic Yang mills equation.  The main result of this paper (Theorem \ref{intro:thm-rigorous}), which will be discussed momentarily, is a first step towards extending the results of  \cite{BGHZ21,CCHS20,CCHS22,Shen21} to hyperbolic equations.

\subsection{Main result and proof ideas}\label{section:intro-main}

In most earlier results on wave maps, the initial data is placed in the $L^2$-based Sobolev spaces $H_x^s \times H_x^{s-1}$. For our purposes, it is more convenient to work in the $L^\infty$-based Hölder spaces $\C_x^s \times \C_x^{s-1}$ (see Definition \ref{prep:def-spaces}). There are two reasons for this: First, the Brownian path $B$ and the white noise velocity $V$ have the same regularity in both Sobolev and Hölder spaces. Second, the $(1+1)$-dimensional linear wave equation is bounded on Hölder spaces, which is not the case for the $(1+d)$-dimensional linear wave equation in higher dimensions.

While the main focus of this paper concerns the wave maps equation with random initial data, our estimates also lead to the following theorem on deterministic well-posedness.

\begin{theorem}[Deterministic well-posedness and mild ill-posedness]\label{intro:thm-deterministic}
Let $(\M,g)$ be a compact Riemannian manifold and let $r\in \R$. 
\begin{enumerate}[label={(\roman*)},leftmargin=8mm]
    \item\label{intro:item-determininistic-well} (Well-posedness) If $r>1/2$, the wave maps equation \eqref{intro:eq-WM} is locally well-posed in $\C_x^r \times \C_x^{r-1}$. 
    \item\label{intro:item-determininistic-ill} (Mild ill-posedness) If $r\leq 1/2$, $\M = \mathbb{S}^{D-1} \subseteq \R^\dimA$, and $D\geq 2$, then the first Picard iterate of the wave maps equation is unbounded on $\C_x^r \times \C_x^{r-1}$. 
\end{enumerate}
\end{theorem}

Since the scaling-critical regularity of \eqref{intro:eq-WM} in H\"{o}lder spaces is given by $r=0$, the mild ill-posedness for $r\leq 1/2$ does not stem from the scaling symmetry. Instead, it is a result of bad high$\times$high$\rightarrow$low-interactions in the nonlinearity. 
Due to the mild ill-posedness for $r\leq 1/2$, we cannot treat the wave maps equation with our random initial data $(B,V)$ using a deterministic contraction-mapping argument.

\begin{remark}
Theorem \ref{intro:thm-deterministic} only yields a rather mild form of ill-posedness. It would therefore be interesting to also prove stronger forms of ill-posedness such as the failure of continuous dependence on the initial data or even norm inflation. Norm inflation may however be difficult to establish in $\C_x^r \times \C_x^{r-1}$ for $0<r\leq 1/2$ since the geometric constraints $\phi(t,x)\in \M$ prevent arbitrary growth in the $L_x^\infty$-norm. Since the main focus of this article is on geometric and probabilistic aspects of the well-posedness theory, we leave the proof of stronger forms of ill-posedness as an open problem (and refer the reader to \cite{K19,F20a,FO20,Oh17,ST20} for more detailed discussions of ill-posedness).
\end{remark}

Before we state the rigorous version of our main result (Theorem \ref{intro:thm-rigorous}), we give a precise definition of the random data. To this end, we first let $W\colon \R \rightarrow \R^\dimA$ be a Euclidean Brownian motion in the ambient space and we fix a reference point $B_0 \in \M$. For each $p\in \M$, we let $P(p)\colon \R^\dimA \rightarrow T_p \M$ be the orthogonal projection onto the tangent space of $\M$ at $p$. Then, the Brownian path $B\colon \R \rightarrow \M$ is defined as the solution to the Stratonovich SDE
\begin{equation}\label{intro:eq-B}
\mathrm{d}B(x)= P(B(x)) \circ \mathrm{d}W(x), \qquad B(0)=B_0\in \M.
\end{equation}
In addition to the Brownian path $B$ itself, we also define smooth approximations $(B^\varepsilon)_{\varepsilon}\colon \R \rightarrow \M$. We implicitly restrict the parameter $\varepsilon$ to dyadic numbers, but do not further reflect this in our notation. We then first define the smooth approximations $(W^\varepsilon)_\varepsilon \colon \R \rightarrow \R^\dimA$ of the Euclidean Brownian motion $W$ by 
\begin{equation*}
W^\varepsilon := P_{\leq \varepsilon^{-1}} W,
\end{equation*}
where $P_{\leq \varepsilon^{-1}}$ is the Littlewood-Paley projection from Definition \ref{prep:def-LittlewoodPaley} below. Then, we define the smooth path $B^\varepsilon \colon \R \rightarrow \M$ as the solution to the classical ODE
\begin{equation}\label{intro:eq-Beps}
\partial_x B^\varepsilon(x) = P(B^\varepsilon(x)) \partial_x W^\varepsilon(x), \qquad B^\varepsilon(0)=B_0 \in \M. 
\end{equation}
In Corollary~\ref{prep:cor_Cloc_convergence_to_BV} below, it is shown that the smooth paths $(B^\varepsilon)_\varepsilon$ converge to the Brownian path $B$ in $\C_{\loc}^{s}$ for $s<1/2$. It remains to define the white noise velocity $V$ and its smooth approximations $(V^\varepsilon)_\varepsilon$. To this end, we let $\Wb\colon \R \rightarrow \R^\dimA$ be an independent copy of $W$ and define $\Wb^\varepsilon := P_{\leq \varepsilon^{-1}} \Wb$. Then, we explicitly define
\begin{equation}\label{intro:eq-Veps}
V^\varepsilon(x) := P(B^\varepsilon(x)) \partial_x \Wb^\varepsilon(x). 
\end{equation}
Due to the projection $P(B^\varepsilon(x))$, it holds that $V^\varepsilon \in (B^\varepsilon)^\ast T\M$, i.e., $V^\varepsilon(x)\in T_{B^\varepsilon(x)} \M$ for all $x\in \R$. In Corollary~\ref{prep:cor_Cloc_convergence_to_BV} below, it is shown that $(V^\varepsilon)_\varepsilon$ converges in $\C_{\loc}^{s-1}$ for $s<1/2$. Therefore, we can define the white noise velocity $V$ as 
\begin{equation}\label{intro:eq-V}
V := \lim_{\varepsilon \rightarrow 0} V_\epsilon. 
\end{equation}
Equipped with the Brownian path $B$, the white noise velocity $V$, and their smooth approximations, we can now state our main result.

\begin{theorem}[Probabilistic local well-posedness]\label{intro:thm-rigorous}
Let $B\colon \R \rightarrow \M$ be the Brownian path,  let $V\in B^\ast T\M$ be the white noise velocity, and let  $(B^\varepsilon)_{\varepsilon>0}$ and $(V^\varepsilon)_{\varepsilon>0}$ be their smooth approximations.
Then, for all $\tau>0$ and $R\geq 1$, there exists an event $\mathcal{E}(\tau,R)$ such that the following two properties hold:
\begin{enumerate}[label={(\roman*)},leftmargin=8mm]
    \item (``High''-probability) We have that  
    \begin{equation*}
        \mathbb{P}\big(  \mathcal{E}(\tau,R) \big) \geq 1 - C R \exp\big(-c \tau^{-c}\big),
    \end{equation*}
where $C=C(\M)\geq 1$ and $c>0$ are constants. 
    \item (LWP) On the event $\mathcal{E}(\tau,R)$, the smooth global solutions $\phi^{\varepsilon}$ of \eqref{intro:eq-WM} with initial data $\phi^{\varepsilon}[0]=(B^{\varepsilon},V^\varepsilon)$ converge in $(\C_t^0 \C_x^s \medcap \C_t^1 \C_x^{s-1})([-\tau,\tau]\times [-R,R] \rightarrow \M)$.
\end{enumerate}
\end{theorem}

We note that the local well-posedness statement in Theorem \ref{intro:thm-rigorous} is slightly non-standard since, in order for $\mathbb{P}\big(  \mathcal{E}(\tau,R) \big)$ to be close to one, $\tau$ needs to be small depending on $R$. Put differently, Theorem \ref{intro:thm-rigorous} does not yield local well-posedness in time, but rather yields local well-posedness in space-time. Since the random data $(B,V)$ does not decay in space and the wave maps equation exhibits finite speed of propagation, this version of local well-posedness is natural. \\

We now describe the main ideas in the proof of Theorem \ref{intro:thm-rigorous}. Throughout this informal discussion, we denote the regularity parameter for the Brownian path by $s<1/2$. Furthermore, we formally set $\varepsilon=0$ and omit smooth cut-off functions from our notation. The first step, which was already used in \cite{KT98,MNT10,Tao00}, is to switch from Cartesian to null-coordinates. We define
\begin{equation*}
u:= x-t \qquad \text{and} \qquad v:=x+t. 
\end{equation*}
Due to d'Alembert's formula, the linear evolution of the initial data $(B,V)$ is given by\footnote{For technical reasons, we will later write the linear evolution as $\phi_{\textup{lin}}=\theta (\phi^+(u)+\phi^{-}(v))$, where $\theta>0$ is a small parameter. Due to this, we will need to adjust the definitions of $\phi^+$ and $\phi^-$, see e.g. \eqref{ansatz:eq-shifted-linear}. In the introduction, we ignore this technicality.}
\begin{equation}\label{intro:eq-lin}
\phi_{\textup{lin}}(u,v) = \phi^+(u)+\phi^-(v), 
\end{equation}
where
\begin{equation}
    \phi^\pm(x) = \frac{1}{2} \Big( B(x) \mp \int_0^x \mathrm{d}y \, V(y) \Big). 
\end{equation}
We refer to $\phi^+$ and $\phi^-$ as the right and left-moving linear waves, respectively. The wave maps equation in null coordinates is given by
\begin{equation}\label{intro:eq-WM-NC}
\begin{cases}
\partial_u \partial_v \phi^k = - \Second^k_{ij}(\phi) \partial_u \phi^i \partial_v \phi^j \\
\phi|_{u=v} = B, \, (\partial_v - \partial_u ) \phi|_{u=v} = V. 
\end{cases}
\end{equation}

The most difficult aspects of our argument are linked to the absence of nonlinear smoothing for \eqref{intro:eq-WM-NC}. In particular, we cannot rely on Bourgain's trick. Instead, we require a more delicate Ansatz, which is related to but different from the random averaging operators in \eqref{intro:eq-DNY}. To motivate this Ansatz, we heuristically discuss low$\times$high and high$\times$high$\rightarrow$low-interactions.

\subsubsection{\protect{Low$\times$high-interactions}}\label{section:intro-lh} 

Our treatment of the low$\times$high-interactions is motivated by the gauge transform of Tao \cite{Tao01_1,Tao01_2}, which was already mentioned in Section \ref{section:related}. While we will also be multiplying the (linear) evolution with a low-frequency modulation, our construction of the modulation differs significantly from \cite{Tao01_1,Tao01_2}.

We first recall from \eqref{intro:eq-lin} that the linear evolution is given by  $\phi_{\textup{lin}}^i(u,v)=\phi^{+,i}(u)+\phi^{-,i}(v)$. We now focus on the right-moving component $\phi^{+,i}(u)$ and restrict to frequencies of size $\sim M$. The corresponding portion of the linear evolution is given by $\phi^{+,i}_M(u):= (P_M \phi^{+,i})(u)$. More generally, let us replace $\phi^{+,i}_M(u)$ by the modulated right-moving wave $A^{+,i}_{M,m}(u,v) \phi^{+,m}_M(u)$. For the moment, we assume that the modulation $A^{+,i}_{M,m}$ is supported on frequencies of size $\sim 1$ in both variables. Inserting this component into the nonlinearity yields
\begin{equation*}
    -\Second^{k}_{ij}(\phi(u,v))\partial_u \big( A_{M,m}^{+,i}(u,v)\phi^{+,m}_M(u) \big) \, \partial_v \phi^{j}(u,v).
\end{equation*}
Since we are currently discussing low$\times$high-interactions, we turn to the sub-term 
\begin{equation}\label{intro:eq-low-high}
- P_{\lesssim 1}^{u,v}\big( \Second^{k}_{ij}(\phi) \partial_v \phi^{j}\big)(u,v) \, \partial_u \big( A_{M,m}^{+,i}(u,v) \phi^{+,m}_M(u) \big),
\end{equation}
where the product $\Second^{k}_{ij}(\phi) \partial_v \phi^{j}$ enters at low frequencies in both variables. Eventually, we will solve the wave maps equation \eqref{intro:eq-WM-NC} using its Duhamel integral formulation. Using explicit calculations (see Proposition \ref{prep:prop-duhamel}), the Duhamel integral of \eqref{intro:eq-low-high} is of the form
\begin{equation}\label{intro:eq-low-high-Duh}
\begin{aligned}
&\Duh\Big[ - P_{\lesssim 1}^{u,v}\big( \Second^{k}_{ij}(\phi) \partial_v \phi^{j}\big)(u,v) \partial_u \big( A_{M,m}^{+,i}(u,v) \phi^{+,m}_M(u) \big) \Big] \\
=& - \Big( \int_u^v \dv^\prime  P_{\lesssim 1}^{u,v}\big( \Second^{k}_{ij}(\phi) \partial_v \phi^{j}\big)(u,v^\prime) A_{M,m}^{+,i}(u,v^\prime) \Big) \phi^{+,m}_M(u) 
+ \{ \textup{error terms}\}. 
\end{aligned}
\end{equation}
The leading term in \eqref{intro:eq-low-high-Duh} has the same regularity as $\phi^+_M(u)$,  which witnesses the absence of nonlinear smoothing. In particular, \eqref{intro:eq-low-high-Duh} cannot be absorbed into a smoother nonlinear remainder. However, the leading term in \eqref{intro:eq-low-high-Duh} is exactly of the same form as our starting point $A^{+,i}_{M,m}(u,v) \phi^{+,m}_M(u)$. Therefore, we can hope to absorb it into the $k$-th component $A^{+,k}_{M,m}(u,v) \phi^{+,m}_M(u)$. By pursuing this idea, one is quickly lead to the ordinary differential equation (ODE)  given by 
\begin{equation}\label{intro:eq-ODE}
\partial_v A^{+,k}_{M,m}(u,v) = -  P_{\lesssim 1}^{u,v}\big( \Second^{k}_{ij}(\phi) \partial_v \phi^{j}\big)(u,v) A_{M,m}^{+,i}(u,v), \qquad A^{+,k}_{M,m}(u,u)= \delta^k_m.
\end{equation}
The initial value $A^{+,k}_{M,m}(u,u) = \delta^k_m$ in \eqref{intro:eq-ODE} is due to the linear evolution, which has to be contained in $A^{+,k}_{M,m}(u,v) \phi^{+,m}_M(u)$. Provided that the pre-factor $\Second^{k}_{ij}(\phi) \partial_v \phi^{\,j}$ is well-defined, the  ODE \eqref{intro:eq-ODE} can easily be solved using the Picard-Lindelöf theorem. Unfortunately, \eqref{intro:eq-low-high} is an oversimplification of all low$\times$high-interactions. As a result, \eqref{intro:eq-ODE} does not present itself as a suitable model for the modulation equations (see Section \ref{section:modulation}). This is due to a further culprit hidden in $\Second^{k}_{ij}(\phi) \partial_v \phi^j$, which is given by
\begin{equation}\label{intro:eq-culprit}
P^{u,v}_{\lesssim 1} \big(\Second^{k}_{ij}(\phi) \big)(u,v) \, \partial_v \phi^{-,j}_N(v).
\end{equation}
Here, $\phi^-$ is the left-moving component of the initial data and we take $1\leq N \leq M$.\footnote{When $N>M$, the roles of $\phi^+$ and $\phi^-$ should be reversed.}
Restricting our attention only to the contribution of \eqref{intro:eq-culprit}, we are led to the ODE
\begin{equation}\label{intro:eq-rough-ODE}
\partial_v A^{+,k}_{M,m}(u,v) = -P^{u,v}_{\lesssim 1} \big(\Second^{k}_{ij}(\phi) \big)(u,v) A^{+,i}_{M,m}(u,v)\,  \partial_v \phi^{-,j}_N(v).
\end{equation}
It is evident from \eqref{intro:eq-rough-ODE} that we should no longer view the modulation $A^{+,k}_{M,m}$ as being supported on frequencies of size $\sim 1$ in the $v$-variable. Unfortunately, \eqref{intro:eq-rough-ODE} cannot be solved using classical theory for ODEs. The reason is that  $\phi^-$ only has regularity $s$, which suggests that $v\mapsto A^+_M(u,v)$ also only has regularity $s$. Since $s<1/2$, the regularity information is insufficient to even define the right-hand side in \eqref{intro:eq-rough-ODE}.  Instead of classical ODE methods, we utilize the para-controlled approach to rough ODEs by Gubinelli, Imkeller, and Perkowski \cite{GIP15}. Due to certain high$\times$high$\rightarrow$low-interactions, which will be described in Section \ref{section:intro-hh}, it is convenient to separate the cases $1\leq N \leq M^{1-\delta}$ and $M^{1-\delta}<N\leq M$. After reversing the roles of the $u$ and $v$-variables, we are lead to the following four terms in our Ansatz:
\begin{enumerate}
    \item The modulated right-moving wave $A_{M,m}^{+,i}(u,v) \phi^{+,m}_M(u)$.
    \item The modulated left-moving wave $A_{N,n}^{-,i}(u,v) \phi^{-,n}_N(v)$. 
    \item The bilinear term $B^i_{M,N,mn}(u,v) \phi^{+,m}_M(u) \phi^{-,n}_N(v)$, which contains left and right-moving components and only occurs when $\min(M,N)\geq \max(M,N)^{1-\delta}$.
    \item The nonlinear remainder $\psi^i(u,v)$, which lives at a higher regularity. 
\end{enumerate}

\subsubsection{\protect{High$\times$high$\rightarrow$low-interactions}}\label{section:intro-hh} 

As stated in Theorem \ref{intro:thm-deterministic}, the one-dimensional wave maps equation is (deterministically) ill-posed in $\C^r$ for all $r\leq 1/2$, which is a result of high$\times$high$\rightarrow$low-interactions. This is in sharp contrast to the deterministic theory for wave maps in high dimensions \cite{NSU03,SS02,Tao01_1}, where high$\times$high$\rightarrow$low-interactions are relatively harmless. In order to go beyond the deterministic theory, we need to rely on probabilistic cancellations for the Brownian path $B$ and the white noise velocity $V$. The main ingredient is the high$\times$high$\rightarrow$low-estimate 
\begin{equation}\label{intro:eq-hhl-BP}
\mathbb{E}\bigg[  \max_{\pm_1, \pm_2} \max_{1\leq i,j \leq \dimA} \sup_{M\sim N} M^s
\Big\| P_M \phi^{\pm_1,i}(x) \partial_x P_N \phi^{\pm_2,j}(x) \Big\|_{\C^{s-1}} \bigg] \lesssim 1,
\end{equation}
where $\pm_1,\pm_2 \in \{ +, - \}$.
We emphasize that since $s+(s-1)<0$, the left-hand side in \eqref{intro:eq-hhl-BP} cannot be bounded using only that $\phi^\pm \in \C^s$. In our analysis, we then encounter two different forms of frequency-resonances. \\

\emph{Frequency-resonances involving only the linear waves $\phi^\pm$:} These terms cannot even be defined using only deterministic estimates, but are relatively harmless once the probabilistic ingredient \eqref{intro:eq-hhl-BP} is taken into account. To illustrate this, let $K\sim M\sim N$ be comparable frequency scales and consider the cubic term
\begin{equation}\label{intro:eq-initial-resonance-1}
P_{\lesssim 1}^u \big(  \phi_K^{+,k}(u) \partial_u\phi^{+,m}_M(u) \big) \partial_v \phi^{-,n}_N(v).
\end{equation}
The term \eqref{intro:eq-initial-resonance-1} naturally occurs after a para-linearization of the second fundamental form $\Second(\phi)$ in  \eqref{intro:eq-WM-NC}. It has $u$-frequency $\sim 1$, $v$-frequency $\sim N$, and, due to \eqref{intro:eq-hhl-BP}, the amplitude 
\begin{equation}\label{intro:eq-initial-resonance-2}
\big\| P_{\lesssim 1}^u \big(  \phi_K^{+,k}(u) \partial_u \phi^{+,m}_M(u) \big) \partial_v \phi^{-,n}_N(v) \big\|_{L^\infty_{u,v}} \lesssim M^{-s} N^{1-s} \sim N^{1-2s}. 
\end{equation}
As our analysis will show (see Lemma \ref{bilinear:lemma-cubic} and Proposition \ref{bilinear:prop-paraless}), \eqref{intro:eq-initial-resonance-2} is sufficient to treat the cubic term \eqref{intro:eq-initial-resonance-1} as a smooth remainder. ~\\

\emph{Frequency-resonances involving the linear $\phi^\pm$ and the remainder $\psi$:} These terms can be defined using only deterministic ingredients. Despite being well-defined, however, the resulting estimates of the resonant term are worse than \eqref{intro:eq-initial-resonance-2}. 

Let $r\in (1/2,1)$ denote the regularity of the smooth remainder $\psi\in \Cprod{r}{r}$ (see Definition \ref{prep:def-spaces} below). Similar as in \eqref{intro:eq-initial-resonance-1}, we let $K\sim M\sim N$ be frequency scales and consider the resonant term
\begin{equation}\label{intro:eq-remainder-resonance-1}
P_{\lesssim 1}^u \big(  P_K^u \psi^k(u,v)\, \partial_u \phi^{+,m}_M(u) \big) \partial_v \phi^{-,n}_N(v). 
\end{equation}
In comparison with \eqref{intro:eq-initial-resonance-1}, we only replaced $\phi^{+,k}_K(u)$ by $P_K^u \psi^k(u,v)$. We make no restrictions on the $v$-frequencies of $\psi^k(u,v)$, and the worst case corresponds to $v$-frequencies of size $\sim 1$. Since $\psi \in \Cprod{r}{r}$ is arbitrary, we can only use direct (deterministic) estimates of \eqref{intro:eq-remainder-resonance-1}, which yield 
\begin{equation}\label{intro:eq-remainder-resonance-2}
    \big\| P_{\lesssim 1}^u \big( P_K^u \psi^k(u,v)\, \partial_u  \phi^{+,m}_M(u) \big) \partial_v \phi^{-,n}_N(v) \big\|_{L^\infty_{u,v}} \lesssim K^{-r} M^{1-s} N^{1-s} \sim N^{1-2s+1-r}. 
\end{equation}
Since $r<1$, the estimate is worse than \eqref{intro:eq-initial-resonance-2}. More importantly, it only holds that 
\begin{equation}\label{intro:eq-remainder-resonance-3}
\begin{aligned}
&\big\| P_{\lesssim 1}^u \big(  P_K^u \psi^k(u,v)\,  \partial_u \phi^{+,m}_M(u) \big) \partial_v \phi^{-,n}_N(v) \big\|_{\Cprod{r-1}{r-1}} \\
\lesssim&
 1^{r-1} N^{r-1} \big\| P_{\lesssim 1}^u \big(  P_K^u \psi^k(u,v)\,  \partial_u \phi^{+,m}_M(u) \big) \partial_v \phi^{-,n}_N(v) \big\|_{L^\infty_{u,v}} \\
 \lesssim& N^{1-2s}.
 \end{aligned}
\end{equation}
Since $s<1/2$, the right-hand side of \eqref{intro:eq-remainder-resonance-3} is unbounded in $N$. As a result, \eqref{intro:eq-remainder-resonance-1} cannot be absorbed back into the nonlinear remainder $\psi$.  Instead, \eqref{intro:eq-remainder-resonance-1} forms an additional contribution to the modulated left-moving wave $A^{-,i}_{N,n}(u,v) \phi^{-,n}_N(v)$. Unfortunately, the high$\times$high$\rightarrow$low-interactions between $\phi^\pm$ and $\psi$ do not only occur in \eqref{intro:eq-remainder-resonance-1}, but are part of many different terms in $\Second^k_{ij}(\phi) \partial_u \phi^i \partial_v \phi^j$. 

We remark that \eqref{intro:eq-remainder-resonance-2} can be improved if the condition $N\sim M$ is replaced by $N\leq M^{1-\delta}$. This is the reason for introducing the bilinear term $B^i_{M,N,mn}(u,v) \phi^{+,m}_M(u) \phi^{-,n}_N(v)$, which isolates the problematic case $N>M^{1-\delta}$.  

\subsubsection{Further remarks}
We now make further remarks and a few comparisons of this article with related works. 
\begin{enumerate}[label={(\roman*)},leftmargin=12mm]
     \item While the $(1+1)$-dimensional wave maps equation is completely integrable (see e.g. \cite{KT98,P76,TU04}), our argument does not rely on complete integrability. In fact, our well-posedness theory also applies to the more general system
    \begin{equation*}
     \partial_\mu \partial^\mu \phi^k = - \Second^k_{ij}(\phi) \partial_\mu \phi^i \partial^\mu \phi^j + F^k(\phi), 
    \end{equation*}
    where $F$ is (the extension to $\mathbb{R}^\dimA$ of) a vector field on $\M$.
    \item The approximations $(B^\epsilon)_{\epsilon>0}$ and $(V^\epsilon)_{\epsilon>0}$ from \eqref{intro:eq-Beps} and \eqref{intro:eq-Veps} and the corresponding wave maps $(\phi^\epsilon)_{\epsilon>0}$ depend on the convolution kernel $\rho$, which is used in the definition of our Littlewood-Paley operators (see Definition \ref{prep:def-LittlewoodPaley}). However, we expect that their limits $B$, $V$, and $\phi$ do not depend on the precise choice of $\rho$. That is, we expect that any convolution kernel $\widebar{\rho}$ satisfying the conditions in Definition \ref{prep:def-LittlewoodPaley} leads to the same limit. This should follow from variants of Lemma \ref{prep:lem-Weps_hhtolow}.\ref{prep:item-Weps_hhtolow-convergence}, Proposition \ref{prep:prop-brownian-path-approximation}.\ref{prep:item-Beps_hhtolow-convergence}, and Proposition \ref{prep:prop-velocity-approximation}.\ref{prep:item-velocity_hhtolow-convergence}, whose proof should allow us to control the differences between stochastic objects based on $\rho$ and $\widebar{\rho}$.
    \item In this article, we view the compact Riemannian manifold $\M$ as an isometrically embedded submanifold of a Euclidean space $\R^{\dimA}$. But neither the isometric embedding nor the Euclidean space $\R^{\dimA}$ are unique, and it is an interesting problem to show that the (law of the) limit $\phi$ does not depend on them.
    \item The modulated linear wave $A_{N,n}^{-,k}(u,v) \phi_N^{-,n}(v)$ shares similarities with the adapted linear evolutions in \cite{B18_2} and the random averaging operators in \cite{DNY19}, see Subsection~\ref{section:literature-dispersive}. However, the modulation $A_N^-$ and linear wave $\phi^-_N$ are probabilistically dependent. There are two reasons for this: First, the geometric constraints on the Brownian path $B$ and white noise velocity $V$ create probabilistic dependencies between low and high frequencies. Second, the high$\times$high$\rightarrow$low-interaction in \eqref{intro:eq-remainder-resonance-2}, which enters into the modulation $A^{-}_N$, depends on high frequency terms in the initial data. To compensate for the lack of probabilistic independence, however, the linear map $\phi^-_N\mapsto A_{N,n}^{-,k}(u,v) \phi_N^{-,n}(u)$ consists of a simple multiplication, which is easier to handle than the random averaging operators in \eqref{intro:eq-DNY}.
    
    \item In \cite{KLS20}, two of the authors in joint work with J. Krieger obtained a probabilistic small data global existence result for the energy-critical Maxwell-Klein-Gordon equation relative to the Coulomb gauge. It is the first global existence result for a geometric wave equation for random initial data at super-critical regularity. Similarly to \cite{B18_2}, the proof is based upon an induction on frequency procedure and an adapted linear-nonlinear decomposition. The latter relies on a delicate global parametrix construction. It is worth noting that already in the deterministic study of the global well-posedness of the energy-critical Maxwell-Klein-Gordon equation, certain low-high interactions are non-perturbative at critical regularity and have to be incorporated into the linear operator.
    In comparison to \cite{KLS20}, a novel feature of this work is that the random data arises naturally from geometric considerations. 
    
    \item In \cite{BR20}, Brze\'{z}niak and Rana consider a stochastic wave maps equation. In null-coordinates on Minkowski space $\R^{1+1}$ and intrinsic coordinates on the Riemannian manifold $\M$, it is given by 
    \begin{equation}\label{intro:eq-stochastic-wm}
    \partial_u \partial_v \phi^k = - \Gamma^k_{ij}(\phi) \partial_u \phi^i \partial_v \phi^j + \frac{1}{4} \sigma^k_i(\phi) \partial_u \partial_v \Xi^i. 
    \end{equation}
    Here, $\Gamma^k_{ij}\colon M \rightarrow \R$ are the Christoffel symbols, $\sigma^k_i\colon \M \rightarrow \R$ are smooth functions, and the stochastic terms $\Xi^i \colon \R^{1+1}_{u,v}\rightarrow \R$ are given by fractional Brownian sheets with Hurst indices $3/4<H_1,H_2<1$. The main result \cite[Theorem 4.3]{BR20}, which builds on \cite{MNT10}, yields the local well-posedness of the stochastic wave maps equation \eqref{intro:eq-stochastic-wm}.
    It would be interesting to see if the methods in this paper could be used to extend their result to Hurst indices $s<H_1,H_2<1$, where $s=1/2-$.
    \item In \cite{BJ22}, which appeared after the preprint of this manuscript, Brze{\'z}niak and Jendrej studied lattice approximations of wave maps into spheres with Brownian initial data. The main result of \cite{BJ22} shows that, as the lattice spacing tends to zero, a subsequence of the discretized solutions converges in law and that the law of the limit is invariant under time-like translations. The main differences between \cite[Theorem 1]{BJ22} and Theorem \ref{intro:thm-rigorous} are that \cite[Theorem 1]{BJ22} is global (rather than local) in time, but Theorem \ref{intro:thm-rigorous} concerns strong (rather than weak) solutions.
    \item As briefly mentioned in Subsection \ref{section:literature-parabolic}, Theorem \ref{intro:thm-rigorous} is a first step towards extending the results in \cite{BGHZ21} from the parabolic to the hyperbolic setting. 
    The most intriguing difference between \cite{BGHZ21} and our work is the use of renormalization techniques, which are not needed in the proof of Theorem \ref{intro:thm-rigorous}. Such techniques, however, may be very relevant in the open problem described below, which requires a more detailed understanding of finite-dimensional approximations of \eqref{intro:eq-WM}. 
\end{enumerate}

\subsection{Open problem: Invariant Gibbs measure}\label{section:open}

Our original motivation to study the $(1+1)$-dimensional wave maps equation was the following problem, which remains unsolved.

\begin{ggibbs}
Prove the existence and invariance of the Gibbs measure for the wave maps equation \eqref{intro:eq-WM}.
\end{ggibbs}

As discussed in Section \ref{section:motivation}, the Gibbs measure for the wave maps equation has not yet been constructed. Aside from the construction of the measure, we encountered another (substantial) difficulty in our attempt to solve this problem. Previous proofs of invariance (see e.g. \cite{B94,B96}) all rely on finite-dimensional approximations of the full system. In the derivation of the finite-dimensional approximation, however, one has to carefully preserve the important properties of the full system. In the case of the wave maps equation, the finite-dimensional approximation should still exhibit the following two structures: 
\begin{enumerate}[label={(\roman*)},leftmargin=12mm]
\item The null structure.
\item The Hamiltonian structure.
\end{enumerate}
It is easy to derive finite-dimensional approximations which preserve either of these two structures, but this seems to be insufficient to solve the geometric Gibbs measure problem. So far, we were unable to derive a finite-dimensional truncation which preserves both structures simultaneously. For related discussions of finite-dimensional approximations of random dispersive equations and singular stochastic PDEs, we refer the reader to \cite{DTV15,NORBS} and \cite{CM18,EH19,FH17,HM12,HM18}, respectively. \\

\textbf{Acknowledgements:} The authors thank Rishabh Gvalani, Florian Kunich, Stephan Luckhaus, Felix Otto, Igor Rodnianski, Angela Stevens, Terence Tao, and Markus Tempelmayr for helpful and interesting discussions. The authors also thank the anonymous referees for valuable comments and suggestions.
B.B. thanks the MPI for Mathematics in the Sciences for support during a visit in the summer of 2021. The three authors thank ICERM, which is supported by NSF grant DMS-1929284,  for support during the semester program on  Hamiltonian Methods in Dispersive and Wave Evolution Equations. 
B.B. was partially supported by the NSF under Grant No. DMS-1926686. 
J.L. was partially supported by NSF grant DMS-1954707. 
G.S. was partially supported by DMS-1764403, DMS-2052651 and the Simons Foundation. 

% \textbf{Data availability:} This article does not have any external supporting data. \\

% \textbf{Competing interests:} The authors do not have any competing interests to declare.

%%%%%%%%%%%%%%% Preparations %%%%%%%%%%%%%%%%%%%%%%%%%%%%%%%%%
\section{Preparations}\label{section:preparations}

\subsection{Notation and parameters} \label{section:parameters}
Let $d\geq 1$ and let $f\in \mathcal{S}(\R^d)$ be a Schwartz function. We define the Fourier and inverse Fourier transform of $f$ by 
\begin{equation*}
\widehat{f}(\xi) = (2\pi)^{-d/2} \int_{\R^d} \mathrm{d}x \,  e^{-ix\xi} f(x)
\quad \text{and} \quad 
\widecheck{f}(x)  = (2\pi)^{-d/2} \int_{\R^d} \mathrm{d}\xi \,  e^{+ix\xi} f(\xi).
\end{equation*}
In the following, we often use dyadic decompositions of frequency space and we refer to the corresponding dyadic scales simply as frequency scales. The frequency scales will be denoted exclusively by capital letters such as $K$, $L$, $M$, and $N$. For two frequency scales $M$ and $N$, we define
\begin{equation}\label{prep:eq-frequency-scales-1}
\begin{aligned}
M\ll N \quad &:\Longleftrightarrow \quad M\leq 2^{-10} N,  \\
M\lesssim N \quad &:\Longleftrightarrow \quad M\leq 2^{10} N,  \\
M\sim N \quad &:\Longleftrightarrow \quad 2^{-10} N < M < 2^{10} N. 
\end{aligned}
\end{equation}
With a slight abuse of notation, we also use ``$\ll$" and ``$\lesssim$" for quantities other than frequency scales, but where the exact definition differs slightly from \eqref{prep:eq-frequency-scales-1}. Let $C$ and $c$ be sufficiently large and small absolute constants, respectively. For any $A,B>0$, we write $A\ll B$ if $A\leq c B$ and  $A\lesssim B$ if $A\leq C B$.  In the following, the precise meaning of ``$\ll$" and ``$\lesssim$" will always be clear from the context. \\

We now turn to the parameters used in our analysis.
The main parameters are given by $r,s,\delta\in \R$, which satisfy 
\begin{equation}\label{prep:eq-parameter-1}
0 < \frac{1}{2} -s \ll \delta \ll \frac{3}{4}-r \ll 1. 
\end{equation}
The parameter $s$ represents the (low) regularity of the initial data and the parameter $r$ represents the (high) regularity of the smoother remainder. The parameter $\delta$ is used in our definition of a modified low-high para-product. For notational convenience, 
we also choose
\begin{equation}\label{prep:eq-parameter-sigma}
\sigma := 100 \delta. 
\end{equation}
In addition, we choose a (less significant) parameter $\eta$ satisfying
\begin{equation}\label{prep:eq-parameter-2}
0< \frac{1}{2}- s  \ll \eta \ll \delta. 
\end{equation} The parameter $\eta$ will mostly be used to ensure the summability over dyadic scales. For notational convenience, we also define 
\begin{equation}\label{prep:eq-parameter-3}
r^\prime=r-\eta.
\end{equation}
During the first reading of the paper, we encourage the reader to mentally replace $s\rightarrow 1/2$, $\eta\rightarrow 0$, $\delta\rightarrow 0$, and $r^\prime,r \rightarrow 3/4$. 

Finally, we let $\theta>0$ be a small parameter which is allowed to depend on $s$, $\eta$, $\delta$, and $r$. It will be used as an upper bound on the size of the re-scaled initial data and therefore plays a different role than the other parameters, which are linked to frequency-scale restrictions and/or regularities. \\

Finally, we define modifications of \eqref{prep:eq-frequency-scales-1}, which involve the parameter $\delta>0$. For all frequency scales $M$ and $N$, we define
\begin{equation}\label{prep:eq-frequency-scales-2}
\begin{aligned}
M\ll_\delta N \quad &:\Longleftrightarrow \quad M\leq N^{1-\delta},  \\
M\lesssim_\delta N \quad &:\Longleftrightarrow \quad M\leq N^{1/(1-\delta)},  \\
M\sim_\delta N \quad &:\Longleftrightarrow \quad N^{1-\delta} < M < N^{1/(1-\delta)}. \\
\end{aligned}
\end{equation}

In the extrinsic formulation of the wave maps equation, we previously encountered the second fundamental form $\Second\colon T\M \times T\M \rightarrow N\M$. 
Our assumption that $\M$ is a smooth, compact Riemannian manifold without boundary guarantees the uniform boundedness of the second fundamental form $\Second$ and of all of its derivatives in the embedding $\M \hookrightarrow \R^\dimA$.
In order to work in the ambient space $\R^\dimA$, we require an extension  $\widetilde{\Second}$ of $\Second$. The extension $\widetilde{\Second}$ is determined by its components $\widetilde{\Second}^k_{ij}\colon \R^\dimA \rightarrow \R$, which can be chosen to satisfy the following properties:
\begin{enumerate}[leftmargin=7mm]
 \item For all $\phi \in \M$ and $V,W \in T_\phi \M$, it holds that $\Second^k(\phi)(V,W)= \widetilde{\Second}^k_{ij}(\phi) V^i W^j$. 
    \item For all $1\leq i,j,k \leq \dimA$, it holds that $\widetilde{\Second}^k_{ij}\in \C^\infty_c(\R^\dimA\rightarrow \R)$ and  $\widetilde{\Second}^k_{ij}=\widetilde{\Second}^k_{ji}$. 
\end{enumerate}
Such an extension $\widetilde{\Second}$ of the second fundamental form can be realized as the Hessian of a smooth and compactly supported extension of the nearest point projection map, which under our assumptions is well-defined in a tubular neighborhood of $\M$, see for instance \cite[Section 2.12.3]{Simon96}.
With a slight abuse of notation, we now identify $\Second$ with $\widetilde{\Second}$ and simply write $\Second$ for both the original second fundamental form and its extension. \\

Throughout this article, we fix a bump function $\chi \in \C^\infty_c(\R\rightarrow [0,1])$ which satisfies
\begin{equation}\label{prep:eq-chi}
\chi(x)=1  \text{ for all } x \in [-2,2] \quad \text{and} \quad \chi(x)=0 \text{ for all } x \not \in [-21/10,21/10].
\end{equation}
Furthermore, we define two functions $\chi^+,\chi^-\colon \R_{u,v}^{1+1}\rightarrow [0,1]$ by 
\begin{equation}\label{prep:eq-chi-pm}
\chi^+(u,v):= \chi(u) \qquad \text{and} \qquad \chi^-(u,v):= \chi(v). 
\end{equation}

\subsection{Function spaces, para-products and product estimates}\label{section:function-product}

In this subsection, we define the Bourgain-type space $\Cprod{\gamma_1}{\gamma_2}$. Despite its simplicity, the space $\Cprod{\gamma_1}{\gamma_2}$ is essential for making use of the null-structure in the wave maps equation. After this definition, we turn to para-products and their estimates in the $\Cprod{\gamma_1}{\gamma_2}$-spaces. 

We let $\C_b^\infty(\R)$ be the space of bounded smooth functions $f\colon \R\rightarrow \R$ with bounded derivatives, i.e., satisfying
\begin{equation*}
\| \partial_x^k f(x) \|_{L^\infty(\R)} \leq C(k,f)<\infty 
\end{equation*}
for all $k\geq 0$. Furthermore, we let $\C_c^\infty(\R)$ be the subspace of $\C_b^\infty(\R)$ consisting of compactly supported smooth functions. We now recall the definition of Littlewood-Paley operators. 

\begin{definition}[Littlewood-Paley operators]\label{prep:def-LittlewoodPaley}
Let $\rho\colon \R \rightarrow \R$ be a smooth even cut-off function satisfying $\rho(\xi)=1$ for all $\xi\in [-7/8,7/8]$ and $\rho(\xi)=0$ for all $\xi\not \in [-9/8,9/8]$. We then define 
\begin{equation}
\rho_1(\xi) := \rho(\xi) \qquad \text{and} \qquad \rho_N(\xi) := \rho(\xi/N) - \rho(2\xi/N) \quad \text{for all } N\geq 2. 
\end{equation}
For any function $f\in \C_b^\infty(\R)$, the Littlewood-Paley projections  $\{ P_N f\}_{N\geq 1}$ are defined as 
\begin{equation}
P_N f  (x) := (\widecheck{\rho}_N \ast f)(x),
\end{equation}
where $\widecheck{\rho}_N$ is the inverse Fourier transform of $\rho_N$. 
In addition, we define the fattened Littlewood-Paley operators $\widetilde{P}_N$ by 
\begin{equation}\label{prep:eq-fattened-LWP}
\widetilde{P}_N := \sum_{\substack{M\colon M\sim N}} P_M. 
\end{equation}
Furthermore, if $f \in \C_b^\infty(\R^{1+1}_{u,v})$, the Littlewood-Paley projections in the $u$ and $v$-variables are defined as 
\begin{equation}
P_N^u f(u,v) = (\widecheck{\rho}_N \ast_u f)(u,v) \qquad \text{and} \qquad P_N^v f(u,v) = (\widecheck{\rho}_N \ast_v f)(u,v),
\end{equation}
where $\ast_u$ and $\ast_v$ denote the convolution in the $u$ and $v$-variable, respectively. 
The fattened Littlewood-Paley projections $\widetilde{P}_N^u$ and  $\widetilde{P}_N^v$ are defined similarly as in \eqref{prep:eq-fattened-LWP}. 

\end{definition}

Equipped with Definition \ref{prep:def-LittlewoodPaley}, we can now define the Bourgain-type spaces $\Cprod{\gamma_1}{\gamma_2}$.

\begin{definition}[Hölder and Bourgain-type spaces]\label{prep:def-spaces}
For any regularity parameter $\gamma \in \R$ and any function $f\colon \R \rightarrow \R$, we define the $\C^\gamma$-norm of $f$ by 
\begin{equation}
\| f \|_{\C^\gamma} = \| f \|_{\C^\gamma(\R)}:= \sup_{N\geq 1} N^{\gamma} \| P_N f(x) \|_{L^\infty_x(\R)}. 
\end{equation}
The corresponding Hölder space $\C^\gamma$ is defined as the completion of $C_b^\infty(\R)$.

For any two regularity parameters $\gamma_1,\gamma_2 \in \R$ and any function $f\colon \R_{u,v}^{1+1} \rightarrow \R$, we define the $\Cprod{\gamma_1}{\gamma_2}$-norm of $f$ by
\begin{equation}\label{prep:eq-product-norm}
\| f \|_{\Cprod{\gamma_1}{\gamma_2}} = \| f \|_{\Cprod{\gamma_1}{\gamma_2}(\R^{1+1}_{u,v})}:= \sup_{N_1,N_2\geq 1} N^{\gamma_1}_1 N^{\gamma_2}_2 \| P_{N_1}^u P_{N_2}^v f(u,v) \|_{L^\infty_{u,v}(\R^{1+1}_{u,v})}. 
\end{equation}
Similar as above, we define the corresponding Bourgain-type space $\Cprod{\gamma_1}{\gamma_2}$ as the completion of $C_b^{\infty}(\R^{1+1}_{u,v})$ under the $\Cprod{\gamma_1}{\gamma_2}$-norm. 
\end{definition}

\begin{remark}
We refer to $\Cprod{\gamma_1}{\gamma_2}$ as a Bourgain-type space since it is the natural $L^\infty$-based analogue of the $L^2$-based Bourgain space $X^{s,b}$ (see e.g. \cite{Bourgain93,Tao06}). 
Even if $\gamma_1=\gamma_2=\gamma \in \R$, the Bourgain-type space $\Cprod{\gamma_1}{\gamma_2}$ does not coincide with the usual Hölder spaces $\C^\gamma(\R^{1+1}_{u,v})$. While \eqref{prep:eq-product-norm} contains the pre-factor
$N_1^{\gamma_1} N_2^{\gamma_2}$, the usual Hölder norm would contain $\max(N_1,N_2)^\gamma$. This difference is essential for making use of the null structure in the wave maps equation. 
\end{remark}

We now turn to the para-product operators. First, we recall the definition of standard low$\times$high, high$\times$high, and high$\times$low para-products (in a single variable). For our purposes, it is convenient to give ourselves more room in the frequency-scales, which leads to the modified para-products below. Finally, we extend the para-products in a single variable to para-products acting on either the $u$ or $v$-coordinate. 

\begin{definition}[Para-product operators]\label{prep:def-paraproduct} In this definition, we define three different kinds of para-product operators. 
\begin{enumerate}
\item Standard para-products: For any $f,g \in \C_b^\infty(\R)$, we define the para-products
\begin{align} 
f  \parall   g &:= \sum_{ M \ll  N}  P_M f \cdot P_N g, \label{prep:eq-ll} \\
f \parasim g &:=  \sum_{ M \sim  N}  P_M f \cdot P_N g, \label{prep:eq-sim} \\ 
f \paragg g &:=  \sum_{ M \gg  N}  P_M f \cdot P_N g. \label{prep:eq-gg}
\end{align} 
In other words, $\parall, \parasim$, and $\paragg$ correspond to the low$\times$high, high$\times$high, and high$\times$low para-product, respectively. Similarly, we define 
\begin{align}
f \paralesssim  g &:= \sum_{ M \lesssim   N}  P_M f \cdot P_N g, \\
f \paragtrsim  g &:= \sum_{ M \gtrsim N}  P_M f \cdot P_N g, \\ 
f \paransim  g &:= \sum_{ M \not\sim N}  P_M f \cdot P_N g.  
\end{align}
 In other words, $f \paralesssim g$ contains low$\times$high and high$\times$high-interactions, $f \paragtrsim g$ contains high$\times$low and high$\times$high-interactions, and $f\paransim g$ contains low$\times$high and high$\times$low-interactions.  
 
 \item Modified para-products: Let $\sigma \in (0,1)$ be as in \eqref{prep:eq-parameter-1} and \eqref{prep:eq-parameter-sigma}. For any $f,g \in \C_b^\infty(\R)$, we also define modified para-products by 
\begin{align}
f \parallsigma g &:= \sum_{ M \leq   N^{1-\sigma}}  P_M f \cdot P_N g, \allowdisplaybreaks[3] \\
f \paragtrsimsigma g &:= \sum_{ M >  N^{1-\sigma}}  P_M f \cdot P_N g, \allowdisplaybreaks[3]\\
f \paradown g &:= \sum_{\substack{M \sim N}} \sum_{K \leq \min(M,N)^\sigma} P_K \big( P_M f \cdot P_N g \big). 
\end{align}
As a result, $\parallsigma$ contains less frequency-interactions than $\paralesssim$ and $\paragtrsimsigma$ contains more frequency-interactions than $\paragtrsim$. The operator $\paradown$ contains only high$\times$high$\rightarrow$low-interactions and is therefore a modification of $\parasim$.

\item Para-products in $u$ and $v$: Finally, if $f,g\in \C_b^\infty(\R^{1+1}_{u,v})$, we define para-product operators in individual variables by adding superscripts to our previous para-product operators. For example, the analogues of  \eqref{prep:eq-ll}, \eqref{prep:eq-sim}, and \eqref{prep:eq-gg} are given by 
\begin{align}
(f \parallu g) (u,v) := \sum_{ M \ll  N}  P_M^u f(u,v) \cdot P_N^u g(u,v), \\
(f \parasimu g) (u,v) := \sum_{ M \sim N}  P_M^u f(u,v) \cdot P_N^u g(u,v), \\
(f \paraggu g) (u,v) := \sum_{ M \gg  N}  P_M^u f(u,v) \cdot P_N^u g(u,v).
\end{align}
All other para-product operators used in this paper, such as $\paransimu$, $\parallsigu$, and $\paradownu$, are defined similarly. 
\end{enumerate}
\end{definition}

We now present the basic bilinear estimate in our product spaces. The results and proofs are similar to the standard para-product estimates in Besov or Hölder spaces, see e.g. \cite[Lemma 2.1]{GIP15} or \cite[Section 2.6]{BCD11}. However, the product structure of the $\Cprod{\gamma_1}{\gamma_2}$-norm offers greater flexibility in regularity parameters.

\begin{proposition}[Bilinear estimates]\label{prep:prop-bilinear}
Let $\alpha_j,\beta_j\in \mathbb{R}\backslash \{ 0 \} $ and $ \gamma_j \in \R$ be regularities, where $j=1,2$, and let $f,g \colon \R_{u,v}^2 \rightarrow \R$. Then, the following estimates hold:
\begin{enumerate}[label={(\roman*)},leftmargin=7mm]
\item \label{prep:item-general} (General case): If 
\begin{equation}\label{prep:eq-bilinear-condition}
\gamma_j \leq \min(\alpha_j, \beta_j) \quad \text{and} \quad \alpha_j + \beta_j  >0 \quad \text{for} \quad j=1,2,
\end{equation}
then we have that 
\begin{equation}\label{prep:eq-bilinear-general}
\| f g \|_{\Cprod{\gamma_1}{\gamma_2}} \lesssim \| f \|_{\Cprod{\alpha_1}{\alpha_2}} \| g \|_{\Cprod{\beta_1}{\beta_2}}. 
\end{equation}
\item \label{prep:item-low-high} (Low$\times$high-improvement): If $\alpha_1 < 0 < \beta_1$, 
\begin{align*}
\gamma_1 &\leq \alpha_1 + (1-\sigma) \beta_1, \qquad &\alpha_1 + \beta_1 >0,    \\
\gamma_2 &\leq \min(\alpha_2,\beta_2), \qquad  &\alpha_2+ \beta_2 >0, 
\end{align*}
then 
\begin{equation}
\| f \paralesssimsigu g \|_{\Cprod{\gamma_1}{\gamma_2}} \lesssim \| f \|_{\Cprod{\alpha_1}{\alpha_2}} \| g \|_{\Cprod{\beta_1}{\beta_2}}. 
\end{equation}
\item \label{prep:item-nonres} (Non-resonant improvement): If 
\begin{align*}
\gamma_1 &\leq \min(\alpha_1,\beta_1,\alpha_1+\beta_1),    \\
\gamma_2 &\leq \min(\alpha_2,\beta_2), \qquad  &\alpha_2+ \beta_2 >0, 
\end{align*}
then 
\begin{equation}
\| f \paransimu g \|_{\Cprod{\gamma_1}{\gamma_2}} \lesssim \| f \|_{\Cprod{\alpha_1}{\alpha_2}} \| g \|_{\Cprod{\beta_1}{\beta_2}}. 
\end{equation}
\item \label{prep:item-res} (Resonant improvement): If 
\begin{align*}
\gamma_1 &\leq \alpha_1+\beta_1,  \qquad &\alpha_1 + \beta_1 >0,   \\
\gamma_2 &\leq \min(\alpha_2,\beta_2), \qquad  &\alpha_2+ \beta_2 >0, 
\end{align*}
then 
\begin{equation}
\| f \parasimu g \|_{\Cprod{\gamma_1}{\gamma_2}} \lesssim \| f \|_{\Cprod{\alpha_1}{\alpha_2}} \| g \|_{\Cprod{\beta_1}{\beta_2}}. 
\end{equation}
\end{enumerate}
\end{proposition}
\begin{remark}
The restriction $\alpha_j,\beta_j \neq 0$ is only imposed to avoid logarithmic corrections to our estimates. In \ref{prep:item-low-high}, the improvement lies in the upper bound on $\gamma_1$, which is weaker than an upper bound by $\min(\alpha_1,\beta_1)$. In \ref{prep:item-nonres}, the improvement lies in the absence of the condition $\alpha_1+\beta_1>0$. In \ref{prep:item-res}, the improvement again lies in the upper bound on $\gamma_1$, which is similar to \ref{prep:item-low-high}.
\end{remark}

\begin{proof}[Proof of Proposition \ref{prep:prop-bilinear}:]

We start with the general case  \ref{prep:item-general}. Using a Littlewood-Paley decomposition, we obtain that 
\begin{align*}
&\| f g \|_{\Cprod{\gamma_1}{\gamma_2}}  \\
=& \sup_{K_1,K_2} \Big( \prod_{a=1,2} K_a^{\gamma_a} \Big)  \| P_{K_1}^u P_{K_2}^v ( f g) \|_{L^\infty_{u,v}} \\
\leq&  \sup_{K_1,K_2} \sum_{\substack{M_1,M_2, \\ N_1,N_2}} \Big( \prod_{a=1,2} K_a^{\gamma_a} \Big)  \| P_{K_1}^u P_{K_2}^v (  P_{M_1}^u P_{M_2}^v f \cdot P_{N_1}^u P_{N_2}^v g) \|_{L^\infty_{u,v}}  \\
\leq& \prod_{a=1,2} \bigg[ 
\sup_{K_a} \sum_{M_a,N_a}
 \Big(  1\{ K_a \hspace{-2pt}\sim \hspace{-2pt} M_a \hspace{-2pt} \gg  \hspace{-2pt} N_a \}+  1\{ K_a\hspace{-2pt} \sim \hspace{-2pt} N_a\hspace{-2pt} \gg \hspace{-2pt} M_a \} 
+ 1\{ M_a \hspace{-2pt}\sim \hspace{-2pt} N_a \hspace{-2pt}\gtrsim  \hspace{-2pt} K_a \} \Big) K_a^{\gamma_a} M_a^{-\alpha_a} N_a^{-\beta_a}  \bigg]  \\
\times& \| f \|_{\Cprod{\alpha_1}{\alpha_2}} \| g \|_{\Cprod{\beta_1}{\beta_2}}. 
\end{align*}
Thus, it suffices to prove the estimate
\begin{equation}\label{prep:eq-sum-scales}
\sup_K \sum_{M,N} 
 \Big(  1\{ K \hspace{-2pt}\sim \hspace{-2pt} M \hspace{-2pt} \gg  \hspace{-2pt} N \}+  1\{ K \hspace{-2pt} \sim \hspace{-2pt} N \hspace{-2pt} \gg \hspace{-2pt} M \} 
+ 1\{ M \hspace{-2pt}\sim \hspace{-2pt} N \hspace{-2pt}\gtrsim  \hspace{-2pt} K \} \Big) K^{\gamma}  M^{-\alpha} N^{-\beta} \lesssim 1 
\end{equation}
for all $\alpha,\beta \in \R\backslash \{0 \}$ and $\gamma \in \R$ satisfying $\gamma\leq \min(\alpha,\beta)$ and $\alpha+\beta>0$. In the following, we write $x_-:=\min(x,0)$. We separate the proof into three sub-cases. \\

\emph{(a): The high$\times$low-interaction $K\sim M\gg N$.} We first note that $\gamma\leq \min(\alpha,\beta)$ and $\alpha+\beta>0$ imply that $\gamma \leq \alpha +\beta_{-}$. Using only that $\gamma \leq \alpha +\beta_{-}$, we obtain 
\begin{align*}
\sup_K \sum_{M,N}   1\{ K \hspace{-2pt}\sim \hspace{-2pt} M \hspace{-2pt} \gg  \hspace{-2pt} N \} K^{\gamma} M^{-\alpha} N^{-\beta} 
= \sup_{K} K^{\gamma-\alpha} \sum_{N} 1\{ N\ll K\} N^{-\beta} 
\lesssim \sup_{K} K^{\gamma-\alpha-\beta_-} =1. 
\end{align*}~\\ 

\emph{(b): The low$\times$high-interaction $K\sim N\gg M$.} The argument is similar as in case (a) and only requires that $\gamma \leq \alpha_- + \beta$. ~\\ 

\emph{(c): The high$\times$high-interaction $M\sim N\gtrsim K$.} Using only the condition $\alpha+\beta>0$, we obtain 
\begin{equation*}
\sup_K \sum_{M,N}   1\{ M \hspace{-2pt}\sim \hspace{-2pt} N \hspace{-2pt} \gtrsim \hspace{-2pt} K \} K^{\gamma} M^{-\alpha} N^{-\beta} 
\lesssim \sup_{K} K^\gamma \sum_{M} 1\{ M \gtrsim K\} M^{-\alpha-\beta} \lesssim \sup_{K} K^{\gamma-\alpha-\beta}. 
\end{equation*}
From $\alpha+\beta>0$, it follows that $\alpha+\beta=\max(\alpha,\beta)+\min(\alpha,\beta)\geq \min(\alpha,\beta)$. Thus, the bound by one follows from the assumption $\gamma\leq \min(\alpha,\beta)$. This completes the proof of \ref{prep:item-general}. \\

The arguments for \ref{prep:item-low-high}, \ref{prep:item-nonres}, and \ref{prep:item-res} only require minor modifications. In \ref{prep:item-low-high}, the frequency-restriction in case (a) is replaced by $ K\sim M \gg N \geq M^{1-\sigma}$. Instead of $\gamma\leq \alpha= \alpha+\beta_-$, we therefore only need the weaker condition $\gamma \leq \alpha + (1-\sigma) \beta$. In \ref{prep:item-nonres}, the high-high interaction has been removed, which means that case (c) no longer appears. As a result, the condition $\alpha+\beta>0$ is not needed. The new restriction on $\gamma$ is a result of the identity
\begin{equation*}
\min(\alpha+\beta_-,\alpha_- +  \beta) = \min(\alpha,\beta,\alpha+\beta). 
\end{equation*}
Finally, in \ref{prep:item-res}, we only encounter case (c), which requires that $\gamma_1\leq \alpha_1+\beta_1$ and $\alpha_1+\beta_1>0$.
\end{proof}

We now present a corollary of the proof of Proposition \ref{prep:prop-bilinear}, which addresses families of functions.

\begin{corollary}[Bilinear estimate for families of functions]\label{prep:cor-bilinear}
Let $\alpha_j,\beta_j\in \mathbb{R}\backslash \{ 0 \} $ and $ \gamma_j \in \R$ be regularities, where $j=1,2$, and assume that 
\begin{equation*}
\gamma_j \leq \min(\alpha_j, \beta_j) \quad \text{and} \quad \alpha_j + \beta_j  >0 \quad \text{for} \quad j=1,2.
\end{equation*}
Let $(f_M)_M \subseteq \Cprod{\alpha_1}{\alpha_2}$ and $(g_M)_M \subseteq \Cprod{\beta_1}{\beta_2}$ be two families of functions and assume that $g_M$ is supported on frequencies $\sim M$ in the u-variable. Then, it holds that 
\begin{equation}\label{prep:eq-bilinear-families}
\Big\| \sum_M f_M g_M \Big\|_{\Cprod{\gamma_1}{\gamma_2}} 
\lesssim \sup_M \big\| f_M \big\|_{\Cprod{\alpha_1}{\alpha_2}}
\, \sup_M \big\| g_M \big\|_{\Cprod{\beta_1}{\beta_2}}.
\end{equation}
\end{corollary}

The most important aspect of \eqref{prep:eq-bilinear-families} is that the right-hand side contains a supremum and not a sum over $M$. 

\begin{proof}
We only sketch the necessary modifications in the proof of Proposition \ref{prep:prop-bilinear}. We first decompose
\begin{equation}\label{prep:eq-families-p1}
\sum_M f_M g_M = \sum_{M} f_M \parallu g_M + \sum_{M} f_M \parasimu g_M + \sum_{M} f_M \paraggu g_M.
\end{equation}
For the low$\times$high-term in \eqref{prep:eq-families-p1}, we have from frequency-support considerations that 
\begin{equation}\label{prep:eq-families-p2}
\Big\| \sum_{M} f_M \parallu g_M \Big\|_{\Cprod{\gamma_1}{\gamma_2}} 
\lesssim \sup_M \Big\| f_M \parallu g_M \Big\|_{\Cprod{\gamma_1}{\gamma_2}}. 
\end{equation}
The right-hand side of \eqref{prep:eq-families-p2} can then be estimated as before. Arguing as in the proof of \mbox{Proposition \ref{prep:prop-bilinear}}, the high$\times$high and low$\times$high-terms in \eqref{prep:eq-families-p1} can be estimated by 
\begin{align}\label{prep:eq-families-p3}
\Big\|  f_M \parasimu g_M \Big\|_{\Cprod{\gamma_1}{\gamma_2}} + \Big\| f_M \paraggu g_M \Big\|_{\Cprod{\gamma_1}{\gamma_2}} 
&\lesssim M^{\max(0,\gamma_1)-\alpha_1-\beta_1} \big\| f_M \big\|_{\Cprod{\alpha_1}{\alpha_2}} 
\,  \big\| g_M \big\|_{\Cprod{\beta_1}{\beta_2}}. 
\end{align}
Since $\alpha_1+\beta_1>0$ and
\begin{equation*}
\gamma_1 - \alpha_1 - \beta_1 = \gamma_1 - \min(\alpha_1,\beta_1)- \max(\alpha_1,\beta_1) \leq - \max(\alpha_1,\beta_1) < 0, 
\end{equation*}
the exponent of $M$ in \eqref{prep:eq-families-p3} is negative, and the contributions are therefore summable in $M$.
\end{proof}

The following estimate is a special case of the bilinear estimate \eqref{prep:eq-bilinear-general}. Due to our frequent use of this estimate, however, we isolate it in the following corollary. 
\begin{corollary}[Multiplication estimate]\label{prep:corollary-multiplication} 
Let $s,r^\prime$, and $r$ be as in \eqref{prep:eq-parameter-1} and \eqref{prep:eq-parameter-3}. For all $f,g \colon \R_{u,v}^{1+1} \rightarrow \R$, it then holds that 
\begin{equation}\label{prep:eq-multiplication}
\| f  g \|_{\Cprod{r-1}{r-1}} \lesssim  \| f \|_{\Cprod{1-r^\prime}{1-r^\prime}} \|  g \|_{\Cprod{r-1}{r-1}} \lesssim
\|  f \|_{\Cprod{s}{s}} \|  g \|_{\Cprod{r-1}{r-1}}. 
\end{equation}
\end{corollary}

The inequality \eqref{prep:eq-multiplication} is called a multiplication estimate since the map $g\mapsto f \cdot g$ is regarded as the multiplication of $g$ with the ``smooth'' function $f$. 

\begin{proof}
This follows  from the bilinear estimate \eqref{prep:eq-bilinear-general}, $s,r^\prime,r\in (0,1)$, $r^\prime<r$, and $1-r^\prime<s$.
\end{proof}

In Proposition \ref{prep:prop-bilinear} and its variants, we have obtained estimates for bilinear products. However, the full nonlinearity is given by 
$\Second^k_{ij}(\phi) \partial_u \phi^i \partial_v \phi^j$
and therefore contains the composition of $\phi$ with a nonlinear function. To treat this composition, we require the following version of Bony's para-linearization in the $\Cprod{\gamma_1}{\gamma_2}$-spaces.

\begin{lemma}[Bony's para-linearization]\label{prep:lemma-bony}
Let $F \in C_b^\infty (\R^{\dimA}\rightarrow \R)$, let $0<\gamma<1$, and let $\phi \in \Cprod{\gamma}{\gamma}$. Then, there exists a constant 
\begin{equation*}
C_{\gamma,F}=C(\gamma,\| F \|_{L^\infty}, \hdots , \| \nabla^{10} F \|_{L^\infty}) 
\end{equation*}
such that the following two properties hold: 
\begin{enumerate}[label={(\roman*)}]
\item \label{prep:lemma-bony-item1} (Composition estimate): It holds that 
\begin{equation}\label{prep:eq-composition}
\| F (\phi) \|_{\Cprod{\gamma}{\gamma}} \leq C_{\gamma,F} ( 1 + \| \phi \|_{\Cprod{\gamma}{\gamma}}  )^9 \| \phi \|_{\Cprod{\gamma}{\gamma}}.
\end{equation}
\item \label{prep:lemma-bony-item2}  (Para-linearization in one variable): 
For all $N> 1$, it holds that 
\begin{align}
\| P_{N}^u ( F(\phi)) - \langle \nabla F (\phi) , P_N^u \phi \rangle \|_{\Cprod{\gamma}{\gamma}} 
&\leq C_{\gamma,F} ~N^{-\gamma} (1+\| \phi \|_{\Cprod{\gamma}{\gamma}})^9 \| \phi \|_{\Cprod{\gamma}{\gamma}}, \label{prep:eq-bony-u} \\ 
\| P_{N}^v ( F(\phi)) - \langle \nabla F(\phi) ,P_N^v \phi \rangle \|_{\Cprod{\gamma}{\gamma}} 
&\leq  C_{\gamma,F} ~N^{-\gamma} (1+\| \phi \|_{\Cprod{\gamma}{\gamma}})^9 \| \phi \|_{\Cprod{\gamma}{\gamma}}. \label{prep:eq-bony-v}
\end{align}
\end{enumerate}
\end{lemma}

The exponents in \eqref{prep:eq-composition}, \eqref{prep:eq-bony-u}, and \eqref{prep:eq-bony-v} are generously large and the exact value is irrelevant for the rest of the paper. We emphasize that \ref{prep:lemma-bony-item2} only contains estimates for $P_N^u$ and $P_N^v$ but not the combined operator $P_{N_1}^u P_{N_2}^v$. The reason is that the size of the error term would only be bounded by inverse powers of $\min(N_1,N_2)$ instead of $\max(N_1,N_2)$, which is not sufficient for our purposes.  

\begin{proof}[Proof sketch:]
We only sketch the argument, which is a minor modification of the Taylor expansion around low-frequencies used in the standard Besov-space version (see e.g.  \cite[Theorem 2.89 and 2.92]{BCD11}). 
To prove \eqref{prep:eq-composition} and \eqref{prep:eq-bony-u}, the simplest approach is to perform the same steps as in the proof of \cite[Theorem 2.89 and 2.92]{BCD11}  in the $u$-variable and then apply the known results \cite[Theorem 2.89 and 2.92]{BCD11} in the remaining $v$-variable. 
To convince the reader that this reduction to the single-variable case is possible, we note that 
\begin{align*}
\| \psi \|_{\Cprod{\gamma}{\gamma}} &= \sup_{N_1,N_2\geq 1} N_1^\gamma N_2^\gamma 
\| P_{N_1}^u P_{N_2}^v \psi \|_{L^\infty_{u,v}} \\
&= \sup_{N_1\geq 1} N_1^{\gamma} \sup_{u\in \R} \| P_{N_1}^u \psi(u,v) \|_{\C_v^\gamma}. 
\end{align*} 
The para-linearization \eqref{prep:eq-bony-v} in the $v$-variable follows by reversing the roles of $u$ and $v$. 
\end{proof}

\subsection{Commutator estimates}
 
In this subsection, we present several commutator estimates. Similar commutator estimates were already heavily used in the para-controlled approach to rough ODEs and singular parabolic SPDEs in \cite{GIP15}. Our setting requires commutator estimates in the Bourgain-type spaces $\Cprod{\gamma_1}{\gamma_2}$ instead of the standard Besov spaces, which only requires minor modifications. 

\begin{lemma}[Frequency-localized commutator estimates]\label{prep:lem-commutator}
Let $\alpha_j,\beta_j\in \mathbb{R}\backslash \{ 0 \}$ and $\gamma_j \in \mathbb{R}$, where $j=1,2$, be regularity parameters which satisfy  $0<\alpha_1,\beta_1,\gamma_1<1$ and  \eqref{prep:eq-bilinear-condition}. Furthermore, let $K\geq 1$ and let $f,g \colon \R_{u,v}^{1+1} \rightarrow \R$. Then, we have that
\begin{align}
\big\| \big[P_{K}^u , f \big]  g \big\|_{\Cprod{\gamma_1}{\gamma_2}} 
&\lesssim K^{\gamma_1-\alpha_1} \| f \|_{\Cprod{\alpha_1}{\alpha_2}} 
\| g \|_{\Cprod{\beta_1}{\beta_2}}\label{prep:eq-commutator-1}, \\
\big\| \big[P_{K}^u , f \big] P_{\gtrsim K}^u g \big\|_{\Cprod{\gamma_1}{\gamma_2}} 
&\lesssim K^{\gamma_1-\alpha_1-\beta_1} \| f \|_{\Cprod{\alpha_1}{\alpha_2}} \| g \|_{\Cprod{\beta_1}{\beta_2}}\label{prep:eq-commutator-2}, \\
\big\| P_K^u ( f g) - P_{\ll K}^u f P_K^u g - P_{K}^u f P_{\ll K}^ u g \big\|_{\Cprod{\gamma_1}{\gamma_2}} 
&\lesssim K^{\gamma_1-\alpha_1-\beta_1} \| f \|_{\Cprod{\alpha_1}{\alpha_2}} \| g \|_{\Cprod{\beta_1}{\beta_2}},\label{prep:eq-commutator-3}\\
\big\| P_K^u ( f g) -  f P_K^u g - P_{K}^u f \,  g \big\|_{\Cprod{\gamma_1}{\gamma_2}} 
&\lesssim K^{\gamma_1-\alpha_1-\beta_1} \| f \|_{\Cprod{\alpha_1}{\alpha_2}} \| g \|_{\Cprod{\beta_1}{\beta_2}}.\label{prep:eq-commutator-4}
\end{align}
\end{lemma}

Since all estimates only involve Littlewood-Paley operators in the $u$-variable, the proof is exactly as in the standard Besov spaces. We only present the argument for the sake of completeness. 
\begin{proof}
We start by proving the estimate 
\begin{equation}\label{prep:eq-commutator-p1}
    \big\| \big[P_{K}^u , f \big] P_{L}^u g \big\|_{\Cprod{\gamma_1}{\gamma_2}} 
\lesssim K^{\gamma_1-\alpha_1} L^{-\beta_1} \| f \|_{\Cprod{\alpha_1}{\alpha_2}} \| g \|_{\Cprod{\beta_1}{\beta_2}}
\end{equation}
for all frequency-scales $K$ and $L$. The estimates \eqref{prep:eq-commutator-1} and \eqref{prep:eq-commutator-2} then follow by summing over $L\geq 1$ and $L \gtrsim K$, respectively. To prove \eqref{prep:eq-commutator-p1}, we distinguish the cases $L\not\sim K$ and $L\sim K$. 

If $L\not \sim K$, we decompose
\begin{equation*}
[P_K^u , f ] P_L^u g= P_K^u( f P_L^u g) = P_K^u \big( P_{\gtrsim K}^u f \, P_L^u g \big). 
\end{equation*}
Then, we have the estimate
\begin{equation}\label{prep:eq-commutator-p2}
\begin{aligned}
 \big\| P_K^u \big( P_{\gtrsim K}^u f \, P_L^u g \big) \big\|_{\Cprod{\gamma_1}{\gamma_2}} 
 &\lesssim K^{\gamma_1} \big\|  P_{\gtrsim K}^u f \, P_L^u g \big\|_{L^\infty_u \C_v^{\gamma_2}} \\
 &\lesssim K^{\gamma_1-\alpha_1} L^{-\beta_1}
  \| f \|_{\Cprod{\alpha_1}{\alpha_2}} \| g \|_{\Cprod{\beta_1}{\beta_2}}.
\end{aligned}
\end{equation}
If $L\sim K$, we decompose 
\begin{equation}\label{prep:eq-commutator-p3}
\begin{aligned}
\, [P_K^u , f ] P_L^u g &= [P_K^u , P_{\lesssim K}^u f ] P_L^u g + [P_K^u , P_{\gg K}^u f ] P_L^u g \\
&=[P_K^u , P_{\lesssim K}^u f ] P_L^u g - P^u_{\gg K} f P_K^u P_L^u g . 
\end{aligned}
\end{equation}
Using standard commutator estimates (see e.g. \cite[Lemma 2.97]{BCD11}), the first term in \eqref{prep:eq-commutator-p3} can be estimated by
\begin{align*}
\big\| [P_K^u , P_{\lesssim K}^u f ] P_L^u g \big\|_{\Cprod{\gamma_1}{\gamma_2}} 
&\lesssim K^{\gamma_1} \big\| [P_K^u , P_{\lesssim K}^u f ] P_L^u g \big\|_{L^\infty_u \C_v^{\gamma_2}} \\
&\lesssim K^{\gamma_1-1} \big\| \partial_u P_{\lesssim K}^u f \big\|_{L^\infty_u \C_v^{\alpha_2}} 
\big\| P_L^u g \big\|_{L^\infty_u \C_v^{\beta_2}} \\
&\lesssim K^{\gamma_1-\alpha_1} L^{-\beta_1}  \| f \|_{\Cprod{\alpha_1}{\alpha_2}} \| g \|_{\Cprod{\beta_1}{\beta_2}}. 
\end{align*}
For the second term in \eqref{prep:eq-commutator-p3}, we estimate
\begin{align*}
\big\| P^u_{\gg K} f P_K^u P_L^u g \big\|_{\Cprod{\gamma_1}{\gamma_2}} 
&\lesssim \big\| P^u_{\gg K} f  \big\|_{\Cprod{\gamma_1}{\gamma_2}} 
\big\| P_K^u P_L^u g \big\|_{L^\infty_u \C_v^{\gamma_2}} \\
&\lesssim K^{\gamma_1-\alpha_1} L^{-\beta_1}  \| f \|_{\Cprod{\alpha_1}{\alpha_2}} \| g \|_{\Cprod{\beta_1}{\beta_2}}. 
\end{align*}
This completes the proof of \eqref{prep:eq-commutator-p1}. It remains to prove \eqref{prep:eq-commutator-3} and \eqref{prep:eq-commutator-4}. To prove  \eqref{prep:eq-commutator-3}, we decompose
\begin{align}
&P_K^u ( f g) - P_{\ll K}^u f P_K^u g - P_{K}^u f P_{\ll K}^ u g \notag \\
=& P_K^u \big( P_{\ll K}^u f P_{\gtrsim K}^u g +  P_{\gtrsim K}^u f P_{\ll K}^u g + P_{\gtrsim K}^u f P_{\gtrsim K}^u g \big) - P_{\ll K}^u f P_K^u g - P_{K}^u f P_{\ll K}^ u g \notag  \\
=& [P_K^u , P_{\ll K}^u f] P_{\gtrsim K}^u g + 
[P_K^u , P_{\ll K}^u g] P_{\gtrsim K}^u f + P_K^u \big(  P_{\gtrsim K}^u f P_{\gtrsim K}^u g \big). \label{prep:eq-commutator-p4}
\end{align}
The first two summands in \eqref{prep:eq-commutator-p4} can be bounded using \eqref{prep:eq-commutator-2}. The third summand in \eqref{prep:eq-commutator-p4} can be bounded using \eqref{prep:eq-commutator-p2}. In order to obtain the last estimate \eqref{prep:eq-commutator-4} from \eqref{prep:eq-commutator-3}, it remains to prove 
\begin{equation}\label{prep:eq-commutator-p5}
\big\| P_{\gtrsim K}^u f P_K^u g \big\|_{\Cprod{\gamma_1}{\gamma_2}} + 
\big\| P_K^u f P_{\gtrsim K}^u g \big\|_{\Cprod{\gamma_1}{\gamma_2}} 
\lesssim K^{\gamma_1-\alpha_1-\beta_1} \| f \|_{\Cprod{\alpha_1}{\alpha_2}} \| g \|_{\Cprod{\beta_1}{\beta_2}}. 
\end{equation}
By symmetry, it suffices to treat the first summand in \eqref{prep:eq-commutator-p5}. Using frequency-support considerations, we have that 
\begin{equation}\label{prep:eq-commutator-p6}
\big\| P_{\gtrsim K}^u f P_K^u g \big\|_{\Cprod{\gamma_1}{\gamma_2}} 
= \big\| \sum_{L\gtrsim K} P_L^u f P_K^u g \big\|_{\Cprod{\gamma_1}{\gamma_2}}  
\lesssim \sup_{L\gtrsim K} L^{\gamma_1} \big\|  P_L^u f P_K^u g \big\|_{L_u^\infty \C_v^{\gamma_2}}. 
\end{equation}
Using $\gamma_1\leq \alpha_1$, the right-hand side of \eqref{prep:eq-commutator-p6} can be bounded as before.
\end{proof}

In the following, we state additional commutator estimates directly in terms of the para-product operators, which essentially follow from Lemma \ref{prep:lem-commutator}. The first corollary will be used in the PDE-analysis (Section \ref{section:prototypical}-\ref{section:duhamel}) and the second corollary will be used in the ODE-analysis (Section \ref{section:modulation}). 

In the following corollary, we prove a commutator estimate for the para-product operator $\parallsigu$.
\begin{corollary}[Commutators for PDE-analysis]\label{prep:cor-commutator-PDE}
Let $s,r^\prime$, and $r$ be as in \eqref{prep:eq-parameter-1} and \eqref{prep:eq-parameter-3}. Then, we have 
for all $f,g,h\in \C_b^\infty(\R^{1+1}_{u,v})$ that
\begin{equation}\label{prep:eq-commutator-PDE}
\big\| f \big( g \parallsigu h \big) - \big( fg \big) \parallsigu h \big\|_{\Cprod{r-1}{r-1}} 
\lesssim \| f \|_{\Cprod{s}{1-r^\prime}} \| g \|_{\Cprod{1-r^\prime}{r-1}} \| h \|_{\Cprod{s-1+\sigma}{1-r^\prime}}. 
\end{equation}
\end{corollary}

\begin{proof}
We first rewrite the argument on the left-hand side of \eqref{prep:eq-commutator-PDE} as \begin{align*}
      f \big( g \parallsigu h \big) - \big( fg \big) \parallsigu h  
      &= \sum_{M} \Big( f \, P_{\leq M^{1-\sigma}}^u g \, P_M^u h - P_{\leq M^{1-\sigma}} \big( f g\big) P_{M}^u h \Big)\\
      &= \sum_{M}  \big[ f , P_{\leq M^{1-\sigma}}^u \big] \big( g \big) \, P_M^u h \\
      &= \sum_{M} \big[  P_{>M^{1-\sigma}}^u, f  \big] \big( g \big) \, P_M^u h . 
 \end{align*}
 Using the bilinear estimate (Proposition \ref{prep:prop-bilinear}) and \eqref{prep:eq-commutator-1}, we have that 
 \begin{align*}
     &\big\| \sum_M \big[  P_{>M^{1-\sigma}}^u, f  \big] \big( g \big) \, P_M^u h \big\|_{\Cprod{r-1}{r-1}} \\
     &\lesssim \sum_M \big\| \big[  P_{>M^{1-\sigma}}^u, f  \big] \big( g \big) \big\|_{\Cprod{1-r^\prime}{r-1}} \big\| P_M^u h \big\|_{\Cprod{r-1}{1-r^\prime}} \\
     &\lesssim \| f \|_{\Cprod{s}{1-r^\prime}} \| g \|_{\Cprod{1-r^\prime}{r-1}} \sum_M M^{(1-\sigma)(1-r^\prime-s)} \big\| P_M^u h \big\|_{\Cprod{r-1}{1-r^\prime}}. 
 \end{align*}
 Since 
 \begin{equation*}
    (1-\sigma) (1-r^\prime-s)+r-1= s-1 + 1-2s + \eta + \sigma (r^\prime+s-1) < s-1 + \sigma,
 \end{equation*}
 this yields the desired estimate.
\end{proof}

We now define three different commutator terms, which will be used in the ODE-analysis.

\begin{definition}[Commutator terms for ODE-analysis]\label{prep:def-commutator-ODE}
For all $f,g,h\in \C_b^\infty(\R^{1+1}_{u,v})$, we define 
\begin{align}
    \Com_{\parall}^v (f,g,h) 
    &:= f ( g \parallv h) - (fg) \parallv h - g (f\parasimv h), \\
    \Com_{\parall,\parasim}^v(f,g,h) 
    &:= (f \parallv g) \parasimv h - f ( g \parasimv h), \\
    \Com_{\parasim}^v (f,g,h) &:= (fg)\parasimv h - f ( g \parasimv h) - g (f \parasimv h). 
\end{align}
\end{definition}

Equipped with Definition \ref{prep:def-commutator-ODE}, we can now state the commutator estimates for the ODE-analysis.

\begin{lemma}[Commutators for ODE-analysis]\label{prep:lem-commutator-ODE}
For all $f,g,h\in \C_b^\infty(\R^{1+1}_{u,v})$, we have that
\begin{align}
\big\| \Com_{\parall}^v(f,g,h) \big\|_{\Cprod{s}{r-1}} 
&\lesssim \| f \|_{\Cprod{s}{s}} \| g\|_{\Cprod{s}{s}} \| h \|_{\Cprod{s}{r-1-s+2\eta}} , \\
\big\| \Com_{\parall,\parasim}^v(f,g,h) \big\|_{\Cprod{s}{r-1}} 
&\lesssim \| f \|_{\Cprod{s}{s}} \| g\|_{\Cprod{s}{s}} \| h \|_{\Cprod{s}{-1+2\eta}} , \\
\big\| \Com_{\parasim}^v(f,g,h) \big\|_{\Cprod{s}{r-1}} 
&\lesssim \| f \|_{\Cprod{s}{s}} \| g\|_{\Cprod{s}{s}} \| h \|_{\Cprod{s}{-1+2\eta}} .
\end{align}
\end{lemma}

\begin{proof}
The three estimates can be derived from Lemma \ref{prep:lem-commutator} and we omit the standard details. Similar estimates can be found in \cite[Lemma 2.4]{GIP15} and \cite[Lemma 2.8]{GIP15}.
\end{proof}

We end this subsection with an estimate for para-products of frequency-localized functions. Strictly speaking, it is not a commutator estimate, but it has a similar flavor. This lemma will be used heavily in the analysis below, since it allows us to freely switch back and forth between para-products and Littlewood-Paley decompositions.

\begin{lemma}[Para-products of frequency-localized functions]\label{prep:lem-para-localized}
Let $\alpha_j,\beta_j\in \mathbb{R}\backslash \{ 0 \} $ and $ \gamma_j \in \R$ be regularities, where $j=1,2$, and let $f,g \in \C_b^\infty(\R_{u,v}^{1+1})$. Furthermore, let $M$ and $N$ be frequency scales.
\begin{itemize}
\item[(i):] If $\gamma_2\leq \min(\alpha_2,\beta_2)$ and $\alpha_2+\beta_2>0$, then it holds that
\begin{equation}\label{prep:eq-localized-1}
\big\| f \parallsigu \widetilde{P}_M^u g - P_{\leq M^{1-\sigma}}^u f \, \widetilde{P}_M^u g \big\|_{\Cprod{\gamma_1}{\gamma_2}} \lesssim M^{\gamma_1-(1-\sigma) \alpha_1 - \beta_1}
\big\| f \big\|_{\Cprod{\alpha_1}{\alpha_2}} \big\| g \big\|_{\Cprod{\beta_1}{\beta_2}}. 
\end{equation}
\item[(ii):] If $\gamma_1\leq \min(\alpha_1,\beta_1)$ and $\alpha_1+\beta_1>0$, then it holds that
\begin{equation}\label{prep:eq-localized-2}
\big\| f \parasimv P_N^v g - \widetilde{P}_N^v f \, P_N^v  g \big\|_{\Cprod{\gamma_1}{\gamma_2}} \lesssim N^{\gamma_2-\alpha_2-\beta_2} \big\| f \big\|_{\Cprod{\alpha_1}{\alpha_2}} \big\| g \big\|_{\Cprod{\beta_1}{\beta_2}}.
\end{equation}
\end{itemize}
\end{lemma}

\begin{proof}
 We start by proving \eqref{prep:eq-localized-1}. To this end, we decompose
 \begin{align*}
    f \parallsigu \widetilde{P}_M^u g - P_{\leq M^{1-\sigma}}^u f \, \widetilde{P}_M^u g 
    &= \sum_{\substack{K,L\colon \\ K \leq L^{1-\sigma}}} 
    P_K^{u} f \, P_L^u \widetilde{P}_M^u g 
    -\sum_{\substack{K,L\colon \\ K \leq M^{1-\sigma}}} 
    P_K^{u} f \, P_L^u \widetilde{P}_M^u g \\
    &= \sum_{\substack{K,L}} \big( 1 \{ K \leq L^{1-\sigma}\}- 1\{ K \leq M^{1-\sigma}\}\big) P_K^{u} f \, P_L^u \widetilde{P}_M^u g. 
 \end{align*}
 The dyadic sum is supported on frequency-scales $K\sim M^{1-\sigma}$ and $L\sim M$. For all such dyadic scales, it holds that 
 \begin{align*}
     \| P_K^{u} f \, P_L^u \widetilde{P}_M^u g \|_{\Cprod{\gamma_1}{\gamma_2}} 
     &\lesssim M^{\gamma_1} \| P_K^{u} f \, P_L^u \widetilde{P}_M^u g \big\|_{L^\infty_u \C_v^{\gamma_2}} \\
     &\lesssim M^{\gamma_1}\| P_K^{u} f\|_{L^\infty_u \C_v^{\alpha_2}} 
     \| P_L^u \widetilde{P}_M^u g \|_{L^\infty_u \C_v^{\beta_2}} \\
     &\lesssim M^{\gamma_1} K^{-\alpha_1} L^{-\beta_1}\big\| f \big\|_{\Cprod{\alpha_1}{\alpha_2}} \big\| g \big\|_{\Cprod{\beta_1}{\beta_2}} \\
     &\lesssim M^{\gamma_1-(1-\sigma)\alpha_1-\beta_1} \big\| f \big\|_{\Cprod{\alpha_1}{\alpha_2}} \big\| g \big\|_{\Cprod{\beta_1}{\beta_2}}. 
 \end{align*}
 This completes the proof of \eqref{prep:eq-localized-1} and we now turn to \eqref{prep:eq-localized-2}. Using linearity, we can replace $f$ in \eqref{prep:eq-localized-2} by $P_K^v f$, where $2^{-20}\leq K/N\leq 2^{20}$. That is, $f$ lives at frequencies comparable to $N$, but with a larger constant than implicit in $K\sim N$. A similar argument as for \eqref{prep:eq-localized-1} shows that 
 \begin{equation}\label{prep:eq-localized-p1}
\big\| P_K^v f \parallv P_N g - P_{\ll N}^v P_K^v f \, P_N^v g \big\|_{\Cprod{\gamma_1}{\gamma_2}} \lesssim N^{\gamma_2-\alpha_2-\beta_2} 
\big\| f \big\|_{\Cprod{\alpha_1}{\alpha_2}} 
\big\| g \big\|_{\Cprod{\beta_1}{\beta_2}}.
 \end{equation}
 Furthermore, it is easy to see that 
 \begin{equation}\label{prep:eq-localized-p2}
\big\| P_K^v f \paraggv P_N g \big\|_{\Cprod{\gamma_1}{\gamma_2}}
+ \big\| P_{\gg N}^v P_K^v f \, P_N^v g \big\|_{\Cprod{\gamma_1}{\gamma_2}}
\lesssim N^{\gamma_2-\alpha_2-\beta_2} 
\big\| f \big\|_{\Cprod{\alpha_1}{\alpha_2}} 
\big\| g \big\|_{\Cprod{\beta_1}{\beta_2}}.
 \end{equation}
 Together with the decomposition
 \begin{align*}
& P_K^v f \parasimv P_N^v g - \widetilde{P}_N^v P_K^v f \, P_N^v  g  \\
=& \big( P_K^v f \, P_N^v g - P_K^v f \parallv P_N^v g - P_K^v f \paraggv P_N^v g \big) 
- \big( P_K^v f  P_N^v g - P_{\ll N}^v P_K^v f \,  P_N^v g - P_{\gg N}^v P_K^v f \, P_N^v g \big) \\
=&  P_{\ll N}^v P_K^v f P_N^v g - P_K^v f \parallv P_N^v g +  P_{\gg N}^v P_K^v f \, P_N^v g - P_K^v f \paraggv P_N^v g, 
 \end{align*}
the estimates \eqref{prep:eq-localized-p1} and \eqref{prep:eq-localized-p2} yield \eqref{prep:eq-localized-2}.
\end{proof}

\subsection{Integrals and traces} 
In this subsection, our main goal is to understand the Duhamel integral of the linear wave equation in null-coordinates. As we will see in Proposition \ref{prep:prop-duhamel} below, the Duhamel integral consists of both integral and trace operators. Before addressing combinations of integral and trace operators, we start by considering them separately. \newline

As is clear from the wave equation \eqref{intro:eq-WM-NC}, the Duhamel integral acts on both $u$ and $v$-variables. Nevertheless, we start by analyzing integral operators in a single variable. 
\begin{definition}[Single-variable integral operator]
For all functions $f\in \C_c^\infty(\R)$, we define the integral
\begin{equation}
\I f (x) := \int_{0}^x \dy ~  f(y) . 
\end{equation}
\end{definition}

The following lemma proves that the integral $\I$ gains one derivative in Hölder spaces. 

\begin{lemma}[Single-variable integral estimate]\label{prep:lemma-single-integral}
For all $f\in C^\infty_b(\R)$,  $\widetilde{\chi} \in C_c^\infty(\R)$, and $\gamma\in (0,\infty)$, it holds that 
\begin{equation}
\big \| \I [ \widetilde{\chi} f ] \big\|_{\C^\gamma} \lesssim_{\widetilde{\chi},\gamma} \| f \|_{\C^{\gamma-1}}. 
\end{equation}
\end{lemma}

\begin{remark}
The cut-off function $\widetilde{\chi}$, which will later be chosen as a fattened version of $\chi$ from \eqref{prep:eq-chi}, should be seen as a technical crutch. Since the wave maps equation exhibits finite speed of propagation, it is always possible to reduce to a compact set. 
\end{remark}

Of course, the gain of derivatives through integration is standard. For a proof of Lemma \ref{prep:lemma-single-integral} in the H\"{o}lder-spaces from Definition \ref{prep:def-spaces}, we refer to \cite[Lemma A.10]{GIP15}.

We now define the integral operators $\I_u$ and $\I_v$, which act on the $u$ and $v$-variables, respectively. 

\begin{definition}[Partial integral operators]
For all functions $f\in C^\infty_c(\R_{u,v}^{1+1})$, we define the partial integral operators by
\begin{align}
\I_u f(u,v) &:= \int_{0}^u \du^\prime f(u^\prime,v), \\
\I_v f(u,v) &:= \int_{0}^v \dv^\prime f(u,v^\prime). 
\end{align}
\end{definition}

The single-variable integral estimate (Lemma \ref{prep:lemma-single-integral}) directly yields estimates for the partial integral operators in our Bourgain-type space $\Cprod{\alpha}{\beta}$.

\begin{lemma}[Partial integral estimate]\label{prep:lemma-partial-integral}
For all $f\in C^\infty(\R_{u,v}^{1+1})$, $\widetilde{\chi}\in C^\infty_c(\R_{u,v}^{1+1})$, $\alpha \in (0,\infty)$, and $\beta\in \R$, it holds that
\begin{equation}\label{prep:eq-partial-integral-1}
\big\| \I_u \big[ \widetilde{\chi} f\big] \big\|_{\Cprod{\alpha}{\beta}} \lesssim_{\alpha,\beta,\widetilde{\chi}} \| f \|_{\Cprod{\alpha-1}{\beta}}.
\end{equation}
If instead $\alpha \in \R$ and $\beta \in (0,\infty)$, then
\begin{equation}\label{prep:eq-partial-integral-2}
\big\| \I_v \big[ \widetilde{\chi} f\big] \big\|_{\Cprod{\alpha}{\beta}} \lesssim_{\alpha,\beta,\widetilde{\chi}} \| f \|_{\Cprod{\alpha}{\beta-1}}. 
\end{equation}
\end{lemma}

\begin{proof}
By symmetry in $u$ and $v$, it suffices to prove \eqref{prep:eq-partial-integral-1}. This estimate directly follows from the single-variable integral estimate from Lemma~\ref{prep:lemma-single-integral} and the commutativity of $\I_u$ and the Littlewood-Paley operators $P_N^v$. 
\end{proof}

This completes the analysis of integral operators and we now turn to trace operators. As mentioned above, the trace operators are examined due to their appearance in the Duhamel integral (see Proposition \ref{prep:prop-duhamel}). 

\begin{definition}[Trace operators]
For any $f\in C^\infty(\R_{u,v}^{1+1})$, we define the trace operator $\Tr$ by 
\begin{equation*}
\Tr f(x) := f(x,x).
\end{equation*}
Furthermore, we define the trace operators $\Tr_u$ and $\Tr_v$ by 
\begin{align}
\Tr_u f(u,v) := f(u,u) \qquad \text{and} \qquad \Tr_v f(u,v) := f(v,v). 
\end{align}
\end{definition}

We now analyze the mapping properties of $\Tr\colon \Cprod{\alpha}{\beta}\rightarrow \C_x^\gamma$. 

\begin{lemma}[Trace estimate]\label{prep:lemma-trace}
Let $\alpha,\beta \in \R\backslash \{0\}$, $\gamma\in \R$, and $f\in \Cprod{\alpha}{\beta}$.  Then, the estimate 
\begin{equation*}
    \| \Tr f \|_{\C_x^\gamma} \lesssim \| f \|_{\Cprod{\alpha}{\beta}}.
    \end{equation*}
    is satisfied under either of the following two conditions: 
\begin{enumerate}[label={(\roman*)}]
    \item\label{prep:item-trace-general} (General) It holds that $\gamma \leq \min(\alpha,\beta)$ and $\alpha+\beta>0$.
    \item\label{prep:item-trace-nonres} (Non-resonant) It holds that $\gamma \leq \min(\alpha,\beta,\alpha+\beta)$ and  $P_M^u P_N^v f = 0$  for  all $M\sim N$.
\end{enumerate}
\end{lemma}

\begin{proof}
Using a Littlewood-Paley decomposition, we have that
\begin{align*}
&\| \Tr f(x) \|_{\C_x^\gamma} \\
=& \sup_{K} K^\gamma \| \big(P_K^x \Tr f \big)(x) \|_{L^\infty_x} \\
\leq& \sup_{K} K^\gamma \sum_{M,N}  \Big\| \Big(P_K^x \Tr\big( P_M^u P_N^v f \big)\Big)(x) \Big\|_{L^\infty_x} \\
\lesssim& \sup_K \sum_{M,N}  \Big[
 \Big(  1\{ K \hspace{-2pt}\sim \hspace{-2pt} M \hspace{-2pt} \gg  \hspace{-2pt} N \}+  1\{ K \hspace{-2pt} \sim \hspace{-2pt} N \hspace{-2pt} \gg \hspace{-2pt} M \} 
+ 1\{ M \hspace{-2pt}\sim \hspace{-2pt} N \hspace{-2pt}\gtrsim  \hspace{-2pt} K \} \Big) K^{\gamma}  \| P_M^u P_N^v f(u,v) \|_{L^\infty_u L^\infty_v}  \Big] \\
\lesssim& \sup_K \sum_{M,N}  \Big[
 \Big(  1\{ K \hspace{-2pt}\sim \hspace{-2pt} M \hspace{-2pt} \gg  \hspace{-2pt} N \}+  1\{ K \hspace{-2pt} \sim \hspace{-2pt} N \hspace{-2pt} \gg \hspace{-2pt} M \} 
+ 1\{ M \hspace{-2pt}\sim \hspace{-2pt} N \hspace{-2pt}\gtrsim  \hspace{-2pt} K \} \Big) K^{\gamma}  M^{-\alpha} N^{-\beta} \Big] \| f \|_{\Cprod{\alpha}{\beta}}. 
\end{align*}
This sum has been previously estimated in \eqref{prep:eq-sum-scales}, and hence this completes the proof in case \ref{prep:item-trace-general}. In case \ref{prep:item-trace-nonres}, which no longer requires $\alpha+\beta>0$, the estimate follows from the argument leading to Proposition \ref{prep:prop-bilinear}.\ref{prep:item-nonres}. 
\end{proof}

Using the trace estimate, we now prove that the $\Cprod{\alpha}{\beta}$-norms control the usual $\C_t^0 \C_x^\gamma$ and $\C_t^1 \C_x^{\gamma-1}$-norms in Cartesian coordinates.

\begin{corollary}[From null to Cartesian coordinates]\label{prep:cor-null-Cartesian}
Let $f\in C^\infty_b(\R^{1+1}_{t,x})$ and $\widetilde{f}\in C^\infty_b(\R^{1+1}_{u,v})$ satisfy \begin{equation}\label{prep:eq-f-ftil}
f(t,x)=\widetilde{f}(x-t,x+t).
\end{equation}
Furthermore, let $\alpha,\beta \in \R\backslash \{0\}$ and let $\gamma \in \R$. Then, the estimate
\begin{equation}\label{prep:eq-f-ftil-estimate}
\| f \|_{\C_t^0 \C_x^\gamma} + \| f \|_{\C_t^1 \C_x^{\gamma-1}} \lesssim \| \widetilde{f} \|_{\Cprod{\alpha}{\beta}}. 
\end{equation}
is satisfied under either of the following two conditions:
\begin{enumerate}[label={(\roman*)}]
    \item\label{prep:item-cnc-general} (General) It holds that $\gamma \leq \min(\alpha,\beta)$ and $\alpha+\beta>1$.
    \item\label{prep:item-cnc-nonres} (Non-resonant) It holds that $\gamma \leq \min(\alpha,\beta,\alpha+\beta)$ and  $P_M^u P_N^v \widetilde{f} = 0$  for  all $M\sim N$.
\end{enumerate}

\end{corollary}

\begin{proof}
In order to utilize our trace estimate (Lemma \ref{prep:lemma-trace}), we define the shift operator $\Lambda_t$ by 
\begin{equation}\label{prep:eq-Lambda}
\big(\Lambda_t \widetilde{f})(u,v) = \widetilde{f}(u-t,v+t). 
\end{equation}
Due to the translation invariance of the Littlewood-Paley operators, $\Lambda_t$ preserves the $\Cprod{\alpha}{\beta}$-norm. From \eqref{prep:eq-f-ftil} and \eqref{prep:eq-Lambda}, it follows that 
\begin{equation*}
f(t,x) = \big(\Tr \Lambda_t \widetilde{f} \, \big)(x) \qquad \text{and} \qquad 
\partial_t f(t,x) = \big(\Tr \partial_v \Lambda_t \widetilde{f} \, \big)(x) - \big(\Tr \partial_u \Lambda_t \widetilde{f} \, \big)(x).
\end{equation*}
Using Lemma \ref{prep:lemma-trace}, we obtain under either of the two conditions \ref{prep:item-cnc-general} or \ref{prep:item-cnc-nonres} that 
\begin{align*}
\| f \|_{\C_t^0 \C_x^\gamma} + \| f \|_{\C_t^1 \C_x^{\gamma-1}} 
\lesssim& \sup_{t\in \R} \Big( \| \Tr \Lambda_t \widetilde{f} \|_{\C_x^\gamma}
    + \| \Tr \partial_u \Lambda_t \widetilde{f} \|_{\C_x^{\gamma-1}}
    + \| \Tr \partial_v \Lambda_t \widetilde{f} \|_{\C_x^{\gamma-1}} \Big) \\
\lesssim & \sup_{t\in \R} \Big( \| \Lambda_t \widetilde{f}  \|_{\Cprod{\alpha}{\beta}}
+ \| \partial_u \Lambda_t \widetilde{f}  \|_{\Cprod{\alpha-1}{\beta}}
+ \| \partial_v \Lambda_t \widetilde{f}  \|_{\Cprod{\alpha}{\beta-1}} \Big) \\
\lesssim& \| \widetilde{f} \|_{\Cprod{\alpha}{\beta}}.
\end{align*}
This completes the proof of \eqref{prep:eq-f-ftil-estimate}.
\end{proof}

Equipped with both integral and trace estimates, we now turn to the Duhamel integral.

\begin{proposition}[Duhamel integral]\label{prep:prop-duhamel}
Let $F\in \C^\infty_b(\R_{u,v}^{1+1})$ and let $\phi\in C^\infty_b(\R_{u,v}^{1+1})$ be a solution of the inhomogeneous linear wave equation 
\begin{equation}
\begin{cases}
\partial_u \partial_v \phi = F \qquad (u,v) \in \R^{1+1}, \\
\phi\big|_{u=v}=0, ~ (\partial_v - \partial_u ) \phi\big|_{u=v}=0. 
\end{cases}
\end{equation}
Then, the solution $\phi$ can be written as 
\begin{equation}\label{prep:eq-duhamel-1}
\phi(u,v) = \Duh[F] := - \int_u^v \dv^\prime \int_u^{v^\prime} \du^\prime F(u^\prime,v^\prime) = 
-\int_u^v \du^\prime \int_{u^\prime}^{v} \dv^\prime F(u^\prime,v^\prime). 
\end{equation}
Equivalently, the solution $\phi$ is also given by 
\begin{equation}\label{prep:eq-duhamel-2}
\phi = \big( \I_v - \Tr_u \I_v \big) \big( \I_u - \Tr_v \I_u\big) \big( F \big) = 
\big( \I_u - \Tr_v \I_u \big) \big( \I_v - \Tr_u \I_v\big) \big(  F \big). 
\end{equation}
In addition, for $\chi^+, \,  \chi^-$ defined in \eqref{prep:eq-chi-pm}, we have the estimate  
\begin{equation}\label{prep:eq-duhamel-3}
\| \chi^+ \chi^- \phi \|_{\Cprod{\gamma_1}{\gamma_2}} \lesssim_{\gamma_1,\gamma_2} \| F \|_{\Cprod{\gamma_1-1}{\gamma_2-1}}
\end{equation}
for all $\gamma_1,\gamma_2\in (0,1)$ satisfying $\gamma_1+\gamma_2>1$. 
\end{proposition}

\begin{proof}
The identities \eqref{prep:eq-duhamel-1} and \eqref{prep:eq-duhamel-2} follow directly from the fundamental theorem of calculus. In order to prove the estimate \eqref{prep:eq-duhamel-3}, we let $\widetilde{\chi} \in C^\infty_c(\R)$ be a fattened version of $\chi$ and define
\begin{equation*}
\widetilde{\chi}^+(u,v):= \widetilde{\chi}(u) \qquad \text{and} 
\qquad \widetilde{\chi}^-(u,v) := \widetilde{\chi}(v). 
\end{equation*}
Due to the explicit representation \eqref{prep:eq-duhamel-1}, it holds that
\begin{equation*}
    \chi^+ \chi^- \phi  = \chi^+ \chi^- \Duh[F]
    = \chi^+ \chi^- \Duh[ \widetilde{\chi}^+ \widetilde{\chi}^- F]. 
\end{equation*}
Using this observation, the estimate \eqref{prep:eq-duhamel-3} follows from the integral estimates (Lemma \ref{prep:lemma-partial-integral}) and the trace estimate (Lemma \ref{prep:lemma-trace}). The condition $\gamma_1+\gamma_2>1$ is needed in order to use the trace estimate for $I_u (\widetilde{\chi} F)$.
\end{proof}

We finish this subsection on integrals with a commutator estimate. To this end, we first make the following definition.

\begin{definition}[Integral commutator]\label{prep:def-integral-commutator}
For all $f,g\in \C_b^\infty(\R^{1+1}_{u,v})$, we define
\begin{equation}\label{prep:eq-integral-commutator}
\Com_{\chi,\parall, \I}^v(f,g) := \chi^+ \chi^- \Big( \I_v[f \parallv \partial_v g](u,v) - \I_v[f\parallv \partial_v g](u,u) - f \parallv g \Big).
\end{equation}
\end{definition}

\begin{lemma}[Integrals of low$\times$high-terms]\label{modulation:lem-integral-lh}
Let $\zeta_{N,n}\colon \R^{1+1}_{u,v}\rightarrow \R^\dimA$ be a family of functions, let $\phi^-\colon \R_v \rightarrow \R^\dimA$, and let $\phi^-_N:=P_N^v \phi^-$. Then, it holds that 
\begin{equation*}
\Big\| \sum_N \Com_{\chi,\parall,\I}^v(\zeta_{N,n}, \phi^{-,n}_N) \Big\|_{\Cprod{s}{r}} 
\lesssim \sup_N \big( \| \zeta_N \|_{\Cprod{s}{s}} \big) \| \phi^- \|_{\C_v^s}.
\end{equation*}
\end{lemma}

A similar estimate is contained in \cite[Lemma B.2]{GIP15}.

\begin{proof}
We first rewrite the expression on the right-hand side of \eqref{prep:eq-integral-commutator}. Using the product formula for $\partial_v$, we obtain that
\begin{align}
&\, \I_v \big[ \zeta_{N,n} \parallv \partial_v \phi^{-,n}_N \big](u,v) 
- \I_v \big[\zeta_{N,n} \parallv \partial_v \phi^{-,n}_N \big](u,u) - 
\big( \zeta_{N,n} \parallv \phi^{-,n}_N \big)(u,v) \notag \\ 
=&\,  \I_v \big[  \zeta_{N,n} \parallv  \partial_v \phi^{-,n}_N \big](u,v^\prime) \Big|_{v^\prime=u}^v  -  \big( \zeta_{N,n} \parallv  \phi^{-,n}_N \big)(u,v) \notag \\
=&\, \I_v \big[ \partial_v \big(  \zeta_{N,n} \parallv  \phi^{-,n}_N \big) \big](u,v^\prime) \Big|_{v^\prime=u}^v 
- \I_v \big[ \partial_v  \zeta_{N,n} \parallv  \phi^{-,n}_N  \big](u,v^\prime) \Big|_{v^\prime=u}^v 
-  \big( \zeta_{N,n} \parallv  \phi^{-,n}_N \big)(u,v).
\label{modulation:eq-integral-lh-q1}
\end{align}
Due to the definition of the integral operator $\I_v$, the combined contribution of the first and third summand in \eqref{modulation:eq-integral-lh-q1} equals
\begin{align*}
&\, \big( \zeta_{N,n} \parallv  \phi^{-,n}_N \big)(u,v) 
 - \big( \zeta_{N,n} \parallv  \phi^{-,n}_N \big)(u,u)
- \big( \zeta_{N,n} \parallv  \phi^{-,n}_N \big)(u,v) \\ 
=&\,  - \big( \zeta_{N,n} \parallv  \phi^{-,n}_N \big)(u,u)
= - \Tr_u \big( \zeta_{N,n} \parallv \phi^{-,n}_N \big)(u). 
\end{align*}
Using the definition of $\Com_{\chi,\parall,\I}^v$, we therefore obtain that
\begin{equation}\label{modulation:eq-integral-lh-p1}
\begin{aligned}
&\Com_{\chi,\parall,\I}^v(\zeta_{N,n}, \phi^{-,n}_N)(u,v) \\
=&-\chi^+ \chi^- \Tr_u \big( \zeta_{N,n} \parallv \phi^{-,n}_N \big)(u) -
 \chi^+ \chi^- \Big( \I_v\big[ \partial_v \zeta_{N,n}\parallv \phi^{-,n}_N\big](u,v^\prime)\Big|_{v^\prime=u}^v\Big).  
\end{aligned}
\end{equation}
We start by estimating the first summand in \eqref{modulation:eq-integral-lh-p1}. Using Lemma \ref{prep:lemma-trace} and Corollary \ref{prep:cor-bilinear}, we obtain that  
\begin{align*}
\Big\| \sum_N \Tr_u \big( \zeta_{N,n} \parallv \phi^{-,n}_N \big)(u,v) \Big\|_{\Cprod{s}{r}} 
&=\Big\| \sum_N \Tr_u \big( \zeta_{N,n} \parallv \phi^{-,n}_N \big)(u) \Big\|_{\C_u^s} \\
&\leq \Big\| \sum_N \zeta_{N,n} \parallv \phi^{-,n}_N \Big\|_{\Cprod{s}{s}} \\
&\lesssim \sup_N \big( \| \zeta_N \|_{\Cprod{s}{s}} \big) \|\phi^{-}\|_{\C_v^s}.
\end{align*}
We now turn to the second summand in \eqref{modulation:eq-integral-lh-p1}. Using  Proposition \ref{prep:prop-bilinear}.\ref{prep:item-nonres}, we have that
\begin{align*}
\big\| \partial_v \zeta_{N,n} \parallv \phi^{-,n}_N \big\|_{\Cprod{s}{r-1}} 
&\lesssim  \big\| P_{\lesssim N}^v \partial_v \zeta_{N} \big\|_{\Cprod{s}{\eta}}
\big\| \phi^{-}_N \big\|_{\C_v^{r-1}} \\
&\lesssim N^{1-s+\eta} N^{r-1-s} \big\| \zeta_N \big\|_{\Cprod{s}{s}} \| \phi^- \|_{\C_v^s}.
\end{align*}
Since $r-2s\approx -1/4$, this contribution is summable in $N$. The desired bound then follows from Lemma \ref{prep:lemma-partial-integral}.
\end{proof}

\subsection{Proof of Theorem \ref{intro:thm-deterministic}.\ref{intro:item-determininistic-well}}
Before we can proceed with the proof of Theorem \ref{intro:thm-deterministic}.\ref{intro:item-determininistic-well}, we require the following time-localization lemma. This lemma will only be used in the deterministic theory, since in the random theory a scaling argument is more convenient.

\begin{lemma}[Time-localization]\label{prep:lem-time-localization}
Let $\gamma_1^\prime,\gamma_1,\gamma_2\in (-1/2,0)$ satisfy $\gamma_1^\prime<\gamma_1$, let $\widetilde{\chi} \in \C^\infty_c(\R)$, and let $\tau>0$. Then, it holds that
\begin{equation}
\Big\| \widetilde{\chi}\big( \tfrac{u-v}{\tau} \big) F \Big\|_{\Cprod{\gamma_1^\prime}{\gamma_2}} \lesssim \tau^{\gamma_1-\gamma_1^\prime} \big\| F \big\|_{\Cprod{\gamma_1}{\gamma_2}}.
\end{equation}
\end{lemma}

\begin{proof}[Proof of Lemma \ref{prep:lem-time-localization}]
 Throughout this proof only, we define
 \begin{equation}\label{prep:eq-Wspace}
 \big\| G \big\|_{\bW_u^{\gamma_1}\bW_v^{\gamma_2}} := \sum_{K_1,K_2} K_1^{\gamma_1} K_2^{\gamma_2} \| P^u_{K_1} P^v_{K_2} G \|_{L^1_u L^1_v}. 
 \end{equation}
 By duality\footnote{While the dual space of $L^\infty_x$ is not $L^1_x$, one can still characterize the $L^\infty_x$-norm as a supremum over integrals against $L^1_x$-normalized functions. This statement generalizes to our functions spaces and is sufficient for the duality argument used here.}, it suffices to prove for all $\gamma_1^\prime,\gamma_1,\gamma_2\in (0,1/2)$ satisfying $\gamma_1^\prime<\gamma_1$ that
 \begin{equation}\label{prep:eq-tlocal-1}
     \big\| \widetilde{\chi}\big( \tfrac{u-v}{\tau} \big) G \big\|_{\bW_u^{\gamma_1^\prime}\bW_v^{\gamma_2}} \lesssim  \tau^{\gamma_1-\gamma_1^\prime}
      \big\| G \big\|_{\bW_u^{\gamma_1}\bW_v^{\gamma_2}}.
 \end{equation}
 We now write $\widetilde{\chi}_\tau(\cdot):= \widetilde{\chi}(\cdot/\tau)$ and decompose
 $\widetilde{\chi}_\tau = \sum_M P_M \widetilde{\chi}_\tau$. 
 For the dyadic components, we have the pointwise estimate 
 \begin{equation*}
  |P_M \widetilde{\chi}_\tau|(y) \lesssim 1\{ M \lesssim \tau^{-1} \} M \tau \langle M y \rangle^{-10} + 1 \{ M \gg \tau^{-1} \} M^{-10} \langle  \tau^{-1} y \rangle^{-10}
 \end{equation*}
for all $y \in \R$. As a result, it follows that 
\begin{equation}\label{prep:eq-tlocal-2}
\| P_M \widetilde{\chi}_\tau \|_{L^1_y} \lesssim \tau^{\gamma_1-\gamma_1^\prime} M^{\gamma_1-\gamma_1^\prime-1} \qquad \text{and} \qquad
\| P_M \widetilde{\chi}_\tau \|_{L^\infty_y} \lesssim \tau^{\gamma_1-\gamma_1^\prime} M^{\gamma_1-\gamma_1^\prime}. 
\end{equation}
Furthermore, we decompose $G=\sum_{K_1,K_2} P^u_{K_1} P^v_{K_2} G$. By inserting both decompositions into the left-hand side of \eqref{prep:eq-tlocal-1}, we obtain that
\begin{equation}\label{prep:eq-tlocal-5}
\begin{aligned}
    &\big\| \widetilde{\chi}\big( \tfrac{u-v}{\tau} \big) G \big\|_{\bW_u^{\gamma_1^\prime}\bW_v^{\gamma_2}} \\
    \lesssim& 
    \sum_{M,K_1,K_2} \big\| P_M\widetilde{\chi}_\tau(u-v)  \, P^u_{K_1} P^v_{K_2} G \big\|_{\bW_u^{\gamma_1^\prime}\bW_v^{\gamma_2}} \\
    \lesssim& \sum_{M,K_1,K_2} \max(M,K_1)^{\gamma_1^\prime} \max(M,K_2)^{\gamma_2}
    \big\| P_M\widetilde{\chi}_\tau(u-v)  \, P^u_{K_1} P^v_{K_2} G \big\|_{L^1_u L^1_v}. 
\end{aligned}
\end{equation}
Using Hölder's inequality and \eqref{prep:eq-tlocal-2}, we have the estimate
\begin{equation}\label{prep:eq-tlocal-3}
\begin{aligned}
\big\| P_M\widetilde{\chi}_\tau(u-v)  \, P^u_{K_1} P^v_{K_2} G \big\|_{L^1_u L^1_v} 
&\lesssim 
\|P_M\widetilde{\chi}_\tau \|_{L^\infty_y} \big\| P^u_{K_1} P^v_{K_2} G \big\|_{L^1_u L^1_v} \\
&\lesssim \tau^{\gamma_1-\gamma_1^\prime} M^{\gamma_1-\gamma_1^\prime} \big\| P^u_{K_1} P^v_{K_2} G \big\|_{L^1_u L^1_v}. 
\end{aligned}
\end{equation}
Using Hölder's inequality, \eqref{prep:eq-tlocal-2}, and Bernstein's estimate, we also have that
\begin{equation}\label{prep:eq-tlocal-4}
\begin{aligned}
\big\| P_M\widetilde{\chi}_\tau(u-v)  \, P^u_{K_1} P^v_{K_2} G \big\|_{L^1_u L^1_v} 
&\lesssim 
\|P_M\widetilde{\chi}_\tau \|_{L^1_y} \min\Big(\big\| P^u_{K_1} P^v_{K_2} G \big\|_{L^1_u L^\infty_v}, \big\| P^u_{K_1} P^v_{K_2} G \big\|_{L^1_v L^\infty_u} \Big) \\
&\lesssim \tau^{\gamma_1-\gamma_1^\prime} M^{\gamma_1-\gamma_1^\prime-1} \min(K_1,K_2) \big\| P^u_{K_1} P^v_{K_2} G \big\|_{L^1_u L^1_v}. 
\end{aligned}
\end{equation}
After combining \eqref{prep:eq-tlocal-5}, \eqref{prep:eq-tlocal-3}, and \eqref{prep:eq-tlocal-4}, it remains to prove that
\begin{align*}
\sup_{K_1,K_2} \sum_M  M^{\gamma_1-\gamma_1^\prime-1} \max(M,K_1)^{\gamma_1^\prime} \max(M,K_2)^{\gamma_2} \min(M,K_1,K_2) K_1^{-\gamma_1} K_2^{-\gamma_2} \lesssim 1.
\end{align*}
This dyadic sum estimate easily follows by distinguishing the four cases $M\lesssim K_1,K_2$, $K_1\lesssim M \lesssim K_2$, $K_2 \lesssim M \lesssim K_1$, and $M\gtrsim K_1,K_2$.
\end{proof}

Equipped with the time-localization lemma, we now turn to the proof of deterministic well-posedness at regularities $r>1/2$. To be more precise, let us elaborate on the notion of well-posedness. Let $\Theta>0$ and let $0<\tau\ll_r \Theta^{-1/(r-1/2)}$. Then, we prove for all initial data $(\phi_0,\phi_1)\colon \R \rightarrow T \M$ satisfying
\begin{equation}\label{prep:eq-det-data}
    \| \phi_0 \|_{\C_x^r}, \| \phi_1 \|_{\C_x^{r-1}} \leq \Theta
\end{equation}
that there exists a (conditionally) unique solution $\phi$ of $\eqref{intro:eq-WM}$ in
\begin{equation*}
\big( \C_t^0 \C_x^r \times \C_t^1 \C_x^{r-1}\big)\big([-\tau,\tau] \times \R \rightarrow T \M\big). 
\end{equation*}
In addition, we prove the continuous dependence of $\phi$ on the initial datum $(\phi_0,\phi_1)$. 

\begin{proof}[Proof of  Theorem \ref{intro:thm-deterministic}.\ref{intro:item-determininistic-well}]
We start the argument with a few standard reductions. Throughout the proof, we can neglect the geometric constraint $\phi(t,x)\in \M$. It can be recovered a-posteriori from the well-posedness theory for smooth initial data and the continuous dependence on the initial data. Due to translation invariance and finite speed of propagation, it is possible to insert  space-time truncations to $|t|\lesssim \tau$ and $|x|\lesssim 1$ into the equation. Finally, Corollary \ref{prep:cor-null-Cartesian} and the condition $r>1/2$ allow us to argue entirely in null coordinates. In total, these reductions lead to the fixed-point problem 
\begin{equation}\label{prep:eq-WM-det-2}
\phi^k(u,v)= (\chi^+ \phi^{+,k})(u) + (\chi^- \phi^{-,k})(v) - \chi^+(u) \chi^-(v)
\Duh\Big[  \chi\big((v-u)/\tau\big) \Second^{k}_{ij}(\phi) 
\partial_u \phi^i \partial_v \phi^j \Big](u,v),
\end{equation}
where $\phi^+$ and $\phi^-$ are the right and left-moving linear waves, respectively.
We now solve \eqref{prep:eq-WM-det-2} using a contraction argument. In order to later gain a power of $\tau$, we introduce the parameter $r^\prime:= (1/2+r)/2$, which is between $1/2$ and $r$. Using this new parameter, we define the norm
\begin{equation*}
\| \phi\|_{\mathcal{S}} = \| \phi \|_{\Cprod{r}{r^\prime}} + \| \phi \|_{\Cprod{r^\prime}{r}}.
\end{equation*}
We recall from \eqref{prep:eq-det-data} that $\Theta>0$ denotes the size of the initial data. For a constant $C=C(r)$, which remains to be chosen, we  define the ball of radius $C\Theta$ by 
\begin{equation*}
    \mathcal{S}_{C\Theta} := \big\{ \phi \colon \| \phi \|_{\Sc} \leq C \Theta \}.
\end{equation*}
Finally, we define a map $\Gamma$, which encodes the right-hand side of \eqref{prep:eq-WM-det-2}, by 
\begin{equation}\label{prep:eq-WM-det-3}
    (\Gamma \phi)^k(u,v) := (\chi^+ \phi^{+,k})(u) + (\chi^- \phi^{-,k})(v) - \chi^+ \chi^-
\Duh\Big[  \chi\big((v-u)/\tau\big) \Second^{k}_{ij}(\phi) 
\partial_u \phi^i \partial_v \phi^j \Big](u,v). 
\end{equation}
In order to complete the proof, it remains to prove that $\Gamma$ is a contraction on $\Sc_{C\Theta}$ and that the resulting fixed-point depends continuously on $(\phi_0,\phi_1)$. Since the contraction property and continuous dependence follow from similar arguments, we only prove that  $\Gamma$ maps $\Sc_{C\Theta}$ back into itself. To this end, we let $\phi \in \Sc_{C\Theta}$ be arbitrary and prove that $\| \Gamma \phi\|_{\Sc}\leq C\Theta$. By symmetry, it suffices to prove that 
\begin{equation}\label{prep:eq-WM-det-1}
\|  \Gamma \phi \|_{\Cprod{r^\prime}{r}} \leq C\Theta/2. 
\end{equation} 
For the linear waves in \eqref{prep:eq-WM-det-3}, it follows directly from the definitions that
\begin{equation}\label{prep:eq-WM-det-4}
\| \chi^+(u) \phi^+(u) \|_{\Cprod{r^\prime}{r}} + \| \chi^-(v) \phi^-(v) \|_{\Cprod{r^\prime}{r}} \lesssim \| \phi_0 \|_{\C_x^r} + \|\phi_1 \|_{\C_x^{r-1}} \lesssim \Theta. 
\end{equation}
For the nonlinear term, it follows from Proposition \ref{prep:prop-duhamel} and Lemma \ref{prep:lem-time-localization} that
\begin{equation}\label{prep:eq-WM-det-5}
\begin{aligned}
&\Big\| \chi^+ \chi^- \Duh\Big[  \chi\big((v-u)/\tau\big) \Second^{k}_{ij}(\phi) 
\partial_u \phi^i \partial_v \phi^j \Big] \Big\|_{\Cprod{r^\prime}{r}} \\ 
\lesssim& \Big\| \chi\big((v-u)/\tau\big) \Second^{k}_{ij}(\phi) 
\partial_u \phi^i \partial_v \phi^j \Big\|_{\Cprod{r'-1}{r-1}} \\
\lesssim& \tau^{r-r^\prime} \big\| \Second^{k}_{ij}(\phi) 
\partial_u \phi^i \partial_v \phi^j \big\|_{\Cprod{r-1}{r-1}}. 
\end{aligned}
\end{equation}
Using the multiplication estimate (Corollary \ref{prep:corollary-multiplication}), the bilinear estimate (Proposition \ref{prep:prop-bilinear}), and the composition estimate (Lemma \ref{prep:lemma-bony}), we obtain that 
\begin{equation}\label{prep:eq-WM-det-6}
\big\| \Second^{k}_{ij}(\phi) 
\partial_u \phi^i \partial_v \phi^j \big\|_{\Cprod{r-1}{r-1}} 
\lesssim \big\| \Second^{k}_{ij}(\phi) \big\|_{\Cprod{r^\prime}{r^\prime}} 
\big\| \partial_u \phi^i \big\|_{\Cprod{r-1}{r^\prime}} \big\| \partial_v \phi^j \big\|_{\Cprod{r^\prime}{r-1}} \lesssim (1+C\Theta)^{10} (C\Theta)^2.  
\end{equation}
By combining \eqref{prep:eq-WM-det-4}, \eqref{prep:eq-WM-det-5}, and \eqref{prep:eq-WM-det-6}, it follows that 
\begin{equation*}
    \big\|  \Gamma \phi \big\|_{\Cprod{r^\prime}{r}} \lesssim_{r,r^\prime} \Theta + \tau^{r-r^\prime} (1+C\Theta)^{10} (C\Theta)^2. 
\end{equation*}
The desired estimate \eqref{prep:eq-WM-det-1} now follows by first choosing $C=C(r)\geq 1$ sufficiently large and then choosing $0<\tau\ll_r \Theta^{-1/(r-1/2)}$ sufficiently small. \\
\end{proof}

Before the end of this subsection, we state a lemma which essentially follows from the previous proof. 

\begin{lemma}[Finite speed of propagation and uniqueness in $\Cprod{r}{r}$]\label{prep:lem-deterministic-uniqueness}
Let $\chi\in \C^\infty(\R)$ satisfy \eqref{prep:eq-chi} and let $\widetilde{\chi}\in \C^\infty_c(\R)$ satisfy $\widetilde{\chi}(x)=1$ for all $x\in [-4,4]$. Furthermore, let $\phi^+,\phi^- \colon \R_x \rightarrow \R^\dimA$ and $\phi,\widetilde{\phi}\colon \R_{u,v}^{1+1}\rightarrow \R^\dimA$ satisfy the following conditions:
\begin{enumerate}[label={(\roman*)},leftmargin=10mm]
    \item The linear waves satisfy $\phi^+,\phi^- \in \C^r_x$. 
    \item The maps $\phi$ and $\widetilde{\phi}$ locally belong to $\Cprod{r}{r}$, i.e., $\widetilde{\chi}^+\widetilde{\chi}^- \phi \in \Cprod{r}{r}$  and
    $\widetilde{\chi}^+\widetilde{\chi}^- \widetilde{\phi} \in \Cprod{r}{r}$.
    \item The map $\phi$ solves the Duhamel integral problem
    \begin{equation*}
        \phi^k(u,v) = \phi^{+,k}(u) + \phi^{-,k}(v) - \Duh \Big[ \Second^k_{ij}(\phi) \partial_u \phi^i \partial_v \phi^j \Big].
    \end{equation*}
      \item The map $\widetilde{\phi}$ solves the localized Duhamel integral problem
    \begin{equation*}
        \widetilde{\phi}^k(u,v) = \phi^{+,k}(u) + \phi^{-,k}(v) - \chi^+ \chi^- \Duh \Big[ \Second^k_{ij}(\widetilde{\phi}) \partial_u \widetilde{\phi}^i \partial_v \widetilde{\phi}^j \Big].
    \end{equation*}
\end{enumerate}
Then, it holds that $\phi(u,v)=\widetilde{\phi}(u,v)$ for all $u,v\in [-2,2]$. 
\end{lemma}

\begin{proof}
The lemma follows from the same estimates as in the proof of Theorem \ref{intro:thm-deterministic}.\ref{intro:item-determininistic-well} and a continuity argument. In the continuity argument, it is best to work with the usual localization of our norms, i.e., 
\begin{equation*}
\| \psi \|_{\Cprod{r}{r}([-\tau,\tau]^2)} := \inf \{ \| \zeta \|_{\Cprod{r}{r}} \colon  \zeta(u,v)=\psi(u,v) \, \text{ for all } u,v \in [-\tau,\tau]^2 \}. 
\end{equation*}
We leave the standard details to the reader.
\end{proof}

\subsection{Brownian paths} \label{section:Brownian}

In this subsection we introduce smooth approximations of the Brownian paths $B \colon \bbR \to \M$ and the white noise velocities $V \in B^\ast T\M$, and we establish quantitative approximation properties that will be needed for the proof of Theorem~\ref{intro:thm-rigorous}.

As explained in the introduction, we take an extrinsic approach to construct Brownian paths $B \colon \bbR \to \M$ on our submanifold $\M \hookrightarrow \mathbb{R}^\dimA$ embedded into the ambient Euclidean space $\mathbb{R}^\dimA$.
To this end we denote by $W \colon \bbR \to \M$ a $\dimA$-dimensional Euclidean Brownian motion defined on a probability space $(\Omega,\mathcal{F},\mathbb{P})$.
Moreover, we introduce for each $p \in \M$ the orthogonal projection $P(p) \colon \mathbb{R}^\dimA \to T_p \M$ from $\mathbb{R}^\dimA$ to the tangent space $T_p \M$, and we fix a reference point $B_0 \in \M$.
Following \cite[Chapter 3.2]{Hsu02} we then obtain a Brownian path $B \colon \bbR \to \M$ by solving the following Stratonovich stochastic differential equation on $\M$
\begin{equation} \label{prep:eq-def-B}
 d B(x) = P\bigl(B(x)\bigr) \circ d W(x), \quad B(0) = B_0 \in \M,
\end{equation}
which is driven by the $\dimA$-dimensional Euclidean Brownian motion $W(x)$\footnote{While it is customary to write $B_x$ or $W_x$ for Brownian motions indexed by the variable $x$, we use the notation $B(x)$, respectively $W(x)$, which we consider to be more in line with the notation in the rest of the paper.}.

\begin{remark}
 Intrinsically, Brownian motion on a Riemannian manifold can be defined as a diffusion process generated by half of the Laplace-Beltrami operator on the manifold. We refer to~\cite[Chapter 3]{Hsu02} for more background. For the extrinsic approach to obtain Brownian motion as a solution to a stochastic differential equation driven by an ambient Euclidean Brownian motion the Stratonovich formulation~\eqref{prep:eq-def-B} is key. Indeed, only the Stratonovich formulation preserves the classical chain rule, which is essential for proving that $B(x)\in \M$. In contrast, the solution to the It\^{o} stochastic differential equation 
 \begin{equation*}
    d X(x) = P(X(x)) d W(x), \quad X(0)=X_0 \in \M,
 \end{equation*}
 does not map into the manifold $\M$.
\end{remark}

 In order to define suitable smooth approximations of the Brownian path $B \colon \bbR \to \M$, we introduce a countable family of smooth functions $W^\varepsilon \colon \bbR \to \bbR^\dimA$, $\varepsilon>0$, that approximate the $\dimA$-dimensional ambient Brownian motion $W$. Specifically, we set 
 \begin{equation*}
  W^\varepsilon := P_{\leq \varepsilon^{-1}} W = K_\varepsilon \ast W, \quad \varepsilon := 2^{-k}, \quad k \in \mathbb{N},
 \end{equation*}
 where $K_\varepsilon(x) := \varepsilon^{-1} \widecheck{\rho}(\varepsilon^{-1} x)$. While we consider the variable $\varepsilon$ to be restricted to dyadic numbers, we do not make this more explicit in our notation.
 We obtain a corresponding family of smooth functions $B^\varepsilon \colon \bbR \to \M$, $\varepsilon > 0$, as solutions to the classical ordinary differential equations 
 \begin{equation} \label{prep:eq-def-Beps}
  \partial_x B^\varepsilon = P(B^\varepsilon) \partial_x W^\varepsilon, \quad B^\varepsilon(0) = B_0 \in \M.
 \end{equation}
 In view of the definition of the Brownian path $B \colon \mathbb{R} \to \M$ as a solution to the Stratonovich differential equation~\eqref{prep:eq-def-B}, the family of smooth functions $(B^\varepsilon)_{\varepsilon > 0}$ should provide good approximations of $B$ as $\varepsilon \to 0$. To quantify this we work with smooth local approximations.
 
For any $\tau > 0$ and $x_0 \in \bbR$, we first define the re-scaled and translated smooth approximation 
\begin{equation} \label{prep:equ-Bepstaux0_def}
 B^\varepsilon_{\tau, x_0}(x) := B^\varepsilon( \tau x+x_0 )
\end{equation}
and the associated re-scaled and translated approximation of the Euclidean Brownian motion
\begin{equation*}
 \begin{aligned}
  W^\varepsilon_{\tau, x_0}(x) := \tau^{-\frac12} \bigl( W^{\varepsilon}(\tau x+x_0) - W^\varepsilon(x_0) \bigr).
 \end{aligned}
\end{equation*}
From \eqref{prep:eq-def-Beps} it follows that $B^\varepsilon_{\tau, x_0} \colon \bbR \to \M$ satisfies the classical ordinary differential equation
\begin{equation*}
 \begin{aligned}
  \partial_x B^\varepsilon_{\tau, x_0} = \tau^{\frac12} P\bigl( B^\varepsilon_{\tau, x_0} \bigr) \partial_x W^\varepsilon_{\tau, x_0}, \quad B_{\tau, x_0}^\varepsilon(0) = B^\varepsilon_{\tau, x_0}(0).
 \end{aligned}
\end{equation*}
Next, we introduce localized versions of the smooth approximations $B^\varepsilon_{\tau,x_0}$. To this end, we let $\chi \in C_c^\infty(\bbR)$ be as in \eqref{prep:eq-chi}. Then we define $B^\varepsilon_{\tau, x_0,\loc} \colon \bbR \to \M$ as the solution to the classical ordinary differential equation 
\begin{equation} \label{prep:equ-Bepstaux0loc_def}
 \begin{aligned}
  \partial_x B^\varepsilon_{\tau, x_0, \loc} = \tau^{\frac12} P\bigl( B^\varepsilon_{\tau, x_0,\loc} \bigr) \partial_x \bigl( \chi(x) W^\varepsilon_{\tau, x_0} \bigr), \quad B_{\tau, x_0, \loc}^\varepsilon(0) = B^\varepsilon_{\tau, x_0}(0).
 \end{aligned}
\end{equation}

As a technical tool we will make use of the following moment bounds for weighted $L^\infty_x$ estimates of the real-valued Brownian motions $W^{j}_{\tau,x_0}$, $1 \leq j \leq \dimA$. 

\begin{lemma} \label{prep:lem-weighted_Linfty_BM}
 Let $\tau > 0$ and $x_0 \in \bbR$ be arbitrary.
 For any $\sigma > \frac12$ there exists a constant $C(\sigma) > 0$ such that for all $p \geq 2$ and for all $1 \leq j \leq D$ we have 
 \begin{equation} \label{prep:equ-weighted_Linfty_BM}
 \begin{aligned}
  \bigl\| \| \jap{x}^{-\sigma} W^j_{\tau, x_0} \|_{L^\infty_x} \bigr\|_{L^p_\omega(\Omega)} \leq C(\sigma) \sqrt{p}.
 \end{aligned}
 \end{equation}
\end{lemma}
\begin{proof}
 By scaling and translation invariance of Brownian motion, the laws of $W^j_{\tau, x_0}$ and $W^j$ are the same under $\mathbb{P}$. Correspondingly, it suffices to prove~\eqref{prep:equ-weighted_Linfty_BM} for $W^j$. 
 Using a dyadic decomposition of the line we obtain 
 \begin{equation*}
 \begin{aligned}
  \bigl\| \| \jap{x}^{-\sigma} W^j \|_{L^\infty_x} \bigr\|_{L^p_\omega} &\leq \Bigl\| \sup_{|x|\leq 1} \, \jap{x}^{-\sigma} |W^j(x)| \Bigr\|_{L^p_\omega} + \sum_{K \geq 1} \Bigl\| \sup_{K \leq |x| \leq 2K} \, \jx^{-\sigma} |W^j(x)| \Bigr\|_{L^p_\omega} \\
  &\lesssim \Bigl\| \sup_{|x|\leq 1} \, |W^j(x)| \Bigr\|_{L^p_\omega} + \sum_{K \geq 1} K^{-\sigma} \Bigl\| \sup_{0 \leq |x| \leq 2K} \, |W^j(x)| \Bigr\|_{L^p_\omega}.
 \end{aligned}
 \end{equation*}
 By Doob's maximal inequality~\cite[Theorem 1.3.8]{KarShr91} we can bound the preceding line by
 \begin{equation*}
 \begin{aligned}
  p' \|W^j(1)\|_{L^p_\omega} + \sum_{K \geq 1} K^{-\sigma} p' \|W^j(2K)\|_{L^p_\omega},
 \end{aligned}
 \end{equation*}
 where $p' = \frac{p}{p-1} \leq 2$ denotes the conjugate exponent of $p \geq 2$. Invoking the well-known moment bounds for standard Euclidean Brownian motion
 \begin{equation*}
 \begin{aligned}
  \|W^j(x)\|_{L^p_\omega(\Omega)} \lesssim \sqrt{p} |x|^{\frac12},
 \end{aligned}
 \end{equation*}
 we conclude uniformly for all $p \geq 2$ 
 \begin{equation*}
 \begin{aligned}
  \bigl\| \| \jap{x}^{-\sigma} W^j \|_{L^\infty_x} \bigr\|_{L^p_\omega} &\lesssim p' \sqrt{p} \Bigl( 1 + \sum_{K \geq 1} K^{-\sigma} K^{\frac12} \Bigr) \lesssim_{\sigma} \sqrt{p},
 \end{aligned}
 \end{equation*}
 which finishes the proof of the lemma.
\end{proof}

Next, we establish high$\times$high$\rightarrow$low-bounds for Euclidean Brownian motions. 

\begin{lemma}[High$\times$high$\rightarrow$low-bounds for Euclidean Brownian motion] \label{prep:lem-Weps_hhtolow}
Let $\tau > 0$, $x_0 \in \bbR$, and $0 < \alpha < \frac12$ be arbitrary.
For any $\lambda > 0$ there exists an event $\widetilde{\mathcal{E}}_\lambda(\tau, x_0)  \subseteq \Omega$ satisfying
\begin{equation} \label{prep:eq-Weps_hhtolow_event}
 \mathbb{P}\big( \widetilde{\mathcal{E}}_\lambda(\tau, x_0) \big) \geq 1 - c^{-1} \exp(-c\lambda)
\end{equation}
and on which the following two estimates hold:
\begin{enumerate}[label={(\roman*)}]
    \item\label{prep:item-Weps_hhtolow} Uniform high$\times$high$\rightarrow$low-bounds: For all $1\leq j_1, j_2 \leq D$, it holds that 
    \begin{equation} \label{prep:eq-Weps_hhtolow_unif_bound}
         \sup_{\varepsilon>0} \sup_{M\sim N} M^{\alpha} \big\| \chi(\cdot)^2 P_M W^{\varepsilon,j_1}_{\tau, x_0} \, \partial_x P_N W^{\varepsilon,j_2}_{\tau, x_0} \big\|_{\C^{\alpha-1}} \leq \lambda. 
    \end{equation}
    \item\label{prep:item-Weps_hhtolow-convergence} Convergence of high$\times$high$\rightarrow$low-term: For all $1\leq j_1, j_2 \leq D$, it holds that 
    \begin{equation} \label{prep:eq-Weps_hhtolow_conv}
     \begin{aligned}
       &\lim_{\varepsilon \to 0} \, \sup_{M\sim N} M^{\alpha} \big\| \chi(\cdot)^2 \bigl( P_M W^{\varepsilon,j_1}_{\tau, x_0} \, \partial_x P_N W^{\varepsilon,j_2}_{\tau, x_0} - P_M W^{j_1}_{\tau, x_0} \, \partial_x P_N W^{j_2}_{\tau, x_0} \bigr) \big\|_{\C^{\alpha-1}} = 0.
     \end{aligned}
    \end{equation}
\end{enumerate}
\end{lemma}
\begin{proof}
 We begin with the proof of the uniform high$\times$high$\rightarrow$low bounds \eqref{prep:eq-Weps_hhtolow_unif_bound}.
 Fix $1 \leq j_1, j_2 \leq D$. We consider the random variable 
 \begin{equation}
 \begin{aligned}
  X := \sup_{\varepsilon>0} \sup_{M\sim N} M^{\alpha} \big\| \chi(\cdot)^2 P_M W^{\varepsilon,j_1} \, \partial_x P_N W^{\varepsilon,j_2} \big\|_{\C^{\alpha-1}},
 \end{aligned}
 \end{equation}
 and define the event $\widetilde{\mathcal{E}}_\lambda(\tau, x_0) := \{ X \leq \lambda \}$.
 For all $p \geq 2$ we seek to establish the moment bounds $\|X\|_{L^p_\omega} \lesssim p$.
 By Chebyshev's inequality these moment bounds %~\eqref{prep:eq-XR_moment_bounds} 
 imply the tail estimate $\mathbb{P}(\widetilde{\mathcal{E}}_\lambda(\tau, x_0)^c) \leq c^{-1} e^{-c \lambda}$ for any $\lambda > 0$ for some absolute constant $c > 0$. 
 The main work now goes into establishing the moment bounds. We reduce their proof to a computation for random Fourier series representations of Brownian motions on unit-sized intervals. This reduction is achieved in several steps. 
 We have
 \begin{equation} \label{prep:eq-XR_moment_bounds_start}
 \begin{aligned}
  \|X\|_{L^p_\omega} \lesssim \sum_{M \sim N} M^{\alpha} \Bigl\| \sup_{\varepsilon>0} \, \big\| \chi(\cdot)^2 P_M W^{\varepsilon,j_1} \, \partial_x P_N W^{\varepsilon,j_2} \big\|_{\C^{\alpha-1}} \Bigr\|_{L^p_\omega}.
 \end{aligned}
 \end{equation}
 
 \medskip 
 \emph{Step 1: Estimating the low frequencies $M \sim N \lesssim 1$.} 
 We first dispense of the low frequency contributions to~\eqref{prep:eq-XR_moment_bounds_start}. Using the embedding $L^\infty \hookrightarrow \C^{\alpha-1}$ and H\"older's inequality, they can be estimated crudely by
 \begin{equation} \label{prep:eq-Weps_hhtolow_step1_1}
 \begin{aligned}
  &\sum_{M \sim N \lesssim 1} M^{\alpha} \Bigl\| \sup_{\varepsilon>0} \, \big\| \chi(\cdot)^2 P_M W^{\varepsilon,j_1} \, \partial_x P_N W^{\varepsilon,j_2} \big\|_{\C^{\alpha-1}} \Bigr\|_{L^p_\omega} \\
  &\lesssim \sum_{M \sim N \lesssim 1} M^{\alpha} \sum_{L, L' \sim M}  \Bigl\| \big\| \chi(\cdot)^2 P_M P_L W^{j_1} \, \partial_x P_N P_{L'} W^{j_2} \big\|_{L^\infty_x} \Bigr\|_{L^p_\omega} \\
  &\lesssim \sum_{M \sim N \lesssim 1} M^{\alpha} \sum_{L, L' \sim M}  \bigl\| \chi(\cdot) P_M P_L \jap{\cdot} \bigr\|_{L^\infty_x \to L^\infty_x} \bigl\| \| \jap{x}^{-1} W^{j_1} \|_{L^\infty_x} \bigr\|_{L^{2p}_\omega} \\
  &\qquad \qquad \qquad \qquad \quad  \times \bigl\| \chi(\cdot) \px P_N P_{L'} \jap{\cdot} \bigr\|_{L^\infty_x \to L^\infty_x} \bigl\| \| \jap{x}^{-1} W^{j_2} \|_{L^\infty_x} \bigr\|_{L^{2p}_\omega}. 
 \end{aligned}
 \end{equation}
 The following operator norm bounds hold uniformly for the low frequency configurations $M \sim N \lesssim 1$ and $L, L' \sim M$, 
 \begin{equation*}
 \begin{aligned}
  \bigl\| \chi(\cdot) P_M P_L \jap{\cdot} \bigr\|_{L^\infty_x \to L^\infty_x} + \bigl\| \chi(\cdot) \px P_N P_{L'} \jap{\cdot} \bigr\|_{L^\infty_x \to L^\infty_x} &\lesssim 1.
 \end{aligned}
 \end{equation*}
 Thus, by Lemma~\ref{prep:lem-weighted_Linfty_BM} the last line of~\eqref{prep:eq-Weps_hhtolow_step1_1} can be bounded by 
 \begin{equation*}
 \begin{aligned}
  &\sum_{M \sim N \lesssim 1} M^{\alpha} \sum_{L, L' \sim M} p \lesssim p,
 \end{aligned}
 \end{equation*}
 as desired.
 
 \medskip 
 \emph{Step 2: Reduction to a random Fourier series computation.}
 We begin to estimate the high-frequency contributions to~\eqref{prep:eq-XR_moment_bounds_start}.
 To this end we introduce a fattened version $\widetilde{\chi}(x)$ of the bump function $\chi(x)$, which is as in \eqref{prep:eq-chi}. For notational purposes, it is convenient to require that $\widetilde{\chi}(x)=1$ for all $x\in [-5/2,5/2]$ and that the support of $\widetilde{\chi}$ is contained in $[-3,3]\subseteq [-\pi,\pi]$.  Then we decompose into
 \begin{equation} \label{prep:eq-Weps_hhtolow_step2_2}
 \begin{aligned}
  &\sum_{M \sim N \gg 1} M^{\alpha} \Bigl\| \sup_{\varepsilon>0} \, \big\| \chi(\cdot)^2 P_M W^{\varepsilon,j_1} \, \partial_x P_N W^{\varepsilon,j_2} \big\|_{\C^{\alpha-1}} \Bigr\|_{L^p_\omega} \\
  &\lesssim \sum_{M \sim N \gg 1} \sum_{L, L' \sim M} M^{\alpha} \Bigl\| \big\| \chi(\cdot)^2 P_M P_L W^{j_1} \, \partial_x P_N P_{L'} W^{j_2} \big\|_{\C^{\alpha-1}} \Bigr\|_{L^p_\omega} \\
  &\lesssim \sum_{M \sim N \gg 1} \sum_{L, L' \sim M} M^{\alpha} \Bigl\| \big\| \chi(\cdot)^2 P_M P_L \bigl( \widetilde{\chi}(\cdot) W^{j_1} \bigr) \, \partial_x P_N P_{L'} \bigl( \widetilde{\chi}(\cdot) W^{j_2} \bigr) \big\|_{\C^{\alpha-1}} \Bigr\|_{L^p_\omega} \\
  &\quad + \sum_{M \sim N \gg 1} \sum_{L, L' \sim M} M^{\alpha} \Bigl\| \big\| \chi(\cdot)^2 P_M P_L \bigl( (1-\widetilde{\chi}(\cdot)) W^{j_1} \bigr) \, \partial_x P_N P_{L'} \bigl( \widetilde{\chi}(\cdot) W^{j_2} \bigr) \big\|_{\C^{\alpha-1}} \Bigr\|_{L^p_\omega} \\
  &\quad + \sum_{M \sim N \gg 1} \sum_{L, L' \sim M} M^{\alpha} \Bigl\| \big\| \chi(\cdot)^2 P_M P_L \bigl( \widetilde{\chi}(\cdot) W^{j_1} \bigr) \, \partial_x P_N P_{L'} \bigl( (1-\widetilde{\chi}(\cdot)) W^{j_2} \bigr) \big\|_{\C^{\alpha-1}} \Bigr\|_{L^p_\omega} \\
  &\quad + \sum_{M \sim N \gg 1} \sum_{L, L' \sim M} M^{\alpha} \Bigl\| \big\| \chi(\cdot)^2 P_M P_L \bigl( (1-\widetilde{\chi}(\cdot)) W^{j_1} \bigr) \, \partial_x P_N P_{L'} \bigl( (1-\widetilde{\chi}(\cdot)) W^{j_2} \bigr) \big\|_{\C^{\alpha-1}} \Bigr\|_{L^p_\omega}.
 \end{aligned}
 \end{equation}
 The first term on the right-hand side of~\eqref{prep:eq-Weps_hhtolow_step2_2} is the main term that will be estimated in the next step. All other terms have at least one input with a mismatched spatial support that allows for a simpler treatment. Indeed, we can bound the second term on the right-hand side of~\eqref{prep:eq-Weps_hhtolow_step2_2} using the embedding $L^\infty \hookrightarrow \C^{\alpha-1}$ by
 \begin{equation} \label{prep:eq-Weps_hhtolow_step2_3}
 \begin{aligned}
  &\sum_{M \sim N \gg 1} \sum_{L, L' \sim M} M^{\alpha} \Bigl\| \big\| \chi(\cdot)^2 P_M P_L \bigl( (1-\widetilde{\chi}(\cdot)) W^{j_1} \bigr) \, \partial_x P_N P_{L'} \bigl( \widetilde{\chi}(\cdot) W^{j_2} \bigr) \big\|_{\C^{\alpha-1}} \Bigr\|_{L^p_\omega} \\
  &\lesssim \sum_{M \sim N \gg 1} \sum_{L, L' \sim M} M^{\alpha}  \bigl\| \chi(\cdot)^2 P_M P_L \bigl( (1-\widetilde{\chi}(\cdot)) \jap{\cdot} ) \bigr\|_{L^\infty_x \to L^\infty_x} \bigl\| \| \jap{x}^{-1} W^{j_1} \|_{L^\infty_x} \bigr\|_{L^{2p}_\omega} \\
  &\qquad \qquad \qquad \qquad \qquad \quad \times \bigl\| \px P_N P_{L'} \bigl( \widetilde{\chi}(\cdot) \jap{\cdot} ) \bigr\|_{L^\infty_x \to L^\infty_x} \bigl\| \| \jap{x}^{-1} W^{j_2} \|_{L^\infty_x} \bigr\|_{L^{2p}_\omega}.
 \end{aligned}
 \end{equation}
 Due to the mismatched spatial supports of the cutoffs $\chi(\cdot)$ and $1-\widetilde{\chi}(\cdot)$, we have  for all $M \sim L \gg 1$ and $A\geq 1$ that,
 \begin{equation*}
 \begin{aligned}
  \bigl\| \chi(\cdot)^2 P_M P_L \bigl( (1-\widetilde{\chi}(\cdot)) \jap{\cdot} ) \bigr\|_{L^\infty_x \to L^\infty_x} \lesssim_{A} M^{-A}.  
 \end{aligned}
 \end{equation*}
 By Bernstein estimates we also have uniformly for all $N \sim L' \gg 1$  that
 \begin{equation*}
 \begin{aligned}
  \bigl\| \px P_N P_{L'} \bigl( \widetilde{\chi}(\cdot) \jap{\cdot} ) \bigr\|_{L^\infty_x \to L^\infty_x} \lesssim N.
 \end{aligned}
 \end{equation*}
 Hence, by Lemma~\ref{prep:lem-weighted_Linfty_BM} we can bound the last line of~\eqref{prep:eq-Weps_hhtolow_step2_3} as desired by
 \begin{equation*}
 \begin{aligned}
  \sum_{M \sim N \gg 1} \sum_{L, L' \sim M} M^{\alpha} \sum_{K \geq 1} K^{\alpha-1} M^{-2} N p \lesssim p.
 \end{aligned}
 \end{equation*}
 The last two terms on the right-hand side of~\eqref{prep:eq-Weps_hhtolow_step2_3} can be estimated analogously.

 \medskip  
 \emph{Step 3: The main random Fourier series computation.} 
 We now turn to estimating the more delicate first term on the right-hand side of~\eqref{prep:eq-Weps_hhtolow_step2_2}.
 To this end we pass to a representation of real-valued Brownian motion in terms of a random Fourier series. 
 Recall that for any countable orthonormal basis $\{\psi_a\}_{a \in \bbN}$ of $L^2_x(\bbR)$, there exists a family  of independent standard Gaussian random variables $\{h_a\}_{a \in \bbN}$ such that
 \begin{equation*}
 \begin{aligned}
  W(x) = \sum_{a=1}^\infty h_a \langle \mathds{1}_{[0,x]}, \psi_a \rangle,
 \end{aligned}
 \end{equation*}
 where $\langle f, g \rangle := \int_{\bbR} f(y) g(y) \, dy$ denotes the inner product on $L^2_x(\bbR)$, see for instance~\cite[Chapter 3.3]{Evans13} or~\cite[Chapter 1]{Nua06}.
 Since the family $\{ \frac{e^{imx}}{\sqrt{2\pi}} \}_{m \in \bbZ}$ forms an orthonormal basis of the (complex) vector space $L^2_x([-\pi,\pi])$, for $x \in [-\pi,\pi]$ the one-dimensional real-valued Brownian motions $W^j(x)$, $1 \leq j \leq \dimA$, can be written as
 \begin{equation*}
 \begin{aligned}
  W^j(x) = \sum_{m \in \bbZ} \frac{g_m^j}{\sqrt{2\pi}} \langle \mathds{1}_{[0,x]}(\cdot), e^{im\cdot} \rangle = \sum_{m \in \bbZ \backslash \{0\} } \frac{g_m^j}{\sqrt{2\pi} i m} e^{imx} + \biggl( \frac{g_0^j}{\sqrt{2\pi}} x -\sum_{m \in \bbZ \backslash \{0\} } \frac{g_m^j}{\sqrt{2\pi} i m} \biggr),  
 \end{aligned}
 \end{equation*}
 where $\{ g^j_m \}_{m\in\bbZ}$, $1 \leq j \leq D$, are families of independent standard complex-valued Gaussian random variables satisfying the constraints $g^j_m = \overline{g^j_{-m}}$. The latter constraints are necessary to ensure that $W^j(x)$ are real-valued Brownian motions. Correspondingly, we have for $x \in [-\pi,\pi]$ and $1 \leq j \leq D$ that
 \begin{equation*}
 \begin{aligned}
  \px \bigl( \widetilde{\chi}(x) W^j(x) \bigr) = \widetilde{\chi}(x) \biggl( \sum_{m \in \bbZ \backslash \{0\} } \frac{g_m^j}{\sqrt{2\pi} } e^{imx} + \frac{g_0^j}{\sqrt{2\pi}} \biggr) + (\partial_x \widetilde{\chi})(x) W^j(x).
 \end{aligned}
 \end{equation*}

 Then we further decompose into
 \begin{equation} \label{prep:eq-Wfreqproj_decomp}
 \begin{aligned}
  \bigl[P_M P_L \bigl( \widetilde{\chi} W^{j} \bigr)\bigr](x) &= \sum_{m \in \bbZ \backslash \{0\} } \frac{g_m^j}{\sqrt{2\pi} i m}  \rho_M(m) \rho_L(m) \widetilde{\chi}(x) e^{imx} + \mathcal{R}_{M,L}^{j}(x),  
 \end{aligned}
 \end{equation}
 where we introduce the remainder term
 \begin{equation*} 
 \begin{aligned}
  \mathcal{R}_{M,L}^{j}(x) := \sum_{m \in \bbZ \backslash \{0\} } \frac{g_m^j}{\sqrt{2\pi} i m} \varphi_{M,L;m}(x) + \biggl( \frac{g_0^j}{\sqrt{2\pi}} \bigl[P_M P_L \bigl( \widetilde{\chi}(y) y \bigr)\bigr](x) - \sum_{m \in \bbZ \backslash \{0\} } \frac{g_m^j}{\sqrt{2\pi} i m} \bigl[ P_M P_L \widetilde{\chi} \bigr](x) \biggr)
 \end{aligned}
 \end{equation*}
 with
 \begin{equation*}
 \begin{aligned}
  \varphi_{M,L;m}(x) := \bigl[ P_M P_L \bigl( \widetilde{\chi}(\cdot) e^{im\cdot} \bigr) \bigr](x) - \rho_M(m) \rho_L(m) \widetilde{\chi}(x) e^{imx}.
 \end{aligned}
 \end{equation*}
 All terms in $\mathcal{R}_{M,L}^{j}$ gain regularity compared with the first term on the right-hand side of~\eqref{prep:eq-Wfreqproj_decomp} in the sense that all terms in $\mathcal{R}_{M,L}^{j}$ come with an additional inverse power of $M$.  Indeed, we have uniformly for all $M \sim L \gg 1$, all $m \in \bbZ \backslash \{0\}$, and all $x \in \bbR$ that
 \begin{equation*}
 \begin{aligned}
  |\varphi_{M,L;m}(x)| \lesssim_{A,B} \jap{x}^{-A} M^{-1} \bigl( 1\{ |m| \lesssim M \} + 1\{|m| \gg M\} \jap{|m|}^{-B} \bigr).
 \end{aligned}
 \end{equation*}
 Moreover, by Bernstein estimates we have $\bigl\|P_M P_L \bigl( \widetilde{\chi}(y) y \bigr)\bigr\|_{L^\infty_x} + \|P_M P_L \widetilde{\chi}\|_{L^\infty_x} \lesssim_A M^{-A}$ uniformly for all $M \sim L \gg 1$.
 Similarly, we decompose 
 \begin{equation} \label{prep:eq-pxWfreqproj_decomp}
 \begin{aligned}
  \bigl[P_N P_{L'} \px \bigl( \widetilde{\chi} W^{j} \bigr)\bigr](x) = \sum_{n \in \bbZ \backslash \{0\} } \frac{g_n^j}{\sqrt{2\pi}}  \rho_N(n) \rho_{L'}(n) \widetilde{\chi}(x) e^{inx} + \widetilde{\mathcal{R}}_{N, L'}^{j}(x)    
 \end{aligned}
 \end{equation}
 with an analogous remainder term that gains regularity compared with the first term on the right-hand side of~\eqref{prep:eq-pxWfreqproj_decomp}.
 We thus arrive at the decomposition
 \begin{equation} \label{prep:eq-Weps_hhtolow_step3_1}
 \begin{aligned}
  &\chi(x)^2 \bigl[ P_M P_L \bigl( \widetilde{\chi} W^{j_1} \bigr) \bigr](x) \, \bigl[ \partial_x P_N P_{L'} \bigl( \widetilde{\chi} W^{j_2} \bigr) \bigr](x) \\ 
  &= \chi(x)^2 \biggl( \sum_{m \in \bbZ \backslash \{0\} } \frac{g_m^{j_1}}{\sqrt{2\pi} i m}  \rho_M(m) \rho_L(m) e^{imx} \biggr) \biggl( \sum_{n \in \bbZ \backslash \{0\} } \frac{g_n^{j_2}}{\sqrt{2\pi}}  \rho_N(n) \rho_{L'}(n) e^{i n x} \biggr) + \{ \text{better} \}, \\
  &= \sum_{m, n \in \bbZ \backslash \{0\} } \frac{g_m^{j_1} g_n^{j_2}}{2\pi i m}  \rho_M(m) \rho_L(m) \rho_N(n) \rho_{L'}(n) \chi(x)^2 e^{i (m+n) x} + \{ \text{better} \}, 
 \end{aligned}
 \end{equation}
 where all product terms with at least one input at better regularity are grouped into the quantity $\{\text{better}\}$. In what follows we only estimate the contribution of the delicate first term and leave the treatment of the contributions of all other more regular terms to the reader.

Let $q\gg 1$ satisfy $\alpha+\frac{1}{q}<\frac{1}{2}$. Using H\"older's inequality in the $\omega$-variable, it suffices to prove the moment estimate for $p\geq q$. Using Bernstein estimates and Minkowski's integral inequality, we obtain that
\begin{equation*}
 \begin{aligned}
  &\sum_{M \sim N \gg 1} \sum_{L, L' \sim M} M^{\alpha} \Bigl\| \big\| \chi(\cdot)^2 P_M P_L \bigl( \widetilde{\chi} W^{j_1} \bigr) \, \partial_x P_N P_{L'} \bigl( \widetilde{\chi} W^{j_2} \bigr) \big\|_{\C^{\alpha-1}} \Bigr\|_{L^p_\omega} \\
  &\lesssim \sum_{M \sim N \gg 1} \sum_{L, L' \sim M} \sum_{K \geq 1} M^{\alpha} K^{\alpha-1} K^{\frac1q} \biggl\| \Bigl\| P_K \bigl( \chi(\cdot)^2 P_M P_L \bigl( \widetilde{\chi} W^{j_1} \bigr) \, \partial_x P_N P_{L'} \bigl( \widetilde{\chi} W^{j_2} \bigr) \bigr) \Big\|_{L^p_\omega} \biggr\|_{L^q_x}.
 \end{aligned}
 \end{equation*} 
 Inserting the first term on the right-hand side of~\eqref{prep:eq-Weps_hhtolow_step3_1} and using Gaussian hypercontractivity yields 
 \begin{equation}\label{prep:eq-Weps-hypercontractivity}
 \begin{aligned}
  &\sum_{M \sim N \gg 1} \sum_{L, L' \sim M} \sum_{K \geq 1} M^{\alpha} K^{\alpha-1} K^{\frac1q} \\
  &\qquad \qquad \times \biggl\| \Bigl\| \sum_{m,n \in \bbZ \backslash \{0\} } \frac{g_m^{j_1} g_n^{j_2}}{2\pi i m}  \rho_M(m) \rho_L(m) \rho_N(n) \rho_{L'}(n) \bigl[ P_K \bigl( \chi(\cdot)^2 e^{i (m+n) \cdot} \bigr) \bigr](x) \Big\|_{L^p_\omega} \biggr\|_{L^q_x} \\
  &\lesssim p \sum_{M \sim N \gg 1} \sum_{L, L' \sim M} \sum_{K \geq 1} M^{\alpha} K^{\alpha-1} K^{\frac1q} \\
  &\quad \times \biggl\| \Bigl\| \sum_{m,n \in \bbZ \backslash \{0\} } \frac{g_m^{j_1} g_n^{j_2}}{2\pi i m}  \rho_M(m) \rho_L(m) \rho_N(n) \rho_{L'}(n) \bigl[ P_K \bigl( \chi(\cdot)^2 e^{i (m+n) \cdot} \bigr)  \bigr](x) \Big\|_{L^2_\omega} \biggr\|_{L^q_x}.
 \end{aligned}
 \end{equation}
 We now decompose the second-order Gaussian chaos in  \eqref{prep:eq-Weps-hypercontractivity} into a non-resonant and resonant component. To be precise, we decompose 
 \begin{align}
     &\sum_{m,n \in \bbZ \backslash \{0\} } \frac{g_m^{j_1} g_n^{j_2}}{2\pi i m}  \rho_M(m) \rho_L(m) \rho_N(n) \rho_{L'}(n) \bigl[ P_K \bigl( \chi(\cdot)^2 e^{i (m+n) \cdot} \bigr)  \bigr](x) \notag \\
     =& \sum_{m,n \in \bbZ \backslash \{0\} } \frac{g_m^{j_1} g_n^{j_2}-\delta_{m+n=0} \delta_{j_1=j_2}}{2\pi i m}\rho_M(m) \rho_L(m) \rho_N(n) \rho_{L'}(n) \bigl[ P_K \bigl( \chi(\cdot)^2 e^{i (m+n) \cdot} \bigr)  \bigr](x) \label{prep:eq-Weps-nonres}\\
     &\quad + \delta_{j_1=j_2} \sum_{m\in \bbZ \backslash \{ 0\}} \frac{1}{2\pi im}
     \rho_M(m) \rho_L(m) \rho_N(m) \rho_{L'}(m) P_K(\chi^2)(x). \label{prep:eq-Weps-res} 
 \end{align}
 We first address the resonant term \eqref{prep:eq-Weps-res}. At first sight, it may seem like this term is of order one, which would lead to a divergence in the sum over $M$. However, since $m\in\bbZ \backslash \{ 0\} \mapsto 1/m$ is odd and the cut-off function 
 $m\in\bbZ \backslash \{ 0\} \mapsto \rho_M(m) \rho_L(m) \rho_N(m) \rho_{L'}(m)$ is even, it holds that 
 \begin{equation*}
    \sum_{m\in \bbZ \backslash \{ 0\}} \frac{1}{2\pi im}
     \rho_M(m) \rho_L(m) \rho_N(m) \rho_{L'}(m) = 0.
 \end{equation*}
 In particular, the resonant term \eqref{prep:eq-Weps-res}  vanishes.
 It remains to treat the contribution of the non-resonant term \eqref{prep:eq-Weps-nonres}. To this end, we note that 
 \begin{equation}\label{prep:eq-Weps-PK-mode}
 \begin{aligned}
  \bigl| \bigl[ P_K \bigl( \chi(\cdot)^2 e^{i (m+n) \cdot} \bigr) \bigr](x) \bigr| \lesssim_{A,B} \jap{x}^{-A} \bigl( 1\{ |m+n| \lesssim K\} + 1\{|m+n| \gg K\} \jap{m+n}^{-B} \bigr)
 \end{aligned}
 \end{equation}
 for all $K\geq 1$, $m,n \in \bbZ$, and $x\in \R$. Using $L^2_\omega$-orthogonality and \eqref{prep:eq-Weps-PK-mode}, we obtain the bound 
 \begin{align}
&p \sum_{M \sim N \gg 1} \sum_{L, L' \sim M} \sum_{K \geq 1} M^{\alpha} K^{\alpha-1} K^{\frac1q}  \notag \\
  &\quad \times \biggl\| \Bigl\| \sum_{m,n \in \bbZ \backslash \{0\} } \frac{g_m^{j_1} g_n^{j_2}-\delta_{m+n=0} \delta_{j_1=j_2}}{2\pi i m}\rho_M(m) \rho_L(m) \rho_N(n) \rho_{L'}(n) \bigl[ P_K \bigl( \chi(\cdot)^2 e^{i (m+n) \cdot} \bigr)  \bigr](x) \Bigr\|_{L^2_\omega} \biggr\|_{L^q_x} \allowdisplaybreaks[3] \notag\\
  &\lesssim p  \sum_{M \sim N \gg 1} \sum_{L, L' \sim M} \sum_{K \geq 1} M^{\alpha-1} K^{\alpha-1} K^{\frac1q} 
   \times \biggl \| \Bigl\| 1\{ |m|\sim M\} 1\{ |n|\sim N\}  P_K \bigl( \chi(\cdot)^2 e^{i (m+n) \cdot} \bigr)  \Bigr\|_{\ell_{m,n}^2} \biggr\|_{L_x^q} \allowdisplaybreaks[3] \notag\\
  &\lesssim p \sum_{M \sim N \gg 1} \sum_{L, L' \sim M} \sum_{K \geq 1} M^{\alpha-1} K^{\alpha-1} K^{\frac1q} \notag \\
  &\quad \times \Bigl\| 1\{ |m|\sim M\} 1\{ |n|\sim N\} \bigl( 1\{ |m+n|\sim K \} + 1\{ |m+n| \gg K \} \langle m + n \rangle^{-2} \bigr) \Bigr\|_{\ell_{m,n}^2}.\label{prep:eq-Weps-suml2}
 \end{align}
 The $\ell_{m,n}^2$-factor can be estimated by 
 \begin{align*}
      &\Bigl\| 1\{ |m|\sim M\} 1\{ |n|\sim N\} \bigl( 1\{ |m+n|\sim K \} + 1\{ |m+n| \gg K \} \langle m + n \rangle^{-2} \bigr) \Bigr\|_{\ell_{m,n}^2} \\
      &\lesssim M^{\frac{1}{2}} K^{\frac{1}{2}} + M^{\frac{1}{2}} \lesssim M^{\frac{1}{2}} K^{\frac{1}{2}}. 
 \end{align*}
 Thus, \eqref{prep:eq-Weps-suml2} can be estimated by
 \begin{equation*}
     p  \sum_{M \sim N \gg 1} \sum_{L, L' \sim M} \sum_{K \geq 1} M^{\alpha-\frac{1}{2}} K^{\alpha-\frac{1}{2}} K^{\frac1q}  \lesssim p.
 \end{equation*}
 This finishes the proof of the uniform bounds \eqref{prep:eq-Weps_hhtolow_unif_bound}. 
 
 Finally, the convergence statement~\eqref{prep:eq-Weps_hhtolow_conv} can be deduced by a variant of the preceding arguments. Specifically, we can exploit that in the difference on the left-hand side of~\eqref{prep:eq-Weps_hhtolow_conv} at least one input must be at frequencies $\geq \varepsilon^{-1}$. This allows us to gain a small power of $\varepsilon$ at every step of the preceding argument and therefore yields convergence as $\varepsilon \to 0$. 
\end{proof}

\begin{remark}
It may be possible to prove Proposition \ref{prep:lem-Weps_hhtolow} by working directly in physical space and with the covariance of the Euclidean Brownian motion $W$. However, the argument using random Fourier series is closer to the literature on random dispersive equations. 
\end{remark}

The next proposition establishes the main properties of the smooth approximations of the Brownian paths.
As an auxiliary tool we need to introduce a canonical smooth and compactly supported extension $P_{ext} \colon \bbR^\dimA \to L(\bbR^D; \bbR^D)$ of the orthogonal projection $P$ introduced above, i.e., $P_{ext}(p)$ is the orthogonal projection to $T_p \M$ for all $p \in \M$.

\begin{proposition}[Smooth approximation of the Brownian paths] \label{prep:prop-brownian-path-approximation}
Let $x_0 \in \bbR$ be arbitrary and let $s < \alpha < \frac12$ with $0 < \frac12 - \alpha \ll 1$ sufficiently small.
There exists a small constant $0 < c(\M) \ll 1$ depending only on $\M$ such that for all $0 < \tau \leq c(\M)$, there exists an event $\mathcal{E}(\tau, x_0) \subseteq \Omega$ satisfying
\begin{equation*}
    \mathbb{P}\big( \mathcal{E}(\tau, x_0) \big) \geq 1 - C \exp(-c \tau^{-c})
\end{equation*}
and on which the following properties hold:
\begin{enumerate}[label={(\roman*)}]
    \item\label{prep:item-approx-1} Convergence: 
    There exists a path $B_{\tau, x_0, \loc} \colon \bbR \rightarrow \M$ such that $B^\varepsilon_{\tau, x_0, \loc} \to B_{\tau, x_0, \loc}$ in $\C^\alpha$.
    \item\label{prep:item-approx-2} Effect of the spatial cut-off: 
    For all $\varepsilon > 0$ and $x \in [-2,2]$, it holds that $B^\varepsilon_{\tau, x_0, \loc}(x) = B^\varepsilon_{\tau, x_0}(x)$ and $B_{\tau, x_0, \loc}(x) = B(\tau x + x_0)$.
    \item\label{prep:item-approx-3} Para-controlled structure: For all $\varepsilon>0$, there exists a vector-valued function $B^{\varepsilon,\#}_{\tau,x_0,\loc} \in \C^{2\alpha}$ such that 
    \begin{equation} %\label{prep:eq-BM-approx-parastructure}
        B^{\varepsilon}_{\tau,x_0,\loc} = \tau^{\frac12} P_{\ext}(B^{\varepsilon}_{\tau,x_0,\loc}) \parall ( \chi(\cdot) W^\varepsilon_{\tau, x_0} ) + B^{\varepsilon, \#}_{\tau,x_0,\loc}.
    \end{equation}
    \item\label{prep:item-approx-4} Uniform bounds: For all $\varepsilon>0$, it holds that 
    \begin{equation} \label{prep:eq-Beps_Calpha_bounds}
    \| B^{\varepsilon}_{\tau,x_0,\loc} - B^{\varepsilon}_{\tau, x_0, \loc}(0) \|_{\C^\alpha} \lesssim_{\M} \tau^{\frac12 - c} 
    \quad \text{and} \quad 
    \| B^{\varepsilon, \#}_{\tau,x_0,\loc} - B^{\varepsilon}_{\tau, x_0, \loc}(0) \|_{\C^{2\alpha}} \lesssim_{\M} \tau^{\frac12-c}.
    \end{equation}
    \item \label{prep:item-Beps_hhtolow} 
    Uniform high$\times$high$\rightarrow$low-bounds: For all $1 \leq j_1, j_2 \leq D$, it holds that 
    \begin{equation} \label{prep:eq-Beps_hhtolow_bound}
         \sup_{\varepsilon>0} \sup_{M\sim N} M^{s} \big\|  P_M \bigl( B^{\varepsilon,j_1}_{\tau, x_0, \loc} - B^{\varepsilon, j_1}_{\tau, x_0, \loc}(0) \bigr) \, \partial_x P_N \bigl( B^{\varepsilon,j_2}_{\tau, x_0,\loc} - B^{\varepsilon, j_2}_{\tau, x_0, \loc}(0) \bigr) \big\|_{\C^{s-1}} \lesssim_{\M} \tau^{1-c}.
    \end{equation}
    \item \label{prep:item-Beps_hhtolow-convergence} 
    Convergence of high$\times$high$\rightarrow$low-term: For all $1\leq j_1, j_2 \leq D$, it holds that 
    \begin{equation} \label{prep:eq-Beps_hhtolow_conv}
     \begin{aligned}
       &\lim_{\varepsilon \to 0} \, \sup_{M\sim N} M^{s} \big\| 
        P_M \bigl( B^{\varepsilon,j_1}_{\tau, x_0, \loc} - B^{\varepsilon, j_1}_{\tau, x_0, \loc}(0) \bigr) \, \partial_x P_N \bigl( B^{\varepsilon,j_2}_{\tau, x_0, \loc} - B^{\varepsilon, j_2}_{\tau, x_0, \loc}(0) \bigr) \\
       &\qquad \qquad \qquad \, \, - P_M \bigl( B^{j_1}_{\tau, x_0, \loc} - B^{j_1}_{\tau, x_0, \loc}(0) \bigr) \, \partial_x P_N \bigl( B^{j_2}_{\tau, x_0, \loc} - B^{j_2}_{\tau, x_0, \loc}(0) \bigr) \big\|_{\C^{s-1}} = 0.
     \end{aligned}
    \end{equation}
\end{enumerate}
\end{proposition}
\begin{proof}
 We begin by recording a Wong-Zakai type approximation result for the Brownian motions $B(x)$, $x \in \bbR$, on $\M$ defined as solutions to the Stratonovich differential equation~\eqref{prep:eq-def-B}. It follows from \cite[Chapter 6, Theorem 7.2]{IkeWat89} that the solutions $B^\varepsilon(x)$, $x \in \bbR$, to the classical ordinary differential equation~\eqref{prep:eq-def-Beps} approximate the Brownian paths $B(x)$, $x \in \bbR$, in the sense that for any $r > 0$ we have 
 \begin{equation} \label{prep:eq-WongZakai-Bapprox}
  \lim_{\varepsilon \to 0} \, \mathbb{E} \, \biggl[ \sup_{|x| \leq r} \bigl|B^\varepsilon(x) - B(x)\bigr|^2 \biggr] = 0.
 \end{equation}
 This identifies the Brownian path $B$ as the limit of the approximating sequence $(B^\varepsilon)_{\varepsilon > 0}$ in a weak sense. 
 
 We now describe how to obtain an event $\mathcal{E}(\tau, x_0) \subset \Omega$ on which the properties (i)--(vi) hold. 
 To this end we use the para-controlled framework from~\cite[Theorem 3.3]{GIP15} for the analysis of rough differential equations of the form
 \begin{equation} \label{prep:eq-RDE-GIP}
 \begin{aligned}
  \partial_x u = F(u) \xi, \quad u(0) = u_0,
 \end{aligned}
 \end{equation}
 where $u_0 \in \bbR^d$, $u \colon \bbR \to \bbR^d$ is a continuous vector-valued function, $\xi \colon \bbR \to \bbR^n$ is a vector-valued distribution with values in $\C^\beta$ for some $\beta \in (1/3, 1)$, and $F \colon \bbR^d \to \mathcal{L}(\bbR^n, \bbR^d)$ is a family of $C^3_b$ vector fields on $\bbR^d$. 
 Correspondingly, we view the family of solutions $B^{\varepsilon}_{\tau, x_0, \loc} \colon \bbR \to \M$, $\varepsilon > 0$, to the ordinary differential equations~\eqref{prep:equ-Bepstaux0loc_def} as $\bbR^\dimA$-valued solutions to the following rough differential equations
 \begin{equation} \label{prep:eq-Bepstaux0loc-for-GIP}
 \begin{aligned}
  \partial_x B^{\varepsilon}_{\tau, x_0,\loc} = \tau^{\frac12} P_{ext}\bigl( B^{\varepsilon}_{\tau, x_0, \loc} \bigr) \partial_x \bigl( \chi W^\varepsilon_{\tau,x_0} \bigr), \quad B^{\varepsilon}_{\tau, x_0,\loc}(0) = B^\varepsilon(x_0),
 \end{aligned}
 \end{equation}
 where $P_{ext} \colon \bbR^D \to L(\bbR^D; \bbR^D)$ is the canonical smooth and compactly supported extension of the orthogonal projection $P$ introduced above. Clearly, $P_{ext} \in C^3_b$.
 Observe that while we only consider the initial condition $B^{\varepsilon}_{\tau, x_0,\loc}(0) = B^\varepsilon( x_0) \in \M$, \eqref{prep:eq-Bepstaux0loc-for-GIP} is well-defined for any initial condition in $\bbR^\dimA$ and falls into the category of rough differential equations~\eqref{prep:eq-RDE-GIP} considered in~\cite[Theorem 3.3]{GIP15}. 
 
 In particular, by Lemma~\ref{prep:lem-Weps_hhtolow} (with the choice $\lambda = \tau^{-c}$) for any $0 < \frac12-\alpha \ll 1$, there exists an event $\mathcal{E}(\tau, x_0) \subset \Omega$ with $\mathbb{P}(\mathcal{E}(\tau, x_0)) \geq 1 - c^{-1} \exp(-c \tau^{-c})$ on which the high-high paraproducts $\bigl( \chi(\cdot) W^{\varepsilon}_{\tau,x_0} \bigr) \parasim \partial_x \bigl( \chi(\cdot) W^{\varepsilon}_{\tau, x_0} \bigr)$ converge\footnote{
 This conclusion can be obtained from a slight variant of the proof of Lemma~\ref{prep:lem-Weps_hhtolow}. We emphasize that the convergence~\eqref{prep:eq-Weps_hhtolow_conv} in the statement of Lemma~\ref{prep:lem-Weps_hhtolow} is stronger than what is needed to conclude the convergence of the high-high paraproducts in $\C^{2\alpha-1}$.}
 in $\C^{2\alpha-1}$ as $\varepsilon \to 0$, and on which we have for $1 \leq j_1, j_2 \leq D$,
 \begin{equation} \label{prep:eq-Wepstaux0_hhtolow}
 \begin{aligned}
  \sup_{\varepsilon > 0} \, \sup_{M \sim N} \, M^\alpha \bigl\| P_M \bigl( \chi(\cdot) W^{\varepsilon, j_1}_{\tau, x_0} \bigr) \partial_x P_N \bigl( \chi(\cdot) W^{\varepsilon, j_2}_{\tau, x_0} \bigr) \bigr\|_{\C^{\alpha-1}} \leq \tau^{-c}.
 \end{aligned}
 \end{equation}
 Additionally, by a much simpler variant of the proof of Lemma~\ref{prep:lem-Weps_hhtolow} we can achieve that on the event $\mathcal{E}(\tau, x_0)$, we have for all integers $1 \leq j \leq D$,
 \begin{equation} \label{prep:eq-xi-theta_bounds}
 \begin{aligned}
  \sup_{\varepsilon > 0} \, \bigl\| \chi(\cdot) W^{\varepsilon, j}_{\tau,x_0} \bigr\|_{\C^\alpha} + \sup_{\varepsilon > 0} \, \bigl\| \partial_x \bigl( \chi(\cdot) W^{\varepsilon, j}_{\tau,x_0} \bigr) \bigr\|_{\C^{\alpha-1}} \leq \tau^{-\frac{c}{2}}.
 \end{aligned}
 \end{equation}
 Then all assumptions of~\cite[Theorem 3.3]{GIP15} are satisfied and it follows that there exists $0 < c(\M) \ll 1$ such that for all $0 < \tau \leq c(\M)$, we obtain on the event $\mathcal{E}(\tau, x_0)$ the existence of a limit path $B_{\tau,x_0,\loc} \colon \bbR \to \M$ with $B_{\tau,x_0,\loc} \in \C^\alpha$ such that $B^\varepsilon_{\tau,x_0,\loc} \to B_{\tau,x_0,\loc}$ in $\C^\alpha$. 
 By uniqueness of solutions to classical ordinary differential equations, we have $B^{\varepsilon}_{\tau, x_0, \loc}(x) = B^\varepsilon_{\tau, x_0}(x)$ for all $x \in [-2,2]$. 
 Additionally, in view of \eqref{prep:eq-WongZakai-Bapprox} we may conclude that $B_{\tau, x_0, \loc}(x) = B(\tau x+x_0)$ for all $x \in [-2,2]$.
 
 Moreover, it follows\footnote{Specifically, in the notation of the proof of~\cite[Theorem 3.3]{GIP15} we have $C_F \simeq \tau^{\frac12} (\|P_{ext}\|_{C^2_b} + \|P_{ext}\|_{C^2_b}^2)$ and $C_\xi \lesssim \tau^{-c}$. To conclude the bounds~\eqref{prep:eq-Beps_Calpha_bounds} we have to impose $C_F C_\xi \ll 1$.}
 from the proof of \cite[Theorem 3.3]{GIP15} that for all $\varepsilon>0$, there exists a vector-valued function $B^{\varepsilon,\#}_{\tau, x_0,\loc} \in \C^{2\alpha}$ such that on $\mathcal{E}(\tau,x_0)$ it holds that
 \begin{equation} \label{prep:eq-BM-approx-parastructure}
  B^{\varepsilon}_{\tau, x_0, \loc} = \tau^{\frac12} P_{\ext}(B^{\varepsilon}_{\tau, x_0, \loc}) \parall \bigl( \chi(\cdot) W^\varepsilon_{\tau,x_0} \bigr) + B^{\varepsilon,\#}_{\tau,x_0,\loc}
 \end{equation}
 with
 \begin{equation} %\label{prep:eq-Beps_Calpha_bounds}
  \| B^{\varepsilon}_{\tau,x_0,\loc} - B^{\varepsilon}_{\tau, x_0, \loc}(0) \|_{\C^\alpha} \lesssim_{\M} \tau^{\frac12 - c} 
  \quad \text{and} \quad 
  \| B^{\varepsilon, \#}_{\tau,x_0,\loc} - B^{\varepsilon}_{\tau, x_0, \loc}(0) \|_{\C^{2\alpha}} \lesssim_{\M} \tau^{\frac12-c}.
 \end{equation}
 Note that since the target manifold $\M$ is compact, we have $|B^{\varepsilon}_{\tau, x_0, \loc}(x)| \leq C(\M)$ for all $x \in \bbR$ for some constant $C(\M) \geq 1$, whence the preceding bounds also imply 
 \begin{equation} \label{prep:equ-Beps_Calpha_crude_bound}
 \begin{aligned}
  \| B^{\varepsilon}_{\tau,x_0,\loc} \|_{\C^\alpha} \lesssim C(\M).
 \end{aligned}
 \end{equation}

 Finally, we turn to the proofs of the uniform high$\times$high$\rightarrow$low-bounds~\eqref{prep:eq-Beps_hhtolow_bound} and the convergence of the high$\times$high$\rightarrow$low-term~\eqref{prep:eq-Beps_hhtolow_conv}. We focus on the proof of \eqref{prep:eq-Beps_hhtolow_bound}. The derivation of~\eqref{prep:eq-Beps_hhtolow_conv} is based on a variant of the argument and is left to the reader.
 First, we record that by Bony's paralinearization in Besov spaces, see for instance \cite[Theorem 2.89 and 2.92]{BCD11}, we have the bound
 \begin{equation} \label{prep:equ-Pext_Beps_Calpha_bound}
 \begin{aligned}
  \| P_{ext}( B^{\varepsilon}_{\tau, x_0, \loc} ) \|_{\C^\alpha} \leq C(P_{ext}, \alpha) (1 + \|B^{\varepsilon}_{\tau, x_0, \loc}\|_{\C^\alpha}^{10}) \|B^{\varepsilon}_{\tau, x_0, \loc}\|_{\C^\alpha},
 \end{aligned}
 \end{equation}
 where $C(P_{ext},\alpha) > 0$ is a constant that depends on $\alpha$ and finitely many derivatives of $P_{ext}$. 
 In what follows we suppress the superscripts $1 \leq j_1, j_2 \leq D$. Inserting the para-controlled structure~\eqref{prep:eq-BM-approx-parastructure}, we find
 \begin{equation} \label{prep:eq-Beps_hhtolow_bound_decomp}
 \begin{aligned}
  &\sup_{\varepsilon>0} \sup_{M\sim N} M^{s} \big\|  P_M \bigl( B^{\varepsilon}_{\tau, x_0, \loc} - B^{\varepsilon}_{\tau, x_0, \loc}(0) \bigr) \, \partial_x P_N \bigl( B^{\varepsilon}_{\tau, x_0, \loc} - B^{\varepsilon}_{\tau, x_0, \loc}(0) \bigr) \big\|_{\C^{s-1}} \\
  &\leq \sup_{\varepsilon>0} \sup_{M\sim N} M^{s} \Big\|  P_M \Bigl( \tau^{\frac12} P_{\ext}(B^{\varepsilon}_{\tau, x_0, \loc}) \parall \bigl( \chi(\cdot) W^\varepsilon_{\tau, x_0} \bigr) \Bigr) \, \partial_x P_N \Bigl( \tau^{\frac12} P_{\ext}(B^{\varepsilon}_{\tau, x_0, \loc}) \parall \bigl( \chi(\cdot) W^\varepsilon_{\tau, x_0} \bigr) \Bigr) \Big\|_{\C^{s-1}} \\
  &\quad + \sup_{\varepsilon>0} \sup_{M\sim N} M^{s} \Big\|  P_M \Bigl( \tau^{\frac12} P_{\ext}(B^{\varepsilon}_{\tau, x_0, \loc}) \parall \bigl( \chi(\cdot) W^\varepsilon_{\tau, x_0} \bigr) \Bigr) \, \partial_x P_N \bigl( B^{\varepsilon, \#}_{\tau, x_0, \loc} - B^{\varepsilon}_{\tau, x_0, \loc}(0) \bigr) \Bigr\|_{\C^{s-1}} \\
  &\quad + \sup_{\varepsilon>0} \sup_{M\sim N} M^{s} \Big\|  P_M \bigl( B^{\varepsilon, \#}_{\tau, x_0, \loc} - B^{\varepsilon}_{\tau, x_0, \loc}(0) \bigr) \, \partial_x P_N \Bigl( \tau^{\frac12} P_{\ext}(B^{\varepsilon}_{\tau, x_0, \loc}) \parall \bigl( \chi(\cdot) W^\varepsilon_{\tau, x_0} \bigr) \Bigr) \Bigr\|_{\C^{s-1}} \\
  &\quad + \sup_{\varepsilon>0} \sup_{M\sim N} M^{s} \big\|  P_M \bigl( B^{\varepsilon, \#}_{\tau, x_0, \loc} - B^{\varepsilon}_{\tau, x_0, \loc}(0) \bigr) \, \partial_x P_N \bigl( B^{\varepsilon, \#}_{\tau, x_0, \loc} - B^{\varepsilon}_{\tau, x_0, \loc}(0) \bigr) \bigr\|_{\C^{s-1}}.
 \end{aligned}
 \end{equation}
 Then the first term on the right-hand side of~\eqref{prep:eq-Beps_hhtolow_bound_decomp} is the main contribution and all other terms on the right-hand side are easier to estimate due to the better $C^{2\alpha}$ regularity of $B^{\varepsilon,\#}_{\tau, x_0, \loc}$. For this reason we omit the details of their treatment. We also only consider the case when the derivative falls onto the high frequency input in the first term on the right-hand side of~\eqref{prep:eq-Beps_hhtolow_bound_decomp}.
 Recalling that $s < \alpha < \frac12$, we decompose this term as follows
 \begin{equation} \label{prep:eq-Beps_hhtolow_bound_decomp_1stterm}
 \begin{aligned}
  &\sup_{\varepsilon>0} \sup_{M\sim N} M^{s} \Big\|  P_M \Bigl( \tau^{\frac12} P_{\ext}(B^{\varepsilon}_{\tau, x_0, \loc}) \parall \bigl( \chi(\cdot) W^\varepsilon_{\tau, x_0} \bigr) \Bigr) \, P_N \Bigl( \tau^{\frac12} P_{\ext}(B^{\varepsilon}_{\tau, x_0, \loc}) \parall \px \bigl( \chi(\cdot) W^\varepsilon_{\tau, x_0} \bigr) \Bigr) \Big\|_{\C^{s-1}} \\
  &\lesssim \sup_{\varepsilon>0} \sup_{M\sim N} \sup_{M' \sim M} \sup_{N' \sim N} \sum_{L \ll M'} \sum_{L' \ll N'} M^s \tau  \Big\| P_M \Bigl( P_L \bigl( P_{\ext}(B^{\varepsilon}_{\tau, x_0, \loc}) \bigr) P_{M'} \bigl( \chi(\cdot) W^\varepsilon_{\tau, x_0} \bigr) \Bigr) \\
  &\qquad \qquad \qquad \qquad \qquad \qquad \qquad \qquad \qquad \times P_N \Bigl( P_{L'} \bigl( P_{\ext}(B^{\varepsilon}_{\tau, x_0, \loc}) \bigr) \partial_x P_{N'} \bigl( \chi(\cdot) W^\varepsilon_{\tau, x_0} \bigr) \Bigr) \Big\|_{\C^{\alpha-1}} \\
  &\lesssim \sup_{\varepsilon>0} \sup_{M\sim N} \sup_{M' \sim M} \sup_{N' \sim N} \sum_{L \ll M'} \sum_{L' \ll N'} M^s \tau  \Big\| P_L \bigl( P_{\ext}(B^{\varepsilon}_{\tau, x_0, \loc}) \bigr) P_M P_{M'} \bigl( \chi(\cdot) W^\varepsilon_{\tau, x_0} \bigr) \\
  &\qquad \qquad \qquad \qquad \qquad \qquad \qquad \qquad \qquad \times P_{L'} \bigl( P_{\ext}(B^{\varepsilon}_{\tau, x_0, \loc}) \bigr) \partial_x P_N P_{N'} \bigl( \chi(\cdot) W^\varepsilon_{\tau, x_0} \bigr) \Big\|_{\C^{\alpha-1}} \\
  &\quad + \sup_{\varepsilon>0} \sup_{M\sim N} \sup_{M' \sim M} \sup_{N' \sim N} \sum_{L \ll M'} \sum_{L' \ll N'} M^s \tau \\
  &\qquad \quad \times \bigg\| \biggl( P_M \Bigl( P_L \bigl( P_{\ext}(B^{\varepsilon}_{\tau, x_0, \loc}) \bigr) P_{M'} \bigl( \chi(\cdot) W^\varepsilon_{\tau, x_0} \bigr) \Bigr) - P_L \bigl( P_{\ext}(B^{\varepsilon}_{\tau, x_0, \loc}) \bigr) P_M P_{M'} \bigl( \chi(\cdot) W^\varepsilon_{\tau, x_0} \bigr) \biggr) \\
  &\qquad \qquad \qquad \qquad \qquad \qquad \qquad \qquad \qquad \qquad \times P_N \Bigl( P_{L'} \bigl( P_{\ext}(B^{\varepsilon}_{\tau, x_0, \loc}) \bigr) \partial_x P_{N'} \bigl( \chi(\cdot) W^\varepsilon_{\tau, x_0} \bigr) \Bigr) \biggr\|_{\C^{\alpha-1}} \\
  &\quad + \bigl\{ \text{similar} \bigr\},
 \end{aligned}
 \end{equation}
 where the quantity $\{\text{similar}\}$ is short-hand notation for two further terms like the second term on the right-hand side of \eqref{prep:eq-Beps_hhtolow_bound_decomp_1stterm} with at least one input featuring a commutator term with improved regularity.
 Then the estimate for the second term on the right-hand side of \eqref{prep:eq-Beps_hhtolow_bound_decomp_1stterm} is straightforward using an analogue of the commutator estimate \eqref{prep:eq-commutator-3} for the standard Besov spaces, see for instance \cite[Lemma 2.3]{GIP15}. We omit the details.
 In order to estimate the first term on the right-hand side of \eqref{prep:eq-Beps_hhtolow_bound_decomp_1stterm} we pick $0 < \nu \ll 1$ to be fixed sufficiently small further below.
 Using the standard para-product estimates in Besov spaces, see for instance~\cite[Lemma 2.1]{GIP15}, we find
 \begin{equation} \label{prep:eq-Beps_hhtolow_bound_decomp_1stterm_2}
 \begin{aligned}
  &\sup_{\varepsilon>0} \sup_{M\sim N} \sup_{M' \sim M} \sup_{N' \sim N} \sum_{L \ll M'} \sum_{L' \ll N'} M^s \tau  \Big\| P_L \bigl( P_{\ext}(B^{\varepsilon}_{\tau, x_0, \loc}) \bigr) P_M P_{M'} \bigl( \chi(\cdot) W^\varepsilon_{\tau, x_0} \bigr) \\
  &\qquad \qquad \qquad \qquad \qquad \qquad \qquad \qquad \qquad \times P_{L'} \bigl( P_{\ext}(B^{\varepsilon}_{\tau, x_0, \loc}) \bigr) \partial_x P_N P_{N'} \bigl( \chi(\cdot) W^\varepsilon_{\tau, x_0} \bigr) \Big\|_{\C^{\alpha-1}} \\
  &\lesssim \sup_{\varepsilon>0} \sup_{M\sim N} \sup_{M' \sim M} \sup_{N' \sim N} \sum_{L \ll M'} \sum_{L' \ll N'} M^{s-\alpha} \Big\| P_L \bigl( P_{\ext}(B^{\varepsilon}_{\tau, x_0, \loc}) \bigr) P_{L'} \bigl( P_{\ext}(B^{\varepsilon}_{\tau, x_0, \loc}) \bigr) \Bigr\|_{C^{1-\alpha+\nu}} \\
  &\qquad \qquad \qquad \qquad \qquad \qquad \qquad \qquad \qquad \times M^\alpha \tau \Bigl\| P_M P_{M'} \bigl( \chi(\cdot) W^\varepsilon_{\tau, x_0} \bigr) \partial_x P_N P_{N'} \bigl( \chi(\cdot) W^\varepsilon_{\tau, x_0} \bigr) \Big\|_{\C^{\alpha-1}}.
 \end{aligned}
\end{equation}
 For the $\C^{1-\alpha+\nu}$ norm of the product of the low frequency pieces we have
 \begin{equation*}
 \begin{aligned}
  &\sup_{M\sim N} \sup_{M' \sim M} \sup_{N' \sim N} \sum_{L \ll M'} \sum_{L' \ll N'} \Big\| P_L \bigl( P_{\ext}(B^{\varepsilon}_{\tau, x_0, \loc}) \bigr) P_{L'} \bigl( P_{\ext}(B^{\varepsilon}_{\tau, x_0, \loc}) \bigr) \Bigr\|_{C^{1-\alpha+\nu}} \\
  &\lesssim \sup_{M\sim N} \sup_{M' \sim M} \sup_{N' \sim N} \sum_{L \ll M'} \sum_{L' \ll N'} \sum_{K \lesssim \max\{L,L'\}} K^{1-\alpha+\nu} L^{-\alpha} (L')^{-\alpha} \bigl\| P_{\ext}(B^{\varepsilon}_{\tau, x_0, \loc}) \bigr\|_{\C^\alpha}^2 \\
  &\lesssim M^{1-2\alpha+\nu} \bigl\| P_{\ext}(B^{\varepsilon}_{\tau, x_0, \loc}) \bigr\|_{\C^\alpha}^2.
 \end{aligned}
 \end{equation*}
 Note that the estimate~\eqref{prep:eq-Wepstaux0_hhtolow} extends to the setting where each input carries an additional Littlewood-Paley projection to a comparable dyadic frequency region. Thus, using \eqref{prep:eq-Wepstaux0_hhtolow} and~\eqref{prep:equ-Pext_Beps_Calpha_bound} we can bound the right-hand side of~\eqref{prep:eq-Beps_hhtolow_bound_decomp_1stterm_2} by
 \begin{equation*}
 \begin{aligned}
  &\sup_{\varepsilon>0} \sup_{M\sim N} \sup_{M' \sim M} \sup_{N' \sim N} M^{s + 1 - 3\alpha +\nu} \bigl\| P_{\ext}(B^{\varepsilon}_{\tau, x_0, \loc}) \bigr\|_{\C^\alpha}^2 \\
  &\qquad \qquad \qquad \qquad \qquad \qquad \qquad \qquad \times M^\alpha \tau \Bigl\| P_M P_{M'} \bigl( \chi(\cdot) W^\varepsilon_{\tau, x_0} \bigr) \partial_x P_N P_{N'} \bigl( \chi(\cdot) W^\varepsilon_{\tau, x_0} \bigr) \Big\|_{\C^{\alpha-1}} \\
  &\lesssim C(P_{ext}, \alpha)^2 (1 + \|B^{\varepsilon}_{\tau, x_0, \loc}\|_{\C^\alpha}^{10})^2 \|B^{\varepsilon}_{\tau, x_0, \loc}\|_{\C^\alpha}^2 \tau^{1-c} \\
   &\lesssim_{\M} \tau^{1-c},
 \end{aligned}
 \end{equation*}
 which yields~\eqref{prep:eq-Beps_hhtolow_bound}. Here we used that for any $0 < \frac12 -s \ll 1$, we may choose $s < \alpha < \frac12$ sufficiently close to $\frac12$ and we may pick $0 < \nu \ll 1$ sufficiently small so that $s + 1 - 3\alpha +\nu < 0$.
 This finishes the proof of the proposition.  
\end{proof}

So far, we have constructed the Brownian path $B\colon \R \rightarrow \M$ and its smooth approximations. We now turn to initial velocities. To this end, we let $\Wb\colon \R\rightarrow \R^\dimA$ be an independent copy of the Euclidean Brownian motion $W$. Then, we define the integrated velocity $\Vc\colon \R \rightarrow \R^\dimA$ by the Stratonovich integral\footnote{Since $B$ and $\Wb$ are independent and hence have vanishing cross-variation, the It\^o and Stratonovich integrals are identical in this case.}
\begin{equation} \label{prep:eq-def-integrated-velocity}
 \Vc(x) := \int_0^x P(B(y)) \circ \mathrm{d}\Wb(y). 
\end{equation}
We define the velocity $V\colon \R\rightarrow \R^\dimA$ as the distributional derivative of $\Vc$. We now define smooth approximations of both $\Vc$ and $V$. To this end, we first define a smooth approximation of the Euclidean Brownian motion by 
\begin{equation*}
\Wb^\varepsilon := P_{\leq \varepsilon^{-1}} \Wb = K_\varepsilon \ast \Wb. 
\end{equation*}
Then, we define the velocity $V^\varepsilon\colon \R\rightarrow \R^\dimA$ and its integral $\Vc^\varepsilon$ by 
\begin{align}
V^\varepsilon(x) &:= P(B^\varepsilon(x)) \partial_x \Wb^\varepsilon(x), \\
\Vc^\varepsilon(x) &:= \int_0^x V^\varepsilon(y) \, \mathrm{d}y. 
\end{align}
We note that $V^\varepsilon \in (B^\varepsilon)^\ast T\M$, i.e., $V^\varepsilon(x) \in T_{B^\varepsilon(x)}\M$ for all $x\in \R$. 
Next, for any $\tau > 0$ and $x_0 \in \bbR$ we define the re-scaled and translated velocity, respectively integrated velocity, by
\begin{equation} \label{prep:equ-Vepstaux0_def}
 V^\varepsilon_{\tau, x_0}(x) := \tau V^\varepsilon( \tau x+x_0 ), \quad \Vc^\varepsilon_{\tau, x_0}(x) := \Vc^\varepsilon(\tau x+x_0) - \Vc^\varepsilon(x_0).
\end{equation}
Introducing the associated re-scaled and translated Euclidean Brownian motion
\begin{equation*}
 \Wb^\varepsilon_{\tau,x_0}(x) := \tau^{-\frac12} \bigl( \Wb^\varepsilon(\tau x+x_0) - \Wb^\varepsilon(x_0) \bigr),
\end{equation*}
it follows that $V^\varepsilon_{\tau,x_0}$ satisfies
\begin{equation*}
 \begin{aligned}
  V^\varepsilon_{\tau,x_0}(x) &= \tau^{\frac12} P(B^\varepsilon_{\tau,x_0}(x)) \partial_x \Wb^\varepsilon_{\tau,x_0}(x).
 \end{aligned}
\end{equation*}

Furthermore, we define the spatially localized versions $\Vc^{\varepsilon}_{\tau,x_0, \loc}$ and $V^\varepsilon_{\tau,x_0, \loc}$ by 
\begin{align}
V^{\varepsilon}_{\tau,x_0,\loc}(x) &:= \tau^{\frac12} P(B^\varepsilon_{\tau,x_0,\loc}(x)) \partial_x \bigl( \chi(x) \Wb^\varepsilon_{\tau,x_0}(x) \bigr), \label{prep:equ-Vepstaux0loc_def} \\
\Vc^{\varepsilon}_{\tau,x_0,\loc}(x) &:= \int_0^x V^{\varepsilon}_{\tau, x_0, \loc}(y) \, \mathrm{d}y, \label{prep:equ-Vc_epstaux0loc_def}
\end{align}
where the smooth cutoff $\chi(x)$ is the same as above. Note that $V^\varepsilon_{\tau, x_0} \in (B^\varepsilon_{\tau, x_0})^\ast T \M$ and $V^\varepsilon_{\tau, x_0, \loc} \in (B^\varepsilon_{\tau, x_0, \loc})^\ast T \M$. In particular, we obtain that $V^\varepsilon_{\tau, x_0,\loc}(x) \in T_{B^\varepsilon_{\tau, x_0}(x)} \M$ for $x \in [-2,2]$.

The main properties of the smooth approximations of the velocity and the integrated velocity are included in the following proposition. 

\begin{proposition}[Smooth approximation of the white noise velocities] \label{prep:prop-velocity-approximation}
Let $x_0 \in \bbR$ be arbitrary and let $s < \alpha < \frac12$ with $0 < \frac12 - \alpha \ll 1$ sufficiently small.
There exists a small constant $0 < c(\M) \ll 1$ depending only on $\M$ such that for all $0 < \tau \leq c(\M)$, there exists an event $\mathcal{E}(\tau, x_0) \subseteq \Omega$ satisfying
\begin{equation*}
    \mathbb{P}\big(\mathcal{E}(\tau,x_0)\big) \geq 1 - C \exp(-c \tau^{-c})
\end{equation*}
and on which the following properties hold:
\begin{enumerate}[label={(\roman*)}]
    \item\label{prep:item-velocity-approx-1} Convergence: 
    There exist $V_{\tau, x_0, \loc}, \Vc_{\tau,x_0,\loc} \colon \bbR \to \mathbb{R}^\dimA$ such that $\px \Vc_{\tau,x_0,\loc} = V_{\tau,x_0,\loc}$ and $(V^\varepsilon_{\tau,x_0,\loc}, \Vc^\varepsilon_{\tau,x_0,\loc}) \to (V_{\tau, x_0, \loc}, \Vc_{\tau,x_0,\loc})$ in $\C^{\alpha-1} \times \C^\alpha$. 
    
    \item\label{prep:item-velocity-approx-2} Effect of spatial cut-off: For all $\varepsilon>0$ and $x \in [-2,2]$, we have
     \begin{alignat}{3}
      V^{\varepsilon}_{\tau,x_0,\loc}(x) &= V^{\varepsilon}_{\tau,x_0}(x), \qquad &V_{\tau,x_0,\loc}(x) &= \tau V(\tau x+x_0), \label{equ:velocity_effect_spatial_cutoff}  \\
      \Vc^\varepsilon_{\tau,x_0,\loc}(x) &= \Vc^\varepsilon_{\tau,x_0}(x), \qquad &\Vc_{\tau,x_0,\loc}(x) &= \Vc(\tau x+x_0) - \Vc(x_0),
     \end{alignat}
     where \eqref{equ:velocity_effect_spatial_cutoff} holds in the sense of distributions.
     
    \item\label{prep:item-velocity-approx-3} Para-controlled structure: For all $\varepsilon>0$ there exist vector-valued functions $V^{\varepsilon,\#}_{\tau,x_0,\loc} \in \C^{2\alpha-1}$ and $\Vc^{\varepsilon,\#}_{\tau,x_0,\loc} \in \C^{2\alpha}$ such that 
    \begin{equation}\label{prep:eq-velocity-approx-parastructure}
    \begin{aligned}
        V^{\varepsilon}_{\tau,x_0,\loc} &= \tau^{\frac12} P_{\ext}(B^{\varepsilon}_{\tau,x_0,\loc}) \parall \partial_x \bigl( \chi(\cdot) \widebar{W}^\varepsilon_{\tau,x_0,\loc} \bigr) + V^{\varepsilon,\#}_{\tau,x_0,\loc}, \\
         \Vc^{\varepsilon}_{\tau,x_0,\loc} &= \tau^{\frac12} P_{\ext}(B^{\varepsilon}_{\tau,x_0,\loc}) \parall \bigl( \chi(\cdot) \widebar{W}^\varepsilon_{\tau,x_0,\loc} \bigr) + \Vc^{\varepsilon,\#}_{\tau,x_0,\loc}. 
    \end{aligned}
    \end{equation}
    \item\label{prep:item-velocity-approx-4} Uniform bounds: For all $\varepsilon>0$, it holds that 
    \begin{alignat}{3}
    \| V^{\varepsilon}_{\tau,x_0,\loc} \|_{\C^{\alpha-1}} &\lesssim_{\M} \tau^{\frac12-c}, 
    \qquad 
    &\| \Vc^{\varepsilon}_{\tau,x_0,\loc} \|_{\C^{\alpha}} &\lesssim_{\M} \tau^{\frac12-c}, 
    \label{prep:eq-velocity_Calpha_bounds-1} \\ 
    \| V^{\varepsilon, \#}_{\tau,x_0,\loc} \|_{\C^{2\alpha-1}} &\lesssim_{\M} \tau^{\frac12-c},
    \qquad
    &\| \Vc^{\varepsilon, \#}_{\tau,x_0,\loc} \|_{\C^{2\alpha}} &\lesssim_{\M} \tau^{\frac12-c}.
    \label{prep:eq-velocity_Calpha_bounds-2}
    \end{alignat}
    \item \label{prep:item-velocity_hhtolow} Uniform high$\times$high$\rightarrow$low-bounds: For all $1 \leq j_1, j_2 \leq D$ and all $\varepsilon > 0$, it holds that 
    \begin{equation} \label{prep:eq-velocity_hhtolow_bound}
     \begin{aligned}
         \sup_{\varepsilon>0} \sup_{M\sim N} M^{s} \big\|  P_M \Vc^{\varepsilon,j_1}_{\tau,x_0,\loc} \, P_N V^{\varepsilon,j_2}_{\tau,x_0,\loc} \big\|_{\C^{s-1}} &\lesssim_{\M} \tau^{\frac12-c}, \\
         \sup_{\varepsilon>0} \sup_{M\sim N} M^{s} \big\|  P_M \bigl( B^{\varepsilon,j_1}_{\tau,x_0,\loc} - B^{\varepsilon,j_1}_{\tau,x_0,\loc}(0) \bigr) \, P_N V^{\varepsilon,j_2}_{\tau,x_0,\loc} \big\|_{\C^{s-1}} &\lesssim_{\M} \tau^{\frac12-c}, \\
         \sup_{\varepsilon>0} \sup_{M\sim N} M^{s} \big\|  P_M \Vc^{\varepsilon,j_1}_{\tau,x_0,\loc} \, \partial_x P_N \bigl( B^{\varepsilon,j_2}_{\tau,x_0,\loc} - B^{\varepsilon,j_2}_{\tau,x_0,\loc}(0) \bigr) \big\|_{\C^{s-1}} &\lesssim_{\M} \tau^{\frac12-c}.
     \end{aligned}
    \end{equation}
    \item \label{prep:item-velocity_hhtolow-convergence} Convergence of high$\times$high$\rightarrow$low-term: For all $1\leq j_1, j_2 \leq D$ and all $\varepsilon > 0$, it holds that 
    \begin{equation} \label{prep:eq-velocity_hhtolow_conv}
     \begin{aligned}
       &\lim_{\varepsilon \to 0} \, \sup_{M\sim N} M^{s} \big\| P_M \Vc^{\varepsilon,j_1}_{\tau,x_0,\loc} \, P_N V^{\varepsilon,j_2}_{\tau,x_0,\loc} - P_M \Vc^{j_1}_{\tau,x_0,\loc} \, P_N V^{j_2}_{\tau,x_0,\loc} \big\|_{\C^{s-1}} = 0, \\
       &\lim_{\varepsilon \to 0} \, \sup_{M\sim N} M^{s} \big\| P_M \bigl( B^{\varepsilon,j_1}_{\tau,x_0,\loc} - B^{\varepsilon,j_1}_{\tau,x_0,\loc}(0) \bigr) \, P_N V^{\varepsilon,j_2}_{\tau,x_0,\loc} \\
       &\qquad \qquad \qquad \qquad \quad - P_M \bigl( B^{j_1}_{\tau,x_0,\loc} - B^{j_1}_{\tau,x_0,\loc}(0) \bigr) \, P_N V^{j_2}_{\tau,x_0,\loc} \big\|_{\C^{s-1}} = 0, \\
       &\lim_{\varepsilon \to 0} \, \sup_{M\sim N} M^{s} \big\| P_M \Vc^{\varepsilon,j_1}_{\tau,x_0,\loc} \, P_N \bigl( B^{\varepsilon,j_2}_{\tau,x_0,\loc} - B^{\varepsilon,j_2}_{\tau,x_0,\loc}(0) \bigr) \\
       &\qquad \qquad \qquad \qquad \quad - P_M \Vc^{j_1}_{\tau,x_0,\loc} \, P_N \bigl( B^{j_2}_{\tau,x_0,\loc} - B^{j_2}_{\tau,x_0,\loc}(0) \bigr) \big\|_{\C^{s-1}} = 0.       
     \end{aligned}
    \end{equation}
\end{enumerate}
\end{proposition}
\begin{proof}
 The assertions can be proved using similar arguments as in the proof of the preceding Proposition~\ref{prep:prop-brownian-path-approximation}. 
 We refer to \cite[Chapter 6, Theorem 7.1]{IkeWat89} for a Wong-Zakai type result for approximations of Stratonovich stochastic integrals. Most of the arguments are in fact simpler because the integrated velocities are defined via stochastic integration, while the Brownian paths are defined in terms of a stochastic differential equation. The details are left to the reader.
\end{proof}

\begin{corollary} \label{prep:cor_Cloc_convergence_to_BV}
 Let $B \colon \bbR \to \M$ be the Brownian path, let $V \in B^\ast T\M$ be the white noise velocity, and let $(B^\varepsilon)_{\varepsilon > 0}$ and $(V^\varepsilon)_{\varepsilon > 0}$ be their smooth approximations. 
 For any $R > 0$ we have almost surely that $(B^\varepsilon, V^\varepsilon)$ converges to $(B, V)$ in $C^s \times C^{s-1} ([-R,R] \to T\M)$ as $\varepsilon \to 0$.
\end{corollary}
\begin{proof}
 The assertion follows by a covering argument from the properties (i)-(ii) of the smooth local approximations established in Proposition~\ref{prep:prop-brownian-path-approximation} and in Proposition~\ref{prep:prop-velocity-approximation}.
\end{proof}

We recall that the (shifted) linear waves corresponding to the initial data $(B^{\varepsilon}_{\tau,x_0,\loc}, V^{\varepsilon}_{\tau,x_0,\loc})$ are given by 
\begin{equation*}
 \phi^{\pm, \varepsilon}_{\tau,x_0,\loc} := \frac{1}{2} \Big( B^{\varepsilon}_{\tau,x_0,\loc} -  B^{\varepsilon}_{\tau,x_0,\loc}(0) \mp \Vc^{\varepsilon}_{\tau,x_0,\loc} \Big),
\end{equation*}
with analogous definitions of the linear waves $\phi^{\pm}_{\tau,x_0,\loc}$ for the initial data $(B_{\tau,x_0,\loc}, V_{\tau,x_0,\loc})$.
In the final proposition of this section, we prove a variant of the high$\times$high$\rightarrow$low-bound, which also allows for shifts in the argument.

\begin{proposition}[High$\times$high$\rightarrow$low-bound with shifts] \label{prep:lem-mixed_hhtolow}
Let $x_0 \in \bbR$ be arbitrary and let $s < \alpha < \frac12$ with $0 < \frac12 - \alpha \ll 1$ sufficiently small.
There exists a small constant $0 < c(\M) \ll 1$ depending only on $\M$ such that for all $0 < \tau \leq c(\M)$, there exists an event $\mathcal{E}(\tau, x_0) \subseteq \Omega$ satisfying
\begin{equation*}
    \mathbb{P}\big(\mathcal{E}(\tau,x_0)\big) \geq 1 - C \exp(-c \tau^{-c})
\end{equation*}
and on which the following properties hold:
\begin{enumerate}
    \item \label{prep:item-mixed_hhtolow} Uniform high$\times$high$\rightarrow$low-bounds: For all $1 \leq j_1, j_2 \leq D$, it holds that 
    \begin{equation} \label{prep:eq-mixed_hhtolow_bound}
    \begin{aligned}
         \sup_{\varepsilon>0} \sup_{t \in \bbR}  \sup_{M\sim N} M^{s} \big\|  P_M \phi^{+,\varepsilon,j_1}_{\tau, x_0, \loc}(x-t) \,  
         \partial_x P_N \phi^{-,\varepsilon,j_2}_{\tau, x_0, \loc}(x+t)
         \big\|_{\C^{s-1}} &\lesssim_{\M} \tau^{\frac12 - c}, \\
             \sup_{\varepsilon>0} \sup_{t \in \bbR}  \sup_{M\sim N} M^{s} \big\|  \partial_x P_M \phi^{+,\varepsilon,j_1}_{\tau, x_0, \loc}(x-t) \,  
          P_N \phi^{-,\varepsilon,j_2}_{\tau, x_0, \loc}(x+t)
         \big\|_{\C^{s-1}} &\lesssim_{\M} \tau^{\frac12 - c}. 
         \end{aligned}
    \end{equation}
    \item \label{prep:item-mixed_hhtolow-convergence} Convergence of high$\times$high$\rightarrow$low-term: For all $1\leq j_1, j_2 \leq D$, it holds that 
    \begin{equation} \label{prep:eq-mixed_hhtolow_conv}
    \begin{aligned}
       \lim_{\varepsilon \to 0} \, \sup_{t \in \bbR}  \sup_{M\sim N} M^{s} \big\|& 
        P_M \phi^{+,\varepsilon,j_1}_{\tau, x_0, \loc}(x-t) \,  \partial_x P_N \phi^{-,\varepsilon,j_2}_{\tau, x_0, \loc}(x+t) \\
        &-P_M \phi^{+,j_1}_{\tau, x_0, \loc}(x-t) \, \partial_x P_N \phi^{-, j_2}_{\tau, x_0, \loc}(x+t) \big\|_{\C^{s-1}} = 0. \\
          \lim_{\varepsilon \to 0} \, \sup_{t \in \bbR}  \sup_{M\sim N} M^{s} \big\|& 
        \partial_x P_M \phi^{+,\varepsilon,j_1}_{\tau, x_0, \loc}(x-t) \,   P_N \phi^{-,\varepsilon,j_2}_{\tau, x_0, \loc}(x+t) \\
        &- \partial_x P_M \phi^{+,j_1}_{\tau, x_0, \loc}(x-t) \,  P_N \phi^{-, j_2}_{\tau, x_0, \loc}(x+t) \big\|_{\C^{s-1}} = 0.
        \end{aligned}
    \end{equation}
\end{enumerate}
\end{proposition}

\begin{remark}\label{prep:rem-mixed}
In Proposition \ref{prep:lem-mixed_hhtolow}, it is essential to consider the mixed term involving both $\phi^+$ and $\phi^-$. In other cases, the probabilistic resonance poses a severe problem.
\end{remark}

\begin{proof}
 The argument is similar to the proofs of Proposition~\ref{prep:prop-brownian-path-approximation} and Proposition~\ref{prep:prop-velocity-approximation}, but this time we get no resonant term (at top order). The reason is that $W+ \widebar{W}$ and $W-\widebar{W}$ are independent due to the rotation invariance of Brownian motion.
 To see this more directly, we check that the cross-covariance matrix vanishes, which suffices to conclude independence for Gaussian random variables.
 Indeed, we have for any $x_1, x_2 \in \bbR$ and any $1 \leq i,j \leq D$ that
 \begin{equation*}
 \begin{aligned}
  &\mathrm{cov}\bigl( (W^i+\widebar{W}^i)(x_1), (W^j-\widebar{W}^j)(x_2) \bigr) \\
  &= \bbE\Bigl[ \Bigl( (W^i+\widebar{W}^i)(x_1) - \bbE\bigl[ (W^i+\widebar{W}^i)(x_1) \bigr] \Bigr) \Bigl( (W^j-\widebar{W}^j)(x_2) - \bbE\bigl[ (W^j-\widebar{W}^j)(x_2) \bigr] \Bigr) \Bigr] \\
  &= \bbE\Bigl[ \Bigl( W^i(x_1) - \bbE\bigl[W^i(x_1)\bigr] \Bigr) \Bigl( W^j(x_2) - \bbE\bigl[W^j(x_2)\bigr] \Bigr) \Bigr] \\
  &\quad \quad - \bbE\Bigl[ \Bigl( \widebar{W}^i(x_1) - \bbE\bigl[\widebar{W}^i(x_1)\bigr] \Bigr) \Bigl( \widebar{W}^j(x_2) - \bbE\bigl[\widebar{W}^j(x_2)\bigr] \Bigr) \Bigr] \\
  &= 0,
 \end{aligned}
 \end{equation*}
 as desired.
\end{proof}

%%%%%%%%%%%%%%%%%%%%% Ansatz %%%%%%%%%%%%%%%%%%%%%%%%%%%%%%%%%
\section{Ansatz}\label{section:ansatz}

In the introduction (Section \ref{section:intro-main}), we gave an informal description of the Ansatz for $\phi$. In this section, we give a rigorous definition of the terms in our previous discussion.

\subsection{The shifted wave map}
In the proof of Theorem \ref{intro:thm-rigorous}, which is a local result, we  utilize some form of smallness. In principle, smallness can be created by localizing to times $|t|\leq \tau \ll 1$. However, due to the complexity of our Ansatz, exhibiting gains in $\tau$ in this fashion is tedious. Instead of localizing to a short time-interval, we use a scaling argument. To this end, we use the scale-invariance, translation-invariance, and finite speed of propagation of \eqref{intro:eq-WM}, which allows us to replace $B^\varepsilon$ and $V^\varepsilon$ by their re-scaled, translated, and localized counterparts (see, e.g., the definition of $B^\varepsilon_{\tau, x_0,\loc}$ in \eqref{prep:equ-Bepstaux0loc_def} and the definition of $V^\varepsilon_{\tau,x_0,\loc}$ in \eqref{prep:equ-Vepstaux0loc_def}). Unfortunately, this alone is not sufficient to yield small data, since the manifold constraint $B^\varepsilon(x)\in \M$ prohibits smallness in $L^\infty$. To resolve this issue, we need to subtract the initial position at $x=0$, and the resulting shifted wave map will be described in this subsection. \\

While our main result (Theorem \ref{intro:thm-rigorous}) is stated just for the random data $(B^\varepsilon,V^\varepsilon)$, and will be proven using the re-scaled, translated, and localized counterparts $B^\varepsilon_{\tau, x_0,\loc}$ and $V^\varepsilon_{\tau, x_0,\loc}$, our argument requires only the regularity and high$\times$high$\rightarrow$low-estimates from Section \ref{section:Brownian}. As a result, it is more convenient to work in a general framework. To this end, we let $\phi_0\colon \R \rightarrow \M$  and $\phi_1\in \phi_0^\ast T\M$. Then, we define the shifted initial data by 
\begin{equation}
\phi_0^\diamond(x):= \phi_0(x)- \phi_0(0) \qquad \text{and} \qquad \phi_1^\diamond(x)=\phi_1(x). 
\end{equation}
We remark that the shifted data satisfies the geometric constraint 
\begin{equation*}
    \phi_1^\diamond \in (\phi_0^\diamond)^\ast T \big( \M - \phi_0(0)\big), 
\end{equation*}
but this will not be used in the rest of the paper. We define the corresponding shifted linear waves $\phi^{\diamond,+}$ and $\phi^{\diamond,-}$ by 
\begin{equation}\label{ansatz:eq-shifted-linear}
\phi^{\diamond,\pm}(x) = \frac{1}{2\theta} \Big( \phi_0^{\diamond}(x) \mp \int_0^x \dy  \, \phi_1^\diamond(y) \Big).
\end{equation}
Here, the parameter $\theta>0$ is as in Section \ref{section:parameters}. The purposes of the $(1/\theta)$-factor in \eqref{ansatz:eq-shifted-linear} will be clear from \eqref{ansatz:eq-shifted-WM-nc} below. We also define the shifted second fundamental form by \begin{equation}\label{ansatz:eq-shifted-second}
\SecondC{k}{ij}(\zeta)= \SecondC{k}{ij}(\zeta;\phi_0(0)) := \Second^k_{ij}(\zeta+\phi_0(0)). 
\end{equation}
For most of this paper, the dependence of $\SecondC{k}{ij}$ on $\phi_0(0)$ is omitted from our notation. Finally, we define the shifted wave map $\phi^\diamond$ by 
\begin{equation}\label{ansatz:eq-shifted-WM}
\phi^\diamond(t,x) := \phi(t,x) - \phi_0(0).
\end{equation}
As a result of our definitions, we see that $\phi^\diamond$ satisfies the shifted wave maps equation
\begin{equation}\label{ansatz:eq-shifted-WM-equation}
\begin{cases}
 \partial_\mu \partial^\mu \phi^{\diamond,k} = - \SecondC{k}{ij}(\phi^\diamond) \partial_\mu \phi^{\diamond,i} \partial^\mu \phi^{\diamond,j} \\
 \phi^\diamond(0,x)= \phi^\diamond_0(x), ~ \partial_t \phi^\diamond(0,x)= \phi_1^\diamond(x). 
 \end{cases}
\end{equation}
In our analysis of \eqref{ansatz:eq-shifted-WM-equation}, we will exclusively work with its Duhamel integral formulation. Due to the mapping properties of the Duhamel integral on Hölder spaces, it is convenient to also insert our cut-off functions $\chi^+$ and $\chi^-$ from \eqref{prep:eq-chi-pm}. Due to finite speed of propagation, this does not alter the behavior of $\phi^{\diamond}$ on $|u|,|v|\leq 2$. In total, this leads to the fixed-point problem
\begin{equation}\label{ansatz:eq-shifted-WM-nc}
\phi^{\diamond,k} = \theta \phi^{\diamond,+,k} + \theta \phi^{\diamond,-,k}
- \chi^+ \chi^- \Duh\Big[ \SecondC{k}{ij}(\phi^\diamond) \partial_u \phi^{\diamond,i} \partial_v \phi^{\diamond,j} \Big]. 
\end{equation}
As stated above, the purpose of the $(1/\theta)$-factor in \eqref{ansatz:eq-shifted-linear} stems from \eqref{ansatz:eq-shifted-WM-nc}. It leads to the small $\theta$-factor in \eqref{ansatz:eq-shifted-WM-nc}, which allows us to consider small modulations in our Ansatz (see Section \ref{section:rigorous-ansatz}). \\

\emph{A word on notation:} Throughout Sections \ref{section:ansatz}-\ref{section:modulation}, we work exclusively with the shifted wave map $\phi^\diamond$ and the shifted linear waves $\phi^{\diamond,+}$ and $\phi^{\diamond,-}$. In order to simplify the notation (see e.g.  Section \ref{section:nonlinear}), we omit the superscript ``$\diamond$" in our notation for $\phi^\diamond$, $\phi^{\diamond,+}$, and $\phi^{\diamond,-}$. Therefore, we simply write $\phi$, $\phi^+$, and $\phi^-$. In Section \ref{section:lwp}, where both the original and shifted wave maps are present, we distinguish between $\phi$ and $\phi^\diamond$. Throughout the whole paper, we include the superscript ``$\diamond$" in $\Second^\diamond$. This is because we do not want to change the definition of a central object in mathematics.

\subsection{The enhanced initial data}
We let $\phi^+$ and $\phi^-$ be the linear waves from \eqref{ansatz:eq-shifted-linear}, where we recall that the superscript ``$\diamond$" is omitted from our notation.   As will be further detailed in Hypothesis \ref{hypothesis:smallness}, our analysis will be based on the following assumptions.
\begin{itemize}
    \item[(1)] Regularity condition: It holds that $\| \phi^+(x) \|_{\C_x^s(\R)}+ \| \phi^-(x) \|_{\C^s_x(\R)}\leq \theta$.
    \item[(2)] High$\times$high$\rightarrow$low-bounds: For $\pm_1,\pm_2 \in \{ +, - \}$ it holds that  
    \begin{align*}
    \max_{\pm_1,\pm_2} \max_{1\leq m,n \leq \dimA} \sup_{M\sim N} M^s 
    \| P_M\phi^{\pm_1,m}(x) \, \partial_x P_N \phi^{\pm_2,n}(x) \|_{\C^{s-1}_x(\R)} &\leq \theta^2, \\
     \max_{1\leq m,n \leq \dimA} \sup_{M\sim N} \sup_{t\in \R} M^s 
    \| P_M\phi^{+,m}(x-t) \, \partial_x P_N \phi^{-,n}(x+t) \|_{\C^{s-1}_x(\R)}
    &\leq \theta^2, \\ 
     \max_{1\leq m,n \leq \dimA} \sup_{M\sim N} \sup_{t\in \R} M^s 
    \| \, \partial_x P_M \phi^{+,m}(x-t)   P_N\phi^{-,n}(x+t) \|_{\C^{s-1}_x(\R)}
    &\leq \theta^2. 
    \end{align*}
\end{itemize}
As previously discussed in Remark \ref{prep:rem-mixed}, we emphasize that the term involving the supremum over $t\in \R$ only covers products of linear waves in different directions. Due to the bilinear nature of the high$\times$high$\rightarrow$low-bound, the two assumptions cannot be captured through a norm or metric on $\C_x^s \times \C_x^s$. As in the previous literature on random dispersive equations (see e.g. the survey \cite{BOP19}), we will interpret the assumptions through an enhanced data set.

\begin{definition}[Enhanced data and $\Ds$]\label{ansatz:definition-data} 
For any $\phi^{+},\phi^-\colon \R \rightarrow \R^{\dimA}$, $\pm_1,\pm_2 \in \{ +, - \}$, $1\leq m,n \leq \dimA$, and $M,N\geq 1$ satisfying $M\sim N$, we define $\Phi^{\pm_1,\pm_2,m,n}_{M,N}\colon \R \rightarrow \R$ by 
\begin{equation}
\Phi^{\pm_1,\pm_2,m,n}_{M,N}(x):= M^s P_M\phi^{\pm_1,m}(x) \, \partial_x P_N \phi^{\pm_2,n}(x).
\end{equation}
In the case $\pm_1 \neq \pm_2$, we also define the shifted $\Phi^{\pm_1,\pm_2,(s),m,n}_{M,N}\colon \R^{1+1}_{t,x} \rightarrow \R$ by 
\begin{equation}
\Phi^{\pm_1,\pm_2,(s),m,n}_{M,N}(x;t):= M^s  P_M\phi^{\pm_1,m}(x-t) \, \partial_x P_N \phi^{\pm_2,n}(x+t).
\end{equation}
We then define 
\begin{align*}
\| ( \phi^+, \phi^-) \|_{\Ds} := \max\bigg(& \| \phi^+ \|_{\C_x^s}+ \| \phi^- \|_{\C^s_x}, 
\max_{\pm_1,\pm_2} \max_{1\leq m,n \leq \dimA} \sup_{M\sim N} \sqrt{ \| \Phi^{\pm_1,\pm_2,m,n}_{M,N}(x) \|_{\C_x^{s-1}}}, \\
&\max_{\pm_1 \neq \pm_2} \max_{1\leq m,n \leq \dimA} \sup_{M\sim N} \sup_{t\in \R} \sqrt{ \| \Phi^{\pm_1,\pm_2,(s),m,n}_{M,N}(x;t) \|_{\C_x^{s-1}}}
\bigg). 
\end{align*}
\end{definition}

We note that the regularity condition and high$\times$high$\rightarrow$low-bound are equivalent to the $\Ds$-bound $\| (\phi^+,\phi^-)\|_{\Ds}\leq \theta$.

\begin{remark}
As mentioned above, $\| \cdot \|_{\Ds}$ is not  actually a norm. However, it can be viewed as the composition of the nonlinear function $\phi^\pm \mapsto (\phi^\pm, \Phi)$, an honest norm, and the nonlinear function $(a,b)\mapsto a+\sqrt{b}$. By including the square-root, we have preserved the  $1$-homogeneity in $\phi^\pm$ but destroyed the $1$-homogeneity in $\Phi$. While  $\| \cdot\|_{\Ds}$ can therefore also not be viewed as a honest norm on the enhanced data $(\phi^\pm,\Phi)$, we will now associate it with a metric on the enhanced data.
\end{remark}

In order to prove the convergence statement in Theorem \ref{intro:thm-rigorous}, it is convenient to introduce continuous functions on $\Ds$. To this end, we first define the distance
\begin{align} \label{ansatz:eq-Ds-distance} 
&\| (\phi_1^+, \phi_1^-); (\phi_2^+, \phi_2^-) \|_{\Ds}  \\
:=&\max\bigg( \| \phi^+_1-\phi^+_2 \|_{\C_x^s}+ \| \phi^-_1-\phi^-_2 \|_{\C^s_x}, \notag\\
&\max_{\pm_1,\pm_2} \max_{1\leq m,n \leq \dimA} \sup_{M\sim N} \sqrt{ \| \Phi^{\pm_1,\pm_2,m,n}_{1,M,N}(x) -\Phi^{\pm_1,\pm_2,m,n}_{2,M,N}(x) \|_{\C_x^{s-1}}},  \notag \\
& \max_{\pm_1 \neq \pm_2} \max_{1\leq m,n \leq \dimA} \sup_{M\sim N} \sup_{t\in \R} \sqrt{ \| \Phi^{\pm_1,\pm_2,(s),m,n}_{1,M,N}(x;t) - \Phi^{\pm_1,\pm_2,(s),m,n}_{2,M,N}(x;t) \|_{\C_x^{s-1}}} \bigg).\notag 
\end{align}
Since square-roots preserve metrics, one can view $\| \cdot ; \cdot\|_{\Ds}$ as a metric on the enhanced data $(\phi^\pm,\Phi)$. 

\begin{definition}[Lipschitz continuous dependence w.r.t. $\Ds$]\label{ansatz:def-cts-dependence}
Let $(\mathbb{X},d)$ be a metric space and consider a function $F\colon \C_b^\infty \times \C_b^\infty \rightarrow \mathbb{X}$. We say that $F$ is Lipschitz continuous w.r.t. $\Ds$ if there exists a constant $\operatorname{Lip}(F)\in[0,\infty)$ such that 
\begin{equation}
d\big(  F(\phi^+_1,\phi^-_1), F(\phi^+_2,\phi^-_2) \big) \leq \operatorname{Lip}(F)
\|(\phi^+_1,\phi^-_1); (\phi^+_2,\phi^-_2) \|_{\Ds}
\end{equation}
for all $(\phi^+_1,\phi^-_1),(\phi^+_2,\phi^-_2)\in \Ds$. In other words, $F$ is Lipschitz continuous as a function of $(\phi^+,\phi^-,\Phi)$.
\end{definition}

While $\| \cdot\|_{\Ds}$ encapsulates the high$\times$high$\rightarrow$low-bound, it does not explicitly contain bounds on the low$\times$high and high$\times$low-para-products. The reason is that the optimal bounds for both of them follow directly from the regularity condition $\phi^+,\phi^- \in \C_x^s$. For the rest of the paper, it is convenient to encapsulate estimates of all frequency interactions in a single lemma.

\begin{lemma}\label{ansatz:lemma-data}
Let $\pm_1,\pm_2 \in \{+,-\}$, let $1\leq m,n \leq \dimA$, and let $M$ and $N$ be any frequency scales. Then, it holds that 
\begin{equation}\label{ansatz:eq-data}
\begin{aligned}
&\| P_M\phi^{\pm_1,m}(x)  \, \partial_x P_N \phi^{\pm_2,n}(x) \|_{\C^{r-1}_x} \\
 \lesssim& \Big( 1\{ M\lesssim N\} M^{-s} N^{r-s} + 1\{ M \gg N\} M^{r-1-s} N^{1-s}  \Big) \| (\phi^+, \phi^-) \|_{\Ds}^2 \\ 
 \lesssim&   M^{-s} N^{r-s}  \| (\phi^+, \phi^-) \|_{\Ds}^2. 
\end{aligned}
\end{equation}
\end{lemma}

\begin{proof}
The second inequality in \eqref{ansatz:eq-data} directly follows from
\begin{equation*}
   1\{ M \gg N\} M^{r-1-s} N^{1-s} \lesssim N^{r-1} M^{-s} N^{1-s} = M^{-s} N^{r-s}.
\end{equation*}
Thus, it remains to prove the first inequality. To this end, we distinguish the two cases $M\sim N$ and $M \not \sim N$. If $M\sim N$, we have that 
\begin{align*}
\| P_M\phi^{\pm_1,m}(x)  \, \partial_x P_N \phi^{\pm_2,n}(x) \|_{\C^{r-1}_x} 
&\lesssim M^{r-s} \| P_M\phi^{\pm_1,m}(x)  \, \partial_x P_N \phi^{\pm_2,n}(x) \|_{\C^{s-1}_x} \\
&\lesssim M^{r-2s} \| ( \phi^+, \phi^-) \|_{\Ds}^2 \\
&\lesssim M^{-s} N^{r-s} \| ( \phi^+, \phi^-) \|_{\Ds}^2.
\end{align*}
If $M\not \sim N$, we have that 
\begin{align*}
\| P_M\phi^{\pm_1,m}(x)  \, \partial_x P_N \phi^{\pm_2,n}(x) \|_{\C^{r-1}_x} 
&\lesssim \max(M,N)^{r-1} \| P_M\phi^{\pm_1,m}(x)  \, \partial_x P_N \phi^{\pm_2,n}(x) \|_{L^\infty_x} \\
&\lesssim \max(M,N)^{r-1}
\| P_M\phi^{\pm_1,m}(x)  \|_{L^\infty_x} 
\| \partial_x P_N \phi^{\pm_2,n}(x) \|_{L^\infty_x} \\
&\lesssim \max(M,N)^{r-1} M^{-s} N^{1-s} \| ( \phi^+, \phi^-) \|_{\Ds}^2.
\end{align*}
By further distinguishing the cases $M\ll N$ and $M\gg N$, this yields the desired estimate. 
\end{proof}

\subsection{Rigorous definition of the Ansatz}\label{section:rigorous-ansatz}

In the introduction, we motivated the form of our Ansatz, which will be written as 
\begin{equation}\label{ansatz:eq-Ansatz}
\begin{aligned}
\phi^i(u,v) 
&= \sum_M A^{+,i}_{M,m}(u,v) \phi^{+,m}_M(u) 
+ \sum_N A^{-,i}_{N,n}(u,v) \phi^{-,n}_N(v) \\
&+ \sum_{\substack{M,N\colon \\ M\sim_\delta N}} B^{i}_{M,N,mn}(u,v) \phi^{+,m}_M(u) \phi^{-,n}_N(v) 
+ \psi^i(u,v). 
\end{aligned}
\end{equation}

Throughout the paper, we impose the following frequency-support conditions on $A^+$, $A^-$, and $B$.

\begin{condition}[Frequency-support conditions]\label{condition:frequency}
For all $1\leq i,m,n \leq \dimA$ and frequency-scales $M$ and $N$, we impose the following frequency-support conditions:
\begin{align}
P^u_{\gg M^{1-\sigma}} A^{+,i}_{M,m} &= 0, \label{ansatz:eq-condition-p} \\
P^v_{\gg N^{1-\sigma}} A^{-,i}_{N,n} &= 0, \label{ansatz:eq-condition-m} \\
P^u_{\gg M^{1-\sigma}} B^{i}_{M,N,mn} = P^v_{\gg N^{1-\sigma}} B^{i}_{M,N,mn} &= 0. \label{ansatz:eq-condition-pm}
\end{align}
\end{condition}

While $A^{+,i}_{M,m}$ has $u$-frequencies $\lesssim M^{1-\sigma}$, we impose no frequency restrictions in the $v$-variable. This is because high$\times$high-products  such as $\phi^{+,m}_M(u) P_M^v \zeta(v)$, where
 $\zeta \in \C^r_v$, only belong to $\Cprod{s}{r}$ but do not belong to $\Cprod{r}{r}$. 
 
 The equations for $A^+,A^-,B$, and $\psi$ will be given further below in Section \ref{section:ansatz-equations}. For now, we concentrate on the norms in which the four unknowns will be measured. While the modulation equations for $A^+$ and $A^-$ will eventually be solved using a para-controlled approach (see Section \ref{section:modulation}), the PDE-analysis in Section \ref{section:prototypical}-\ref{section:duhamel} only requires the following norms.

\begin{definition}[Modulation norms]\label{ansatz:def-modulation-norms}
Let $A^+=(A^{+,i}_{M,m})$, $A^-=(A^{-,j}_{N,n})$, and $B=(B^k_{M,N,mn})$. Then, we define 
\begin{align}
\big\| (A^+,A^-,B) \big\|_{\Mod} &:= \sup_M \| A_M^+ \|_{\Mod_M^+} + \sup_N \| A_N^- \|_{\Mod_N^-} + \max_{1\leq k,m,n \leq \dimA} \sup_{M\sim_\delta N} \| B^k_{M,N,mn} \|_{\Cprod{s}{s}},
\end{align}
where 
\begin{align}
\| A_M^+ \|_{\Mod_M^+}&= 
\max_{1\leq i,m \leq \dimA} \Big( \| P_{\lesssim M^{1-\delta}}^v A_{M,m}^{+,i} \|_{\Cprod{s}{s}}+ \| P_{\gg M^{1-\delta}}^v A^{+,i}_{M,m} \|_{\Cprod{s}{r}} \Big),  \label{Ansatz:eq-Mod+}\\
\| A_N^- \|_{\Mod_N^-}&= 
\max_{1\leq j,n \leq \dimA} \Big( \| P_{\lesssim N^{1-\delta}}^u A_{N,n}^{-,j} \|_{\Cprod{s}{s}}+ \| P_{\gg N^{1-\delta}}^u A^{-,j}_{N,n} \|_{\Cprod{r}{s}} \Big). \label{Ansatz:eq-Mod-}
\end{align}
\end{definition}

Using the definition of the modulation norms, we directly obtain the following estimates for $A^+$ and $A^-$. 
\begin{corollary}[$\delta$-gain for $A^\pm$]\label{ansatz:cor-delta-Apm}
For all dyadic scales $M,N\geq 1$, it holds that 
\begin{align}
\| \partial_v A_M^+  \|_{\Cprod{s}{r-1}} 
\lesssim M^{(1-\delta)(r-s)} \| A_M^+  \|_{\Mod_M^+}, \label{ansatz:eq-prep-Ap} \\
\| \partial_u A_N^-  \|_{\Cprod{r-1}{s}} 
\lesssim N^{(1-\delta)(r-s)} \| A_N^-  \|_{\Mod_N^-} . \label{ansatz:eq-prep-Am}
\end{align}
\end{corollary}

We refer to the $(1-\delta)$-factor in the exponents as a $\delta$-gain. On a microscopic level, it is a direct consequence of the modified frequency-cutoffs in \eqref{Ansatz:eq-Mod+} and \eqref{Ansatz:eq-Mod-}. On a macroscopic level, the $\delta$-gain is made possible by including the bilinear term $ B^{i}_{M,N,mn} \phi^{+,m}_M \phi^{-,n}_N$ in our Ansatz. 

\begin{proof}
It suffices to prove \eqref{ansatz:eq-prep-Ap}, since the argument for \eqref{ansatz:eq-prep-Am} is similar. For all $M$, we have that 
\begin{align*}
\| \partial_v A_M^+  \|_{\Cprod{s}{r-1}} 
&\leq \| \partial_v P_{\lesssim M^{1-\delta}}^v A_M^+  \|_{\Cprod{s}{r-1}} 
+\|  \partial_v P_{\gg M^{1-\delta}}^v A_M^+  \|_{\Cprod{s}{r-1}} \\
&\lesssim M^{(1-\delta)(r-s)} \|  P_{\lesssim M^{1-\delta}}^v A_M^+  \|_{\Cprod{s}{s}} 
+\|  P_{\gg M^{1-\delta}}^v A_M^+  \|_{\Cprod{s}{r}} \\
&\lesssim M^{(1-\delta)(r-s)} \| A_M^+  \|_{\Mod_M^+}. 
\end{align*}
\end{proof}

 In Section \ref{section:prototypical}-\ref{section:modulation}, we will prove numerous estimates for the components in our Ansatz. To this end, it is helpful to collect all assumptions on the initial data and components in a single hypothesis, which can then be referenced easily. This hypothesis will ultimately form the basis of a contraction mapping argument. 
 
 \begin{hypothesis}[Smallness]\label{hypothesis:smallness}
Let $\theta$ be the smallness parameter introduced in Section \ref{section:parameters}. Then, we assume that 
\begin{equation*}
  \max\bigg( \| (\phi^+,\phi^-) \|_{\Ds}, \tfrac{1}{5}   \sup_{M} \| A_M^+ \|_{\Mod_M^+},  \tfrac{1}{5} \sup_{N} \| A_N^- \|_{\Mod_N^-}, \tfrac{1}{5}
  \sup_{M\sim_\delta N} \| B_{M,N} \|_{\Cprod{s}{s}}, \| \psi \|_{\Cprod{r}{r}} \bigg)
  \leq \theta. 
\end{equation*}
\end{hypothesis}
The $\frac{1}{5}$-factors in Hypothesis \ref{hypothesis:smallness} are included because of the $\theta \delta^a_m$ and $\theta \delta^a_n$ terms in \eqref{ansatz:eq-modulation-p} and \eqref{ansatz:eq-modulation-m} below.
The following notation, which describes the different components in our Ansatz, is used to organize the case-analysis in the null-form estimates (Section \ref{section:bilinear}).

 \begin{definition}[Types]
Let $\zeta\colon \R^{1+1}_{u,v}\rightarrow \R^{\dimA}$ be one of the four functions  in our Ansatz.  Then, we define $\Type\zeta \in \{ \tp, \tm , \tpm, \ts\}$ by 
\begin{equation}
\Type \zeta := 
\begin{cases}
\begin{tabular}{ll}
$\tp$  &if $\zeta= A_M^{+} \phi^{+}_M$ for some $M$,\\ 
$\tm $  &if $\zeta= A_N^{-} \phi^{-}_N$ for some $N$,\\ 
$\tpm$ &if $\zeta=B_{M,N} \phi^+_{M} \phi^{-}_{N}$ for some $M\sim_\delta N$, \\ 
$\ts$ 		&if $\zeta=\psi$. 
\end{tabular}
\end{cases}
\end{equation}
\end{definition} 

After proving several multi-linear estimates below, we will eventually need to sum over the dyadic frequency scales. This summation will be possible since all relevant estimates will exhibit a small gain in the frequency-scale exponents. To track this gain with minimal notational effort, we make the following definition. 

\begin{definition}[Gains]
Let $\eta$ be as in \eqref{prep:eq-parameter-2} and let $\zeta\colon \R^{1+1}_{u,v}\rightarrow \R^{\dimA}$ be one of the four functions  in our Ansatz. Then, we define the gain corresponding to $\zeta$ as 
\begin{equation}\label{gain0}
\Gain(\zeta) := 
\begin{cases}
\begin{tabular}{ll}
$M^{-\eta} $  &if $\zeta= A_M^{+} \phi^{+}_M$,\\ 
$N^{-\eta} $  &if $\zeta= A_N^{-} \phi^{-}_N$,\\ 
$(MN)^{-\eta}$ &if $\zeta=B_{M,N} \phi^+_{M} \phi^{-}_{N}$, \\ 
$1$ 		&if $\zeta=\psi$. 
\end{tabular}
\end{cases}
\end{equation}
\end{definition}

We now prove that the Ansatz \eqref{ansatz:eq-Ansatz} and  Hypothesis \ref{hypothesis:smallness} are sufficient to control the $\C_t^0 \C_x^s$ and $\C_t^1 \C_x^{s-1}$-norms of $\phi$, which will be used in the proof of Theorem \ref{intro:thm-rigorous}. 

\begin{lemma}[Regularity in space-time]\label{ansatz:lem-ctcx}
Let Hypothesis \ref{hypothesis:smallness} be satisfied. Then, it holds that 
\begin{equation}\label{ansatz:eq-ctcx-1}
\| \phi \|_{\C_t^0 \C_x^s(\R\times \R)} 
+ \| \phi \|_{\C_t^1 \C_x^{s-1}(\R\times \R)}
\lesssim \theta. 
\end{equation}
Furthermore, $\phi$ depends Lipschitz continuously on the linear waves $\phi^+$ and $\phi^-$, the modulations $A^+$, $A^-$, and $B$, and the nonlinear remainder $\psi$ with respect to the natural norm on $\Ds \times \Mod\times \Cprod{r}{r}$. 
\end{lemma}

We remark that the proof of Lemma \ref{ansatz:lem-ctcx} is the only part of our argument which requires the supremum over $t\in \R$ in the $\Ds$-norm. 

\begin{proof}
 We only prove the estimate \eqref{ansatz:eq-ctcx-1}, since the Lipschitz-continuous dependence follows from a minor variant of the argument. The desired estimate for the first, second, and fourth summand in \eqref{ansatz:eq-Ansatz} follows directly from Corollary \ref{prep:cor-null-Cartesian}. As a result, it remains to prove that 
 \begin{align}
\| B^{k}_{M,N,mn}(x-t,x+t) \phi^{+,m}_M(x-t) \phi^{-,n}_N(x+t) \|_{\C_t^0 \C_x^s(\R\times \R)} &\lesssim M^{-\eta} \theta^3, \label{ansatz:eq-ctcx-2}  \\
\| B^{k}_{M,N,mn}(x-t,x+t) \phi^{+,m}_M(x-t) \phi^{-,n}_N(x+t) \|_{\C_t^1 \C_x^{s-1}(\R\times \R)} &\lesssim M^{-\eta} \theta^3.\label{ansatz:eq-ctcx-3} 
 \end{align}
In order to prove the first inequality \eqref{ansatz:eq-ctcx-2}, we estimate 
 \begin{align*}
    &\| B^{k}_{M,N,mn}(x-t,x+t) \phi^{+,m}_M(x-t) \phi^{-,n}_N(x+t) \|_{\C_t^0 \C_x^s(\R\times \R)} \\
    \lesssim& \max(M,N)^s \| B^{k}_{M,N,mn}(u,v) \phi^{+,m}_M(u) \phi^{-,n}_N(v) \|_{L^\infty_u L^\infty_v} \\
    \lesssim& \max(M,N)^s (MN)^{-s} \theta^3.  
 \end{align*}
 Since $M\sim_\delta N$, this contribution is acceptable. We now turn to the second inequality \eqref{ansatz:eq-ctcx-3} and decompose
 \begin{align}
 &\partial_t \Big(  B^{k}_{M,N,mn}(x-t,x+t) \, \phi^{+,m}_M(x-t) \, \phi^{-,n}_N(x+t) \Big) 
 \notag \\
 =&  (\partial_v - \partial_u ) \big(B^{k}_{M,N,mn}\big)(x-t,x+t) \, \phi^{+,m}_M(x-t) \, \phi^{-,n}_N(x+t)\label{ansatz:eq-ctcx-4}  \\
 +& B^{k}_{M,N,mn}(x-t,x+t) \, \partial_u\phi^{+,m}_M(x-t) \, \phi^{-,n}_N(x+t) \label{ansatz:eq-ctcx-5} \\
 +& B^{k}_{M,N,mn}(x-t,x+t)\,  \phi^{+,m}_M(x-t) \, \partial_v\phi^{-,n}_N(x+t). \label{ansatz:eq-ctcx-6} 
 \end{align}
 For the first summand \eqref{ansatz:eq-ctcx-4}, we have that 
 \begin{align*}
&\| (\partial_v - \partial_u ) \big(B^{k}_{M,N,mn}\big)(x-t,x+t) \, \phi^{+,m}_M(x-t) \, \phi^{-,n}_N(x+t) \|_{\C_t^0 \C_x^{s-1}} \\
\lesssim&\| (\partial_v - \partial_u ) \big(B^{k}_{M,N,mn}\big)(u,v) \, \phi^{+,m}_M(u) \, \phi^{-,n}_N(v) \|_{L^\infty_u L^\infty_v} \lesssim \max(M,N)^{1-s} (MN)^{-s}  \theta^3. 
 \end{align*}
 Since $M\sim_\delta N$ and $1-3s\approx -1/2$, this contribution is acceptable. For the second summand, we obtain from the definition of the $\Ds$-norm that
 \begin{align*}
     &\| B^{k}_{M,N,mn}(x-t,x+t) \, \partial_u\phi^{+,m}_M(x-t) \, \phi^{-,n}_N(x+t) 
     \|_{\C_t^0 \C_x^{s-1}} \\
     \lesssim& \| B^{k}_{M,N,mn}(x-t,x+t) \|_{\C_t^0 \C_x^{1-s+\eta}} 
     \| \partial_u\phi^{+,m}_M(x-t) \, \phi^{-,n}_N(x+t)  \|_{\C_t^0 \C_x^{s-1}} \\ 
     \lesssim& \max(M,N)^{1-2s+\eta} M^{-s} \theta^3. 
 \end{align*}
  As in our estimate of \eqref{ansatz:eq-ctcx-4}, this contribution is acceptable. Since the estimate for the third summand \eqref{ansatz:eq-ctcx-6} is similar, this completes the proof.
\end{proof}

 \subsection{The modulation and forced wave maps equations}\label{section:ansatz-equations}
We now state the equations for the unknowns $A^+,A^-,B$, and $\psi$. The equations are split into a set of modulation equations for $A^+,A^-$, and $B$ and a forced wave maps equation for $\psi$. Before we can define the modulation equations, we need to define the auxiliary matrices $F^+,F^-,$ and $\G$, which appear as driving forces in the modulation equations.
 
\begin{definition}[The forcing terms $F^+$, $F^-$, and $\G$]\label{ansatz:def-F}
Let $1\leq a,m,n \leq \dimA$ and let $M$ and $N$ be dyadic frequency scales. Then, we define $F^+$ by 
\begin{align}
F_{M,m}^{+,a} 
&= \sum_{N\leq M^{1-\delta}} A_{M,m}^{+,i} \SecondC{a}{ij}(\phi) \partial_v \big( A_{N,n}^{-,j} \phi^{-,n}_N \big)  \label{nonlinear:eq-Fp-1} \allowdisplaybreaks[4] \\
&+ \sum_{\substack{\Type \zeta =\\ \tp, \tpm, \ts}} A_{M,m}^{+,i} \SecondC{a}{ij}(\phi) \partial_v \zeta^j \label{nonlinear:eq-Fp-2} \allowdisplaybreaks[4] \\
&+ \sum_{N\sim_\delta M} \big( \SecondC{a}{ij}(\phi) A_{M,m}^{+,i} A_{N,n}^{-,j} \big) \parasimv \partial_v \phi^{-,n}_N \label{nonlinear:eq-Fp-3} \allowdisplaybreaks[4]\\
&+ \sum_{N\sim_\delta M} \SecondC{a}{ij}( \phi) B_{M,N,mn}^i \big( \phi_{N}^{-,n} \paradownv \partial_v \psi^j \big). \label{nonlinear:eq-Fp-4} \allowdisplaybreaks[4]
\end{align}
Similarly, we define $F^-$ by 
\begin{align}
F_{N,n}^{-,a} 
&= \sum_{M\leq N^{1-\delta}} A_{N,n}^{-,j} \SecondC{a}{ij}(\phi) \partial_u \big( A_{M,m}^{+,i} \phi^{+,m}_M \big)  \label{nonlinear:eq-Fm-1} \allowdisplaybreaks[4] \\
&+ \sum_{\substack{\Type \zeta =\\ \tm, \tpm, \ts}} A_{N,n}^{-,j} \SecondC{a}{ij}(\phi) \partial_u \zeta^i \label{nonlinear:eq-Fm-2} \allowdisplaybreaks[4] \\
&+ \sum_{M\sim_\delta N} \big( \SecondC{a}{ij}(\phi) A_{M,m}^{+,i} A_{N,n}^{-,j} \big) \parasimu \partial_u \phi^{+,m}_M \label{nonlinear:eq-Fm-3} \allowdisplaybreaks[4]\\
&+ \sum_{M\sim_\delta N} \SecondC{a}{ij}( \phi) B_{M,N,mn}^j \big( \phi_{M}^{+,m} \paradownu \partial_u \psi^i \big) \label{nonlinear:eq-Fm-4}.
\end{align}
Finally, we define $\G$ by 
\begin{equation}\label{ansatz:eq-G}
G^a_{M,N,mn} = - P_{\leq M^{1-\sigma}}^u P_{\leq N^{1-\sigma}}^v \Big( \Second^a_{ij}(\phi) A^{+,i}_{M,m} A^{-,j}_{N,n} \Big). 
\end{equation}
\end{definition}
Instead of $\G$, we could have denoted the right-hand side of \eqref{ansatz:eq-G} by $F^{+,-}$. Since the modulation equation involving $\G$ is quite different from the modulation equations involving $F^+$ and $F^-$ (see Definition \ref{ansatz:def-modulation-eqs}), we choose our notation to reflect this difference. 

For notational convenience, we also define 
\begin{equation}\label{ansatz:eq-G-chi}
G^a_{\chi,M,N,mn} = - P_{\leq M^{1-\sigma}}^u P_{\leq N^{1-\sigma}}^v \Big( \chi^+ \chi^- \Second^a_{ij}(\phi) A^{+,i}_{M,m} A^{-,j}_{N,n} \Big). 
\end{equation}

While the precise definitions of $F^+$, $F^-$, and $G$ are quite complicated, their form will be motivated in Section \ref{section:modulation}. For now, they are simply convenient notation for stating the modulation equations. 

\begin{definition}[Modulation equations]\label{ansatz:def-modulation-eqs}
The modulation equations for $A^+$, $A^-$, and $B$ are given by
\begin{align}
A^{+,a}_{M,m}(u,v) &= 
\theta \delta^a_m 
- P^{u}_{\leq M^{1-\sigma}} \chi^+ \chi^- \int_u^v \dv^\prime P_{\leq M^{1-\sigma}}^u F_{M,m}^{+,a}(u,v^\prime) \label{ansatz:eq-modulation-p},\\
A^{-,a}_{N,n}(u,v) &= 
\theta \delta^a_n 
+ P^{v}_{\leq N^{1-\sigma}}  \chi^+ \chi^- \int_u^v \du^\prime P_{\leq N^{1-\sigma}}^v F_{N,n}^{-,a}(u^\prime,v) \label{ansatz:eq-modulation-m}, \\
B^{b}_{M,N,mn}(u,v) 
&=  1\{ M\sim_\delta N \} \, G^b_{\chi,M,N,mn}(u,v) \label{ansatz:eq-modulation-pm}, 
\end{align}
for all $1\leq a,b,m,n \leq \dimA$ and frequency scales $M$ and $N$. 
\end{definition}

As mentioned before,  we will solve the modulation equations using the para-controlled approach of \cite{GIP15}. However, the para-controlled structure of the modulations will only be used in \mbox{Section \ref{section:modulation}}. In the PDE-analysis, which spans the majority of the paper, we will only utilize the following well-posedness result for the modulation equations.

\begin{proposition}[Local well-posedness of the modulation equations]\label{ansatz:prop-modulation}
Let $\phi^+,\phi^- \colon  \R \rightarrow \R^{\dimA}$ and $\psi \colon \R_{u,v}^{1+1}\rightarrow \R^\dimA$ satisfy
\begin{equation}
\| ( \phi^+, \phi^-) \|_{\Ds}\leq \theta \quad \text{and} \quad \| \psi \|_{\Cprod{r}{r}} \leq \theta. 
\end{equation}
Then, there exists a solution $(A^+,A^-,B)$ of the system of modulation equations \eqref{ansatz:eq-modulation-p}, \eqref{ansatz:eq-modulation-m}, and \eqref{ansatz:eq-modulation-pm} satisfying
\begin{equation}
\| A_M^+ - \theta \operatorname{Id} \|_{\Mod_M^+}, \| A_N^- - \theta \operatorname{Id} \|_{\Mod_N^-}, \| B_{M,N}\|_{\Cprod{s}{s}} \leq \theta 
\end{equation}
for all frequency-scales $M$ and $N$. Furthermore, the modulations $(A^+,A^-,B)$ depend Lipschitz continuously on $(\phi_0(0),\phi^+,\phi^-,\psi)$ in $\M \times \Ds \times \Cprod{r}{r}$ (in the sense of Definition \ref{ansatz:def-cts-dependence}). 
\end{proposition}

To emphasize the dependence on $\psi$, we denote the modulations from Proposition \ref{ansatz:prop-modulation} by 
\begin{equation*}
A^+[\psi]=( A^{+,i}_{M,m}[\psi]), \quad A^-[\psi]=( A^{-,j}_{N,n}[\psi]), \quad \text{and} 
\quad B[\psi]=( B^{k}_{M,N,mn}[\psi]).
\end{equation*}

Similarly, we write 
\begin{equation}\label{ansatz:eq-phi}
\begin{aligned}
\phi^i[\psi](u,v) 
&= \sum_M A^{+,i}_{M,m}[\psi](u,v) \phi^{+,m}_M(u) 
+ \sum_N A^{-,i}_{N,n}[\psi](u,v) \phi^{-,n}_N(v) \\
&+ \sum_{\substack{M,N\colon \\ M\sim_\delta N}} B^{i}_{M,N,mn}[\psi](u,v) \phi^{+,m}_M(u) \phi^{-,n}_N(v) 
+ \psi^i(u,v). 
\end{aligned}
\end{equation}
While $A^+[\psi]$, $A^-[\psi]$, $B[\psi]$, and $\phi[\psi]$ also depend on the initial position $\phi_0(0)$ and the linear waves $\phi^+$ and $\phi^-$, this dependence has been neglected from our notation. It now remains to state the (forced) wave maps equation for the nonlinear remainder $\psi$, which is determined by the wave maps equation \eqref{ansatz:eq-shifted-WM-nc} and our Ansatz.

\begin{definition}[Forced wave maps equation for $\psi$]\label{ansatz:def-forced-wm}
The forced wave maps equation for the nonlinear remainder $\psi$ is given by 
\begin{equation}\label{ansatz:eq-forced-wm}
\psi^k = \theta \phi^{+,k} + \theta \phi^{-,k} - \chi^+ \chi^- \Duh\Big[ \SecondC{k}{ij}(\phi[\psi]) \partial_u \phi^i[\psi] \partial_v \phi^j[\psi] \Big] - (\phi^k[\psi]-\psi^k). 
\end{equation}
\end{definition}

Due to the modulation equations for $A^+,A^-$, and $B$, there are significant cancellations between the Duhamel integral and the difference $\phi^k[\psi]-\psi^k$ in \eqref{ansatz:eq-forced-wm}. The main result on the forced wave maps equation \eqref{ansatz:eq-forced-wm}, which utilizes these cancellations, is the following proposition. 

\begin{proposition}[Local well-posedness of the forced wave maps equation]\label{ansatz:prop-forced-wm}
Let $\phi^+,\phi^-\colon \R\rightarrow \R^\dimA$ satisfy 
$\big\| (\phi^+,\phi^-) \big\|_{\Ds}\leq \theta$.
Then, there exists a solution $\psi \in \Cprod{r}{r}$ of \eqref{ansatz:eq-forced-wm} satisfying 
$\| \psi \|_{\Cprod{r}{r}} \leq \theta$.
Furthermore, $\psi$ depends Lipschitz continuously on $(\phi_0(0),\phi^+,\phi^-)$ in $\M\times \Ds$ (in the sense of Definition \ref{ansatz:def-cts-dependence}).
\end{proposition}

The local well-posedness of the forced wave maps equation will be proved in Section \ref{section:lwp} below.

%%%%%%%%%%%%%%%%%%%%%% Key estimates %%%%%%%%%%%%%%%%%%%%%%%%%

\section{Prototypical estimates}\label{section:prototypical}

In this section, we prove several prototypical estimates. In contrast to the estimates in Section \ref{section:preparations}, the following prototypical estimates are specific to our Ansatz.  They will then be used repeatedly in our analysis of the null-form  $\partial_u \phi^i \partial_v \phi^j$ (in Section \ref{section:bilinear}) or the full nonlinearity $\Second^k_{ij}(\phi) \partial_u \phi^i \partial_v \phi^j$ (in Section \ref{section:nonlinear}). In Section \ref{section:favorable}, we estimate terms containing derivatives on the (low-frequency) modulations $A^+,A^-$, or $B$. In the following, such derivatives are sometimes called ``favorable".   In Section \ref{section:prot-bil-tril}, we prove prototypical bilinear and trilinear estimates. In Section \ref{section:unfortunate}, we isolate an unfortunate resonance between the initial data $\phi^\pm$ and the nonlinear remainder $\psi$. In Section \ref{section:generalizations}, we show that the types $\tp$ and $\tpm$ can sometimes be treated in a unified fashion. 

\subsection{Favorable derivatives}\label{section:favorable}

We now estimate favorable terms, i.e., terms with derivatives on the (low-frequency) modulations $A^+$, $A^-$, and $B$. With the following lemma, the favorable terms can essentially be treated as smooth remainders.

\begin{lemma}[Favorable derivatives]\label{key:lemma-favorable}  ~
\begin{itemize}[leftmargin=13mm]
\item[$\tp$:] \label{key:item-favorable-p} For all $M\geq 1$, it holds that 
\begin{align}
\big\| \partial_u A^+_M \, \phi^+_M \big\|_{\Cprod{r-1}{s}} &\lesssim M^{r-2s} \| A^+_M \|_{\Cprod{s}{s}} \| \phi^+\|_{\C_u^s}, \label{key:eq-favorable-p1}\\
\big\|  \partial_u A^+_M \, \phi^+_M \big\|_{\Cprod{r-1}{r}} &\lesssim
 M^{2r-3s} \| A_M^+ \|_{\Mod_M^+} \| \phi^+\|_{\C_u^s}. \label{key:eq-favorable-p2}
\end{align}
\item[$\tpm$:]  For all $M_1,M_2$ satisfying $M_1\sim_\delta M_2$, it holds that  
\begin{equation}\label{key:eq-favorable-pm-1}
\begin{aligned}
&\big\| \partial_u B_{M_1,M_2} \, \phi^+_{M_1} \, \phi^-_{M_2}  \|_{\Cprod{r-1}{r}} 
+\big\| \partial_v B_{M_1,M_2} \, \phi^+_{M_1} \, \phi^-_{M_2}  \|_{\Cprod{r}{r-1}} \\
\lesssim& \max(M_1,M_2)^{2r-3s} \|B_{M_1,M_2} \|_{\Cprod{s}{s}} \| \phi^+\|_{\C_u^s} \| \phi^- \|_{\C_v^s}. 
\end{aligned}
\end{equation}
In addition, we have the modified estimate
\begin{equation}\label{key:eq-favorable-pm-2}
\begin{aligned}
&\big\| \partial_u B_{M_1,M_2} \, \phi^+_{M_1} \, \phi^-_{M_2}  \|_{\Cprod{r-1}{1-r^\prime}} 
+\big\| \partial_v B_{M_1,M_2} \, \phi^+_{M_1} \, \phi^-_{M_2}  \|_{\Cprod{1-r^\prime}{r-1}} \\
\lesssim& \max(M_1,M_2)^{1-3s} \|B_{M_1,M_2} \|_{\Cprod{s}{s}} \| \phi^+\|_{\C^s_u} \| \phi^- \|_{\C^s_v}. 
\end{aligned}
\end{equation}
\end{itemize}
\end{lemma}

\begin{proof}
We treat the two cases separately. 

\emph{$\tp$-estimate:} Since $A_M^+$ is supported on $u$-frequencies $\lesssim M^{1-\sigma}$, it holds that 
\begin{align*}
\big\| \partial_u A^+_M \,  \phi^+_M \big\|_{\Cprod{r-1}{s}} 
&\lesssim M^{r-1} \big\|  \partial_u A^+_M \,  \phi^+_M \big\|_{L^\infty_u \C_v^{s}} \\
&\lesssim M^{r-1} \Big( \sum_{M_1\lesssim M^{1-\sigma}} \| \partial_u P_{M_1}^u A_M^+ \|_{L^\infty_u \C_v^{s}} \Big) \| \phi^+_M \|_{L^\infty_u} \\
&\lesssim M^{r-1-s} \Big(  \sum_{M_1\lesssim M^{1-\sigma}} M_1^{1-s} \Big) \| A_M^+ \|_{\Cprod{s}{s}} \|\phi^+_M \|_{\C^s_u} \\
&\lesssim M^{r-2s} \| A_M^+ \|_{\Cprod{s}{s}} \|\phi^+_M \|_{\C^s_u} . 
\end{align*}
This completes the proof of \eqref{key:eq-favorable-p1}. The proof of \eqref{key:eq-favorable-p2} is similar and the only additional ingredient is Corollary \eqref{ansatz:cor-delta-Apm}, which yields 
\begin{equation*}
 \| A_M^+  \|_{\Cprod{s}{r}} \lesssim  M^{(1-\delta)(r-s)} \| A_M^+  \|_{\Mod_M^+} 
 \lesssim M^{r-s} \| A_M^+  \|_{\Mod_M^+}.   
\end{equation*}
\emph{$\tpm$-estimate:} We only prove \eqref{key:eq-favorable-pm-1}, since the proof of \eqref{key:eq-favorable-pm-2} is essentially identical. Due to the symmetry in $u$ and $v$, it suffices to bound the first summand. Since $B_{M_1,M_2}$ is supported on $u$-frequencies $\lesssim M_1^{1-\sigma}$ and $v$-frequencies $\lesssim M_2^{1-\sigma}$,  we have that 
\begin{align*}
\big\|  \partial_u B_{M_1,M_2} \, \phi^+_{M_1} \, \phi^-_{M_2}  \|_{\Cprod{r-1}{r}} 
&\lesssim M_1^{r-1} M_2^{r} \big\|  \partial_u B_{M_1,M_2} \, \phi^+_{M_1} \, \phi^-_{M_2}  \|_{L^\infty_{u,v}} \\
&\lesssim M_1^{r-1} M_2^{r} \| \partial_u B_{M_1,M_2} \|_{L^\infty_{u,v}} \| \phi^+_{M_1}\|_{\C^0}  \| \phi^-_{M_2}\|_{\C^0} \\
&\lesssim M_1^{r-1+(1-\sigma)(1-s)-s} M_2^{r-s} \|B_{M_1,M_2} \|_{\Cprod{s}{s}} \| \phi^+\|_{\C^s} \| \phi^- \|_{\C^s}.
\end{align*}
Since $M_1\sim_\delta M_2$ and $\sigma=100\delta$, this is acceptable. 
\end{proof}

As an application of Lemma \ref{key:lemma-favorable} and our commutator estimates, we obtain the following corollary.

\begin{corollary}\label{key:cor-parall-mod}
For all $M\geq 1$ and $\zeta \in \Cprod{1-r^\prime}{r-1}$, it holds that 
\begin{equation}\label{key:eq-parall-mod}
\begin{aligned}
&\Big\| \zeta \parallsigu \partial_u \big( A_{M,m}^{+,i} \phi^{+,m}_M \big) -  \big( A_{M,m}^{+,i} \zeta \big) \parallsigu \partial_u \phi^{+,m}_M \Big\|_{\Cprod{r-1}{r-1}} \\
\lesssim& M^{r-2s+\sigma} \| P_{\leq M}^u \zeta \|_{\Cprod{s}{r-1}} \| A_{M,m}^+ \|_{\Cprod{s}{s}} \| \phi^+ \|_{\C_u^s}. 
\end{aligned}
\end{equation}
\end{corollary}

\begin{proof}
Since $A^{+,i}_{M,m}$ is supported on $u$-frequencies $\lesssim M^{1-\sigma}$, we have that 
\begin{equation*}
    P_{>M}^u \zeta \parallsigu \partial_u \big( A_{M,m}^{+,i} \phi^{+,m}_M \big) =  \big( A_{M,m}^{+,i} P_{>M}^u\zeta \big) \parallsigu \partial_u \phi^{+,m}_M =0.
\end{equation*}
Thus, we can replace $\zeta$ by $P_{\leq M}^u\zeta$, which we now omit from our notation.
Using Lemma \ref{prep:lem-para-localized}, we obtain that 
\begin{align*}
&\big\|  \zeta \parallsigu \partial_u \big( A_{M,m}^{+,i} \phi^{+,m}_M \big) - P_{\leq M^{1-\sigma}}^u  \zeta\, \partial_u \big( A_{M,m}^{+,i} \phi^{+,m}_M \big) \big\|_{\Cprod{r-1}{r-1}} \\
+& \big\| \big( A_{M,m}^{+,i} \zeta \big) \parallsigu \partial_u \phi^{+,m}_M
-P_{\leq M^{1-\sigma}}^u \big( A_{M,m}^{+,i} \zeta \big) \, \partial_u \phi^{+,m}_M \big\|_{\Cprod{r-1}{r-1}} \\
\lesssim& M^{r-2s+\sigma} \| \zeta \|_{\Cprod{s}{r-1}} \| A_{M,m}^+ \|_{\Cprod{s}{s}} \| \phi^+ \|_{\C_u^s},
\end{align*}
which is an acceptable contribution to \eqref{key:eq-parall-mod}. Thus, it remains to prove that 
\begin{align*}
&\big\| P_{\leq M^{1-\sigma}}^u  \zeta\, \partial_u \big( A_{M,m}^{+,i} \phi^{+,m}_M \big) 
- P_{\leq M^{1-\sigma}}^u \big( A_{M,m}^{+,i} \zeta \big) \, \partial_u \phi^{+,m}_M
\big\|_{\Cprod{r-1}{r-1}} \\
\lesssim& M^{r-2s+\sigma} \| \zeta \|_{\Cprod{s}{r-1}} \| A_{M,m}^+ \|_{\Cprod{s}{s}} \| \phi^+ \|_{\C_u^s}.
\end{align*}
To this end, we decompose 
\begin{equation}\label{key:eq-parall-mod-p1}
\begin{aligned}
& P_{\leq M^{1-\sigma}}^u  \zeta\, \partial_u \big( A_{M,m}^{+,i} \phi^{+,m}_M \big) 
- P_{\leq M^{1-\sigma}}^u \big( A_{M,m}^{+,i} \zeta \big) \, \partial_u \phi^{+,m}_M \\
=& P_{\leq M^{1-\sigma}}^u  \zeta\, \partial_u  A_{M,m}^{+,i} \, \phi^{+,m}_M  +
[ A_{M,m}^{+,i}, P_{\leq M^{1-\sigma}}^u   ] \big( \zeta \big) \, \partial_u \phi^{+,m}_M. 
\end{aligned}
\end{equation}
The first summand in \eqref{key:eq-parall-mod-p1} can be estimated using the bilinear estimate (Proposition \ref{prep:prop-bilinear}) and Lemma \ref{key:lemma-favorable}. Using frequency-support considerations and the commutator estimate (Lemma \ref{prep:lem-commutator}), the second summand in \eqref{key:eq-parall-mod-p1} can be estimated by
\begin{align*}
\big\| [ A_{M,m}^{+,i}, P_{\leq M^{1-\sigma}}^u   ] \big( \zeta \big) \, \partial_u  \phi^{+,m}_M \big\|_{\Cprod{r-1}{r-1}} 
&\lesssim M^{r-1-\eta} \big\| [ A_{M,m}^{+,i}, P_{\leq M^{1-\sigma}}^u   ] \big( \zeta \big) \, \partial_u  \phi^{+,m}_M \big\|_{\Cprod{\eta}{r-1}} \\
&\lesssim M^{r-1-\eta} \big\| [ A_{M,m}^{+,i}, P_{\leq M^{1-\sigma}}^u   ] \big( \zeta \big)  \big\|_{\Cprod{\eta}{r-1}} \big\| \partial_u  \phi^{+,m}_M \big\|_{\C_u^{\eta}} \\
&\lesssim M^{r-1-\eta} M^{(1-\sigma)(\eta-s)} M^{\eta-s}  \| \zeta \|_{\Cprod{s}{r-1}} \| A_{M,m}^+ \|_{\Cprod{s}{s}} \| \phi^+ \|_{\C_u^s}. 
\end{align*}
Since $\eta\ll \sigma$, this is an acceptable contribution to \eqref{key:eq-parall-mod}.
\end{proof}

\subsection{Prototypical bilinear and trilinear estimates}\label{section:prot-bil-tril}

The next lemma essentially addresses the null form estimates in the $\tp$-$\ts$ and $\tm$-$\ts$ cases (see Section \ref{section:bilinear}). However, the same estimates can also be used to treat certain error terms and therefore are placed in this central location.  

\begin{lemma}[Prototypical bilinear estimates]\label{key:lemma-bilinear}
Let $1\leq i \leq \dimA$ and let $M$ and $N$  be dyadic scales. Furthermore, let  $\zeta\in \C_b^\infty(\R^{1+1}_{u,v})$. 
\begin{itemize}[leftmargin=20mm]
\item[$\tp$-$\ts$:]  It holds that 
\begin{equation}\label{key:eq-basic-bilinear-p1}
\begin{aligned}
\| \partial_u \big( A^{+,i}_{M,m} \phi^{+,m}_M\big) \, \zeta \|_{\Cprod{r-1}{r-1}} 
\lesssim M^{r-s} \| A^+_M\|_{\Cprod{s}{s}} \| \phi^+_M \|_{\C^s_u} \| \zeta \|_{\Cprod{r}{r-1}}. 
\end{aligned}
\end{equation}
By further restricting to the low$\times$high-interactions in the $u$-variable, we have the improved estimate 
\begin{equation}\label{key:eq-basic-bilinear-p2}
\begin{aligned}
\| \partial_u \big( A^{+,i}_{M,m} \phi^{+,m}_M \big) \paralesssimsigu \zeta  \|_{\Cprod{r-1}{r-1}} 
&\lesssim M^{1-r-s+\eta} \| A^+_M\|_{\Cprod{s}{s}} \| \phi^+_M \|_{\C^s_u} \| \zeta \|_{\Cprod{r}{r-1}} \\
&\lesssim M^{-\eta}\| A^+_M\|_{\Cprod{s}{s}} \| \phi^+_M \|_{\C^s_u} \| \zeta \|_{\Cprod{r}{r-1}}. 
\end{aligned}
\end{equation}
\item[$\tm$-$\ts$:] It holds that
\begin{equation}
\| \partial_u \big( A^{-,i}_{N,n} \phi^{-,n}_N \big) \zeta \|_{\Cprod{r-1}{r-1}} 
\lesssim N^{-\eta} \| A^-_N \|_{\Mod_N^-} \| \phi^-_N \|_{\C^s_v} \| \zeta \|_{\Cprod{1-r^\prime}{r-1}}.
\end{equation}
\end{itemize}
\end{lemma}

\begin{proof}[Proof of Lemma \ref{key:lemma-bilinear}:] 
We treat the two cases separately. \newline

\emph{$\tp$-$\ts$-estimate:} Using the  bilinear estimate (Proposition \ref{prep:prop-bilinear}), we obtain that
\begin{equation}\label{key:eq-bilinear-proof1}
\begin{aligned}
\| \partial_u \big( A^{+,i}_{M,m} \phi^{+,m}_M \big)   \zeta \|_{\Cprod{r-1}{r-1}} 
&\lesssim \|  \partial_u \big( A^{+,i}_{M,m} \phi^{+,m}_M \big) \|_{\Cprod{r-1}{s}} \|  \zeta \|_{\Cprod{1-r^\prime}{r-1}} \\
&\lesssim M^{r-s}  \| A_M^+ \|_{\Cprod{s}{s}} \| \phi^+ \|_{\C^s_u} \| \zeta \|_{\Cprod{1-r^\prime}{r-1}}. 
\end{aligned}
\end{equation}
This proves the first estimate \eqref{key:eq-basic-bilinear-p1}. Using the low$\times$high-improvement of the bilinear estimate, i.e., Proposition \ref{prep:prop-bilinear}.\ref{prep:item-low-high}, we obtain that 
\begin{align*}
\| \partial_u \big( A^{+,i}_{M,m} \phi^{+,m}_M \big) \paralesssimsigu \zeta  \|_{\Cprod{r-1}{r-1}}  
&\lesssim \|  \partial_u \big( A^{+,i}_{M,m} \phi^{+,m}_M \big) \|_{\Cprod{-r^\prime}{s}} \| \zeta \|_{\Cprod{r}{r-1}} \\
&\lesssim M^{1-s-r^\prime}\| A^+_M\|_{\Cprod{s}{s}} \| \phi^+_M \|_{\C^s_u} \| \zeta \|_{\Cprod{r}{r-1}}.
\end{align*}

\emph{$\tm$-$\ts$-estimate:} We decompose 
\begin{align*}
 \partial_u \big( A^{-,i}_{N,n} \phi^{-,n}_N \big)  \zeta  
 = \big( \partial_u A^{-,i}_{N,n} \phi^{-,n}_N \big) \paransimv \zeta  + \big( \partial_u A^{-,i}_{N,n} \phi^{-,n}_N \big) \parasimv \zeta.
\end{align*}
 Using the non-resonant improvement in the bilinear estimate, i.e., Proposition \ref{prep:prop-bilinear}.\ref{prep:item-nonres}, we have that 
\begin{align*}
\big\| \big( \partial_u A^{-,i}_{N,n} \phi^{-,n}_N \big) \paransimv \zeta  \big\|_{\Cprod{r-1}{r-1}} 
&\lesssim \|  \partial_u A^{-,i}_{N,n} \phi^{-,n}_N  \|_{\Cprod{r-1}{\eta}} \| \zeta \|_{\Cprod{1-r^\prime}{r-1}}  \\
&\lesssim \|  \partial_u A^{-,i}_{N,n}\|_{\Cprod{r-1}{s}} \| \phi^{-,n}_N \|_{\C_v^{\eta}}  \| \zeta \|_{\Cprod{1-r^\prime}{r-1}}  \\
&\lesssim M^{r-2s+\eta} \| A_M^- \|_{\Mod_M^-} \| \phi^- \|_{\C^s_v} \| \zeta \|_{\Cprod{1-r^\prime}{r-1}} . 
\end{align*}
Since $r-2s\approx -1/4$, this term is acceptable. Using the resonant improvement of the bilinear estimate, i.e., Proposition \ref{prep:prop-bilinear}.\ref{prep:item-res}, and Corollary \ref{ansatz:cor-delta-Apm}, it holds that 
\begin{align*}
\| \big( \partial_u A^{-,i}_{N,n} \phi^{-,n}_N \big) \parasimv \zeta \|_{\Cprod{r-1}{r-1}} 
&\lesssim 
\| \partial_u A^{-,i}_{N,n} \phi^{-,n}_N \|_{\Cprod{r-1}{-r^\prime}} \| \zeta \|_{\Cprod{1-r^\prime}{r-1}} \\
&\lesssim N^{-s-r^\prime} \| A^{-,i}_{N,n} \|_{\Cprod{r}{s}} \| \phi^{-,n}_N \|_{\C_v^s}  \| \zeta \|_{\Cprod{1-r^\prime}{r-1}}  \\
&\lesssim N^{(1-\delta)(r-s)-s-r^\prime}  \| A^-_N \|_{\Mod_N^-} \| \phi^-_N \|_{\C^s_v} \| \zeta \|_{\Cprod{1-r^\prime}{r-1}}. 
\end{align*}
Since 
\begin{equation*}
(1-\delta)(r-s)-s-r^\prime = 1 -2s + \eta - \delta (r-s) \leq - \delta/10,
\end{equation*}
this is acceptable. 

\end{proof}

We now present a prototypical trilinear estimate in the initial data $\phi^\pm$.  

\begin{lemma}[Prototypical trilinear estimate]\label{bilinear:lemma-cubic}
Let $K,M$, and $N$ be dyadic scales and let $1\leq k,m,n \leq \dimA$. Then, it holds that 
\begin{equation}\label{bilinear:lemma-cubic-1}
\begin{aligned}
\| \phi_K^{+,k}(u) \partial_u \phi^{+,m}_M(u)   \partial_v \phi^{+,n}_N(v) \|_{\Cprod{r-1}{r-1}} \lesssim K^{-s} M^{r-s} N^{r-s} \cdot \| (\phi^+, \phi^-) \|_{\Ds}^3. 
\end{aligned}
\end{equation}
In particular, the condition $K\geq \max(M,N)^{1-\sigma}$ implies the stronger estimate  
\begin{equation}\label{bilinear:lemma-cubic-2}
\begin{aligned}
\| \phi_K^{+,k}(u) \partial_u \phi^{+,m}_M(u)   \partial_v \phi^{+,n}_N(v) \|_{\Cprod{r-1}{r-1}} \lesssim (KMN)^{-\eta}\cdot \| (\phi^+, \phi^-) \|_{\Ds}^3. 
\end{aligned}
\end{equation}
A similar estimate also holds with $\phi^+_K(u)$ replaced by $\phi_K^-(v)$. 
\end{lemma}

As will be clear from the proof, Lemma \ref{bilinear:lemma-cubic} easily follows from the definition of $\Ds$. Despite its simplicity, the estimate is essential for our analysis. Most importantly, it shows that our Ansatz does not need to include an explicit cubic term in $\phi^\pm$. 

\begin{proof}
Using Lemma \ref{ansatz:lemma-data}, it holds that 
\begin{align*}
\| \phi_K^{+,k} \partial_u \phi^{+,m}_M   \partial_v \phi^{+,n}_N \|_{\Cprod{r-1}{r-1}} 
=  \| \phi_K^{+,k} \partial_u \phi^{+,m}_M \|_{\C_u^{r-1}} \| \partial_v \phi^{+,n}_N \|_{\C_v^{r-1}} 
\lesssim K^{-s} M^{r-s} N^{r-s} \| (\phi^+, \phi^-) \|_{\Ds}^3. 
\end{align*}
This implies the first estimate. If $K\geq \max(M,N)^{1-\sigma}$, the desired second estimate follows directly from our parameter conditions \eqref{prep:eq-parameter-1} and \eqref{prep:eq-parameter-2}. 
\end{proof}

\subsection{Unfortunate resonances}\label{section:unfortunate}

In this subsection, we describe and estimate frequency-resonances between the initial data $\phi^\pm$ and a smoother function $\zeta$. These frequency-resonances between $\phi^\pm$ and $\zeta$ are quite different from the frequency-resonances between $\phi^\pm$ and itself. Using only the regularity assumption $\phi^\pm \in \C^s_x$, the resonant product $\phi^{\pm,i} \parasim \partial_x\phi^{\pm,j}$ is ill-defined. For our geometric random data, however, the resonant portion has been controlled using probabilistic arguments (see Section \ref{section:Brownian}). In fact, Proposition~\ref{prep:prop-brownian-path-approximation} and Proposition~\ref{prep:prop-velocity-approximation} yield significant decay of the high$\times$high-para-product in the high frequency scale. In contrast, the resonant portion $\phi^{-,k} \parasimv \partial_v \zeta^n$ is well-defined using only the regularity assumptions $\phi^- \in \C_v^s$ and $\zeta \in \Cprod{s}{r}$. However, it exhibits worse decay in the high frequency scale than $\phi^{\pm,i} \parasim \partial_x\phi^{\pm,j}$. The resulting contribution to the nonlinearity is a $\delta$-power away from $\Cprod{r-1}{r-1}$ (see Proposition \ref{nonlinear:prop-unfortunate}) and therefore has to be absorbed in our Ansatz. 
While this does not create any serious (conceptual) difficulties,  it feels unfortunate to barely miss  $\Cprod{r-1}{r-1}$, which is why we call this term the unfortunate resonance.

\begin{lemma}[Unfortunate resonances]\label{key:lemma-unfortunate}~
 Let $1\leq k,n \leq \dimA$, let $K$ be a dyadic scale, and let $\zeta \in \Cprod{s}{r}$. The unfortunate resonance satisfies
\begin{equation}\label{key:eq-unfortunate-v1}
\big\| \phi_K^{-,k} \paradownv \partial_v \zeta^n \big \|_{\Cprod{s}{r-1}} \lesssim K^{1-s-r} \| \phi^- \|_{\C^s_v} \| \zeta \|_{\Cprod{s}{r}}. 
\end{equation}
After eliminating the unfortunate resonances, the product obeys the better estimate 
\begin{equation}\label{key:eq-unfortunate-v2}
    \big\| \phi_K^{-,k} \partial_v \zeta^n - \phi_K^{-,k} \paradownv \partial_v \zeta^n \big \|_{\Cprod{s}{r-1}} \lesssim K^{-(r-s)-\sigma/6} \| \phi^- \|_{\C^s_v} \| \zeta \|_{\Cprod{s}{r}}. 
\end{equation}
In total, the product satisfies the combined estimate
\begin{equation}\label{key:eq-unfortunate-v3}
    \big\| \phi_K^{-,k}  \partial_v \zeta^n  \big \|_{\Cprod{s}{r-1}} \lesssim K^{1-s-r} \| \phi^- \|_{\C^s_v} \| \zeta \|_{\Cprod{s}{r}}.
\end{equation}
\end{lemma}

\begin{remark}\label{key:remark-unfortunate}
The gain in \eqref{key:eq-unfortunate-v2} is a result of the $v$-frequency of $\phi_K^{-,k} \partial_v \zeta^n - \phi_K^{-,k} \paradownv \partial_v \zeta^n$, which is bounded below by $K^\sigma$. After adjusting the norm on the left-hand side in \eqref{key:eq-unfortunate-v1} or \eqref{key:eq-unfortunate-v2}, a similar gain is possible by inserting frequency-projections in the $u$-variable. For instance, it holds that
\begin{equation*}
 \big\| P_{\geq K^\sigma}^u \big( \phi_K^{-,k}  \partial_v \zeta^n \big) \big \|_{\Cprod{1-r^\prime}{r-1}} \lesssim K^{\sigma (1-r^\prime-s)} \big\|  \phi_K^{-,k}  \partial_v \zeta^n  \big \|_{\Cprod{s}{r-1}}
 \lesssim K^{-(r-s)-\sigma/6} \| \phi^- \|_{\C^s_v} \| \zeta \|_{\Cprod{s}{r}}. 
 \end{equation*}
\end{remark}

\begin{proof}
We first prove \eqref{key:eq-unfortunate-v1}, which is the easier estimate. It holds that
\begin{align*}
   \big\| \phi_K^{-,k} \paradownv \partial_v \zeta^n \big \|_{\Cprod{s}{r-1}} 
    &\lesssim \| \phi_K^{-,k} \paradownv \partial_v  \widetilde{P}_K^v \zeta^n \big \|_{\Cprod{s}{r-1}} \\
    &\lesssim \| \phi_K^{-,k} \|_{L_v^\infty} \| \partial_v  \widetilde{P}_K^v \zeta^n \big \|_{L_v^\infty \C_u^s} \\
    &\lesssim K^{1-s-r} \| \phi^- \|_{\C_v^s} \| \zeta \|_{\Cprod{s}{r}}. 
\end{align*}
We now turn to the non-resonant estimate \eqref{key:eq-unfortunate-v2}. To this end, we decompose 
\begin{align}
 \phi_K^{-,k} \partial_v \zeta^n - \phi_K^{-,k} \paradownv \partial_v \zeta^n  
=& \phi_K^{-,k} \paransimv \partial_v \zeta^n  \label{key:eq-unfortunate-p1}\\
+& \phi_K^{-,k} \parasimv \partial_v \zeta^n  -  \phi_K^{-,k} \paradownv \partial_v \zeta^n. \label{key:eq-unfortunate-p2}
\end{align}
The contribution of \eqref{key:eq-unfortunate-p1} can be estimated using Proposition \ref{prep:prop-bilinear}.\ref{prep:item-nonres}, which yields 
\begin{align*}
  \Big\| \phi_K^{-,k} \paransimv \partial_v \zeta^n \Big\|_{\Cprod{s}{r-1}} 
  &\lesssim \| \phi^{-,k}_K \|_{\C_v^\eta} \| \partial_v \zeta \|_{\Cprod{s}{r-1}} \\
&\lesssim K^{-(s-\eta)}  \| \phi^- \|_{\C_v^s}\| \zeta \|_{\Cprod{s}{r}}.
\end{align*}
Since $-s\approx -1/2$ and $-(r-s)\approx-1/4$, this term is acceptable. The contribution of \eqref{key:eq-unfortunate-p2} is estimated by 
\begin{align*}
  \Big\| \phi_K^{-,k} \parasimv \partial_v \zeta^n  -  \phi_K^{-,k} \paradownv \partial_v \zeta^n \Big\|_{\Cprod{s}{r-1}} 
  &\lesssim K^{\sigma (r-1)} \big\| \phi_K^{-,k} \parasimv \partial_v \zeta^n\big\|_{\Cprod{s}{0}} \\
  &\lesssim K^{\sigma (r-1)} \big\| \phi_K^{-,k} \big\|_{L_v^\infty}
  \| \widetilde{P}_K^v \partial_v \zeta^n \big\|_{L_v^\infty \C_u^s} \\
  &\lesssim K^{\sigma (r-1)-s-(r-1)} \| \phi^- \|_{\C_v^s} \| \zeta \|_{\Cprod{s}{r-1}}. 
\end{align*}
Since
\begin{equation*}
\sigma (r-1) -s - (r-1) = - (r-s) + 1-2s - \sigma (1-r) \leq -(r-s) - \sigma/6,
\end{equation*}
this term is acceptable. The combined estimate \eqref{key:eq-unfortunate-v3} follows directly from \eqref{key:eq-unfortunate-v1}, \eqref{key:eq-unfortunate-v2}, and the triangle inequality. 
\end{proof}

\subsection{\texorpdfstring{Generalizations of $\tp$ and $\tpm$}{Generalizations of our terms}}\label{section:generalizations}
In several (but not all) estimates below, it is possible to treat the $\tp$ and $\tpm$-terms  simultaneously, which makes the case analysis more efficient. This essentially means that only the right-moving component of the wave is relevant. The imaginative reader may think of the next lemma as a violation of common rules for pedestrians, since it allows us to look in only one direction before safely crossing a street.

\begin{lemma}[``Looking in only one direction"]\label{key:lemmma-direction}~
Let $1\leq i \leq \dimA$, $M,M_1,M_2$ be dyadic scales satisfying $M=M_1\sim_\delta M_2$, and let
\begin{equation*}
\zeta^{+,i}(u,v)= A_{M,m}^{+,i}(u,v) \phi^{+,m}_M(u) \qquad \text{or} \qquad \zeta^{+,i}(u,v) = B_{M_1,M_2,m_1m_2}^i(u,v) \phi_{M_1}^{+,m_1}(u) \phi_{M_2}^{-,m_2}(v). 
\end{equation*}
Then, it holds that 
\begin{equation*}
\zeta^{+,i}= E_{M,m}^{+,i}(u,v) \phi_M^{+,m}(u),
\end{equation*}
where $E_M^{+,i}$ satisfies either the estimates
\begin{equation}\label{key:eq-direction-11}
\begin{aligned}
\| E_{M}^{+,i} \|_{\Cprod{s}{s}} &\lesssim \| A_M^+ \|_{\Cprod{s}{s}}, \\
\| E_{M}^{+,i} \|_{\Cprod{s}{r}} &\lesssim M^{(1-\delta)(r-s)} \| A_M^+ \|_{\Mod_N^+},
\end{aligned}
\end{equation}
or the estimates
\begin{equation}\label{key:eq-direction-12}
\begin{aligned}
    \| E_{M}^{+,i} \|_{\Cprod{s}{s}} &\lesssim \| B_{M_1,M_2}\|_{\Cprod{s}{s}} \| \phi^-_{M_1}\|_{\C_v^s}, \\
    \| E_{M}^{+,i} \|_{\Cprod{s}{r}} &\lesssim M^{(1+2\delta)(r-s)}
    \| B_{M_1,M_2}\|_{\Cprod{s}{s}} \| \phi^-_{M_1}\|_{\C_v^s}.
\end{aligned}
\end{equation}
\end{lemma}

\begin{proof}
We treat the $\tp$ and $\tpm$ cases separately. In the $\tp$ case, we set $E_{M,m}^{+,i}:= A^{+,i}_{M,m}$. Then, the first line in \eqref{key:eq-direction-11} is trivial and the second line in \eqref{key:eq-direction-11} follows from Corollary \ref{ansatz:cor-delta-Apm}. 

We now turn to the $\tpm$-case. Here, we define $E_{M,m}^{+,i}=B_{M_1,M_2,m_1m_2}^i \phi^{-,m_2}_{M_2}$, where $M_1=M$ and $m_1=m$. The first line in \eqref{key:eq-direction-12} then follows directly from the bilinear estimate (Proposition \ref{prep:prop-bilinear}). Due to the frequency-support conditions on $B$, we have that 
\begin{equation*}
\big\| E_{M,m}^{+,i}\big\|_{\Cprod{s}{r}}  \lesssim M_2^{r-s}  \big\| E_{M,m}^{+,i}\big\|_{\Cprod{s}{s}}.
\end{equation*}
Since $M_2 \leq M_1^{1/(1-\delta)}=M^{1/(1-\delta)}$, this yields the second line in \eqref{key:eq-direction-12}. 
\end{proof}

\section{Bilinear analysis}\label{section:bilinear}

In this section, we analyze two different bilinear expressions. In Section \ref{section:nullform}, we analyze the null-form $\partial_u \phi^i \partial_v \phi^j$, which is the most important component of the nonlinearity. In Section \ref{section:product}, we analyze the product $\phi^k \partial_u \phi^i$. Together with Bony's para-linearization (Lemma \ref{prep:lemma-bony}), this will yield sufficient information on $\SecondC{k}{ij}(\phi) \partial_u \phi^i$. 

\subsection{\protect{\texorpdfstring{The null-form $\partial_u \phi^i \partial_v \phi^j$}{The null-form}}}\label{section:nullform}

In this section, we analyze the null-form $\partial_u \phi^i \partial_v \phi^j$. The main goal is to identify the terms in $\partial_u \phi^i \partial_v \phi^j$ which belong to $\Cprod{r-1}{r-1}$, since they can be treated as smooth remainders. The rough terms which cannot be placed in $\Cprod{r-1}{r-1}$ will later be absorbed in our Ansatz, i.e, the terms $A_M^+ \phi_M^+$, $A_M^- \phi^-_M$, and $B_{M_1,M_2} \phi^{+}_{M_1} \phi^{-}_{M_2}$. Unfortunately, the argument requires a tedious case-by-case analysis, which is displayed in Figure \ref{figure:cases}. The available estimates are illustrated through the graphical symbols \raisebox{2pt}{\scalebox{0.6}{\crossmark}}, $\paralesssimsigu$, $\scalebox{1.25}{$\lnot$}\paradownv$, and 
\raisebox{2pt}{\scalebox{0.7}{\checkmark}}, which will now be explained. 

The main term $\tp$-$\tm$ cannot be controlled in $\Cprod{r-1}{r-1}$. While certain favorable sub-terms can be removed, it is more efficient to postpone a detailed analysis until the pre-factor $\SecondC{k}{ij}(\phi)$ is included. The graphical symbol \raisebox{2pt}{\scalebox{0.6}{\crossmark}} illustrates that this term is not controlled here. The cases $\tp$-$\tp$, $\tp$-$\tpm$, and $\tp$-$\ts$ can be controlled as long as the second argument is restricted to high $u$-frequencies. Thus, we can treat the $\paralesssimsigu$-portion of this term, which is also used as a graphical symbol. In the $\tpm$-$\ts$-term, the only problem lies in the unfortunate resonance between $\phi^-$ and $\partial_v \psi$. Since the unfortunate resonance is given by $\paradownv$, we use $\scalebox{1.25}{$\lnot$}\paradownv$ to illustrate this case. The remaining cases (on or above the diagonal in Figure \ref{figure:cases}) can be controlled in $\Cprod{r-1}{r-1}$. These terms will not need to be considered further and therefore deserve the symbol \raisebox{2pt}{\scalebox{0.7}{\checkmark}}. 

Finally, we note that all cases below the diagonal in Figure \ref{figure:cases} have been left empty. The results in these cases can be obtained from the symmetry in the $u$ and $v$-variables and do not need to be considered separately.

\begin{figure}[t]
\begin{tabular}{P{1cm}|P{1cm}|P{1cm}|P{1cm}|P{1cm}|}
\scalebox{1}{
\begin{tabular}{l}
\hspace{8pt} $\partial_v $ \\[-12pt]
$\partial_u $ \hspace{14pt}
\end{tabular}}	& $\tm$ & $\tp$ & $\tpm$ & $\ts$  \\[3pt] \hline 
$\tp$  &   	 \crossmark	 &	\scalebox{1.2}{$\paralesssimsigu$} &	\scalebox{1.2}{$\paralesssimsigu$}      &		\scalebox{1.2}{$\paralesssimsigu$}    \\[3pt] \hline 
$\tm$		& 	 &	\checkmark	& \checkmark	    &	\checkmark   \\[3pt] \hline 
$\tpm$	& 	 &		& \checkmark	    &	\scalebox{1.5}{$\lnot$}\scalebox{1.25}{$\paradownv$}  \\[3pt] \hline  
$\ts$	& 	 &		&	    &	 \checkmark	   \\[3pt] \hline 
\end{tabular}
\caption{Overview of cases in the null-form estimate.}\label{figure:cases}
\end{figure}

\begin{proposition}[\raisebox{1pt}{\scalebox{0.6}{\crossmark}}-estimate]\label{bilinear:prop-cross}
Let Hypothesis \ref{hypothesis:smallness} be satisfied, let $1\leq i,j \leq \dimA$, and let $M,N$ be frequency scales. Then, it holds that 
\begin{equation}\label{bilinear:eq-bilinear-pm-1}
\begin{aligned}
\Big\| \partial_u \big( A_{M,m}^{+,i} \phi^{+,m}_M \big)  \partial_v \big( A_{N,n}^{-,j} \phi^{-,n}_N\big) \Big\|_{\Cprod{r-1}{r-1}} 
\lesssim M^{r-s} N^{r-s} \theta^4.
\end{aligned}
\end{equation}
If $M\sim_\delta N$, we also have that  
\begin{equation}\label{bilinear:eq-bilinear-pm-2}
\begin{aligned}
&\Big\| \partial_u \big( A_{M,m}^{+,i} \phi^{+,m}_M \big)  \partial_v \big( A_{N,n}^{-,j} \phi^{-,n}_N \big)
- A_{M,m}^{+,i} \,   A_{N,n}^{-,j}\,\partial_u \phi^{+,m}_M \,  \partial_v \phi^{-,n}_N
\Big\|_{\Cprod{r-1}{r-1}} 
\lesssim (MN)^{-\eta} \theta^4. 
\end{aligned}
\end{equation}
\end{proposition}

\begin{remark}
Since $r$ is bigger than $s$, the first estimate \eqref{bilinear:eq-bilinear-pm-1} does not provide a uniform upper bound. Nevertheless, it illustrates the gap to $\Cprod{r-1}{r-1}$. The second estimate \eqref{bilinear:eq-bilinear-pm-2} is the first step in our analysis of the main term. A more detailed analysis, however, is postponed until Section \ref{section:main-term}, since it is more efficient to directly treat 
\begin{equation*}
  \SecondC{k}{ij}(\phi) A_{M,m}^{+,i} \,   A_{N,n}^{-,j}\,\partial_u \phi^{+,m}_M \,  \partial_v \phi^{-,n}_N . 
\end{equation*}
\end{remark}

\begin{proof}
We first prove \eqref{bilinear:eq-bilinear-pm-1}. Using the bilinear estimate (Proposition \ref{prep:prop-bilinear}) and frequency-support considerations, we have that 
\begin{align*}
\big\| \partial_u \big( A_{M,m}^{+,i} \phi^{+,m}_M \big)  \partial_v \big( A_{N,n}^{-,j} \phi^{-,n}_N \big) \big\|_{\Cprod{r-1}{r-1}} 
\lesssim& \big\|  \partial_u \big( A_{M,m}^{+,i} \phi^{+,m}_M \big)  \big\|_{\Cprod{r-1}{s}} \big\| \partial_v \big( A_{N,n}^{-,j} \phi^{-,n}_N \big) \big\|_{\Cprod{s}{r-1}} \\
\lesssim & M^{r-s} N^{r-s}  \big\|   A_{M,m}^{+,i} \phi^{+,m}_M \big\|_{\Cprod{s}{s}} \big \|  A_{N,n}^{-,j} \phi^{-,n}_N \big\|_{\Cprod{s}{s}} \\
\lesssim & M^{r-s} N^{r-s} \| A_M^+ \|_{\Cprod{s}{s}} \| \phi^+ \|_{\C^s} \| A_N^{-} \|_{\Cprod{s}{s}} \| \phi^-\|_{\C^s}. 
\end{align*}
We now turn to \eqref{bilinear:eq-bilinear-pm-2}. To this end, we decompose
\begin{align}
&\partial_u \big( A_{M,m}^{+,i} \phi^{+,m}_M \big)  \partial_v \big( A_{N,n}^{-,j} \phi^{-,n}_N \big) 
- A_{M,m}^{+,i} \,   A_{N,n}^{-,j}\,\partial_u \phi^{+,m}_M \,  \partial_v \phi^{-,n}_N   \notag \\
=& \partial_u \big( A_{M,m}^{+,i} \phi^{+,m}_M \big) \, \partial_v A_{N,n}^{-,j} \, \phi^{-,n}_N  \label{bilinear:eq-bilinear-pm-p1} \\
+&  \partial_u A_{M,m}^{+,i} \, \phi^{+,m}_M  \, \partial_v \big( A_{N,n}^{-,j} \, \phi^{-,n}_N \big) \label{bilinear:eq-bilinear-pm-p2} \\ 
-&  \partial_u A_{M,m}^{+,i} \, \phi^{+,m}_M  \, \partial_v A_{N,n}^{-,j} \, \phi^{-,n}_N  \label{bilinear:eq-bilinear-pm-p3}.  
\end{align}
Using the prototypical bilinear estimate  (Lemma \ref{key:lemma-bilinear}) and Lemma \ref{key:lemma-favorable}, the first summand \eqref{bilinear:eq-bilinear-pm-p1} is estimated by 
\begin{align*}
     \big\| \partial_u \big( A_{M,m}^{+,i} \phi^{+,m}_M \big) \, \partial_v A_{N,n}^{-,j} \, \phi^{-,n}_N \big\|_{\Cprod{r-1}{r-1}} 
     \lesssim& M^{r-s} \| A_M^+ \|_{\Cprod{s}{s}} \| \phi^+ \|_{\C_u^s} \big\|  \partial_v A_{N,n}^{-,j} \, \phi^{-,n}_N \big\|_{\Cprod{1-r^\prime}{r-1}} \\
     \lesssim& M^{r-s} N^{r-2s} \theta^4.
\end{align*}
Since $M\sim_\delta N$ and $3s-2r\gg \delta$, this term is acceptable. By symmetry, the second summand  \eqref{bilinear:eq-bilinear-pm-p2}  obeys the same estimate. Finally, the third summand \eqref{bilinear:eq-bilinear-pm-p3} can be estimated easily  using the bilinear estimate (Proposition \ref{prep:prop-bilinear}) and  Lemma \ref{key:lemma-favorable}.
\end{proof}

We now turn to the $\paralesssimsigu$-cases in Figure \ref{figure:cases}. 

\begin{proposition}[$\paralesssimsigu$-estimates] \label{bilinear:prop-paraless}
Let Hypothesis \ref{hypothesis:smallness} be satisfied, let $1\leq i,j \leq \dimA$, and let $M$ be a frequency scale. Furthermore, let $\zeta\colon \R_{u,v}^{1+1} \rightarrow \R^{\dimA}$ be of type $\tp$,$\tpm$, or $\ts$. Then, with   $\Gain(\zeta)$ as in \eqref{gain0}, it holds that 
\begin{equation}\label{nonlinear:eq-low-high-1}
\Big\| \partial_u \big( A_{M,m}^{+,i}  \phi^{+,m}_M\big) \paralesssimsigu \partial_v \zeta^j \Big\|_{\Cprod{r-1}{r-1}} 
\lesssim M^{-\eta}   \Gain(\zeta)  \theta^3. 
\end{equation}
Without the low$\times$high-para-product, we only have the weaker estimate 
\begin{equation}\label{nonlinear:eq-low-high-2}
\Big\| \partial_u \big( A_{M,m}^{+,i}  \phi^{+,m}_M\big)  \partial_v \zeta^j \Big\|_{\Cprod{r-1}{r-1}} 
\lesssim M^{r-s}  \Gain(\zeta) \theta^3. 
\end{equation}
\end{proposition}

\begin{proof}
The case $\tp$-$\ts$ has already been treated in Lemma \ref{key:lemma-bilinear}. Thus, it remains to treat the cases $\tp$-$\tp$ and $\tp$-$\tpm$, which can be done simultaneously. Using Lemma \ref{key:lemmma-direction}, we can write 
\begin{equation}
\zeta^j(u,v) = E_{N,n}^{+,j}(u,v) \phi_N^{+,n}(u),
\end{equation}
where $E_N^+$ satisfies either \eqref{key:eq-direction-11} or \eqref{key:eq-direction-12}. In both cases, we have that 
\begin{equation}
\| E_N^+ \|_{\Cprod{s}{r}}  \lesssim N^{(1+2\delta)(r-s)} \theta. 
\end{equation}
The main step in this argument is to prove
\begin{equation}\label{bilinear:eq-paraless-p3}
\big\| \partial_u \big( A_{M,m}^{+,i}  \phi^{+,m}_M\big)  \partial_v \big(  E_{N,n}^{+,j} \phi_N^{+,n}\big) \big\|_{\Cprod{r-1}{r-1}} \lesssim M^{r-s} N^{r-2s+4\delta} \Gain(\zeta) \theta^4.
\end{equation}
Once \eqref{bilinear:eq-paraless-p3} has been proven, the first estimate \eqref{nonlinear:eq-low-high-1} follows by considering $N\gtrsim M^{1-\sigma}$ and the second estimate \eqref{nonlinear:eq-low-high-2} follows by considering all $N\geq 1$. The additional Littlewood-Paley operators, which enter through the para-product $\scalebox{0.8}{\paralesssimsigu}$, can be handled using Lemma \ref{prep:lem-commutator} and Lemma \ref{prep:lem-para-localized}. 

Using the product rule, we decompose 
\begin{align}
 \partial_u \big( A_{M,m}^{+,i}  \phi^{+,m}_M\big)  \partial_v \big(  E_{N,n}^{+,j} \phi_N^{+,n}\big) 
 =& A_{M,m}^{+,i}  \partial_u \phi^{+,m}_M \partial_v E_{N,n}^{+,j}  \phi_N^{+,n} \label{bilinear:eq-paraless-p1} \\
 +& \partial_u A_{M,m}^{+,i}   \phi^{+,m}_M \partial_v E_{N,n}^{+,j}  \phi_N^{+,n}\label{bilinear:eq-paraless-p2}.
\end{align}
We start by estimating the first summand \eqref{bilinear:eq-paraless-p1}, which is the main term. Using the multiplication estimate (Corollary \ref{prep:corollary-multiplication}), the bilinear estimate (\ref{prep:prop-bilinear}), and Hypothesis \ref{hypothesis:smallness}, we have that 
\begin{align*}
\big\| A_{M,m}^{+,i}  \partial_u \phi^{+,m}_M \partial_v E_{N,n}^{+,j}  \phi_N^{+,n}  \big\|_{\Cprod{r-1}{r-1}} 
\lesssim&   \big\| A_{M,m}^{+,i}\big\|_{\Cprod{s}{s}}
\big\| \partial_u \phi^{+,m}_M \partial_v E_{N,n}^{+,j}  \phi_N^{+,n}  \big\|_{\Cprod{r-1}{r-1}} \\
\lesssim& \big\| A_{M,m}^{+,i}\big\|_{\Cprod{s}{s}}
\big\| \partial_v E_{N,n}^{+,j} \big\|_{\Cprod{s}{r-1}}
\big\| \partial_u \phi^{+,m}_M  \phi_N^{+,n}  \big\|_{\C_u^{r-1}} \\
\lesssim& M^{r-s} N^{-s} \| A_M^+ \|_{\Cprod{s}{s}}  \| E_N^+ \|_{\Cprod{s}{r}} \| ( \phi^+, \phi^- ) \|_{\Ds}^2 \\
\lesssim& M^{r-s} N^{(1+2\delta)(r-s)-s}  \theta^4. 
\end{align*}
This is an acceptable contribution to \eqref{bilinear:eq-paraless-p3}. For the second summand \eqref{bilinear:eq-paraless-p2}, we have that
\begin{align*}
&\big\|  \partial_u A_{M,m}^{+,i}   \phi^{+,m}_M \partial_v E_{N,n}^{+,j}  \phi_N^{+,n} \big\|_{\Cprod{r-1}{r-1}} \\
\lesssim&  \big\| \partial_u A_{M,m}^{+,i}\big\|_{\Cprod{\eta}{s}} 
\big\|  \phi^{+,m}_M \big\|_{\C_u^\eta} \big\| \partial_v E_{N,n}^{+,j} \big\|_{\Cprod{\eta}{r-1}} \big\| \phi_N^{+,n}\big\|_{\C_u^\eta} \\
\lesssim& M^{1-s+\eta} M^{-s+\eta} N^{(1+2\delta)(r-s)} N^{-s+\eta}  \| A_M^+ \|_{\Cprod{s}{s}} \| \phi^+ \|_{\C_u^s} \| E_{N,n}^{+,j} \|_{\Cprod{s}{r}} \| \phi^+ \|_{\C_u^s} \\
\lesssim& M^{1-2s+2\eta} N^{r-2s+4\delta} \theta^4.
\end{align*}
This is an acceptable contribution to \eqref{bilinear:eq-paraless-p3}. 
\end{proof}

We now turn to the $\scalebox{1.25}{$\lnot$}\paradownv$-case in Figure \ref{figure:cases}.

\begin{proposition}[$\scalebox{1.25}{$\lnot$}\paradownv$-estimate] \label{bilinear:prop-tpm-ts}
Let Hypothesis \ref{hypothesis:smallness} be satisfied, let $1\leq i,j \leq \dimA$, and let $M_1,M_2$ be frequency scales satisfying $M_1\sim_\delta M_2$. Furthermore, let $\zeta \in \Cprod{s}{r}$. Then, it holds that 
\begin{equation}\label{bilinear:eq-fortunate}
\begin{aligned}
&\Big\| \partial_u \big( B_{M_1,M_2,m_1m_2}^i \phi^{+,m_1}_{M_1} \phi^{-,m_2}_{M_2} \big) \partial_v \zeta^j 
- \Big(  B_{M_1,M_2,m_1m_2}^i \big( \phi^{-,m_2}_{M_2} \paradownv \partial_v \zeta^j \big) \Big) \parallsigu\partial_u \phi^{+,m_2}_{M_2} \Big\|_{\Cprod{r-1}{r-1}} \\
&\lesssim (M_1 M_2)^{-\eta} \theta^3  \| \zeta \|_{\Cprod{s}{r}}. 
\end{aligned}
\end{equation}
Furthermore, the unfortunate resonance satisfies  
\begin{equation}\label{bilinear:eq-unfortunate}
\Big\| B_{M_1,M_2,m_1m_2}^i \big( \phi^{-,m_2}_{M_2} \paradownv \partial_v \zeta^j \big) \Big\|_{\Cprod{s}{r-1}}
\lesssim M_2^{1-s-r} \theta^2 \| \zeta \|_{\Cprod{s}{r}}. 
\end{equation}
\end{proposition}

\begin{proof}
 The estimate of the unfortunate resonance \eqref{bilinear:eq-unfortunate} follows directly from the multiplication estimate (Corollary \ref{prep:corollary-multiplication}) and Lemma \ref{key:lemma-unfortunate}. Thus, it remains to prove \eqref{bilinear:eq-fortunate}. To this end, we decompose 
 \begin{align}
 &\partial_u \big( B_{M_1,M_2,m_1m_2}^i \phi^{+,m_1}_{M_1} \phi^{-,m_2}_{M_2} \big) \partial_v \zeta^j 
- \Big(  B_{M_1,M_2,m_1m_2}^i \big( \phi^{-,m_2}_{M_2} \paradownv \partial_v \zeta^j \big) \Big) \parallsigu\partial_u \phi^{+,m_2}_{M_2} \notag \\
=& B_{M_1,M_2,m_1m_2}^i \Big( \phi_{M_2}^{-,m_2} \partial_v \zeta^j - \phi_{M_2}^{-,m_2} \paradownv \partial_v \zeta^j \Big) \partial_u \phi^{+,m_1}_{M_1} \label{bilinear:eq-tpm-ts-p1} \\ 
+&  \Big( B_{M_1,M_2,m_1m_2}^i \big( \phi_{M_2}^{-,m_2} \paradownv \partial_v \zeta^j  \big) \Big) \paragtrsimsigu \partial_u \phi^{+,m_1}_{M_1}  \label{bilinear:eq-tpm-ts-p2}\\
 +& \partial_u B_{M_1,M_2,m_1m_2}^i \phi^{+,m_1}_{M_1} \phi_{M_2}^{-,m_2}  \partial_v \zeta^j.  \label{bilinear:eq-tpm-ts-p3}
 \end{align}
 We estimate the three terms \eqref{bilinear:eq-tpm-ts-p1}, \eqref{bilinear:eq-tpm-ts-p2}, and \eqref{bilinear:eq-tpm-ts-p3} separately. Using Lemma \ref{key:lemma-unfortunate}, the first summand \eqref{bilinear:eq-tpm-ts-p1} is estimated by 
 \begin{align*}
    &\Big\| B_{M_1,M_2,m_1m_2}^i \Big( \phi_{M_2}^{-,m_2} \partial_v \zeta^j - \phi_{M_2}^{-,m_2} \paradownv \partial_v \zeta^j \Big) \partial_u \phi^{+,m_1}_{M_1} \Big\|_{\Cprod{r-1}{r-1}} \\
    \lesssim&\| B_{M_1,M_2,m_1m_2}^i \|_{\Cprod{s}{s}} \| \phi_{M_2}^{-,m_2} \partial_v \zeta^j - \phi_{M_2}^{-,m_2} \paradownv \partial_v \zeta^j  \|_{\Cprod{s}{r-1}}  \| \partial_u \phi^{+,m_1}_{M_1} \|_{\Cprod{r-1}{s}} \\
    \lesssim& M_1^{r-s} M_2^{-(r-s)-\sigma/6} \theta^3 \| \zeta \|_{\Cprod{s}{r}}. 
 \end{align*}
 Since $M_1\sim_\delta M_2$ and $\sigma=100\delta$, this term is acceptable. We now turn to the second summand \eqref{bilinear:eq-tpm-ts-p2}. To this end, we first note that the tuple $(\alpha_1,\beta_1)=(s,-s+\eta)$ satisfies 
 \begin{equation*}
     r-1 \leq \alpha_1 + (1-\sigma)\beta_1 \qquad \text{and} \qquad \alpha_1 + \beta_1 >0. 
 \end{equation*}
 Using Proposition \ref{prep:prop-bilinear}.\ref{prep:item-low-high} and Lemma \ref{key:lemma-unfortunate}, it follows that 
 \begin{align*}
    &\Big\|    \Big( B_{M_1,M_2,m_1m_2}^i \big( \phi_{M_2}^{-,m_2} \paradownv \partial_v \zeta^j  \big) \Big) \paragtrsimsigu \partial_u \phi^{+,m_1}_{M_1} \Big\|_{\Cprod{r-1}{r-1}} \\
    \lesssim&  \big\|   B_{M_1,M_2,m_1m_2}^i \big( \phi_{M_2}^{-,m_2} \paradownv \partial_v \zeta^j  \big) \big\|_{\Cprod{s}{r-1}} \big\| \partial_u \phi^{+,m_1}_{M_1} \big\|_{\Cprod{-s+\eta}{s}} \\
    \lesssim& M_1^{1-2s+\eta} M_2^{1-s-r} \| B \|_{\Cprod{s}{s}} \| \phi^- \|_{\C_v^s} \| \zeta \|_{\Cprod{s}{r}} \| \phi^+ \|_{\C_u^s} \\
    \lesssim& M_1^{1-2s+\eta} M_2^{1-s-r} \theta^3  \| \zeta \|_{\Cprod{s}{r}}. 
 \end{align*}
 Since $M_1\sim_\delta M_2$ and $1-s-r \approx -1/4$, this is acceptable. It remains to estimate the third summand \eqref{bilinear:eq-tpm-ts-p3}. Using the bilinear estimate (Proposition \ref{prep:prop-bilinear}) and Lemma \ref{key:lemma-favorable}, we have that 
 \begin{align*}
     &\big\| \partial_u B_{M_1,M_2,m_1m_2}^i \phi^{+,m_1}_{M_1} \phi_{M_2}^{-,m_2}  \partial_v \zeta^j \big\|_{\Cprod{r-1}{r-1}} \\
     &\lesssim \big\| \partial_u B_{M_1,M_2,m_1m_2}^i \phi^{+,m_1}_{M_1} \phi_{M_2}^{-,m_2} \big\|_{\Cprod{r-1}{1-r^\prime}}  \big\| \partial_v \zeta^j  \big\|_{\Cprod{r-1}{s}} \\
     &\lesssim M_{\textup{max}}^{1-3s} \| B_{M_1,M_2}\|_{\Cprod{s}{s}} \| \phi^+ \|_{\C_u^s} \| \phi^{-} \|_{\C_v^s} \| \zeta \|_{\Cprod{s}{r}} \\
     &\lesssim M_{\textup{max}}^{1-3s} \theta^3  \| \zeta \|_{\Cprod{s}{r}}. 
 \end{align*}
 Since $1-3s\approx -1/2$, this term is acceptable. 
\end{proof}

We now turn to the final proposition in our case analysis, which addresses the \raisebox{1pt}{\scalebox{0.7}{\checkmark}}-terms.

\begin{proposition}[\raisebox{1pt}{\scalebox{0.7}{\checkmark}}-estimates]\label{bilinear:prop-check}
Let Hypothesis \ref{hypothesis:smallness} be satisfied and let $1\leq i,j \leq \dimA$. Furthermore, let $\zeta_1,\zeta_2 \colon \R^{1+1}_{u,v} \rightarrow \R^{\dimA}$ satisfy 
\begin{equation*}
\Type (\zeta_1,\zeta_2) \in \big\{ \text{\tm-\tp, \tm-\tpm, \tm-\ts, \tpm-\tpm, \ts-\ts} \big\}. 
\end{equation*}
Then, it holds that 
\begin{equation}
\| \partial_u \zeta_1^i \partial_v \zeta_2^j \|_{\Cprod{r-1}{r-1}} \lesssim  \Gain(\zeta_1) 
\Gain(\zeta_2) \theta^2. 
\end{equation}
\end{proposition}

\begin{proof}
We treat the five different cases in this proposition separately. \newline

\emph{\tm-\tp~estimate:} In this case, we have that 
\begin{align*}
\zeta^i_1(u,v)=  A_{M,m}^{-,i}(u,v) \phi^{-,m}_{M}(v) \quad
\text{and} \quad \zeta^j_2(u,v) &= A_{N,n}^{+,j}(u,v)  \phi^{+,n}_{N}(u). 
\end{align*}
for some $M,N\geq 1$. Using the bilinear estimate (Proposition \ref{prep:prop-bilinear}) twice and Corollary \ref{ansatz:cor-delta-Apm}, we obtain that
\begin{align*} 
\| \partial_u   A_{M,m}^{-,i} \, \phi^{-,m}_{M} \, \partial_v  A_{N,n}^{+,j} \, \phi_N^{+,n} \|_{\Cprod{r-1}{r-1}}
&\lesssim \| \phi^{-,m}_{M}  \, \phi_N^{+,n} \|_{\Cprod{1-r^\prime}{1-r^\prime}}  \| \partial_u   A_{M,m}^{-,i} \, \partial_v  A_{N,n}^{+,j} \|_{\Cprod{r-1}{r-1}}  \\
&\lesssim \| \phi^{-,m}_{M}   \, \phi_N^{+,n} \|_{\Cprod{1-r^\prime}{1-r^\prime}}  \| \partial_u   A_{M,m}^{-,i} \|_{\Cprod{r-1}{s}}  \| \partial_v  A_{N,n}^{+,j} \|_{\Cprod{s}{r-1}} \\
&\lesssim (MN)^{1-r^\prime-s+(1-\delta) (r-s)}  \| \phi^+ \|_{\C^s} \| \phi^- \|_{\C^s} \| A_M^- \|_{\Mod_M^-} \| A_N^+ \|_{\Mod_N^+} \\
&\lesssim (MN)^{1-2s+\eta-\delta (r-s)} \theta^4. 
\end{align*}
Since 
\begin{equation*}
1-2s +\eta - \delta (r-s) \leq 1-2s +\eta - \delta/8 \leq - \delta/10, 
\end{equation*}
this contribution is acceptable.\footnote{This is the shortest argument which yields an acceptable contribution. By using frequency-support considerations for $A_M^-$ and $A_N^+$, we can obtain better decay in $M$ and $N$, and thus this term is less serious than our argument suggests.}  \newline

\emph{\tm-\tpm~estimate:} In this case, we have that 
\begin{align*}
\zeta^i_1(u,v)=  A_{M,m}^{-,i}(u,v) \phi^{-,m}_{M}(v) \quad
\text{and} \quad \zeta^j_2(u,v) &= B_{N_1,N_2,n_1n_2}^{j}(u,v) \phi^{+,n_1}_{N_1}(u) \phi^{-,n_2}_{N_2}(v)
\end{align*}
for some $M,N_1,N_2\geq 1$ satisfying $N_1\sim_\delta N_2$. Using the product rule, we obtain 
\begin{align}
 \partial_u \Big( A_{M,m}^{-,i} \phi^{-,m}_{M} \Big) \partial_v \Big(  B_{N_1,N_2,n_1n_2}^{j} \phi^{+,n_1}_{N_1} \phi^{-,n_2}_{N_2}\Big)   
=&  \partial_u  A_{M,m}^{-,i} \phi^{-,m}_{M}   B_{N_1,N_2,n_1n_2}^{j} \phi^{+,n_1}_{N_1} \partial_v \phi^{-,n_2}_{N_2} \label{nonlinear:eq-check-proof5} \\ 
+&  \partial_u  A_{M,m}^{-,i} \phi^{-,m}_{M}   \partial_v B_{N_1,N_2,n_1n_2}^{j} \phi^{+,n_1}_{N_1}  \phi^{-,n_2}_{N_2} \label{nonlinear:eq-check-proof6} .
\end{align}
We start with the first summand \eqref{nonlinear:eq-check-proof5}, which is the main term. Using the multiplication estimate (Corollary \ref{prep:corollary-multiplication}), the bilinear estimate (Proposition \ref{prep:prop-bilinear}), and the properties of the linear waves (Corollary \ref{ansatz:lemma-data}), we obtain that 
\begin{align*}
 &\| \partial_u  A_{M,m}^{-,i} \phi^{-,m}_{M}   B_{N_1,N_2,n_1n_2}^{j} \phi^{+,n_1}_{N_1} \partial_v \phi^{-,n_2}_{N_2} \|_{\Cprod{r-1}{r-1}} \\
\lesssim & \| B_{N_1,N_2,n_1n_2}^{j} \|_{\Cprod{s}{s}}  \| \partial_u  A_{M,m}^{-,i} \phi^{-,m}_{M}    \phi^{+,n_1}_{N_1} \partial_v \phi^{-,n_2}_{N_2} \|_{\Cprod{r-1}{r-1}} \\
\lesssim &  \| B_{N_1,N_2,n_1n_2}^{j} \|_{\Cprod{s}{s}}  \| \partial_u  A_{M,m}^{-,i}     \phi^{+,n_1}_{N_1}  \|_{\Cprod{\eta}{s}}  \| \phi^{-,m}_{M}  \partial_v \phi^{-,n_2}_{N_2} \|_{\Cprod{\eta}{r-1}} \\
\lesssim & M^{1-s+\eta} N_1^{-s+\eta} \big( 1\{ M \lesssim N_2 \} N_2^{r-s} M^{-s} + 1\{ M \gg N_2 \} M^{r-1-s} N_2^{1-s} \big) \\
\times & \| A_M^+ \|_{\Cprod{s}{s}}  \| B_{N_1,N_2} \|_{\Cprod{s}{s}} \| \phi^+ \|_{\C^s} \| (\phi^+,\phi^-) \|_{\Ds}^2 \\
\lesssim&  M^{1-s+\eta} N_1^{-s+\eta} \big( 1\{ M \lesssim N_2 \} N_2^{r-s} M^{-s} + 1\{ M \gg N_2 \} M^{r-1-s} N_2^{1-s} \big) \, \theta^5. 
\end{align*}
Using $N_1 \sim_\delta N_2$, the pre-factor can be estimated by 
\begin{align*}
&M^{1-s+\eta} N_1^{-s+\eta} \big( 1\{ M \lesssim N_2 \} N_2^{r-s} M^{-s} + 1\{ M \gg N_2 \} M^{r-1-s} N_2^{1-s} \big) \\
&\lesssim 1\{ M \lesssim N_2 \} M^{1-2s+\eta} N_2^{r-2s+\eta+\delta} +  1\{ M \gg N_2 \}  M^{r-2s+\eta} N_2^{1-2s+\eta+\delta} \\
&\lesssim \max(M,N_1,N_2)^{r+1-4s+4\delta}. 
\end{align*} 
Since $r+1-4s\approx -1/4$, this completes our estimate of \eqref{nonlinear:eq-check-proof5}. The estimate for the minor term \eqref{nonlinear:eq-check-proof6} follows directly from the $\tm$-case in Lemma \ref{key:lemma-bilinear} and the $\tpm$-case in Lemma \ref{key:lemma-favorable}. \newline

\emph{\tm-\ts~estimate:} This estimate directly follows from the $\tm$-$\ts$-case in Lemma \ref{key:lemma-bilinear}. \newline

\emph{\tpm-\tpm~estimate:} We remark that this term obeys better estimates than the $\tpm$-$\ts$-term (see Proposition \ref{bilinear:prop-tpm-ts}). The reason is that the structure of $\tpm$ allows for better bounds on certain high$\times$high$\rightarrow$low-interactions. 

In this case, we have that 
\begin{align*}
\zeta^i_1(u,v) = B_{M_1,M_2,m_1m_2}^{i}(u,v) \phi^{+,m_1}_{M_1}(u) \phi^{-,m_2}_{M_2}(v) \quad
\text{and} \quad \zeta^j_2(u,v) &= B_{N_1,N_2,n_1n_2}^{j}(u,v) \phi^{+,n_1}_{N_1}(u) \phi^{-,n_2}_{N_2}(v)
\end{align*}
for some $M_1,M_2,N_1,N_2\geq 1$ satisfying $M_1 \sim_\delta M_2$ and $N_1 \sim_\delta N_2$. To simplify the notation, we set $M_{\textup{max}}:= \max(M_1,M_2)$ and $N_{\textup{max}}:=\max(N_1,N_2)$. 
By symmetry, we can assume that $M_{\textup{max}}\leq N_{\textup{max}}$. Using the product rule, we have that 
\begin{align}
&\partial_u \Big( B_{M_1,M_2,m_1m_2}^{i}  \phi^{+,m_1}_{M_1}  \phi^{-,m_2}_{M_2} \Big)    \partial_v \Big( B_{N_1,N_2,n_1n_2}^{j}   \phi^{+,n_1}_{N_1}   \phi^{-,n_2}_{N_2} \Big) \notag \\
=&  B_{M_1,M_2,m_1m_2}^{i}   \partial_u \phi^{+,m_1}_{M_1}  \phi^{-,m_2}_{M_2}   B_{N_1,N_2,n_1n_2}^{j}  \phi^{+,n_1}_{N_1} \partial_v \phi^{-,n_2}_{N_2} \label{nonlinear:eq-check-proof1} \\
 +&   B_{M_1,M_2,m_1m_2}^{i}   \partial_u  \phi^{+,m_1}_{M_1} \phi^{-,m_2}_{M_2} \,    \partial_v B_{N_1,N_2,n_1n_2}^{j}   \phi^{+,n_1}_{N_1}   \phi^{-,n_2}_{N_2} \label{nonlinear:eq-check-proof2} \\
 +&  \partial_u B_{M_1,M_2,m_1m_2}^{i}     \phi^{+,m_1}_{M_1} \phi^{-,m_2}_{M_2} \,     B_{N_1,N_2,n_1n_2}^{j}   \phi^{+,n_1}_{N_1}  \partial_v \phi^{-,n_2}_{N_2} \label{nonlinear:eq-check-proof3} \\
  +& \partial_u B_{M_1,M_2,m_1m_2}^{i}   \phi^{+,m_1}_{M_1}    \phi^{-,m_2}_{M_2}   \partial_v B_{N_1,N_2,n_1n_2}^{j}    \phi^{+,n_1}_{N_1}  \phi^{-,n_2}_{N_2}. \label{nonlinear:eq-check-proof4} 
\end{align}
We now estimate the four terms \eqref{nonlinear:eq-check-proof1}-\eqref{nonlinear:eq-check-proof4} separately. We begin with the main term \eqref{nonlinear:eq-check-proof1}. Using the multiplication estimate (Corollary \ref{prep:corollary-multiplication}) and Lemma \ref{ansatz:lemma-data}, we obtain that 
\begin{align*}
 &\| B_{M_1,M_2,m_1m_2}^{i}   \partial_u \phi^{+,m_1}_{M_1}   \phi^{-,m_2}_{M_2}  B_{N_1,N_2,n_1n_2}^{j}  \phi^{+,n_1}_{N_1}  \partial_v \phi^{-,n_2}_{N_2} \|_{\Cprod{r-1}{r-1}} \\
 &\lesssim \| B_{M_1,M_2,m_1m_2}^{i} \|_{\Cprod{s}{s}}  \| B_{N_1,N_2,n_1n_2}^{j} \|_{\Cprod{s}{s}} 
 \|  \partial_u \phi^{+,m_1}_{M_1}   \phi^{-,m_2}_{M_2}    \phi^{+,n_1}_{N_1}  \partial_v \phi^{-,n_2}_{N_2}  \|_{\Cprod{r-1}{r-1}}  \\
 &\lesssim M_1^{r-s} N_1^{-s} M_2^{-s} N_2^{r-s}    \| B_{M_1,M_2} \|_{\Cprod{s}{s}}  \| B_{N_1,N_2}  \|_{\Cprod{s}{s}}  \| (\phi^+, \phi^- ) \|_{\Ds}^4  \\
 &\lesssim M_{\textup{max}}^{r-2s+\delta}   N_{\textup{max}}^{r-2s+\delta} 
\| B_{M_1,M_2} \|_{\Cprod{s}{s}}  \| B_{N_1,N_2}  \|_{\Cprod{s}{s}}  \| (\phi^+, \phi^- ) \|_{\Ds}^4 \\
&\lesssim  M_{\textup{max}}^{r-2s+\delta}   N_{\textup{max}}^{r-2s+\delta} \theta^6. 
\end{align*}
Since $r-2s\approx -1/4$, this contribution is acceptable.  
We now turn to the second summand \eqref{nonlinear:eq-check-proof2}, which is relatively easy. Using the bilinear estimate (Proposition \ref{prep:prop-bilinear}) and Lemma \ref{ansatz:lemma-data}, we have that 
\begin{align*}
&\|  B_{M_1,M_2,m_1m_2}^{i}   \partial_u  \phi^{+,m_1}_{M_1} \phi^{-,m_2}_{M_2} \,    \partial_v B_{N_1,N_2,n_1n_2}^{j}   \phi^{+,n_1}_{N_1}   \phi^{-,n_2}_{N_2} \|_{\Cprod{r-1}{r-1}} \\
&\lesssim \|  B_{M_1,M_2,m_1m_2}^{i} \|_{\Cprod{s}{s}} \|   \partial_u  \phi^{+,m_1}_{M_1} \phi^{+,n_1}_{N_1} \|_{\C_u^{r-1}} \| \partial_v B_{N_1,N_2,n_1n_2}^{j}  \phi^{-,m_2}_{M_2}  \phi^{-,n_2}_{N_2} \|_{\Cprod{s}{r-1}} \\
&\lesssim M_1^{r-s} N_1^{-s} N_2^{1-s} M_2^{-s} N_2^{-s} \theta^6 \\
&= M_1^{r-s} M_2^{-s} N_1^{-s} N_2^{1-2s} \theta^6.
\end{align*}
Since $r-2s\approx -1/4$ and $1-3s\approx -1/2$, this contribution is acceptable.
The third summand \eqref{nonlinear:eq-check-proof3} can be treated similarly. The fourth summand  \eqref{nonlinear:eq-check-proof4} can easily be controlled through the bilinear estimate (Proposition \ref{prep:prop-bilinear}) and Lemma \ref{key:lemma-favorable}. \newline

\emph{\ts-\ts~estimate:} This estimate directly follows from the bilinear estimate (Proposition \ref{prep:prop-bilinear}). 
\end{proof}

%Old labels: Resonant estimate
%\label{nonlinear:eq-tpm-ts}
%\label{nonlinear:eq-res}
%\label{nonlinear:eq-res-estimate}

\subsection{\protect{\texorpdfstring{The product $\phi^k \partial_u \phi^i$}{The product}}}\label{section:product}

The following lemma controls $\phi^k \scalebox{0.9}{\paragtrsimsigu}\partial_u \big( A_{M,m}^{+,i} \phi^{+,m}_M\big)$ in $\Cprod{r-1}{1-r^\prime}$. Together with Bony's paralinearization (Lemma \ref{prep:lemma-bony}), it will be used in Section \ref{section:nonlinear} to control $\SecondC{a}{ij}(\phi) \partial_u\big( A_{M,m}^{+,i} \phi^{+,m}_M\big)$.

\begin{lemma}\label{bilinear:lemma-product}
Let Hypothesis \ref{hypothesis:smallness} be satisfied, let $1\leq i,k \leq \dimA$, and let $K,M$ be frequency scales satisfying $K\gtrsim M^{1-\sigma}$. Let $\HH_M^+\colon \R_{u,v}^{1+1}\rightarrow \R^{D\times D}$ satisfy $P_{\gg M^{1-\sigma}}^u H_M^+ = 0$. Furthermore, let $\zeta\colon \R_{u,v}^{1+1} \rightarrow \R^{\dimA}$ be of one of the four types $\tm$, $\tp$, $\tpm$, or $\ts$. Then, it holds that 
\begin{equation}\label{bilinear:eq-product}
\begin{aligned}
&\Big\| P_{K}^u \big( \zeta^k \big) \partial_u \big( H_{M,m}^{+,i} \phi_M^{+,m} \big) \Big\|_{\Cprod{r-1}{1-r^\prime}} \\
\lesssim&  \Big( 1\{ K \lesssim M\}  M^{1-s-r} + 1\{ K \gg M\} K^{2(r-1+\eta)} \Big) \Gain(\zeta) \theta^2 \big\| \HH_M^+ \big\|_{\Cprod{s}{s}} \\
\lesssim& K^{-\eta} M^{1-s-r+\eta}\Gain(\zeta) \theta^2 \big\| \HH_M^+ \big\|_{\Cprod{s}{s}}.
\end{aligned}
\end{equation}
In the individual cases, we have the following more detailed estimates, which are stated only for the (difficult) regime $M^{1-\sigma}\leq K \lesssim M$. \newline
\begin{itemize}[leftmargin=13mm]
    \item[$\tm$:] If $\zeta=A_{L,\ell}^{-} \phi^{-,\ell}_{L}$ for some $L\geq 1$, we have that 
    \begin{align*}
        \Big\| P_{K}^u \big( A_{L,\ell}^{-,k} \phi_{L}^{-,\ell} \big) \partial_u \big( H_{M,m}^{+,i} \phi^{+,m}_M \big) \Big\|_{\Cprod{r-1}{1-r^\prime}} 
        &\lesssim \big( L^{-\eta} M^{-(r-s)-\delta/8} + L^{1-r^\prime-s} M^{1-s-r} \big) \theta^3 \big\| \HH_M^+ \big\|_{\Cprod{s}{s}} \\
        &\lesssim L^{-\eta} M^{1-s-r} \theta^3 \big\| \HH_M^+ \big\|_{\Cprod{s}{s}}. 
    \end{align*}
    \item[$\tp$:] If $\zeta=A_{L,\ell}^{+,k} \phi^{+,\ell}_L$ for some $L\geq 1$, then
   \begin{align*}
    &\Big\| P_{K}^u \big( A_{L,\ell}^{+,k} \phi_{L}^{+,\ell} \big) \partial_u \big( H_{M,m}^{+,i} \phi^{+,m}_M \big) \Big\|_{\Cprod{r-1}{s}} 
        \lesssim L^{-\eta} M^{r-2s+\sigma} \theta^3 \big\| \HH_M^+ \big\|_{\Cprod{s}{s}}.
    \end{align*}
    \item[$\tpm$:] If $\zeta=B_{L_1,L_2,\ell_1\ell_2} \phi^{+,\ell_1}_{L_1} \phi^{-,\ell_2}_{L_2}$ with $L_1 \sim_\delta L_2$,  then 
    \begin{align*}
    &\Big\| P_{K}^u \big( B_{L_1,L_2,\ell_1\ell_2}^k \phi^{+,\ell_1}_{L_1} \phi^{-,\ell_2}_{L_2} \big) \partial_u \big( \HH_{M,m}^{+,i}\phi^{+,m}_M \big) \Big\|_{\Cprod{r-1}{s}} 
        \lesssim (L_1 L_2)^{-\eta} M^{r-2s+\sigma} \theta^4 \big\| \HH_M^+ \big\|_{\Cprod{s}{s}}.
    \end{align*}
    \item[$\ts$:] If $\zeta=\psi$, then 
    \begin{align*}
      &\Big\|   P_{K}^u \psi^k \partial_u \big( \HH_{M,m}^{+,i} \phi^{+,m}_M \big) 
        -  \HH_{M,m}^{+,i} \big( P_{K}^u \psi^k \paradownu \partial_u \phi_{M}^{+,m} \big) \Big\|_{\Cprod{r-1}{s}} 
        \lesssim  M^{-(r-s)-\sigma/6} \theta^2 \big\| \HH_M^+ \big\|_{\Cprod{s}{s}}. 
    \end{align*}
    Furthermore, the resonant part satisfies 
    \begin{align*}
       & \Big\|  \HH_{M,m}^{+,i}  \big( P_{K}^u \psi^k \paradownu \partial_u \phi_{M}^{+,m} \big) \Big\|_{\Cprod{r-1}{s}} 
       \lesssim  M^{1-s-r} \theta^2 \big\| \HH_M^+ \big\|_{\Cprod{s}{s}}. 
    \end{align*}
\end{itemize}
\end{lemma}

\begin{remark}
The matrix $\HH_M^+$ will be chosen either as the modulation $A_{M}^+$ or the identity $\operatorname{Id}$. By introducing the auxiliary variable $\HH_M^+$, we are able to treat both cases simultaneously. 
\end{remark}

\begin{proof}
We split the argument into several steps. In the first three steps, we treat the case $\tm$, the two cases $\tp$ and $\tpm$ simultaneously, and the case $\ts$. In the final step, we combine the estimates to prove
\eqref{bilinear:eq-product}. \newline

\emph{Step 1: The case $\tm$.} 
We first note that 
\begin{equation*}
P_K^u \big( A_{L,\ell}^{-,k} \phi_{L}^{-,\ell} \big) 
= \big(P_K^uA_{L,\ell}^{-,k}\big) \, \phi_{L}^{-,\ell}
\end{equation*}

We further distinguish the cases $M^{1-\sigma}\leq K \ll M$ and $K\sim M$. 

\emph{Step 1.(a): $M^{1-\sigma}\leq K \ll M$}. We use Proposition \ref{prep:prop-bilinear}.\ref{prep:item-nonres}, which yields
\begin{align*}
    &\Big\|  P_{K}^u  A_{L,\ell}^{-,k}   \, \phi_{L}^{-,\ell} \partial_u \big( H_{M,m}^{+,i} \phi^{+,m}_M \big) \Big\|_{\Cprod{r-1}{1-r^\prime}} \\
    \lesssim&
    \Big\|  P_{K}^u  A_{L,\ell}^{-,k}   \, \phi_{L}^{-,\ell} \Big\|_{\Cprod{\eta}{1-r^\prime}} 
    \Big\| \partial_u \big( H_{M,m}^{+,i} \phi^{+,m}_M \big) \Big\|_{\Cprod{r-1}{s}} \\
    \lesssim& K^{-s+\eta} L^{1-r^\prime-s} M^{r-s} \| A_L^- \|_{\Cprod{s}{s}} \| ( \phi^+, \phi^- ) \|_{\Ds} \| \HH_M^+\|_{\Cprod{s}{s}} \| ( \phi^+, \phi^- ) \|_{\Ds} \\
    \lesssim& K^{-s+\eta} L^{1-r^\prime-s} M^{r-s} \theta^3 \| \HH_M^+\|_{\Cprod{s}{s}}.
\end{align*}
Since $K\gtrsim M^{1-\sigma}$, we have that 
\begin{align*}
    K^{-s+\eta} L^{1-r^\prime-s} M^{r-s} \lesssim L^{-\eta} M^{r-2s+\sigma} \lesssim L^{-\eta} M^{-(r-s)-\delta/8},
\end{align*}
where we used that $3s-2r\gg \sigma\geq \delta$.

\emph{Step 1.(b): $K \sim M$.} We estimate 
\begin{align*}
    \big\| P_K^u A_{L,\ell}^{-,k}  \, \phi_{L}^{-,\ell} \partial_u \big( H_{M,m}^{+,i} \phi^{+,m}_M \big) \big\|_{\Cprod{r-1}{1-r^\prime}} 
    &\lesssim  \|  P_{K}^u A_L^- \|_{L^\infty_u \C_v^s} \| \phi_L^- \|_{\C_v^{1-r^\prime}} \big\|  \partial_u \big( H_{M,m}^{+,i} \phi^{+,m}_M \big) \big\|_{L^\infty_u \C_v^s} \\
    &\lesssim  L^{1-r^\prime-s} M^{1-s}  \theta^2 \| P^u_{K}  A_L^- \|_{L^\infty_u \C_v^s} 
    \| \HH_M^+ \|_{\Cprod{s}{s}} . 
\end{align*}
For the $A_L^-$-factor, the condition $K\sim M$ yields that 
\begin{align*}
\| P_K^u  A_L^- \|_{L^\infty_u \C_v^s} 
&\leq \| P_K^u P_{\lesssim L^{1-\delta}}^u A_L^- \|_{L^\infty_u \C_v^s} 
+\| P_K^u P_{\gg L^{1-\delta}}^u A_L^- \|_{L^\infty_u \C_v^s} \\
&\lesssim   1\{ M\lesssim L^{1-\delta}\} M^{-s} \| P_{\lesssim L^{1-\delta}}^u A_L^- \|_{\Cprod{s}{s}}  
+ M^{-r} \| P_{\gg L^{1-\delta}}^u A_L^- \|_{\Cprod{r}{s}} \\
&\lesssim \big(  1\{ M\lesssim L^{1-\delta}\} M^{-s}+ M^{-r} \big) \| A_L^- \|_{\Mod_L^-} \\
&\lesssim \big(  1\{ M\lesssim L^{1-\delta}\} M^{-s}+ M^{-r} \big) \theta. 
\end{align*}
The total pre-factor can then be estimated by 
\begin{align*}
&L^{1-r^\prime-s} M^{1-s} \big( 1\{ M \lesssim L^{1-\delta}\} M^{-s} +M^{-r} \big) \\
=& 1\{ M\lesssim L^{1-\delta}\} L^{1-r^\prime-s} M^{1-2s} + L^{1-r^\prime-s} M^{1-s-r} \\
\lesssim & L^{-\eta} M^{-(1-\delta)^{-1} (r^\prime+s-1) + 1-2s} + L^{1-r^\prime-s} M^{1-s-r}. 
\end{align*}
Since
\begin{align*}
    -(1-\delta)^{-1} (r^\prime+s-1) + 1-2s
    &\leq -\big( 1 + \tfrac{3\delta}{4}\big) (r^\prime+s-1) +1-2s \\
    &=-(r-s) - \tfrac{3\delta}{4} (r^\prime+s-1) + \eta + 2 (1-2s) \\
    &\leq -(r-s) - \delta/8,
\end{align*}
this yields the desired estimate. \newline

\emph{Step 2: The cases $\tp$ and $\tpm$.}  Using Lemma \ref{key:lemmma-direction}, we can assume that 
\begin{equation*}
\zeta^k(u,v) = E_{L,\ell}^{+,k}(u,v) \phi_L^{+,\ell}(u)
\end{equation*}
for some $L\geq 1$, where $E_L^+$ satisfies \eqref{key:eq-direction-11} or \eqref{key:eq-direction-12}. In particular, $\| E_L^+ \|_{\Cprod{s}{s}} \lesssim \theta$. Due to the frequency-support conditions on $E_{L,\ell}^{+,k}$, it follows that $K\sim L$. Then, we decompose 
\begin{align}
     P^u_K  \Big( E_{L,\ell}^{+,k} \phi_L^{+,\ell}  \Big) \partial_u \big( H_{M,m}^{+,i} \phi_M^{+,m} \big) 
     =&   E_{L,\ell}^{+,k} \,  P^u_K\phi_L^{+,\ell}  H_{M,m}^{+,i} \partial_u \phi_M^{+,m} \label{bilinear:eq-product-p1}\\
     +& \big[ P^u_K,  E_{L,\ell}^{+,k} \big] \Big(  \phi_L^{+,\ell} \Big) \,  H_{M,m}^{+,i} \, \partial_u \phi_M^{+,m} \label{bilinear:eq-product-p2} \\
     +& P^u_K  \Big(  E_{L,\ell}^{+,k} \phi_L^{+,\ell}  \Big)  \partial_u H_{M,m}^{+,i} \,  \phi_M^{+,m}. \label{bilinear:eq-product-p3}
\end{align}
We start by treating \eqref{bilinear:eq-product-p1}, which is the main term. To this end, we estimate
\begin{align*}
 \Big\|  E_{L,\ell}^{+,k} \,  P^u_K\phi_L^{+,\ell}  H_{M,m}^{+,i} \partial_u \phi_M^{+,m} \Big\|_{\Cprod{r-1}{s}} 
 &\lesssim \|  E_{L,\ell}^{+,k} \|_{\Cprod{s}{s}} \| \HH_{M,m}^{+,i} \|_{\Cprod{s}{s}}
 \|  P^u_K\phi_L^{+,\ell}  \partial_u \phi_M^{+,m} \|_{\C_u^{r-1}} \\
 &\lesssim M^{r-s} L^{-s} \|  E_{L}^{+} \|_{\Cprod{s}{s}} \| \HH_{M}^{+} \|_{\Cprod{s}{s}} \| ( \phi^+, \phi^- ) \|_{\Ds}^2\\
 &\lesssim M^{r-s} L^{-s} \theta^3 \| \HH_{M}^{+} \|_{\Cprod{s}{s}}. 
\end{align*}
Since $M^{1-\sigma}\leq K\lesssim L$, we have that
\begin{align*}
M^{r-s} L^{-s} \lesssim M^{r-2s+\sigma} L^{-4\eta},
\end{align*}
which is acceptable. We now turn to the commutator term \eqref{bilinear:eq-product-p2}.
Using Lemma \ref{prep:lem-commutator}, we obtain that 
\begin{align*}
&\Big\| \big[ P^u_K,  E_{L,\ell}^{+,k} \big] \Big(  \phi_L^{+,\ell} \Big) \,  H_{M,m}^{+,i} \, \partial_u \phi_M^{+,m} \Big\|_{\Cprod{r-1}{s}} \\
&\lesssim \Big\| \big[ P^u_K,  E_{L,\ell}^{+,k} \big] \Big(  \phi_L^{+,\ell} \Big) \Big\|_{\Cprod{1-r^\prime}{s}} \Big\| H_{M,m}^{+,i} \, \partial_u \phi_M^{+,m} \Big\|_{\Cprod{r-1}{s}} \\
&\lesssim K^{1-r^\prime-2s+4\eta} M^{r-s} \| E_L^+ \|_{\Cprod{s}{s}} \| \phi^+_L \|_{\C_u^{s-4\eta}} \| \HH_M^+ \|_{\Cprod{s}{s}} \| \phi^+_M \|_{\C_u^s}  \\
&\lesssim L^{-4\eta}K^{1-r^\prime-2s+4\eta} M^{r-s} \theta^3 \| \HH_M^+ \|_{\Cprod{s}{s}}.  
\end{align*}
Since $K\geq M^{1-\sigma}$, we have that 
\begin{equation*}
L^{-4\eta}K^{1-r^\prime-2s+4\eta} M^{r-s} \lesssim L^{-4\eta} M^{r-s-(1-\sigma) (1-r-2s+5\eta)} \lesssim L^{-4\eta} M^{1-3s+2\sigma}. 
\end{equation*}
Since $1-3s\approx -1/2$ and $r-2s\approx -1/4$, this contribution is acceptable. We now turn to the last term \eqref{bilinear:eq-product-p3}, which contains a favorable derivative. Using\footnote{While $\HH_M^+$ is not necessarily $A_M^+$, the same argument applies.} Lemma \ref{key:lemma-favorable}, it can be estimated by 
\begin{align*}
&\Big\| P^u_K  \Big(  E_{L,\ell}^{+,k} \phi_L^{+,\ell}  \Big)  \partial_u H_{M,m}^{+,i} \,  \phi_M^{+,m} \Big\|_{\Cprod{r-1}{s}} \\
&\lesssim \Big\| P^u_K  \Big(  E_{L,\ell}^{+,k} \phi_L^{+,\ell}  \Big) \Big\|_{\Cprod{1-r^\prime}{s}} 
\Big\| \partial_u H_{M,m}^{+,i} \,  \phi_M^{+,m} \Big\|_{\Cprod{r-1}{s}} \\
&\lesssim L^{1-r^\prime-s} M^{r-2s} \| E_L^+ \|_{\Cprod{s}{s}} \| \HH_M^+ \|_{\Cprod{s}{s}} \| ( \phi^+, \phi^- ) \|_{\Ds}^2 \\
&\lesssim  L^{1-r^\prime-s} M^{r-2s} \theta^3 \| \HH_M^+ \|_{\Cprod{s}{s}}. 
\end{align*}
Since $1-r^\prime-s\approx -1/4$, this term is clearly acceptable.\newline

\emph{Step 3: The case $\ts$.} We begin with the estimate of the non-resonant portion. To this end, we decompose 
\begin{align}
 &P_{K}^u \psi^k \partial_u \big( \HH_{M,m}^{+,i} \phi^{+,m}_M \big) 
-  \HH_{M,m}^{+,i} \big( P_{K}^u \psi^k \paradownu \partial_u \phi_{M}^{+,m} \big) \notag  \\
=&  \HH_{M,m}^{+,i} \big(  P_{K}^u \psi^k \partial_u \phi_{M}^{+,m}  - P_{K}^u \psi^k \paradownu \partial_u \phi_{M}^{+,m} \big)  
+ P_{K}^u \psi^k \partial_u \HH_{M,m}^{+,i} \, \phi^{+,m}_M . \label{bilinear:eq-product-p4}
\end{align}
The first summand in \eqref{bilinear:eq-product-p4} can be estimated using the multiplication estimate (Corollary \ref{prep:corollary-multiplication}) and Lemma \ref{key:lemma-unfortunate}. The second summand in \eqref{bilinear:eq-product-p4} can be estimated using the bilinear estimate (Proposition \ref{prep:prop-bilinear}) and Lemma \ref{key:lemma-favorable}. 

The estimate of the resonant term $ \HH_{M,m}^{+,i}  \big( P_{K}^u \psi^k \paradownu \partial_u \phi_{M}^{+,m} \big)$ directly follows from the multiplication estimate (Corollary \ref{prep:corollary-multiplication}) and the inequality \eqref{key:eq-unfortunate-v1} in Lemma \ref{key:lemma-unfortunate}. \newline

\emph{Step 4: The conclusion.} If $K\lesssim M$, the combined estimate \eqref{bilinear:eq-product} follows directly from the individual estimates and  $\max(-(r-s),r-2s+\sigma)<1-s-r$. In the remaining case $K\gg M$, we have the easier estimate
\begin{align*}
\big\| P_{K}^u \big( \zeta^k \big) \partial_u \big( H_{M,m}^{+,i} \phi_M^{+,m} \big) \big\|_{\Cprod{r-1}{1-r^\prime}} 
&\lesssim K^{r-1} \big\| P_{K}^u \big( \zeta^k \big) \partial_u \big( H_{M,m}^{+,i} \phi_M^{+,m} \big) \big\|_{L^\infty_u\C_v^{1-r^\prime}} \\
&\lesssim K^{r-1} \big\| P_{K}^u  \zeta^k \big\|_{L^\infty_u \C_v^{1-r^\prime}}
\big\| \partial_u \big( H_{M,m}^{+,i} \phi_M^{+,m} \big)\big\|_{L^\infty_u \C_v^{1-r^\prime}} \\
&\lesssim K^{r-1-s+\eta} M^{r-s} \| \zeta\|_{\Cprod{s-\eta}{1-r^\prime}} \| \HH_M^+ \|_{\Cprod{s}{s}} \| ( \phi^+, \phi^- ) \|_{\Ds} \\
&\lesssim K^{2r-2+1-2s+\eta} \Gain(\zeta) \theta^2 \| \HH_M^+ \|_{\Cprod{s}{s}}.
\end{align*}
Since $1-2s\leq \eta$, this yields the desired estimate.
\end{proof}

\begin{corollary}\label{bilinear:cor-product}
Let Hypothesis \ref{hypothesis:smallness} be satisfied, let $1\leq a,b,m \leq \dimA$ and let $K,L,M$ be frequency-scales satisfying $K\sim M$. Furthermore, let $\zeta_1\colon \R_{u,v}^{1+1}\rightarrow \R^\dimA$ and let $\zeta_2$ be of type $\tm$, $\tp$, $\tpm$, or $\ts$. Then, it holds that 
\begin{align*}
\big\| P_L^v \big( \zeta_1^a P_K^u \zeta_2^b \big) \partial_u \phi^{+,m}_M \big\|_{\Cprod{r-1}{1-r^\prime}} 
\lesssim& \big( M^{-(r-s)-\delta/8} + L^{1-s-r+2\eta} M^{1-s-r} \big) \| \zeta_1 \|_{\Cprod{s}{s}} \Gain(\zeta_2) \theta^2. 
\end{align*}
\end{corollary}

\begin{proof}
When applying Lemma \ref{bilinear:lemma-product} in the following argument, we always take $\HH_M^+=\operatorname{Id}$. We distinguish two different cases according to the type of $\zeta_2$. \newline

\emph{Case 1: $\zeta_2$ is of type $\tm$.} We then decompose $\zeta_1^a= P_{\ll L}^v \zeta_1^a+P_{\gtrsim L}^v \zeta_1^a$. To estimate the contribution of the low-frequency term, we note that
\begin{align*}
\big\| P_L^v \big( P_{\ll L}^v \zeta_1^a \, P_K^u \zeta_2^b \big) \partial_u \phi^{+,m}_M \big\|_{\Cprod{r-1}{1-r^\prime}} 
&=\big\| P_L^v \big( P_{\ll L}^v \zeta_1^a \,  P_K^u \widetilde{P}_L^v \zeta_2^b  \partial_u \phi^{+,m}_M  \big) \big\|_{\Cprod{r-1}{1-r^\prime}} \\
&\leq \big\|P_{\ll L}^v \zeta_1^a \,  P_K^u \widetilde{P}_L^v \zeta_2^b \, \partial_u \phi^{+,m}_M \big\|_{\Cprod{r-1}{1-r^\prime}} \\
&\leq \big\| \zeta_1^a  \big\|_{\Cprod{s}{s}} 
\big\|   P_K^u \widetilde{P}_L^v \zeta_2^b \, \partial_u \phi^{+,m}_M \big\|_{\Cprod{r-1}{1-r^\prime}} \\
&= \big\| \zeta_1^a  \big\|_{\Cprod{s}{s}} 
\big\|   \widetilde{P}_L^v \big( P_K^u  \zeta_2^b \, \partial_u \phi^{+,m}_M  \big) \big\|_{\Cprod{r-1}{1-r^\prime}}. 
\end{align*}
The desired estimate then follows from the $\tm$-estimate in Lemma \ref{bilinear:lemma-product}. For the high-frequency term in $\zeta_1$, we estimate 
\begin{align*}
\big\| P_L^v \big( P_{\gtrsim L}^v \zeta_1^a \, P_K^u \zeta_2^b \big) \, \partial_u \phi^{+,m}_M \big\|_{\Cprod{r-1}{1-r^\prime}} 
&=\big\| P_L^v \big( P_{\gtrsim L}^v \zeta_1^a \,  P_K^u  \zeta_2^b \,  \partial_u \phi^{+,m}_M  \big) \big\|_{\Cprod{r-1}{1-r^\prime}} \\
&\leq \big\|  P_{\gtrsim L}^v \zeta_1^a \,  P_K^u \zeta_2^b  \, \partial_u \phi^{+,m}_M   \big\|_{\Cprod{r-1}{1-r^\prime}} \\
&\lesssim 
\big\|  P_{\gtrsim L}^v \zeta_1^a\big\|_{\Cprod{1-r^\prime}{1-r^\prime}}
\big\| P_K^u \zeta_2^b  \, \partial_u \phi^{+,m}_M   \big\|_{\Cprod{r-1}{1-r^\prime}} \\
&\lesssim L^{1-r^\prime-s} \big\| \zeta^a_1 \big\|_{\Cprod{s}{s}}
\big\| P_K^u \zeta_2^b  \, \partial_u \phi^{+,m}_M   \big\|_{\Cprod{r-1}{1-r^\prime}}. 
\end{align*}
The desired estimate then follows from \eqref{bilinear:eq-product}. \newline

\emph{Case 2: $\zeta_2$ is of type $\tp$, $\tpm$, or $\ts$.} In this case, we estimate 
\begin{equation*}
\big\| P_L^v \big( \zeta_1^a P_K^u \zeta_2^b \big) \partial_u \phi^{+,m}_M \big\|_{\Cprod{r-1}{1-r^\prime}} 
\lesssim  L^{1-r^\prime-s}  \big\|  \zeta_1^a \big\|_{\Cprod{s}{s}} \big\| P_K^u \zeta_2^b \, \partial_u \phi^{+,m}_M \big\|_{\Cprod{r-1}{s}}. 
\end{equation*}
Thus, the desired estimate follows from the $\tp$, $\tpm$, or $\ts$-estimate in Lemma \ref{bilinear:lemma-product}. 
\end{proof}
%%%%%%%%%%%%%%%%%%%%%%%%%%%%%%%%% Nonlinear estimates %%%%%%%%%%%%%%%%%%%%%%%%%%%%%%

\section{The full nonlinearity}\label{section:nonlinear}

In the previous section, we analyzed both the null form $\partial_u \phi^i \partial_v \phi^j$ and the product $\SecondC{a}{ij}(\phi) \partial_u \phi^i$. In this section, we treat the full $\SecondC{k}{ij}(\phi) \partial_u \phi^i \partial_v \phi^j$. For the most part, the desired estimates follow directly by combining estimates from Section \ref{section:bilinear}. The primary exception is the term
\begin{equation}\label{nonlinear:eq-main-term}
\SecondC{a}{ij}(\phi) \partial_u \big( A^{+,i}_{M,m}\phi^{+,m}_M \big)
\partial_v \big( A^{-,j}_{N,n} \phi^{-,n}_N \big),
\end{equation}
which was only briefly considered in Proposition \ref{bilinear:prop-cross}. As was mentioned previously, the reason for not considering \eqref{nonlinear:eq-main-term} before is that the case-analysis is more efficient when $\SecondC{a}{ij}(\phi)$ is already included. The main result of this section is the following theorem.

\begin{theorem}[Nonlinearity]\label{nonlinear:thm}
Let Hypothesis \ref{hypothesis:smallness} be satisfied and let $1\leq a \leq D$. Then, it holds that 
\begin{align*}
\Second_{ij}^a(\phi) \, \partial_u \phi^i \, \partial_v \phi^j 
= - \hspace{-1ex}\sum_{M\sim_\delta N} \hspace{-1ex} G_{M,N,mn}^a \partial_u \phi^{+,m}_M \partial_v \phi^{-,n}_N 
+ \sum_M F^{+,a}_{M,m} \parallsigu \partial_u \phi^{+,m}_M + \sum_N F_{N,n}^{-,a} \parallsigv  \partial_v \phi^{-,n}_N  + \ENL^{a}, 
\end{align*}
where  $G_{M,N,mn}^a$, $F^{+,a}_{M,m}$, and $F_{N,n}^{-,a}$ are as in Definition \ref{ansatz:def-F} and the remainder term $\ENL$ satisfies
\begin{equation*}
\big\| \ENL^{a} \big\|_{\Cprod{r-1}{r-1}} \lesssim \theta^2. 
\end{equation*}
\end{theorem}

\subsection{The main term}\label{section:main-term}

We start by analyzing the main term \eqref{nonlinear:eq-main-term}. By symmetry in the $u$ and $v$-variables, we can assume that $N\leq M$. We further split the analysis in the two cases $N\leq M^{1-\delta}$ and $M^{1-\delta}<N\leq M$. 

\subsubsection{The case $N\leq M^{1-\delta}$}
This case is relatively simple, as it will only contribute to the modulated version of $\phi^{+,m}_M$. 
\begin{proposition}[Main term for $N\leq M^{1-\delta}$]\label{nonlinear:prop-main-parall}
Let Hypothesis \ref{hypothesis:smallness} be satisfied, let $1\leq a \leq \dimA$, and  let $M$ and $N$ be frequency scales satisfying $N\leq M^{1-\delta}$. Then, it holds that 
\begin{equation}\label{nonlinear:eq-main-ll}
\begin{aligned}
&\Big\| \SecondC{a}{ij}(\phi) 
\partial_u \big( A_{M,m}^{+,i} \phi^{+,m}_M \big) 
\partial_v \big(A_{N,n}^{-,j}\phi^{-,n}_N \big) 
- \Big( \SecondC{a}{ij}(\phi)  A_{M,m}^{+,i} \partial_v \big(A_{N,n}^{-,j}\phi^{-,n}_N \big) \Big) \parallsigu  \partial_u \big(  \phi^{+,m}_M \big)  \Big\|_{\Cprod{r-1}{r-1}} \\
&\lesssim (MN)^{-\eta} \theta^4. 
\end{aligned}
\end{equation}
\end{proposition}

\begin{proof} We start the proof with two steps which reduce \eqref{nonlinear:eq-main-ll} to a frequency-localized version which does not contain para-product operators. Using Corollary \ref{key:cor-parall-mod}, it holds that 
\begin{align*}
&\Big\| \Big( \SecondC{a}{ij}(\phi)  \partial_v \big(A_{N,n}^{-,j}\phi^{-,n}_N \big) \Big) \parallsigu  \partial_u \big(  A_{M,m}^{+,i} \phi^{+,m}_M \big)  
- \Big( \SecondC{a}{ij}(\phi)  A_{M,m}^{+,i} \partial_v \big(A_{N,n}^{-,j}\phi^{-,n}_N \big) \Big) \parallsigu  \partial_u \big(  \phi^{+,m}_M \big)  \Big\|_{\Cprod{r-1}{r-1}} \\
&\lesssim M^{r-2s+\sigma} 
\big\| \SecondC{a}{ij}(\phi)  \partial_v \big(A_{N,n}^{-,j}\phi^{-,n}_N \big) \big\|_{\Cprod{s}{r-1}} 
\big\| A_{M,m}^{+,i} \big\|_{\Cprod{s}{s}} 
\big\| \phi^{+,m}_M \big\|_{\C_u^s} \\
&\lesssim M^{r-2s+\sigma} N^{r-s} \theta^4.
\end{align*}
Since $N\leq M$ and $3s-2r\gg \sigma$,  this is an acceptable contribution to \eqref{nonlinear:eq-main-ll}. Using Lemma \ref{prep:lem-para-localized}, it holds that 
\begin{align*}
&\Big\| \scalebox{0.95}{$\Big( \SecondC{a}{ij}(\phi) \partial_v \big(A_{N,n}^{-,j}\phi^{-,n}_N \big) \Big) \parallsigu  \partial_u \big( A_{M,m}^{+,i} \phi^{+,m}_M \big)  - 
P_{\leq M^{1-\sigma}}^u \Big( \SecondC{a}{ij}(\phi) \partial_v \big(A_{N,n}^{-,j}\phi^{-,n}_N \big) \Big) \partial_u \big( A_{M,m}^{+,i} \phi^{+,m}_M \big)$} \Big\|_{\Cprod{r-1}{r-1}} \\
&\lesssim M^{-(1-\sigma) s} 
\Big\| \SecondC{a}{ij}(\phi) \partial_v \big(A_{N,n}^{-,j}\phi^{-,n}_N \big) \Big\|_{\Cprod{s}{r-1}} \Big\| \partial_u \big( A_{M,m}^{+,i} \phi^{+,m}_M \big)\Big\|_{\Cprod{r-1}{s}} \\
&\lesssim M^{-(1-\sigma) s} N^{r-s} M^{r-s} \theta^4 \\
&\lesssim (MN)^{-\eta} \theta^4,
\end{align*}
where we used $3s-2r \gg \sigma$. 
In order to prove \eqref{nonlinear:eq-main-ll}, it therefore remains to prove that  
\begin{equation*}
\Big\|   P_{>M^{1-\sigma}}^u \Big( \SecondC{a}{ij}(\phi) \partial_v \big(A_{N,n}^{-,j}\phi^{-,n}_N \big) \Big)  \partial_u \big( A_{M,m}^{+,i} \phi^{+,m}_M \big)  \Big\|_{\Cprod{r-1}{r-1}} 
\lesssim (MN)^{-\eta} \theta^4. 
\end{equation*}
Using Lemma \ref{prep:lem-commutator}, it holds that
\begin{align}
P_{>M^{1-\sigma}}^u \Big( \SecondC{a}{ij}(\phi) \partial_v \big(A_{N,n}^{-,j}\phi^{-,n}_N \big) \Big) 
=& P_{>M^{1-\sigma}}^u \Big( \SecondC{a}{ij}(\phi)  \Big) \,  \partial_v \big(A_{N,n}^{-,j}\phi^{-,n}_N \big) \label{nonlinear:eq-main-ll-p1} \\
+& \SecondC{a}{ij}(\phi) \, \partial_v P_{>M^{1-\sigma}}^u \big(A_{N,n}^{-,j}\phi^{-,n}_N \big) \label{nonlinear:eq-main-ll-p2}\\
+& \mathcal{E}^{a}_{M,N,i}, \label{nonlinear:eq-main-ll-p3}
\end{align}
where the error term satisfies 
\begin{equation*}
\big\| \mathcal{E}^{a}_{M,N,i} \big\|_{\Cprod{1-r^\prime}{r-1}} \lesssim M^{(1-\sigma)(1-r^\prime-2s)} N^{r-s} \theta^2. 
\end{equation*}
We now estimate the contributions of \eqref{nonlinear:eq-main-ll-p1}, \eqref{nonlinear:eq-main-ll-p2}, and \eqref{nonlinear:eq-main-ll-p3} separately. \newline

\emph{Step 1: Contribution of \eqref{nonlinear:eq-main-ll-p1}.}
Using the bilinear estimate (Proposition \ref{prep:prop-bilinear}) and the product estimate (Lemma \ref{bilinear:lemma-product}), it holds that 
\begin{align*}
&\Big\| P_{>M^{1-\sigma}}^u \Big( \SecondC{a}{ij}(\phi)  \Big) \, \partial_u \big( A_{M,m}^{+,i} \phi^{+,m}_M \big)\,  \partial_v \big(A_{N,n}^{-,j}\phi^{-,n}_N \big) \Big\|_{\Cprod{r-1}{r-1}} \\
\lesssim& \Big\| P_{>M^{1-\sigma}}^u \Big( \SecondC{a}{ij}(\phi)  \Big) \, \partial_u \big( A_{M,m}^{+,i} \phi^{+,m}_M \big) \Big\|_{\Cprod{r-1}{1-r^\prime}} \big\|  \partial_v \big(A_{N,n}^{-,j}\phi^{-,n}_N \big)  \big\|_{\Cprod{1-r^\prime}{r-1}} \\
\lesssim& M^{1-s-r+\eta} N^{r-s} \theta^4. 
\end{align*}
Since $N\leq M^{1-\delta}$, we have that 
\begin{equation}\label{nonlinear:eq-main-ll-p4}
M^{1-s-r+\eta} N^{r-s}  \lesssim M^{1-s-r+\eta+(1-\delta)(r-s)} = M^{1-2s+\eta-\delta (r-s)} \lesssim M^{-\delta/6},
\end{equation}
which is acceptable. 

\emph{Step 2: Contribution of \eqref{nonlinear:eq-main-ll-p2}.} We first note that
\begin{equation*}
P_{>M^{1-\sigma}}^u \big(A_{N,n}^{-,j}\phi^{-,n}_N \big) = \big(P_{>M^{1-\sigma}}^u A_{N,n}^{-,j}\big)\, \phi^{-,n}_N . 
\end{equation*}
Then, we recall that  $N^{1-\delta} \ll N\leq M^{1-\delta} \ll M $. Using the definition of $\Mod_N^-$, it follows that 
\begin{equation*}
\big\|  A_{N,n}^{-,j} \big\|_{\Cprod{s}{s}}
+\big\| \widetilde{P}_M^u A_{N,n}^{-,j} \big\|_{\Cprod{r}{s}} \lesssim \theta. 
\end{equation*}
Using the multiplication estimate (Corollary \ref{prep:corollary-multiplication}), we have that 
\begin{align*}
  &\big\| \SecondC{a}{ij}(\phi) \, \partial_u \big( A_{M,m}^{+,i} \phi^{+,m}_M \big)\, \partial_v \big( P_{>M^{1-\sigma}}^u A_{N,n}^{-,j} \, \phi^{-,n}_N \big) \big\|_{\Cprod{r-1}{r-1}} \\
  \lesssim& \big\|  \partial_u \big( A_{M,m}^{+,i} \phi^{+,m}_M \big)\, \partial_v \big(P_{>M^{1-\sigma}}^u A_{N,n}^{-,j}\,\phi^{-,n}_N \big) \big\|_{\Cprod{r-1}{r-1}}.
\end{align*}
We now further separate the resonant and non-resonant terms in the $u$-variable. Using the bilinear estimate (Proposition \ref{prep:prop-bilinear}.\ref{prep:item-res}), the resonant term can be estimated by
\begin{align*}
&\big\|  \partial_u \big( A_{M,m}^{+,i} \phi^{+,m}_M \big) \parasimu \partial_v \big(P_{>M^{1-\sigma}}^u A_{N,n}^{-,j}\,\phi^{-,n}_N \big) \big\|_{\Cprod{r-1}{r-1}} \\
\lesssim& \big\| \partial_u \big( A_{M,m}^{+,i} \phi^{+,m}_M \big) \big\|_{\Cprod{-r^\prime}{s}}
\big\| \partial_v \big( \widetilde{P}_M^u A_{N,n}^{-,j}\,\phi^{-,n}_N \big) \big\|_{\Cprod{r}{s}} \\
\lesssim& M^{1-s-r^\prime} N^{r-s} \big\| A_M^+ \big\|_{\Cprod{s}{s}} \| \phi^+ \|_{\C_u^s} \| \widetilde{P}_M^u A_N^- \|_{\Cprod{r}{s}} \| \phi^- \|_{\C_v^s} \\
\lesssim& M^{1-s-r^\prime} N^{r-s} \theta^4.
\end{align*}
Using \eqref{nonlinear:eq-main-ll-p4}, this term is acceptable.
Using Proposition \ref{prep:prop-bilinear}.\ref{prep:item-nonres}, the non-resonant term can be estimated by 
\begin{align*}
&\big\|  \partial_u \big( A_{M,m}^{+,i} \phi^{+,m}_M \big) \paransimu \partial_v \big(P_{>M^{1-\sigma}}^u A_{N,n}^{-,j}\,\phi^{-,n}_N \big) \big\|_{\Cprod{r-1}{r-1}} \\
\lesssim& \big\|  \partial_u \big( A_{M,m}^{+,i} \phi^{+,m}_M \big) \|_{\Cprod{r-1}{s}} 
\big\| \partial_v \big(P_{>M^{1-\sigma}}^u A_{N,n}^{-,j}\,\phi^{-,n}_N \big) \big\|_{\Cprod{\eta}{r-1}} \\
\lesssim& M^{r-s} N^{r-s} \big\| A_{M,m}^{+,i} \big\|_{\Cprod{s}{s}} \big\| \phi^{+,m}_M \big\|_{\C_u^s} \big\| P_{>M^{1-\sigma}}^u A_{N,n}^{-,j} \big\|_{\Cprod{\eta}{s}} 
\big\| \phi^{-,n}_N \big\|_{\C_v^s} \\
\lesssim& M^{r-s} M^{(1-\sigma)(\eta-s)} N^{r-s} \theta^4.
\end{align*}
Since $N\leq M$ and $3s-2r\gg  \sigma$, this is acceptable.

\emph{Step 3: Contribution of \eqref{nonlinear:eq-main-ll-p3}.}
Using the bilinear estimate (Proposition \ref{prep:prop-bilinear}), we have that 
\begin{align*}
\big\| \mathcal{E}^{a}_{M,N,i} \partial_u \big( A_{M,m}^{+,i} \phi^{+,m}_M \big)  \big\|_{\Cprod{r-1}{r-1}} 
&\lesssim \big\|  \mathcal{E}^{a}_{M,N,i}\big\|_{\Cprod{1-r^\prime}{r-1}}
\big\| \partial_u \big( A_{M,m}^{+,i} \phi^{+,m}_M \big)  \big\|_{\Cprod{r-1}{1-r^\prime}} \\
&\lesssim  M^{(1-\sigma)(1-r^\prime-2s)} M^{r-s} N^{r-s} \theta^4. 
\end{align*}
Since $N\leq M$ and 
\begin{equation*}
    (1-\sigma)(1-r^\prime-2s)+r-s+r-s\approx -1/4, 
\end{equation*}
this term is acceptable. 

\end{proof}

\subsubsection{The case $M\sim_\delta N$} This case is slighly harder, since it contributes to the bilinear term containing $\phi^{+,m}_M \phi^{-,n}_N$, the modulated version of $\phi^{+,m}_M$, and the modulated version of $\phi^{-,n}_N$.

\begin{proposition}[Main term for $M\sim_\delta N$]\label{nonlinear:prop-main-sim}
Let Hypothesis \ref{hypothesis:smallness} be satisfied, let $1\leq i,j \leq \dimA$, and let $M,N$ be two frequency-scales satisfying $M\sim_\delta N$. For all $1\leq a,m,n \leq \dimA$, define 
\begin{equation}\label{nonlinear:eq-main-zeta}
    \zeta^a_{M,N,mn} := \SecondC{a}{ij}(\phi) A_{M,m}^{+,i} A_{N,n}^{-,j}. 
\end{equation}
Then, it holds that 
\begin{align}
 \SecondC{a}{ij}(\phi) \partial_u \big( A_{M,m}^{+,i} \phi^{+,m}_M \big)
 \partial_v\big( A_{N,n}^{-,j}\phi^{-,n}_N \big) 
 &=-G^a_{M,N,mn} \big(  \partial_u \phi^{+,m}_M \, \partial_v \phi_{N}^{-,n} \big) \label{nonlinear:eq-main-1}\\
&+ \big(\zeta^a_{M,N,mn} \parasimv \partial_v \phi_{N}^{-,n} \big) \parallsigu \partial_u \phi^{+,m}_M \label{nonlinear:eq-main-2}\\
&+\big(\zeta^a_{M,N,mn} \parasimu \partial_u \phi^{+,m}_M  \big) \parallsigv \partial_v \phi_{N}^{-,n} \label{nonlinear:eq-main-3}\\
&+\mathcal{E}^{a}_{M,N},\label{nonlinear:eq-main-4}
\end{align}
where $G^a_{M,N,mn}$ is as in Definition \ref{ansatz:def-F} and the error term $\mathcal{E}^{a}_{M,N}$ satisfies 
\begin{equation*}
\big\| \mathcal{E}^{a}_{M,N} \big\|_{\Cprod{r-1}{r-1}} \lesssim (MN)^{-\eta} \theta^4. 
\end{equation*}
\end{proposition}

Before starting with the proof of Proposition \ref{nonlinear:prop-main-sim}, we require the following auxiliary lemma. In Lemma \ref{nonlinear:lemma-gpf}, the function $\zeta^a_{M,N,mn}$ from \eqref{nonlinear:eq-main-zeta} is replaced by a general function $\zeta \in \Cprod{s}{s}$. In the spirit of Definition \ref{prep:def-paraproduct}, we define the para-product operators $\parallsiguv$ and $\parasimuv$ by 
\begin{align*}
f \parallsiguv g &:= 
\sum_{\substack{M_1,N_1 \colon \\ M_1 \leq N_1^{1-\sigma}}} 
 \sum_{\substack{M_2,N_2 \colon \\ M_2 \leq N_2^{1-\sigma}}}
 P_{M_1}^u P_{M_2}^v f \, P_{N_1}^u P_{N_2}^v g, \\
 f \parasimuv g &:= 
\sum_{\substack{M_1,N_1 \colon \\ M_1 \sim N_1}} 
 \sum_{\substack{M_2,N_2 \colon\\ M_2 \sim N_2}}
 P_{M_1}^u P_{M_2}^v f \, P_{N_1}^u P_{N_2}^v g. 
\end{align*}

%GPF stands for general pre-factor
\begin{lemma}\label{nonlinear:lemma-gpf}
Let Hypothesis \ref{hypothesis:smallness} be satisfied. Let $1\leq m,n\leq D$, let $M,N$ be frequency-scales satisfying $M\sim_\delta N$, and let $\zeta \in \Cprod{s}{s}$. Then, we can write 
\begin{align}
\zeta \partial_u \phi^{+,m}_M \, \partial_v \phi_{N}^{-,n} 
&= \zeta \parallsiguv \big(  \partial_u \phi^{+,m}_M \, \partial_v \phi_{N}^{-,n} \big) \label{nonlinear:eq-gpf-1}\\
&+ \big( \zeta \parasimv \partial_v \phi_{N}^{-,n} \big) \parallsigu \partial_u \phi^{+,m}_M \label{nonlinear:eq-gpf-2}\\
&+\big( \zeta \parasimu \partial_u \phi^{+,m}_M  \big) \parallsigv \partial_v \phi_{N}^{-,n} \label{nonlinear:eq-gpf-3}\\
&+\zeta \parasimuv \big(  \partial_u \phi^{+,m}_M \, \partial_v \phi_{N}^{-,n}  \big) \label{nonlinear:eq-gpf-4}\\
&+\mathcal{E}^{m,n}_{M,N},\label{nonlinear:eq-gpf-5}
\end{align}
where the error term $\mathcal{E}^{m,n}_{M,N}$ satisfies 
\begin{equation*}
\big\| \mathcal{E}^{m,n}_{M,N} \big\|_{\Cprod{r-1}{r-1}} \lesssim (MN)^{-\eta} \theta^2 \| \zeta \|_{\Cprod{s}{s}}. 
\end{equation*}
\end{lemma}

\begin{proof}[Proof of Lemma \ref{nonlinear:lemma-gpf}:]
Using a Littlewood-Paley decomposition, we write 
\begin{equation*}
\zeta = \sum_{K,L} P_K^u P_L^v \zeta. 
\end{equation*}
We now split all possible frequency-scales $K$ and $L$ into seven regions as follows: 
\begin{align}
2^{\mathbb{N}_0} \times 2^{\mathbb{N}_0} 
=& \big\{ (K,L) \colon K \leq M^{1-\sigma},\, L \leq N^{1-\sigma} \big\} \label{nonlinear:eq-gpf-region-1} \\
&\dot\medcup \, \big\{ (K,L) \colon K \leq M^{1-\sigma},\, L \sim N \big\} \label{nonlinear:eq-gpf-region-2} \\
&\dot\medcup \, \big\{ (K,L) \colon K \sim M,\, L \leq N^{1-\sigma} \big\} \label{nonlinear:eq-gpf-region-3} \\
&\dot\medcup \, \big\{ (K,L) \colon K \sim M,\, L \sim N \big\} \label{nonlinear:eq-gpf-region-4} \\
&\dot\medcup \, \big\{ (K,L) \colon K \sim M,\, N^{1-\sigma} < L \not\sim N \big\} \label{nonlinear:eq-gpf-region-5} \\
&\dot\medcup \, \big\{ (K,L) \colon  M^{1-\sigma} < K \not \sim M,\, L \sim N \big\} \label{nonlinear:eq-gpf-region-6} \\
&\dot\medcup \, \big\{ (K,L) \colon K > M^{1-\sigma} \text{ or } L>N^{1-\sigma},\, 
K\not \sim M, \, L \not \sim N\big\} \label{nonlinear:eq-gpf-region-7}. 
\end{align}
The contributions of the regions 
\eqref{nonlinear:eq-gpf-region-1}-\eqref{nonlinear:eq-gpf-region-4} will lead to the main terms 
\eqref{nonlinear:eq-gpf-1}-\eqref{nonlinear:eq-gpf-4}. The remaining contributions of 
\eqref{nonlinear:eq-gpf-region-5}, 
\eqref{nonlinear:eq-gpf-region-6}, 
and \eqref{nonlinear:eq-gpf-region-7} will contribute exclusively towards the error term  $\mathcal{E}^{m,n}_{M,N}$. We now treat the individual contributions separately. \newline

\emph{Contribution of \eqref{nonlinear:eq-gpf-region-1}:} 
The contribution of this frequency region is given by 
\begin{equation*}
\sum_{K\leq M^{1-\sigma}} \sum_{L \leq N^{1-\sigma}} P_K^u P_L^v \zeta \, \partial_u \phi^{+,m}_M \partial_v \phi^{-,n}_N = \big( P_{\leq M^{1-\sigma}}^u P_{\leq N^{1-\sigma}}^v \zeta\big) \, \partial_u \phi^{+,m}_M \partial_v \phi^{-,n}_N. 
\end{equation*}
Using a variant of Lemma \ref{prep:lem-para-localized}, it holds that
\begin{align*}
&\big\| \big( P_{\leq M^{1-\sigma}}^u P_{\leq N^{1-\sigma}}^v \zeta\big) \, \partial_u \phi^{+,m}_M \partial_v \phi^{-,n}_N- \zeta \parallsiguv \big(  \partial_u \phi^{+,m}_M \, \partial_v \phi_{N}^{-,n} \big) \big\|_{\Cprod{r-1}{r-1}} \\
\lesssim& \min(M,N)^{-(1-\sigma)s} M^{r-s} N^{r-s} \| \zeta \|_{\Cprod{s}{s}} \theta^2. 
\end{align*}
Since $M\sim_\sigma N$ and $3s-2r\gg \sigma$, this is acceptable. \\

\emph{Contributions of \eqref{nonlinear:eq-gpf-region-2}, \eqref{nonlinear:eq-gpf-region-3}, and \eqref{nonlinear:eq-gpf-region-4}:} The argument is similar as for the contribution of \eqref{nonlinear:eq-gpf-region-1} and only relies on Lemma \ref{prep:lem-para-localized}.\\

\emph{Contribution of \eqref{nonlinear:eq-gpf-region-5}:} Using Proposition \ref{prep:prop-bilinear}.\ref{prep:item-nonres}, we have that 
\begin{align*}
\sum_{\substack{K,L \colon \\ K \sim M, \, \\N^{1-\sigma} < L \not \sim N}} 
\big\| P_K^u P_L^v \zeta \, \partial_u \phi^{+,m}_M \, \partial_v \phi_{N}^{-,n} \big\|_{\Cprod{r-1}{r-1}} 
\lesssim& \sum_{\substack{K,L \colon \\ K \sim M, \, \\N^{1-\sigma} < L \not \sim N}} 
\big\| P_K^u P_L^v \zeta \big\|_{\Cprod{\eta}{\eta}}
\big\| \partial_u \phi^{+,m}_M \, \partial_v \phi_{N}^{-,n} \big\|_{\Cprod{\eta}{r-1}} \\
\lesssim& \sum_{\substack{K,L \colon \\ K \sim M, \, \\N^{1-\sigma} < L \not \sim N}} 
K^{\eta-s} L^{\eta-s} M^{1+\eta-s} N^{r-s} \theta^2 \| \zeta \|_{\Cprod{s}{s}} \\
\lesssim& M^{1+2\eta-2s} N^{r-s-(1-\sigma)(\eta-s)} \theta^2 \| \zeta \|_{\Cprod{s}{s}}. 
\end{align*}
Since $M\sim_\delta N$ and 
\begin{equation*}
    1+2\eta-2s + r-s-(1-\sigma)(\eta-s) \approx -1/4,
\end{equation*}
this is acceptable. \newline

\emph{Contribution of \eqref{nonlinear:eq-gpf-region-6}:} Using the symmetry in the $u$ and $v$-variables, this follows from the same argument as for \eqref{nonlinear:eq-gpf-region-5}. \newline

\emph{Contribution of \eqref{nonlinear:eq-gpf-region-7}:}
By symmetry, it suffices to estimate the contribution for $K>M^{1-\sigma}$. Using Proposition \ref{prep:prop-bilinear}.\ref{prep:item-nonres}, it holds 
that 
\begin{align*}
\sum_{\substack{K,L \colon \\ M^{1-\sigma} < K \not \sim M, \, \\L \not \sim N}} 
\hspace{-3pt} \big\| P_K^u P_L^v \zeta \, \partial_u \phi^{+,m}_M \, \partial_v \phi_{N}^{-,n} \big\|_{\Cprod{r-1}{r-1}} 
\lesssim& \hspace{-3pt} \sum_{\substack{K,L \colon \\ M^{1-\sigma} < K \not \sim M, \, \\L \not \sim N}}  \hspace{-6pt}
\big\| P_K^u P_L^v \zeta \big\|_{\Cprod{\eta}{1-r^\prime}}
\big\| \partial_u \phi^{+,m}_M \, \partial_v \phi_{N}^{-,n} \big\|_{\Cprod{r-1}{r-1}} \\
\lesssim& \sum_{\substack{K,L \colon \\ M^{1-\sigma} < K \not \sim M, \, \\L \not \sim N}} 
K^{\eta-s} L^{1-r^\prime-s} M^{r-s} N^{r-s} \theta^2 \| \zeta \|_{\Cprod{s}{s}} \\
\lesssim& M^{(1-\sigma) (\eta-s)+r-s} N^{r-s} \theta^2 \| \zeta \|_{\Cprod{s}{s}}. 
\end{align*}
Since $M\sim_\delta N$ and $3s-2r \gg \sigma $, this term is acceptable.
\end{proof}

Equipped with Lemma \ref{nonlinear:lemma-gpf}, we can now prove Proposition \ref{nonlinear:prop-main-sim}.

\begin{proof}[Proof of Proposition \ref{nonlinear:prop-main-sim}:]
We separate the proof into several steps.\newline

\emph{Step 1: Removing terms with favorable derivatives.}
Using the multiplication estimate (Corollary \ref{prep:corollary-multiplication}) and Proposition \ref{bilinear:prop-cross}, it holds that 
\begin{align*}
&\big\| \SecondC{a}{ij}(\phi) \partial_u \big( A_{M,m}^{+,i} \phi^{+,m}_M \big)
 \partial_v\big( A_{N,n}^{-,j}\phi^{-,n}_N \big) - \SecondC{a}{ij}(\phi)  A_{M,m}^{+,i}  A_{N,n}^{-,j} \partial_u  \phi^{+,m}_M \partial_v \phi^{-,n}_N \|_{\Cprod{r-1}{r-1}} \\
 \lesssim& \big\|  \SecondC{a}{ij}(\phi) \big\|_{\Cprod{s}{s}} 
 \big\| \partial_u \big( A_{M,m}^{+,i} \phi^{+,m}_M \big) \partial_v\big( A_{N,n}^{-,j}\phi^{-,n}_N \big)  
 - A_{M,m}^{+,i}  A_{N,n}^{-,j} \partial_u  \phi^{+,m}_M \partial_v \phi^{-,n}_N \big\|_{\Cprod{r-1}{r-1}}\\
 \lesssim& (MN)^{-\eta} \theta^4. 
\end{align*}
Thus, we have reduced the main term to 
\begin{equation*}
    \SecondC{a}{ij}(\phi)  A_{M,m}^{+,i}  A_{N,n}^{-,j} \partial_u  \phi^{+,m}_M \partial_v \phi^{-,n}_N 
    = \zeta^a_{M,N,mn} \partial_u\phi^{+,m}_M \partial_v \phi^{-,n}_N, 
\end{equation*}
where $\zeta^a_{M,N,mn}$ is as in \eqref{nonlinear:eq-main-zeta}. \newline

\emph{Step 2: Applying Lemma \ref{nonlinear:lemma-gpf}.} Using Lemma \ref{nonlinear:lemma-gpf}, it holds that 
\begin{align}
     \SecondC{a}{ij}(\phi) \partial_u \big( A_{M,m}^{+,i} \phi^{+,m}_M \big)
 \partial_v\big( A_{N,n}^{-,j}\phi^{-,n}_N \big) 
 &= \zeta^a_{M,N,mn} \parallsiguv \big(  \partial_u \phi^{+,m}_M \, \partial_v \phi_{N}^{-,n} \big) \label{nonlinear:eq-main-q1}\\
&+ \big(\zeta^a_{M,N,mn} \parasimv \partial_v \phi_{N}^{-,n} \big) \parallsigu \partial_u \phi^{+,m}_M \label{nonlinear:eq-main-q2}\\
&+\big(\zeta^a_{M,N,mn} \parasimu \partial_u \phi^{+,m}_M  \big) \parallsigv \partial_v \phi_{N}^{-,n} \label{nonlinear:eq-main-q3}\\
&+ \zeta^a_{M,N,mn} \parasimuv \big(  \partial_u \phi^{+,m}_M \, \partial_v \phi_{N}^{-,n} \big) \label{nonlinear:eq-main-q4}\\
&+\widetilde{\mathcal{E}}^{a}_{M,N}, \notag
\end{align}
where the error term satisfies
\begin{equation*}
\big\| \widetilde{\mathcal{E}}^{a}_{M,N} \big\|_{\Cprod{r-1}{r-1}} \lesssim (MN)^{-\eta} \theta^2 \| \zeta^a_{M,N,mn}\|_{\Cprod{s}{s}} \lesssim (MN)^{-\eta} \theta^4. 
\end{equation*}
Using Lemma \ref{prep:lem-para-localized} and Definition \ref{ansatz:def-F}, the first summand \eqref{nonlinear:eq-main-q1} leads to \eqref{nonlinear:eq-main-1} and an acceptable error term. The second and third summands \eqref{nonlinear:eq-main-q2} and \eqref{nonlinear:eq-main-q3} coincide with \eqref{nonlinear:eq-main-2} and \eqref{nonlinear:eq-main-3}, respectively. As a result, it only remains to prove that 
\begin{equation}\label{nonlinear:eq-main-p2}
\big\|  \zeta^a_{M,N,mn} \parasimuv \big(  \partial_u \phi^{+,m}_M \, \partial_v \phi_{N}^{-,n} \big) \big\|_{\Cprod{r-1}{r-1}} \lesssim (MN)^{-\eta} \theta^4. 
\end{equation}
The double-resonance estimate \eqref{nonlinear:eq-main-p2} requires much more detailed information than only $ \zeta^a_{M,N,mn}\in \Cprod{s}{s}$, which explains why this term was not estimated in Lemma \ref{nonlinear:lemma-gpf}. \newline

\emph{Step 3: The double-resonance estimate \eqref{nonlinear:eq-main-p2}.} Using a variant of Lemma \ref{prep:lem-para-localized}, it suffices to estimate
\begin{align*}
\sum_{\substack{K,L \colon \\ K\sim M, L \sim N}} P_K^u P_L^v \zeta^a_{M,N,mn} \, 
 \partial_u \phi^{+,m}_M \, \partial_v \phi_{N}^{-,n}. 
 \end{align*}
 In the following, we fix frequency scales $K\sim M$ and $L \sim N$. Using \eqref{prep:eq-commutator-4} in Lemma \ref{prep:lem-commutator}, we have that 
 \begin{align}
P_K^u P_L^v \zeta^a_{M,N,mn}
&= P_K^u P_L^v \big( \SecondC{a}{ij} (\phi) \big) \, A_{M,m}^{+,i} A_{N,n}^{-,j} \label{nonlinear:eq-main-p3} \\
&+ P_K^u \big( \SecondC{a}{ij} (\phi) \big)  \,  P_L^v \big( A_{M,m}^{+,i} A_{N,n}^{-,j}  \big) \label{nonlinear:eq-main-p4} \\
&+ P_L^v \big( \SecondC{a}{ij} (\phi) \big) \,  P_K^u \big( A_{M,m}^{+,i} A_{N,n}^{-,j}  \big) \label{nonlinear:eq-main-p5} \\
&+  \SecondC{a}{ij} (\phi)  \, P_K^u P_L^v \big( A_{M,m}^{+,i} A_{N,n}^{-,j}  \big) \label{nonlinear:eq-main-p6} \\
&+ \widetilde{\zeta}^{a}_{M,N,mn} \label{nonlinear:eq-main-p7},
 \end{align}
 where $\widetilde{\zeta}^{a}_{M,N,mn}$ satisfies
 \begin{equation*}
   \big\|  \widetilde{\zeta}^{a}_{M,N,mn} \big\|_{\Cprod{s}{s}} \lesssim \min(K,L)^{-s}
  \| \SecondC{a}{ij} (\phi) \|_{\Cprod{s}{s}} \| A_{M,m}^{+,i} \|_{\Cprod{s}{s}} \| A_{N,n}^{-,j} \|_{\Cprod{s}{s}} \\
  \lesssim \min(M,N)^{-s} \theta^{2}. 
 \end{equation*}
 The contribution of the error term $\widetilde{\zeta}^{a}_{M,N,mn}$ can be estimated by 
 \begin{align*}
     \big\|  \widetilde{\zeta}^{a}_{M,N,mn}  \partial_u \phi^{+,m}_M \, \partial_v \phi_{N}^{-,n} \big\|_{\Cprod{r-1}{r-1}} 
     &\lesssim \big\| \widetilde{\zeta}^{a}_{M,N,mn}   \big\|_{\Cprod{s}{s}}  \big\| \partial_u \phi^{+,m}_M \, \partial_v \phi_{N}^{-,n} \big\|_{\Cprod{r-1}{r-1}} \\
     &\lesssim \min(M,N)^{-s} M^{r-s} N^{r-s} \theta^4.
 \end{align*}
 Since $M\sim_\delta N$ and $3s-2r\gg \delta$, this contribution is acceptable. Thus, it remains to treat the contributions of \eqref{nonlinear:eq-main-p3}, \eqref{nonlinear:eq-main-p4}, \eqref{nonlinear:eq-main-p5}, and \eqref{nonlinear:eq-main-p6}, which will be handled separately. \newline
 
 \emph{Step 3.(a): Contribution of \eqref{nonlinear:eq-main-p3}.}
Due to the multiplication estimate (Corollary \ref{prep:corollary-multiplication}), it suffices to prove that 
\begin{equation}\label{nonlinear:eq-main-p10}
\big\| P_K^u P_L^v \big( \SecondC{a}{ij}(\phi) \big) \, \partial_u \phi^{+,m}_M \, \partial_v \phi^{-,n}_N \big\|_{\Cprod{r-1}{r-1}} \lesssim (MN)^{-\eta} \theta^2.
\end{equation}
Using the symmetry in the $u$ and $v$-variables, we can assume that $M\geq N$. We now break the symmetry of the term in \eqref{nonlinear:eq-main-p10} by decomposing
\begin{equation}\label{nonlinear:eq-main-p11}
P_K^u P_L^v \big( \SecondC{a}{ij}(\phi) \big) 
= P_L^v \big( \partial_k \SecondC{a}{ij}(\phi) P_K^u \phi^k \big) +  P_L^v \big( P_K^u \SecondC{a}{ij}(\phi) - \partial_k \SecondC{a}{ij}(\phi) P_K^u \phi^k \big).
\end{equation}
Using Bony's paralinearization (Lemma \ref{prep:lemma-bony}), the second summand in \eqref{nonlinear:eq-main-p11} is controlled by
\begin{align*}
 &\big\|  P_L^v \big( P_K^u \SecondC{a}{ij}(\phi) - \partial_k \SecondC{a}{ij}(\phi) P_K^u \phi^k \big) \big\|_{\Cprod{s}{1-r^\prime}} \\
 \lesssim& L^{1-r^\prime-s}
 \big\|  P_L^v \big( P_K^u \SecondC{a}{ij}(\phi) - \partial_k \SecondC{a}{ij}(\phi) P_K^u \phi^k \big) \big\|_{\Cprod{s}{s}}\\
 \lesssim& L^{1-r^\prime-s} K^{-s}.
\end{align*}
By using the bilinear estimate (Proposition \ref{prep:prop-bilinear}) and Proposition \ref{bilinear:prop-cross}, this yields an acceptable contribution to \eqref{nonlinear:eq-main-p10}. Thus, we have further reduced \eqref{nonlinear:eq-main-p10} to the estimate
\begin{equation}\label{nonlinear:eq-main-p12}
\big\| P_L^v \big( \partial_k\SecondC{a}{ij}(\phi)  P_K^u \phi^k \big) \, \partial_u \phi^{+,m}_M \, \partial_v \phi^{-,n}_N \big\|_{\Cprod{r-1}{r-1}} \lesssim (MN)^{-\eta} \theta^2. 
\end{equation}
Using the bilinear estimate (Proposition \ref{prep:prop-bilinear}), it holds that 
\begin{align*}
&\big\| P_L^v \big( \partial_k\SecondC{a}{ij}(\phi)  P_K^u \phi^k \big) \, \partial_u \phi^{+,m}_M \, \partial_v \phi^{-,n}_N \big\|_{\Cprod{r-1}{r-1}} \\
\lesssim& \big\| P_L^v \big( \partial_k\SecondC{a}{ij}(\phi)  P_K^u \phi^k \big) \, \partial_u \phi^{+,m}_M \big\|_{\Cprod{r-1}{1-r^\prime}} 
\big\| \partial_v \phi^{-,n}_N \big\|_{\Cprod{s}{r-1}} \\
\lesssim& N^{r-s}   \big\| P_L^v \big( \partial_k\SecondC{a}{ij}(\phi)  P_K^u \phi^k \big) \, \partial_u \phi^{+,m}_M  \big\|_{\Cprod{r-1}{1-r^\prime}} \, \theta.
\end{align*}
Using $K\sim M$, $L\sim N$, and Corollary \ref{bilinear:cor-product}, we have that 
\begin{align*}
&N^{r-s}  \big\| P_L^v \big( \partial_k\SecondC{a}{ij}(\phi)  P_K^u \phi^k \big) \, \partial_u \phi^{+,m}_M  \big\|_{\Cprod{r-1}{1-r^\prime}} \, \theta  \\
\lesssim& N^{r-s} \big( M^{-(r-s)-\delta/8} + N^{1-s-r+2\eta} M^{1-s-r}) \theta^3 \\
\lesssim&  \Big( \big(NM^{-1}\big)^{r-s} M^{-\delta/8} + N^{1-2s+2\eta} M^{1-s-r} \Big) \theta^3. 
\end{align*}
Using our assumption $M\geq N$, this term is acceptable.

\emph{Step 3.(b): Contribution of \eqref{nonlinear:eq-main-p4}.} Using the multiplication estimate (Corollary \ref{prep:corollary-multiplication}) and the bilinear estimate (Proposition \ref{prep:prop-bilinear}), we have that 
\begin{equation}\label{nonlinear:eq-main-p8}
\begin{aligned}
&\big\| P_K^u \big( \SecondC{a}{ij} (\phi) \big)  \,  P_L^v \big( A_{M,m}^{+,i} A_{N,n}^{-,j}  \big) 
\, \partial_u \phi^{+,m}_M \, \partial_v \phi_{N}^{-,n}  \big\|_{\Cprod{r-1}{r-1}} \\
&\lesssim \big\| P_L^v \big( A_{M,m}^{+,i} A_{N,n}^{-,j}  \big) \big\|_{\Cprod{1-r^\prime}{1-r^\prime}} 
\big\| P_K^u \big( \SecondC{a}{ij} (\phi) \big) \partial_u \phi^{+,m}_M \, \partial_v \phi_{N}^{-,n}  \big\|_{\Cprod{r-1}{r-1}} \\
&\lesssim L^{1-r^\prime-s} \big\| A_{M,m}^{+,i}  \big\|_{\Cprod{s}{s}}  \big\| A_{N,n}^{-,j}\big\|_{\Cprod{s}{s}} 
\big\| P_K^u \big( \SecondC{a}{ij} (\phi) \big) \partial_u \phi^{+,m}_M \big\|_{\Cprod{r-1}{1-r^\prime}} \| \phi^{-,n}_N \|_{\C_v^{r-1}} \\
&\lesssim L^{1-r^\prime-s} N^{r-s}  \big\| P_K^u \big( \SecondC{a}{ij} (\phi) \big) \partial_u \phi^{+,m}_M \big\|_{\Cprod{r-1}{1-r^\prime}} \theta^3. 
\end{aligned}
\end{equation}
Using Lemma \ref{bilinear:lemma-product} with $H=\operatorname{Id}$ and Bony's paralinearization,  it holds that 
\begin{equation}\label{nonlinear:eq-main-p9}
\big\| P_K^u \big( \SecondC{a}{ij} (\phi) \big) \partial_u \phi^{+,m}_M \big\|_{\Cprod{r-1}{1-r^\prime}}  \lesssim M^{1-s-r} \theta. 
\end{equation}
By combining \eqref{nonlinear:eq-main-p8} and \eqref{nonlinear:eq-main-p9}, it follows that
\begin{equation*}
    \big\| P_K^u \big( \SecondC{a}{ij} (\phi) \big)  \,  P_L^v \big( A_{M,m}^{+,i} A_{N,n}^{-,j}  \big) 
\, \partial_u \phi^{+,m}_M \, \partial_v \phi_{N}^{-,n}  \big\|_{\Cprod{r-1}{r-1}}  
\lesssim L^{1-r^\prime-s} N^{r-s} M^{1-s-r} \theta^4.
\end{equation*}
Since $L\sim N$, $M\sim_\delta N$, and 
\begin{equation*}
1-r^\prime-s+r-s+1-s-r= 2-3s-r+\eta \approx -1/4,
\end{equation*}
this is acceptable.\newline

\emph{Step 3.(c): Contribution of \eqref{nonlinear:eq-main-p5}.}
By symmetry in the $u$ and $v$-variable, this follows from the same argument as in Step 3.(b). \newline

\emph{Step 3.(d): Contribution of \eqref{nonlinear:eq-main-p6}.} By symmetry, we may assume that $M\geq N$. Since $A^{+,i}_{M,m}$ only has $u$-frequencies bounded by $M^{1-\sigma}$, we can replace $A_{N,n}^{-,j}$ by $P_{\gg N^{1-\delta}}^u A_{N,n}^{-,j}$. It follows that 
\begin{align*}
&\big\| \SecondC{a}{ij} (\phi)  \, P_K^u P_L^v \big( A_{M,m}^{+,i} \, P_{\gg N^{1-\delta}}^u A_{N,n}^{-,j}  \big) 
\, \partial_u \phi^{+,m}_M \, \partial_v \phi_{N}^{-,n} \big\|_{\Cprod{r-1}{r-1}} \\
\lesssim& \big\| \SecondC{a}{ij} (\phi)  \big\|_{L^\infty_{u,v}} 
\big\|  P_K^u P_L^v \big( A_{M,m}^{+,i} \, P_{\gg N^{1-\delta}}^u A_{N,n}^{-,j} \big) \big\|_{L^\infty_{u,v}} 
\big\| \partial_u \phi^{+,m}_M \, \partial_v \phi_{N}^{-,n}  \big\|_{L^\infty_{u,v}} \\
\lesssim& K^{-s} L^{-s} M^{1-s} N^{1-s}  N^{-(1-\delta) (r-s)} 
\| A_{M}^+ \|_{\Cprod{s}{s}} \| P_{\gg N^{1-\delta}}^u A_{N}^- \|_{\Cprod{r}{s}}
\| \phi^+ \|_{\C_u^s} \| \phi^- \|_{\C_v^s} \\
\lesssim& M^{1-2s} N^{1-2s}  N^{-(1-\delta) (r-s)} \theta^4. 
\end{align*}
Since $M\sim_\delta N$ and $2(1-2s)-(1-\delta)(r-s)\approx -1/4$, this term is acceptable.

\end{proof}
\subsection{The para-controlled and resonance terms}

We now control the terms $\SecondC{k}{ij}(\phi) \partial_u \zeta_1^i \partial_v \zeta_2^j$, where the tuple $(\zeta_1,\zeta_2)$ falls into the \scalebox{0.9}{$\paralesssimsigu$} and $\scalebox{1.25}{$\lnot$}\scalebox{0.9}{$\paradownv$}$-cases in Figure \ref{figure:cases}. 

\begin{proposition}[The para-controlled term]\label{nonlinear:prop-para}
Let Hypothesis \ref{hypothesis:smallness} be satisfied, let $1\leq k\leq \dimA$, and let $M$ be a frequency scale. Furthermore, let $\zeta$ be of type $\tp$, $\tpm$, or $\ts$.  Then, it holds that 
\begin{equation}\label{nonlinear:eq-para-1}
\begin{aligned}
\Big\| \SecondC{k}{ij}(\phi) \partial_u \big(A^{+,i}_{M,m}  \phi^{+,m}_M \big) \partial_v \zeta^j - \big( \SecondC{k}{ij}(\phi) A^{+,i}_{M,m} \partial_v \zeta^j \big) \paralesssimsigu  \partial_u \big(  \phi^{+,m}_M\big) \Big\|_{\Cprod{r-1}{r-1}} 
\lesssim M^{-\eta} \Gain(\zeta) \theta^3. 
\end{aligned}
\end{equation}
We also have the frequency-localized variant 
\begin{equation}\label{nonlinear:eq-para-2}
\Big\| P_K^u \big( \SecondC{k}{ij}(\phi)\partial_v \zeta^j\big)  
\partial_u \big( A^{+,i}_{M,m} \phi^{+,m}_M\big) \Big\|_{\Cprod{r-1}{r-1}} \\
\lesssim K^{-\eta} \Gain(\zeta) \theta^3. 
\end{equation}
for all frequency-scales $K$ satisfying $K\gtrsim M^{1-\sigma}$. 
\end{proposition}

\begin{proof}
Due to Lemma \ref{prep:lem-para-localized}, Corollary \ref{key:cor-parall-mod}, and elementary estimates, it suffices to prove \eqref{nonlinear:eq-para-2}. To this end, we split we split $P_{K}^u \big(\SecondC{k}{ij}(\phi) \partial_v \zeta^j\big)$ into a low$\times$high, a high$\times$low, and an error term. To be precise, we write 
\begin{equation}\label{nonlinear:eq-para-p1}
    P_{K}^u \big( \SecondC{k}{ij}(\phi) \partial_v \zeta^j \big) 
    =: P_{\ll K}^u \big( \SecondC{k}{ij}(\phi) \big) \partial_v P_K^u \zeta^j 
    + P_{K}^u \big( \SecondC{k}{ij}(\phi) \big) \partial_v P_{\ll K}^u \zeta^j 
    + \mathcal{E}_{K,i}^{k},
\end{equation}
which serves as a definition of the error term $\mathcal{E}^k_{K,i}$. We now split the argument into three steps corresponding to the three summands in \eqref{nonlinear:eq-para-p1}. \newline

\emph{Step 1: The contribution of the low$\times$high-term.} Using the multiplication estimate (Corollary \ref{prep:corollary-multiplication}) and the composition estimate (Lemma \ref{prep:lemma-bony}), we have that 
\begin{align*}
&\big\| P_{\ll K}^u \big( \SecondC{k}{ij}(\phi) \big) 
\partial_u \big( A^{+,i}_{M,m} \phi^{+,m}_M\big) \partial_v P_K^u \zeta^j  \big\|_{\Cprod{r-1}{r-1}} \\
&\lesssim \big\| \SecondC{k}{ij}(\phi) \big\|_{\Cprod{s}{s}} 
\big\|  \partial_u \big( A^{+,i}_{M,m} \phi^{+,m}_M\big) \partial_v P_K^u \zeta^j \big\|_{\Cprod{r-1}{r-1}} \\
&\lesssim (1+\| \phi\|_{\Cprod{s}{s}})^{10}
\big\|  \partial_u \big( A^{+,i}_{M,m} \phi^{+,m}_M\big) \partial_v P_K^u \zeta^j \big\|_{\Cprod{r-1}{r-1}} \\
&\lesssim 
\big\|  \partial_u \big( A^{+,i}_{M,m} \phi^{+,m}_M\big) \partial_v P_K^u \zeta^j \big\|_{\Cprod{r-1}{r-1}}. 
\end{align*}
It remains to prove that 
\begin{equation*}
    \big\|  \partial_u \big( A^{+,i}_{M,m} \phi^{+,m}_M\big) \partial_v P_K^u \zeta^j \big\|_{\Cprod{r-1}{r-1}} \\
    \lesssim K^{-\eta} \Gain(\zeta) \theta^3. 
\end{equation*}
If $M^{1-\sigma} \lesssim K \lesssim M$, this directly follows from Proposition \ref{bilinear:prop-paraless}. If $K\gg M$, this follows easily from Proposition \ref{prep:prop-bilinear}.\ref{prep:item-nonres}. Indeed, 
\begin{align}
    \big\|  \partial_u \big( A^{+,i}_{M,m} \phi^{+,m}_M\big) \partial_v P_K^u \zeta^j \big\|_{\Cprod{r-1}{r-1}} 
    &\lesssim \big\|  \partial_u \big( A^{+,i}_{M,m} \phi^{+,m}_M\big) \big\|_{\Cprod{\eta}{s}} \big\| \partial_v P_K^u \zeta^j \big\|_{\Cprod{r-1}{r-1}} \notag \\
    &\lesssim M^{1-s+\eta} K^{r-1-(1-r^\prime)} \| A_M^+ \|_{\Cprod{s}{s}} \|\phi^+ \|_{\C_u^s} \| P_K^u \zeta \|_{\Cprod{1-r^\prime}{r-1}} \notag \\
    &\lesssim K^{2r-1-s+2\eta} \theta^2 \| P_K^u \zeta\|_{\Cprod{1-r^\prime}{r-1}}. \label{nonlinear:eq-para-p9}
\end{align}
By inserting the types $\tp$, $\tpm$, or $\ts$, it is easy to see that
\begin{equation*}
     \| P_K^u \zeta\|_{\Cprod{1-r^\prime}{r-1}} \lesssim K^{2\delta} \Gain(\zeta) \theta.
\end{equation*}
After inserting this back into \eqref{nonlinear:eq-para-p9} and using $3/4-r\gg \delta$, this completes the first step. \newline

\emph{Step 2: The contribution of the high$\times$low-term.} Using the standard bilinear estimate (Proposition \ref{prep:prop-bilinear}), we have that 
\begin{align}
   &\big\|  P_{K}^u \big( \SecondC{k}{ij}(\phi) \big) \partial_u \big(A_{M,m}^{+,i} \phi^{+,m}_M \big) \partial_v P_{\ll K}^u \zeta^j  \big\|_{\Cprod{r-1}{r-1}} \notag \\
   &\lesssim  \big\| P_{K}^u \big( \SecondC{k}{ij}(\phi) \big) \partial_u \big(A_{M,m}^{+,i} \phi^{+,m}_M \big) \big\|_{\Cprod{r-1}{1-r^\prime}} 
   \big \| \partial_v P_{\ll K}^u \zeta^j \big\|_{\Cprod{1-r^\prime}{r-1}}.\label{nonlinear:eq-para-p2}
\end{align}
Similar as in Step 1, it is easy to see for $\zeta$ of type $\tp$, $\tpm$, or $\ts$ that
\begin{equation}\label{nonlinear:eq-para-p3}
    \big \| \partial_v P_{\ll K}^u \zeta^j \big\|_{\Cprod{1-r^\prime}{r-1}} \lesssim K^{2\delta} \Gain(\zeta) \theta. 
\end{equation}
The $2\delta$-loss in \eqref{nonlinear:eq-para-p3} will easily be absorbed through our estimate of the first factor in \eqref{nonlinear:eq-para-p3}, which we now present. Inspired by Bony's paralinearization formula, we decompose
\begin{equation}\label{nonlinear:eq-para-p4}
     P_{K}^u \big( \SecondC{k}{ij}(\phi) \big) 
     = \Big( \big(\partial_\ell \SecondC{k}{ij} \big)(\phi) P_K^u \phi^\ell \Big) +  
     \Big( P_{K}^u \big( \SecondC{k}{ij}(\phi) \big) - \big(\partial_\ell \SecondC{k}{ij} \big)(\phi) P_K^u \phi^\ell \Big). 
\end{equation}
The contribution of the first term in \eqref{nonlinear:eq-para-p4} is estimated using the multiplication estimate (Corollary \ref{prep:corollary-multiplication}) and the product estimate (Lemma \ref{bilinear:lemma-product}), which yield 
\begin{equation}\label{nonlinear:eq-para-p5}
\begin{aligned}
    &\big\| \big(\partial_\ell \SecondC{k}{ij} \big)(\phi) P_K^u \phi^\ell \partial_u \big(A_{M,m}^{+,i} \phi^{+,m}_M \big) \big\|_{\Cprod{r-1}{1-r^\prime}} \\
    &\lesssim \big\| \big(\partial_\ell \SecondC{k}{ij} \big)(\phi) \big\|_{\Cprod{s}{s}} 
    \big\| P_K^u \phi^\ell \partial_u \big(A_{M,m}^{+,i} \phi^{+,m}_M \big) \big\|_{\Cprod{r-1}{1-r^\prime}} \\
    &\lesssim \big( 1\{ K \lesssim M \} M^{1-s-r} + 1\{ K \gg M \} K^{2(r-1+\eta)} \big) \theta^3. 
\end{aligned}
\end{equation}
Since $1-s-r \approx -1/4$ and $2(r-1+\eta)\approx -1/2$, the gain in \eqref{nonlinear:eq-para-p5} easily beats the $2\delta$-loss in \eqref{nonlinear:eq-para-p3}. The contribution of the second term in \eqref{nonlinear:eq-para-p4} is estimated using Bony's paralinearization estimate (Lemma \ref{prep:lemma-bony}), which yields 
\begin{equation}\label{nonlinear:eq-para-p6}
\begin{aligned}
&\big\| \Big( P_{K}^u \big( \SecondC{k}{ij}(\phi) \big) - \big(\partial_\ell \SecondC{k}{ij} \big)(\phi) P_K^u \phi^\ell \Big) \partial_u \big(A_{M,m}^{+,i} \phi^{+,m}_M \big) \big\|_{\Cprod{r-1}{1-r^\prime}} \\
&\lesssim \big\|  P_{K}^u \big( \SecondC{k}{ij}(\phi) \big) - \big(\partial_\ell \SecondC{k}{ij} \big)(\phi) P_K^u \phi^\ell \big\|_{\Cprod{s}{s}} \| \partial_u \big(A_{M,m}^{+,i} \phi^{+,m}_M \big)  \|_{\Cprod{r-1}{s}} \\
&\lesssim K^{-s} M^{r-s} \theta^3. 
\end{aligned}
\end{equation}
Since $K\gtrsim M^{1-\sigma}$ and $r-2s\approx -1/4$, the gain in \eqref{nonlinear:eq-para-p6} clearly beats the $2\delta$-loss in \eqref{nonlinear:eq-para-p3}. \newline

\emph{Step 3: The error term.} 
Using the commutator estimate (Lemma \ref{prep:lem-commutator}), we have that 
\begin{equation}\label{nonlinear:eq-para-p7}
\big\| \mathcal{E}_{K,i}^{k} \big\|_{\Cprod{\eta}{r-1}} \lesssim K^{r-1-s+\eta+2\delta} 
\big\| \SecondC{k}{ij}(\phi) \big\|_{\Cprod{s}{s}} \big\| \partial_v \zeta^j \big\|_{\Cprod{1-r-2\delta}{r-1}}. 
\end{equation}
Similar as in Step 1 or 2, it is easy to see for $\zeta$ of type $\tp$, $\tpm$, or $\ts$ that 
\begin{equation}\label{nonlinear:eq-para-p8}
\big\| \partial_v \zeta^j \big\|_{\Cprod{1-r-2\delta}{r-1}} \lesssim \Gain(\zeta) \theta. 
\end{equation}
Together with the bilinear estimate (Proposition \ref{prep:prop-bilinear}), we obtain from \eqref{nonlinear:eq-para-p7} and \eqref{nonlinear:eq-para-p8} that
\begin{align*}
    \big\| \mathcal{E}_{K,i}^{k} \partial_u \big( A_{M,m}^{+,i} \phi^{+,m}_M \big)\big\|_{\Cprod{r-1}{r-1}}  
    &\lesssim \big\| \mathcal{E}_{K,i}^{k} \big\|_{\Cprod{\eta}{r-1}} 
    \big\| \partial_u \big( A_{M,m}^{+,i} \phi^{+,m}_M \big) \big\|_{\Cprod{\eta}{s}} \\
    &\lesssim K^{r-1-s+\eta+2\delta} M^{1-s+\eta} \Gain(\zeta) \theta^3. 
\end{align*}
Since $K\gtrsim M^{1-\sigma}$, $r-1-s+\eta+2\delta \approx -3/4$, and $1-s+\eta\approx 1/2$, this term is acceptable. 
\end{proof}

We now extend the estimate of the  $\tpm$-$\ts$-case from the null form to the full nonlinearity. As described in Section \ref{section:unfortunate}, the only problematic term is the unfortunate resonance.

\begin{proposition}[Resonant term]\label{nonlinear:prop-unfortunate}
Let Hypothesis \ref{hypothesis:smallness} be satisfied, let $1\leq k \leq \dimA$, and let $M_1$ and $M_2$ be frequency scales satisfying $M_1 \sim_\delta M_2$. Then, it holds that
\begin{align}
&\Big\|  \SecondC{k}{ij}(\phi) \partial_u \big( B_{M_1,M_2,m_1m_2}^i \phi^{+,m_1}_{M_1} \phi^{-,m_2}_{M_2} \big) \partial_v \psi^j   \notag \\
&- \Big( \SecondC{k}{ij}(\phi) B_{M_1,M_2,m_1m_2}^i \big( \phi^{-,m_2}_{M_2}  \paradownv \partial_v \psi^j \big) \Big) \paralesssimsigu \partial_u \phi^{+,m_1}_{M_1} \Big\|_{\Cprod{r-1}{r-1}} \lesssim (M_1 M_2)^{-\eta} \theta^4.\label{nonlinear:eq-unfortunate-1}
\end{align}
Furthermore, the resonant part satisfies 
\begin{equation}\label{nonlinear:eq-unfortunate-2}
   \Big\| \SecondC{k}{ij}(\phi) B_{M_1,M_2,m_1m_2}^i \big( \phi^{-,m_2}_{M_2}  \paradownv \partial_v \psi^j \big) \Big\|_{\Cprod{s}{r-1}} \lesssim M_2^{1-r-s+2\eta} \theta^3. 
\end{equation}
\end{proposition}

\begin{proof}
We first prove \eqref{nonlinear:eq-unfortunate-1}. To this end, we decompose \begin{align}
 &\SecondC{k}{ij}(\phi) \partial_u \big( B_{M_1,M_2,m_1m_2}^i \phi^{+,m_1}_{M_1}\phi^{-,m_2}_{M_2}\big) \partial_v \psi^j \\
 &- \Big( \SecondC{k}{ij}(\phi) B_{M_1,M_2,m_1m_2}^i \big( \phi^{-,m_2}_{M_2} \paradownv \partial_v \psi^j \big) \Big) \paralesssimsigu \partial_u \phi^{+,m_1}_{M_1}\notag \\
 =& \bigg[\SecondC{k}{ij}(\phi)
 \bigg( \partial_u \big( B_{M_1,M_2,m_1m_2}^i \phi^{+,m_1}_{M_1}\phi^{-,m_2}_{M_2}\big) \partial_v \psi^j \label{nonlinear:eq-resonant-p1} \\
 &- \Big(  B_{M_1,M_2,m_1m_2}^i \big( \phi^{-,m_2}_{M_2} \paradownv \partial_v \psi^j \big) \Big) \paralesssimsigu \partial_u \phi^{+,m_1}_{M_1}\bigg) \bigg] \notag \\
 +& \bigg[ \SecondC{k}{ij}(\phi) \bigg(\Big(  B_{M_1,M_2,m_1m_2}^i \big( \phi^{-,m_2}_{M_2} \paradownv \partial_v \psi^j \big) \Big) \paralesssimsigu \partial_u \phi^{+,m_1}_{M_1}\bigg)  \label{nonlinear:eq-resonant-p2}  \\
 &- \Big( \SecondC{k}{ij}(\phi) B_{M_1,M_2,m_1m_2}^i \big( \phi^{-,m_2}_{M_2} \paradownv \partial_v \psi^j \big) \Big) \paralesssimsigu \partial_u \phi^{+,m_1}\bigg]. \notag
 \end{align}
 We estimate the two terms \eqref{nonlinear:eq-resonant-p1} and \eqref{nonlinear:eq-resonant-p2} separately. The first term \eqref{nonlinear:eq-resonant-p1} can be controlled using the multiplication estimate (Corollary \ref{prep:corollary-multiplication}) and Proposition \ref{bilinear:prop-tpm-ts}. Using Corollary \ref{prep:cor-commutator-PDE} and Proposition \ref{bilinear:prop-tpm-ts}, we have that 
 \begin{align*}
\bigg\| & \SecondC{k}{ij}(\phi) \bigg(\Big(  B_{M_1,M_2,m_1m_2}^i \big( \phi^{-,m_2}_{M_2}  \paradownv \partial_v \psi^j \big) \Big) \paralesssimsigu \partial_u \phi^{+,m_1}_{M_1} \bigg)  \\
&- \Big( \SecondC{k}{ij}(\phi) B_{M_1,M_2,m_1m_2}^i \big( \phi^{-,m_2}_{M_2}  \paradownv \partial_v \psi^j \big) \Big) \paralesssimsigu \partial_u \phi^{+,m_1}_{M_1} \Big\|_{\Cprod{r-1}{r-1}} \\
\lesssim& \big\| \SecondC{k}{ij}(\phi) \big\|_{\Cprod{s}{s}} 
\big\|  B_{M_1,M_2,m_1m_2}^i \big( \phi^{-,m_2}_{M_2}  \paradownv \partial_v \psi^j \big) \big\|_{\Cprod{s}{r-1}} 
\big\| \partial_u \phi^{+,m_1}_{M_1}\big\|_{\C_u^{s-1+2\delta}} \\
\lesssim& M_1^{2\delta} M_2^{1-s-r} \theta^4. 
 \end{align*}
 Since $1-s-r\approx -1/4$, this term is acceptable. The remaining estimate \eqref{nonlinear:eq-unfortunate-2} for the resonant term follows directly from the multiplication estimate (Corollary \ref{prep:corollary-multiplication}) and Proposition \ref{bilinear:prop-tpm-ts}.

\end{proof}

\subsection{The remainder terms}

It remains to extend the \raisebox{1pt}{\scalebox{0.7}{\checkmark}}-estimates from the null form to the full nonlinearity.

\begin{proposition} \label{nonlinear:prop-check}
Let Hypothesis \ref{hypothesis:smallness} be satisfied, let $1\leq k\leq \dimA$, and let 
$\zeta_1,\zeta_2\colon \R_{u,v}^{1+1}\rightarrow \R^\dimA$ satisfy
\begin{align*}
\Type (\zeta_1,\zeta_2) \in \big\{ \text{\tm-\tp, \tm-\tpm, \tm-\ts, \tpm-\tpm, \ts-\ts} \big\}. 
\end{align*}
Then, it holds that 
\begin{equation*}
\big\| \SecondC{k}{ij}(\phi) \partial_u \zeta^{i}_1 \partial_v \zeta^{j}_2 \big\|_{\Cprod{r-1}{r-1}} \lesssim \Gain(\zeta_1)\Gain(\zeta_2) \theta^2. 
\end{equation*}
\end{proposition}

\begin{proof}
 Using the multiplication estimate (Corollary \ref{prep:corollary-multiplication}) and the composition estimate (Lemma \ref{prep:lemma-bony}), we obtain that 
 \begin{align*}
     \big\| \SecondC{k}{ij}(\phi) \partial_u \zeta^{i}_1 \partial_v \zeta^{j}_2 \big\|_{\Cprod{r-1}{r-1}} 
     &\lesssim \big\| \SecondC{k}{ij}(\phi) \big\|_{\Cprod{s}{s}} 
     \big\|\partial_u \zeta^{i}_1 \partial_v \zeta^{j}_2 \big\|_{\Cprod{r-1}{r-1}} \\
     &\lesssim (1+\| \phi\|_{\Cprod{s}{s}})^{10}  \big\|\partial_u \zeta^{i}_1 \partial_v \zeta^{j}_2 \big\|_{\Cprod{r-1}{r-1}} \\
     &\lesssim \big\|\partial_u \zeta^{i}_1 \partial_v \zeta^{j}_2 \big\|_{\Cprod{r-1}{r-1}}. 
 \end{align*}
 The desired estimate now follows directly from Proposition \ref{bilinear:prop-check}.
\end{proof}

\subsection{Proof of Theorem \ref{nonlinear:thm}}

By combining the earlier propositions, we now prove the main theorem of this section. 

\begin{proof}[Proof of Theorem \ref{nonlinear:thm}]
It only remains to combine our previous estimates. To this end, we decompose
\begin{equation*}
\SecondC{a}{ij}(\phi) \partial_u \phi^i \partial_v \phi^j 
= \SecondC{a}{ij}(\phi) \sum_{\substack{ \Type \zeta_1 = \\ \tp, \tm, \tpm,\ts}} \sum_{\substack{ \Type \zeta_2 = \\ \tm, \tp, \tpm,\ts}} \partial_u \zeta^i_1 \partial_v \zeta^j_2.  
\end{equation*}
Before continuing with the proof, we encourage the reader to review Figure \ref{figure:cases}, which gives an overview of the relevant cases. \newline

\emph{Case $\tp$-$\tm$:} In this case, $\zeta_1^i=A^{+,i}_{M,m}\phi^{+,m}_M$ and $\zeta^j_2=A^{-,j}_{N,n} \phi^{-,n}_N$ for some $M,N\geq 1$. If $N\leq M^{1-\delta}$, we use Proposition \ref{nonlinear:prop-main-parall}, which contributes \eqref{nonlinear:eq-Fp-1} in $F^+$. If $M\sim_\delta N$, we use Proposition \ref{nonlinear:prop-main-sim}, which contributes $-G_{M,N,mn}^a \partial_u \phi^{+,m}_M \partial_v \phi^{-,n}_N$, \eqref{nonlinear:eq-Fp-3} in $F^+$, and \eqref{nonlinear:eq-Fm-3} in $F^-$. If $M\leq N^{1-\delta}$, which use the symmetry in the $u$ and $v$-variables and Proposition \ref{nonlinear:prop-main-parall}, which contributes \eqref{nonlinear:eq-Fm-1} in $F^-$. \newline

\emph{The cases $\tp$-$\tp$, $\tp$-$\tpm$, $\tp$-$\ts$, $\tm$-$\tm$, $\tpm$-$\tm$, and $\ts$-$\tm$:} Using the symmetry in the $u$ and $v$-variables, all six cases are addressed by Proposition \ref{nonlinear:prop-para}. The corresponding contributions to $F^\pm$ are \eqref{nonlinear:eq-Fp-2} and \eqref{nonlinear:eq-Fm-2}. \newline

\emph{The cases $\tpm$-$\ts$ and $\ts$-$\tpm$:} Using the symmetry in the $u$ and $v$-variables, both cases are addressed by Proposition \ref{nonlinear:prop-unfortunate}. The corresponding contributions to $F^\pm$ are \eqref{nonlinear:eq-Fp-4} and \eqref{nonlinear:eq-Fm-4}. \newline

\emph{The remaining cases:} Using the symmetry in the $u$ and $v$-variables, all remaining cases are addressed by Proposition \ref{nonlinear:prop-check}. The corresponding contribution only enters into the remainder. 
\end{proof}

%%%%%%%%%%%%%%%%%%%%%%%%%%%%%%%%%%%%%%%%%%%%%%%% Duhamel integral approximation %%%%%%%%%%%%%%%%%%%%%%%%%%%%%%%%

\section{Duhamel integral approximation}\label{section:duhamel}

In Section \ref{section:nonlinear}, we obtained a detailed description of the nonlinearity $\SecondC{a}{ij}(\phi) \partial_u \phi^i \partial_v \phi^j$. In order to close the fixed point arguments for $A^+$, $A^-$, and $\psi$, however, we need a description of the localized Duhamel integral
\begin{equation}
\chi^+ \chi^- \Duh[ \SecondC{a}{ij}(\phi) \partial_u \phi^i \partial_v \phi^j].
\end{equation}
 From Proposition \ref{prep:prop-duhamel}, we already know that the Duhamel integral maps the remainder $\ENL$ in Theorem \ref{nonlinear:thm} into $\Cprod{r}{r}$, which is an acceptable error. It therefore remains to analyze the Duhamel integral of the three structured components in Theorem \ref{nonlinear:thm}. The main result of this section is contained in the following proposition.

\begin{proposition}[Duhamel integral approximation]\label{duhamel:prop-approx}
Let Hypothesis \ref{hypothesis:smallness} be satisfied and let $1\leq a\leq \dimA$. Then, it holds that 
\begin{align}
- \chi^+ \chi^- \Duh\big[  \Second_{ij}^a(\phi) \partial_u \phi^i \partial_v \phi^v \big]  
&= \sum_{M\sim_\delta N} G_{\chi,M,N,mn}^a \partial_u \phi^{+,m}_M \partial_v \phi^{-,n}_N \label{duhamel:eq-main1} \\
&- \sum_M P_{\leq M^{1-\sigma}}^u \Big( \chi^+ \chi^-  \int_u^v \dv^\prime P_{\leq M^{1-\sigma}}^u F_{M,m}^{+,a}(u,v^\prime) \Big) \phi^{+,m}_M(u) \label{duhamel:eq-main2}  \\
&+\sum_N  P_{\leq N^{1-\sigma}}^v \Big( \chi^+ \chi^-  \int_u^v \du^\prime P_{\leq N^{1-\sigma}}^{v} F_{N,n}^{-,a}(u^\prime,v) \Big) \phi^{-,n}_N(v) \label{duhamel:eq-main3} \\
&+ \RINT^a, \label{duhamel:eq-R}
\end{align}
where the remainder satisfies
\begin{equation*}
\big\| \RINT^a \big\|_{\Cprod{r}{r}} \lesssim \theta^3. 
\end{equation*}
\end{proposition}

The additional frequency projections outside the integrals in \eqref{duhamel:eq-main2} and \eqref{duhamel:eq-main3} may seem superfluous. However, they are necessary due to the $u$-dependence of $\int_u^v \mathrm{d}v^\prime (\hdots)$, $v$-dependence of $\int_u^v \mathrm{d}u^\prime (\hdots)$, and the cut-off functions $\chi^+$ and $\chi^-$.

\subsection{\protect{\texorpdfstring{The low$\times$high-term}{The low x high-term}}}

In this subsection, we analyze the contribution of $F_{M,m}^{+,a}\parallsigu \partial_u \phi^{+,m}_M$ to the Duhamel integral. The desired estimate is included in the next lemma, which is the main estimate in this subsection.

\begin{lemma}[Duhamel integral approximation for the low$\times$high-term]\label{duhamel:lem-low-high}
Let Hypothesis \ref{hypothesis:smallness} be satisfied, let $1\leq a\leq \dimA$, and let $M$ be a frequency-scale. Then, it holds that 
\begin{equation}
\begin{aligned}
&\Big\| \chi^+ \chi^- \Duh\big[ F_{M,m}^{+,a} \parallsigu \partial_u \phi^{+,m}_M \big] - P_{\leq M^{1-\sigma}}^u \Big(  \chi^+ \chi^- \int_u^v \dv^\prime P_{\leq M^{1-\sigma}}^u F_{M,m}^{+,a}(u,v^\prime) \Big) \phi^{+,m}_M\Big\|_{\Cprod{r}{r}}\\ &\lesssim M^{-\eta} \theta^3. 
\end{aligned}
\end{equation}
\end{lemma}

Before we can prove Lemma \ref{duhamel:lem-low-high}, we prove two auxiliary lemmas.

\begin{lemma}[Structured representation of Duhamel integral I]\label{duhamel:lem-structured-I}
Let $1\leq m \leq \dimA$, let $M\geq 1$, let $\zeta \in \Cprod{s}{r-1}$. Then, it holds that 
\begin{equation}\label{duhamel:eq-structured-I}
\Duh\big[ \zeta \partial_u \phi^{+,m}_M \big] 
= \Big( \int_u^v \dv^\prime \zeta(u,v^\prime) \Big) \phi^{+,m}_M 
- \I \big[ \Tr\big( \zeta \phi_M^{+,m} \big) \big]\Big|_{u}^v 
- \Duh\big[ \partial_u \zeta \phi_M^{+,m} \big]. 
\end{equation}
\end{lemma}

\begin{proof}
Using Proposition \ref{prep:prop-duhamel} and integration by parts, it holds that 
\begin{align*}
&\Duh\big[\zeta \partial_u \phi^{+,m}_{M} \big] \\
&= - \int_u^v \dv^\prime \int_u^{v^\prime} \du^\prime \zeta(u^\prime,v^\prime) \, \partial_u \phi^{+,m}_M(u^\prime) \\
&= - \int_u^v \dv^\prime  \Big( \zeta(u^\prime,v^\prime) \phi^{+,m}_M(u^\prime)\big|_{u^\prime=u}^{v^\prime} \Big) 
+ \int_u^v \dv^\prime \int_u^{v^\prime} \du^\prime \partial_u\zeta(u^\prime,v^\prime) \,  \phi^{+,m}_M(u^\prime) \\
&= \Big(  \int_u^v \dv^\prime \zeta(u,v^\prime) \Big) \phi^{+,m}_M(u) 
-\int_u^v \dv^\prime \zeta(v^\prime,v^\prime) \phi^{+,m}_M(v^\prime) 
+ \int_u^v \dv^\prime \int_u^{v^\prime} \du^\prime \partial_u\zeta(u^\prime,v^\prime) \,  \phi^{+,m}_M(u^\prime).
\end{align*}
Using the definitions of the integral operator $\I$ and trace operator $\Tr$, this yields \eqref{duhamel:eq-structured-I}.
\end{proof}

In the proof of Proposition \ref{duhamel:prop-approx}, the first summand in \eqref{duhamel:eq-structured-I} will be responsible for the non-perturbative term \eqref{duhamel:eq-main2}. In order to estimate the second and third summand in \eqref{duhamel:eq-structured-I}, we prove the following estimate for $F_{M,m}^{+,a}$. 

\begin{lemma}[Bounds on $F_{M,m}^{+,a}$]\label{duhamel:lem-F}
Let Hypothesis \ref{hypothesis:smallness} be satisfied, let $1\leq a,m \leq \dimA$, and let $M$ be a frequency scale. Then, it holds that 
\begin{align}
\big\| P_{\leq M^{1-\sigma}}^u F_{M,m}^{+,a} \big\|_{\Cprod{s}{r-1}}
&\lesssim M^{r-s+2\eta} \theta^2, \label{duhamel:eq-F-1} \\
\big\| \Tr \big( F_{M,m}^{+,a} \parallsigu \phi^{+,m}_M \big) \big\|_{\C_x^{r-1}} 
&\lesssim M^{-\eta} \theta^3. \label{duhamel:eq-F-2}
\end{align}
\end{lemma}

\begin{proof}
We first recall that $F_{M,m}^{+,a}$ consists of four different summands, which are given by \eqref{nonlinear:eq-Fp-1}, \eqref{nonlinear:eq-Fp-2}, \eqref{nonlinear:eq-Fp-3}, and \eqref{nonlinear:eq-Fp-4}. We now prove the two estimates \eqref{duhamel:eq-F-1} and \eqref{duhamel:eq-F-2} simultaneously, but distinguish between the four summands in $F_{M,m}^{+,a}$. \newline

\emph{Case 1: Estimate for \eqref{nonlinear:eq-Fp-1}.} We first control the contribution to \eqref{duhamel:eq-F-1}. Using the bilinear estimate, we have that 
\begin{align*}
\big\| A_{M,m}^{+,i} \Second_{ij}^a(\phi) \partial_v \big( A_{N,n}^{-,j} \phi^{-,n}_N \big) \big\|_{\Cprod{s}{s}} 
&\lesssim \big\| A_{M,m}^{+,i} \big\|_{\Cprod{s}{s}}\big\| \Second_{ij}^a(\phi) \big\|_{\Cprod{s}{s}} \big\| \partial_v \big( A_{N,n}^{-,j} \phi^{-,n}_N \big) \|_{\Cprod{s}{r-1}}\\
&\lesssim N^{r-s} \theta^3.
\end{align*}
After summing over $N\lesssim M^{1-\delta}$, this yields an acceptable contribution to \eqref{duhamel:eq-F-1}. We now turn to \eqref{duhamel:eq-F-2}. Using the trace estimate (Lemma \ref{prep:lemma-trace}), we have that 
\begin{align*}
&\Big\| \Tr \Big( \Big( A^{+,i}_{M,m} \SecondC{a}{ij}(\phi) \partial_v \big( A_{N,n}^{-,j} \phi^{-,n}_N \big) \Big) \parallsigu \phi^{+,m}_M \Big) \Big\|_{\C_x^{r-1}} \\
\lesssim& \Big\|  \Big( A^{+,i}_{M,m} \SecondC{a}{ij}(\phi) \partial_v \big( A_{N,n}^{-,j} \phi^{-,n}_N \big) \Big) \parallsigu \phi^{+,m}_M \Big\|_{\Cprod{r-1}{1-r^\prime}} \\
\lesssim& \big\| A^{+,i}_{M,m} \SecondC{a}{ij}(\phi) \big\|_{\Cprod{s}{s}} \big\|  \partial_v \big( A_{N,n}^{-,j} \phi^{-,n}_N \big) \big\|_{\Cprod{s}{1-r^\prime}} \big\|  \phi^{+,m}_M  \big\|_{\C_u^{r-1}} \\
\lesssim& N^{2-r^\prime-s} M^{r-1-s} \theta^4. 
\end{align*}
Since $N\leq M^{1-\delta}$, we have that 
\begin{align*}
    N^{2-r^\prime-s} M^{r-1-s} \lesssim M^{(1-\delta)(2-r^\prime-s)+r-1-s} \lesssim M^{1-2s+\eta-\delta/6}.
\end{align*}
Since $\delta \gg \eta \gg 1-2s$, this is acceptable.

\emph{Case 2: Estimate for \eqref{nonlinear:eq-Fp-2}.} We first decompose
\begin{align}
&P_{\leq M^{1-\sigma}}^u \big( A_{M,m}^{+,i} \SecondC{a}{ij}(\phi)  \partial_v \zeta^j \big) \notag \\
=&P_{\leq M^{1-\sigma}}^u \big( A_{M,m}^{+,i} \SecondC{a}{ij}(\phi) P_{\lesssim M^{1-\sigma}}^u \partial_v \zeta^j \big)  
+P_{\leq M^{1-\sigma}}^u \big( A_{M,m}^{+,i} \SecondC{a}{ij}(\phi) P_{\gg M^{1-\sigma}}^u \partial_v \zeta^j \big)   \notag \\
=&P_{\leq M^{1-\sigma}}^u \big( A_{M,m}^{+,i} \SecondC{a}{ij}(\phi) P_{\lesssim M^{1-\sigma}}^u \partial_v \zeta^j \big)  
+P_{\lesssim M^{1-\sigma}}^u \big( [ P_{\leq M^{1-\sigma}}^u , A_{M,m}^{+,i} \SecondC{a}{ij}(\phi) ] P_{\gg M^{1-\sigma}}^u \partial_v \zeta^j \big) 
\label{duhamel:eq-F-p3}
\end{align}
We estimate the contributions of the first and second summand in  \eqref{duhamel:eq-F-p3} separately. \newline

We start with the first summand in \eqref{duhamel:eq-F-p3}. In order to prove \eqref{duhamel:eq-F-1} and \eqref{duhamel:eq-F-2} simultaneously, we let $\gamma \in \{ 1-r^\prime,s \}$. Then, it holds that 
\begin{equation}\label{duhamel:eq-F-p4}
\begin{aligned}
&\big\| P_{\leq M^{1-\sigma}}^u \big( A_{M,m}^{+,i} \SecondC{a}{ij}(\phi) P_{\lesssim M^{1-\sigma}}^u \partial_v \zeta^j \big) \big\|_{\Cprod{\gamma}{r-1}} \\
\lesssim& \big\|  A_{M,m}^{+,i} \SecondC{a}{ij}(\phi) P_{\lesssim M^{1-\sigma}}^u \partial_v \zeta^j \big\|_{\Cprod{\gamma}{r-1}} \\
\lesssim& \big\| A_{M,m}^{+,i} \big\|_{\Cprod{s}{s}} 
\big\| \SecondC{a}{ij}(\phi) \big\|_{\Cprod{s}{s}} 
\big\| P_{\lesssim M^{1-\sigma}}^u \partial_v \zeta \big\|_{\Cprod{\gamma}{r-1}}.
\end{aligned}
\end{equation}
By inserting the types $\tp$, $\tpm$, and $\ts$, it is easy to see that 
\begin{equation*}
\big\| P_{\lesssim M^{1-\sigma}}^u \partial_v \zeta \big\|_{\Cprod{1-r^\prime}{r-1}}
\lesssim M^{\delta} \Gain(\zeta) \theta \quad \text{and} \quad 
\big\| P_{\lesssim M^{1-\sigma}}^u \partial_v \zeta \big\|_{\Cprod{s}{r-1}}
\lesssim M^{r-s+2\eta} \Gain(\zeta) \theta.
\end{equation*}
Inserting this into \eqref{duhamel:eq-F-p4}, we obtain the two inequalities
\begin{align}
\big\| P_{\leq M^{1-\sigma}}^u \big( A_{M,m}^{+,i} \SecondC{a}{ij}(\phi) P_{\lesssim M^{1-\sigma}}^u \partial_v \zeta^j \big) \big\|_{\Cprod{1-r^\prime}{r-1}} 
&\lesssim M^{\delta} \Gain(\zeta) \theta^2, \label{duhamel:eq-F-p5} \\
\big\| P_{\leq M^{1-\sigma}}^u \big( A_{M,m}^{+,i} \SecondC{a}{ij}(\phi) P_{\lesssim M^{1-\sigma}}^u \partial_v \zeta^j \big) \big\|_{\Cprod{s}{r-1}} 
&\lesssim M^{r-s+2\delta} \Gain(\zeta) \theta^2 .\label{duhamel:eq-F-p6}
\end{align}
Due to \eqref{duhamel:eq-F-p6}, we see that the contribution to \eqref{duhamel:eq-F-2} is acceptable. Using the trace estimate (Lemma \ref{prep:lemma-trace}) and \eqref{duhamel:eq-F-p5}, it also holds that
\begin{equation}\label{duhamel:eq-F-p7}
\begin{aligned}
&\Big\| \Tr \Big( P_{\leq M^{1-\sigma}}^u \big( A_{M,m}^{+,i} \SecondC{a}{ij}(\phi) P_{\lesssim M^{1-\sigma}}^u \partial_v \zeta^j \big) \phi^{+,m}_M \Big)(x) \Big\|_{\C_x^{r-1}} \\
\lesssim& \big\| P_{\leq M^{1-\sigma}}^u \big( A_{M,m}^{+,i} \SecondC{a}{ij}(\phi) P_{\lesssim M^{1-\sigma}}^u \partial_v \zeta^j \big) \phi^{+,m}_M \big\|_{\Cprod{1-r^\prime}{r-1}} \\
\lesssim& \big\|  P_{\leq M^{1-\sigma}}^u \big( A_{M,m}^{+,i} \SecondC{a}{ij}(\phi) P_{\lesssim M^{1-\sigma}}^u \partial_v \zeta^j \big) \big\|_{\Cprod{1-r^\prime}{r-1}} 
\big\| \phi^{+,m}_M \big\|_{\C_u^{1-r^\prime}} \\
\lesssim& M^{\delta+1-r^\prime-s} \theta^2.
\end{aligned}
\end{equation}
Since $\delta+1-r^\prime-s\approx -1/4$, this is acceptable.\newline

We now turn to the second summand in \eqref{duhamel:eq-F-p3}. Using the commutator estimate (Lemma \ref{prep:lem-commutator}), we have that 
\begin{align*}
&\Big\| P_{\lesssim M^{1-\sigma}}^u \big( [ P_{\leq M^{1-\sigma}}^u , A_{M,m}^{+,i} \SecondC{a}{ij}(\phi) ] P_{\gg M^{1-\sigma}}^u \partial_v \zeta^j \big) \Big\|_{\Cprod{s}{r-1}} \\
\lesssim& M^{(1-\sigma)(s-\eta)} \Big\|  [ P_{\leq M^{1-\sigma}}^u , A_{M,m}^{+,i} \SecondC{a}{ij}(\phi) ] P_{\gg M^{1-\sigma}}^u \partial_v \zeta^j  \Big\|_{\Cprod{\eta}{r-1}} \\
\lesssim& M^{(1-\sigma)(s-\eta)} M^{(1-\sigma)(\eta-s)} M^{-2(1-\sigma)\eta} \big\| A_{M,m}^{+,i} \big\|_{\Cprod{s}{s}} 
\big\| \SecondC{a}{ij}(\phi) \big\|_{\Cprod{s}{s}} 
\big\| \partial_v \zeta^j  \big\|_{\Cprod{3\eta}{s}} \\
\lesssim& M^{-\eta} \theta \big\| \partial_v \zeta^j  \big\|_{\Cprod{3\eta}{s}} .
\end{align*}
By inserting the types $\tp$, $\tpm$, and $\ts$ into $\zeta$, it is easy to see that 
\begin{equation*}
    \big\| \partial_v \zeta^j  \big\|_{\Cprod{3\eta}{s}} \lesssim \Gain(\zeta) \theta.
\end{equation*}
Thus, this yields an acceptable contribution to \eqref{duhamel:eq-F-1}. The contribution to \eqref{duhamel:eq-F-2} can be controlled using the trace estimate (Lemma \ref{prep:lemma-trace}) similar as in \eqref{duhamel:eq-F-p7}.

\emph{Case 3: Estimate for \eqref{nonlinear:eq-Fp-3}.} To prove \eqref{duhamel:eq-F-1}, we use Proposition \ref{prep:prop-bilinear}.\ref{prep:item-res}, which yields 
\begin{equation}\label{duhamel:eq-F-p2}
\begin{aligned}
&\big\| \big( \SecondC{a}{ij}(\phi) A_{M,m}^{+,i} A_{N,n}^{-,j} \big) \parasimv \partial_v \phi_N^{-,n} \big\|_{\Cprod{s}{\eta}} \\
\lesssim& \big\|  \big( \SecondC{a}{ij}(\phi) A_{M,m}^{+,i} A_{N,n}^{-,j}  \big) \big\|_{\Cprod{s}{s}} \big\| \partial_v \phi_N^{-,n} \big\|_{\Cprod{s}{\eta-s}} \\
\lesssim& N^{1-2s+\eta} 
\big\| \SecondC{a}{ij}(\phi) \big\|_{\Cprod{s}{s}} 
\big\| A_{M,m}^{+,i} \big\|_{\Cprod{s}{s}} 
\big\| A_{N,n}^{-,j}  \big\|_{\Cprod{s}{s}} 
\big\| \phi^{-,n}_N \big\|_{\C_v^s} \\
\lesssim& N^{1-2s+\eta} \theta^{3}.
\end{aligned}
\end{equation}
Since $N\lesssim M^{1-\delta}$ and $1-2s+\eta \ll r-s$, this yields an acceptable contribution to \eqref{duhamel:eq-F-1}. For \eqref{duhamel:eq-F-2}, we note that 
\begin{align*}
&\Big\| \Tr \Big(  \Big( \big( \SecondC{a}{ij}(\phi) A_{M,m}^{+,i} A_{N,n}^{-,j} \big) \parasimv \partial_v \phi_N^{-,n} \Big) \parallsigu \phi^{+,m}_M \Big)(x) \Big\|_{\C_x^{r-1}} \\
\lesssim& \Big\|  \Big( \big( \SecondC{a}{ij}(\phi) A_{M,m}^{+,i} A_{N,n}^{-,j} \big) \parasimv \partial_v \phi_N^{-,n} \Big) \parallsigu \phi^{+,m}_M \Big\|_{\Cprod{\eta}{\eta}} \\
\lesssim& M^{\eta-s} \Big\| \big( \SecondC{a}{ij}(\phi) A_{M,m}^{+,i} A_{N,n}^{-,j} \big) \parasimv \partial_v \phi_N^{-,n} \Big\|_{\Cprod{\eta}{\eta}} \big\| \phi^{+,m}_M \big\|_{\C_u^s}.
\end{align*}
By reusing \eqref{duhamel:eq-F-p2}, this yields an acceptable contribution to \eqref{duhamel:eq-F-2}. \newline

\emph{Case 4: Estimate for \eqref{nonlinear:eq-Fp-4}.} Using the bilinear estimate (Proposition \ref{prep:prop-bilinear}) and Lemma \ref{key:lemma-unfortunate}, we have that 
\begin{align*}
&\big\| \SecondC{a}{ij}(\phi) B_{M,N,mn}^i ( \phi_{N}^{-,n} \paradownv \partial_v \psi^j ) \big\|_{\Cprod{s}{s}} \\
\lesssim& \big\| \SecondC{a}{ij}(\phi) \big\|_{\Cprod{s}{s}} 
\big\| B_{M,N,mn}^i \big\|_{\Cprod{s}{s}} 
\big\| \phi_{N}^{-,n} \paradownv \partial_v \psi^j \big\|_{\Cprod{s}{s}} \\
\lesssim& N^{\sigma (s+1-r)}\big\| \SecondC{a}{ij}(\phi) \big\|_{\Cprod{s}{s}} 
\big\| B_{M,N,mn}^i \big\|_{\Cprod{s}{s}} 
\big\| \phi_{N}^{-,n} \paradownv \partial_v \psi^j \big\|_{\Cprod{s}{r-1}} \\
\lesssim& N^{\sigma (s+1-r)} N^{1-s-r} \theta^3.
\end{align*}
Since $\sigma (s+1-r)+1-s-r\approx -1/4$, this yields an acceptable contribution to \eqref{duhamel:eq-F-1}. The contribution to \eqref{duhamel:eq-F-2} can then be estimated as in Case 3. 
\end{proof}

Equipped with Lemma \ref{duhamel:lem-structured-I} and Lemma \ref{duhamel:lem-F}, we are now ready to prove Lemma \ref{duhamel:lem-low-high}.

\begin{proof}[Proof of Lemma \ref{duhamel:lem-low-high}]
Using Lemma \ref{duhamel:lem-structured-I}, it holds that 
\begin{align}
& \chi^+ \chi^- \Duh\big[ F_{M,m}^{+,a} \parallsigu \partial_u \phi^{+,m}_M \big] - P_{\leq M^{1-\sigma}}^u \Big( \chi^+ \chi^- \int_u^v \dv^\prime P_{\leq M^{1-\sigma}}^u F_{M,m}^{+,a}(u,v^\prime) \Big) \phi^{+,m}_M \notag \\
=& - \chi^+ \chi^- \I \big[ \Tr \big( P_{\leq M^{1-\sigma}}^u F_{M,m}^{+,a} \, \phi^{+,m}_M \big) \big]\Big|_{x=u}^v \label{duhamel:eq-low-high-p1}\\
+& \chi^+ \chi^- \Duh\big[ \partial_u \big( P_{\leq M^{1-\sigma}}^u F_{M,m}^{+,a}\big) \, \phi^{+,m}_M \big] \label{duhamel:eq-low-high-p2} \\
+& P_{> M^{1-\sigma}}^u \Big( \chi^+ \chi^- \int_u^v \dv^\prime P_{\leq M^{1-\sigma}}^u F_{M,m}^{+,a}(u,v^\prime) \Big) \phi^{+,m}_M \label{duhamel:eq-low-high-p3}.
\end{align}
We first estimate \eqref{duhamel:eq-low-high-p1}. To this end, we let $\widetilde{\chi}$ be a fattened version of $\chi$. Using Lemma \ref{prep:lemma-single-integral} and Lemma \ref{duhamel:lem-F}, we have that 
\begin{align*}
\Big\|  \chi^+ \chi^- \I \big[ \Tr \big( P_{\leq M^{1-\sigma}}^u F_{M,m}^{+,a} \, \phi^{+,m}_M \big) \big]\Big|_{x=u}^v \Big\|_{\Cprod{r}{r}} 
\lesssim& \big\| \I \big[ \widetilde{\chi} \Tr \big( P_{\leq M^{1-\sigma}}^u F_{M,m}^{+,a} \, \phi^{+,m}_M \big) \big] \big\|_{\C_x^r} \\
\lesssim& \big\| \Tr \big( P_{\leq M^{1-\sigma}}^u F_{M,m}^{+,a} \, \phi^{+,m}_M \big) \big\|_{\C_x^{r-1}} \\
\lesssim& M^{-\eta} \theta^3. 
\end{align*}
We now turn to \eqref{duhamel:eq-low-high-p2}. Using Proposition \ref{prep:prop-duhamel}, Proposition \ref{prep:prop-bilinear}.\ref{prep:item-nonres}, and Lemma \ref{duhamel:lem-F}, we have that
\begin{align*}
\Big\| \chi^+ \chi^- \Duh\big[ \partial_u \big( P_{\leq M^{1-\sigma}}^u F_{M,m}^{+,a}\big) \, \phi^{+,m}_M \big] \Big\|_{\Cprod{r}{r}} 
\lesssim& \big\| \partial_u \big( P_{\leq M^{1-\sigma}}^u F_{M,m}^{+,a}\big) \, \phi^{+,m}_M \big\|_{\Cprod{r-1}{r-1}} \\
\lesssim& \big\| \partial_u \big( P_{\leq M^{1-\sigma}}^u F_{M,m}^{+,a}\big)\big\|_{\Cprod{\eta}{r-1}} \big\| \phi^{+,m}_M \big\|_{\C_u^{r-1}} \\
\lesssim& M^{1+\eta-s} M^{r-1-s} \big\|  P_{\leq M^{1-\sigma}}^u F_{M,m}^{+,a} \big\|_{\Cprod{s}{r-1}} \big\| \phi^{+,m}_M \big\|_{\C_u^{s}} \\
\lesssim& M^{2r-3s+3\eta} \theta^3.
\end{align*}
Since $3s-2r\gg \eta$, this is acceptable. It remains to estimate \eqref{duhamel:eq-low-high-p3}. Using Proposition \ref{prep:prop-bilinear}.\ref{prep:item-low-high}, we have that 
\begin{align*}
&\Big\| P_{> M^{1-\sigma}}^u \Big( \chi^+ \chi^- \int_u^v \dv^\prime P_{\leq M^{1-\sigma}}^u F_{M,m}^{+,a}(u,v^\prime) \Big) \phi^{+,m}_M \Big\|_{\Cprod{r}{r}} \\
\lesssim& \, \Big\| \chi^+ \chi^- \int_u^v \dv^\prime P_{\leq M^{1-\sigma}}^u F_{M,m}^{+,a}(u,v^\prime)  \Big\|_{\Cprod{r}{r}} \big\| \phi^{+,m}_{M} \big\|_{\C_u^{\sigma r}} \\
\lesssim& \, M^{\sigma r -s} \theta \Big\| \chi^+ \chi^- \int_u^v \dv^\prime P_{\leq M^{1-\sigma}}^u F_{M,m}^{+,a}(u,v^\prime)  \Big\|_{\Cprod{r}{r}}. 
\end{align*}
Using the trace and integral estimates (Lemma \ref{prep:lemma-trace} and Lemma \ref{prep:lemma-partial-integral}) as well as Lemma \ref{duhamel:lem-F}, we have that 
\begin{align*}
M^{\sigma r -s} \theta \Big\| \chi^+ \chi^- \int_u^v \dv^\prime P_{\leq M^{1-\sigma}}^u F_{M,m}^{+,a}(u,v^\prime)  \Big\|_{\Cprod{r}{r}} 
&\lesssim  M^{\sigma r -s} \theta \big\|  P_{\leq M^{1-\sigma}}^u F_{M,m}^{+,a} \big\|_{\Cprod{r}{r-1}} \\
&\lesssim   M^{\sigma r -s} M^{(1-\sigma)(r-s)} M^{r-s} \theta^3. 
\end{align*}
Since $3s-2r\gg \sigma$, this is acceptable.
\end{proof}

\subsection{\protect{\texorpdfstring{The high$\times$high-term}{The high x high-term}}}

We now analyze the contribution of $\G_{M,N,mn}^a \partial_u \phi^{+,m}_M \partial_v \phi^{-,n}_N$ to the Duhamel estimate. The following lemma constitutes the main estimate of this subsection. 

\begin{lemma}[Duhamel integral approximation of the high$\times$high-term]\label{duhamel:lem-high-high}
Let Hypothesis \ref{hypothesis:smallness} be satisfied, let $1\leq a\leq \dimA$, and let $M,N$ be frequency-scales satisfying $M\sim_\delta N$. Then, it holds that 
\begin{equation}
\big\|\chi^+ \chi^- \Duh\big[ G_{M,N,mn}^a \partial_u \phi^{+,m}_M \, \partial_v \phi^{-,n}_N \big]
- G_{\chi,M,N,mn}^a \phi^{+,m}_M \,  \phi^{-,n}_N \big\|_{\Cprod{r}{r}} \lesssim (MN)^{-\eta} \theta^3. 
\end{equation}
\end{lemma}

Here, $G_{\chi,M,N,mn}^a$ is as in \eqref{ansatz:eq-G-chi}. Before we proceed with the proof of Lemma \ref{duhamel:lem-high-high}, we prove the following algebraic lemma.

\begin{lemma}\label{duhamel:lem-structured-II}
Let $M,N$ be frequency-scales satisfying $M\sim_\delta N$. Then, it holds that 
\begin{align}
&\chi^+ \chi^- \Duh\big[ \G_{M,N,mn}^a \partial_u \phi^{+,m}_M \, \partial_v \phi^{-,n}_N \big]
- G_{\chi,M,N,mn}^a \phi^{+,m}_M \,  \phi^{-,n}_N  \notag\\
=& (\chi^+ \chi^- G_{M,N,mn}^a- G_{\chi,M,N,mn}^a) \phi^{+,m}_M \,  \phi^{-,n}_N \label{duhamel:eq-hh-G-comm} \\
-& \chi^+ \chi^- \bigg[ \Tr_u \Big( \G_{M,N,mn}^a \phi^{+,m}_M \,  \phi^{-,n}_N  \Big) \label{duhamel:eq-hh-1}\\
+& \I_v \Big[ \partial_v G_{M,N,mn}^{a} \phi^{-,n}_N \Big](u,x) \Big|_{x=u}^v \, \phi^{+,m}_M(u)\label{duhamel:eq-hh-2} \\
+& \I \Big[ \Tr \Big( \G_{M,N,mn}^a \phi^{+,m}_M \partial_v \phi^{-,n}_N \Big) \Big](x) \Big|_{x=u}^v \label{duhamel:eq-hh-3}\\
+& \Duh \Big[ \partial_u \G_{M,N,mn}^a \phi^{+,m}_M \,  \partial_v \phi^{-,n}_N \Big] \bigg]. \label{duhamel:eq-hh-4}
\end{align}
\end{lemma}

\begin{proof}[Proof of Lemma \ref{duhamel:lem-structured-II}]
 After applying Lemma \ref{duhamel:lem-structured-I} with $\zeta=\G_{M,N,mn}^a \partial_v \phi^{-,n}_N$, it only remains to show that 
 \begin{align*}
&\Big( \int_u^v \dv^\prime \G_{M,N,mn}^a(u,v^\prime) \partial_v \phi^{-,n}_N(v^\prime) \Big) \phi^{+,m}_M(u) \\
=&  \G_{M,N,mn}^a(u,v) \phi^{+,m}_M(u)\,  \phi^{-,n}_N(v) 
-\Tr_u \Big( \G_{M,N,mn}^a \phi^{+,m}_M \,  \phi^{-,n}_N  \Big)(u,v) \\ 
-& \I_v \Big[ \partial_v G_{M,N,mn}^{a} \phi^{-,n}_N \Big](u,x) \Big|_{x=u}^v \, \phi^{+,m}_M(u).
 \end{align*}
 This follows directly from an integration by parts. 
\end{proof}

\begin{proof}[Proof of Lemma \ref{duhamel:lem-high-high}]
 Throughout the proof, we let $\widetilde{\chi}, \widetilde{\chi}^+$, and $\widetilde{\chi}^-$ be fattened versions of $\chi$, $\chi^+$, and $\chi^-$, respectively. Using Lemma \ref{duhamel:lem-structured-II}, it remains to estimate 
 \eqref{duhamel:eq-hh-G-comm}, \eqref{duhamel:eq-hh-1}, \eqref{duhamel:eq-hh-2}, \eqref{duhamel:eq-hh-3}, and \eqref{duhamel:eq-hh-4} in $\Cprod{r}{r}$. \newline

 \emph{Estimate of \eqref{duhamel:eq-hh-G-comm}:}  From the definition of $\G^a_{\chi,M,N,mn}$, it follows that 
 \begin{align*}
    &(\chi^+ \chi^- G_{M,N,mn}^a- G_{\chi,M,N,mn}^a) \phi^{+,m}_M \,  \phi^{-,n}_N \\
    &= [ \chi^+ \chi^-, P_{\leq M^{1-\sigma}}^u P_{\leq N^{1-\sigma}}^v ] \Big(  \Second^a_{ij}(\phi) A^{+,i}_{M,m} A^{+,j}_{N,n} \Big)  \phi^{+,m}_M \,  \phi^{-,n}_N. 
 \end{align*}
 Using $M\sim_\delta N$, this term can easily be controlled using Lemma \ref{prep:lem-commutator}. \newline

 \emph{Estimate of \eqref{duhamel:eq-hh-1}:} We have that 
 \begin{align*}
  \Big\| \chi^+ \chi^- \Tr_u \Big( \G_{M,N,mn}^a \phi^{+,m}_M \,  \phi^{-,n}_N  \Big) \Big\|_{\Cprod{r}{r}} 
   &\lesssim \max(M,N)^r \big\| \G_{M,N,mn}^a \phi^{+,m}_M \,  \phi^{-,n}_N \big\|_{L^\infty_{u,v}} \\
  &\lesssim \max(M,N)^r M^{-s} N^{-s}  \theta^3.
 \end{align*}
 Since $M\sim_\delta N$ and $r-2s\approx -1/4$, this is acceptable. \newline
 
 \emph{Estimate of \eqref{duhamel:eq-hh-2}:} Using the gain of a derivative through the integral, we have that  
 \begin{align*}
  &\Big\|  \chi^+ \chi^- \I_v \Big[ \partial_v G_{M,N,mn}^{a} \phi^{-,n}_N \Big](u,x) \Big|_{x=u}^v \, \phi^{+,m}_M(u) \Big\|_{\Cprod{r}{r}} \\
  &\lesssim \max(M,N)^{2r} \Big\| \I_v \Big[ \widetilde{\chi}^+ \widetilde{\chi}^- \partial_v G_{M,N,mn}^{a} \phi^{-,n}_N \Big](u,x) \Big|_{x=u}^v \, \phi^{+,m}_M(u) \Big\|_{L^\infty_{u,v}}  \\
  &\lesssim \max(M,N)^{2r} \big\| \I_v \big[ \widetilde{\chi}^+ \widetilde{\chi}^- \partial_v G_{M,N,mn}^{a} \phi^{-,n}_N \big](u,v) \big\|_{L^\infty_{u,v}} \big\| \phi^{+,m}_M(u) \big\|_{L^\infty_u} \\
  &\lesssim \max(M,N)^{2r} N^{-1} N^{1-s} N^{-s} M^{-s} \theta^3 \\
  &= \max(M,N)^{2r} M^{-s} N^{-2s}  \theta^3. 
  \end{align*}
  Since $3s-2r\gg \delta$, this is acceptable. \newline
  
  \emph{Estimate of \eqref{duhamel:eq-hh-3}:} Using Lemma \ref{prep:lemma-single-integral}, we have that 
 \begin{align*}
  \Big\| \chi^+ \chi^- \I \Big[ \Tr \Big( \G_{M,N,mn}^a \phi^{+,m}_M \partial_v \phi^{-,n}_N \Big) \Big](x) \Big|_{x=u}^v \Big\|_{\Cprod{r}{r}} 
  &\lesssim \Big\| \I \Big[ \widetilde{\chi} \Tr \Big( \G_{M,N,mn}^a \phi^{+,m}_M \partial_v \phi^{-,n}_N \Big) \Big](x) \Big\|_{\C_x^r} \\
  &\lesssim \big\| \Tr \big( \G_{M,N,mn}^a \phi^{+,m}_M \partial_v \phi^{-,n}_N\big)  \big\|_{\C_x^{r-1}}\\
  &\lesssim \big\|   \Tr \G_{M,N,mn}^a \big\|_{\C_x^s} \| \phi^{+,m}_M(x) \partial_v \phi^{-,n}_N(x)  \big\|_{\C_x^{r-1}} \\
  &\lesssim M^{-s} N^{r-s}  \theta^3. 
  \end{align*}
  Since $M\sim_\delta N$ and $r-2s\approx-1/4$, this term is acceptable. \newline
  
  \emph{Estimate of \eqref{duhamel:eq-hh-4}:} Using Proposition \ref{prep:prop-duhamel} and frequency-support considerations, it holds that 
  \begin{align*}
  \Big\| \chi^+ \chi^- \Duh \Big[ \partial_u \G_{M,N,mn}^a \phi^{+,m}_M \,  \partial_v \phi^{-,n}_N \Big] \Big\|_{\Cprod{r}{r}}
  &\lesssim \big\| \partial_u \G_{M,N,mn}^a \phi^{+,m}_M \,  \partial_v \phi^{-,n}_N \big\|_{\Cprod{r-1}{r-1}} \\
  &\lesssim (MN)^{r-1} \big\|  \partial_u \G_{M,N,mn}^a \phi^{+,m}_M \, \partial_v \phi^{-,n}_N \big\|_{L^\infty_{u,v}} \\
  &\lesssim (MN)^{r-1} M^{1-s} M^{-s} N^{1-s}  \theta^3 \\
  &\lesssim M^{r-2s} N^{r-s}  \theta^3. 
  \end{align*}
  Since $M\sim_\delta N$ and $3s-2r\gg \delta$, this term is acceptable.
\end{proof}

\subsection{Proof of Proposition \ref{duhamel:prop-approx}} \nopagebreak[4]
\begin{proof}[Proof of Proposition \ref{duhamel:prop-approx}]
 Using Theorem \ref{nonlinear:thm}, we have that 
\begin{align}
&- \chi^+ \chi^- \Duh \big[\Second_{ij}^a(\phi) \, \partial_u \phi^i \, \partial_v \phi^j \big] \notag  \allowdisplaybreaks[3] \\ 
=&  \chi^+ \chi^- \sum_{M\sim_\delta N} \Duh \big[ \G_{M,N,mn}^a \partial_u \phi^{+,m}_M \partial_v \phi^{-,n}_N \big]  \allowdisplaybreaks[3] \label{duhamel:eq-approx-p1} \\
-& \chi^+ \chi^- \sum_M \Duh \big[  F^{+,a}_{M,m} \parallsigu \partial_u \phi^{+,m}_M \big]
-  \chi^+ \chi^- \sum_N\Duh \big[ F_{N,n}^{-,a} \parallsigv  \partial_v \phi^{-,n}_N  \big]  \allowdisplaybreaks[3] \label{duhamel:eq-approx-p2} \\
+& \chi^+ \chi^- \Duh \big[ \ENL^a\big], \label{duhamel:eq-approx-p3}
\end{align}
where the error term $\ENL$ satisfies $\|\ENL \|_{\Cprod{r-1}{r-1}} \lesssim \theta^2$. The contributions \eqref{duhamel:eq-approx-p1}, \eqref{duhamel:eq-approx-p2}, and \eqref{duhamel:eq-approx-p3} can all be adressed using previous lemmas. \newline

\emph{Contribution of \eqref{duhamel:eq-approx-p1}:} This term is addressed in Lemma \ref{duhamel:lem-high-high}. It yields the first non-perturbative term \eqref{duhamel:eq-main1} and contributes to the remainder $\RINT$. \newline

\emph{Contribution of \eqref{duhamel:eq-approx-p2}:} Using the symmetry in the $u$ and $v$-variables, both summands in \eqref{duhamel:eq-approx-p2} can be treated using  Lemma \ref{duhamel:lem-low-high}. It yields the two non-perturbative terms \eqref{duhamel:eq-main2} and \eqref{duhamel:eq-main3} and contributes to the remainder $\RINT$. \newline

\emph{Contribution of \eqref{duhamel:eq-approx-p3}:} Using Proposition \ref{prep:prop-duhamel}, this term only contributes to the remainder $\RINT$. 
\end{proof}

%%%%%%%%%%%%%%%%%%%%% Low-frequency modulations %%%%%%%%%%%%%%%%%%%%%

\section{Modulation equations} \label{section:modulation}

In the previous section, we decomposed the Duhamel integral of the nonlinearity into the three structured components \eqref{duhamel:eq-main1}, \eqref{duhamel:eq-main2}, and \eqref{duhamel:eq-main3} and the nonlinear remainder \eqref{duhamel:eq-R}. As described in Section \ref{section:lwp} below, the nonlinear remainder \eqref{duhamel:eq-R} will be absorbed into the $\psi$-portion of our Ansatz. The three structured components \eqref{duhamel:eq-main1}, \eqref{duhamel:eq-main2}, and \eqref{duhamel:eq-main3}, which cannot be treated perturbatively, are eliminated through the modulation equations. The modulation equations, which were previously stated in Definition \ref{ansatz:def-modulation-eqs}, are given by 
\begin{align}
A^{+,a}_{M,m}(u,v) &= 
\theta \delta^a_m 
- P^{u}_{\leq M^{1-\sigma}} \chi^+ \chi^- \int_u^v \dv^\prime P_{\leq M^{1-\sigma}}^u F_{M,m}^{+,a}(u,v^\prime) \label{modulation:eq-p},\\
A^{-,a}_{N,n}(u,v) &= 
\theta \delta^a_n 
+ P^{v}_{\leq N^{1-\sigma}}  \chi^+ \chi^- \int_u^v \du^\prime P_{\leq N^{1-\sigma}}^v F_{N,n}^{-,a}(u^\prime,v) \label{modulation:eq-m}, \\
B^{b}_{M,N,mn}(u,v) 
&=  1\{ M\sim_\delta N \} \G^{b}_{\chi,M,N,mn} \label{modulation:eq-pm}.
\end{align}
Here, $F^+$, $F^-$, and $\G$ are as in Definition \ref{ansatz:def-F}. We also recall that $\theta \delta^a_m$ and $\theta \delta^a_n$ in \eqref{modulation:eq-p} and \eqref{modulation:eq-m} are a result of the linear evolution. 

The goal of this section is to prove the local well-posedness of the modulation equation, i.e., Proposition \ref{ansatz:prop-modulation}. As described in the introduction (Section \ref{section:intro-main}), the modulation equations cannot be treated using classical ODE methods. Instead, we rely on the para-controlled approach of Gubinelli, Imkeller, and Perkowski \cite{GIP15}. To this end, we now derive para-controlled versions of the modulation equations.

\subsection{Para-controlled modulation equations}

We make the para-controlled Ansatz
\begin{align}
A_{M,m}^{+,a} &= \sum_{N\leq M^{1-\delta}} X_{M,N,mn}^{+,a} \parallv \phi^{-,n}_N + Y^{+,a}_{M,m}, \label{modulation:eq-ansatz-p}\\
A_{N,n}^{-,a} &= \sum_{M\leq N^{1-\delta}} X_{N,M,nm}^{-,a} \parallu \phi^{+,m}_M + Y^{-,a}_{N,n}. \label{modulation:eq-ansatz-m}
\end{align}
To guide the reader, we mention that our analysis below yields uniform bounds for
\begin{equation*}
X_{M,N}^+, X_{N,M}^- \in \Cprod{s}{s}, \quad Y_M^+\in \Cprod{s}{r}, \quad \text{and} \quad  Y_{N}^- \in \Cprod{r}{s}.
\end{equation*}

Motivated by the frequency-support conditions for $A^+,A^-$, and $B$ (Conditions \ref{condition:frequency}), we impose the following frequency-support conditions on the para-controlled modulations.

\begin{condition}[Frequency-support conditions for $X^\pm$ and $Y^\pm$]
For all $1\leq i,m,n\leq \dimA$ and all frequency scales $M$ and $N$, we impose that
\begin{align}
P^u_{\gg M^{1-\sigma}} X^{+,i}_{M,N,mn} &=0, \label{modulation:eq-freq-cond-X+} \\
P^u_{\gg M^{1-\sigma}} Y^{+,i}_{M,m} &=0, \label{modulation:eq-freq-cond-Y+} \\
P^v_{\gg N^{1-\sigma}} X^{-,i}_{N,M,nm} &=0, \label{modulation:eq-freq-cond-X-} \\
P^v_{\gg N^{1-\sigma}} Y^{-,i}_{N,n} &=0. \label{modulation:eq-freq-cond-Y-}
\end{align}
\end{condition}

We now state the para-controlled modulation equations for $X_{M,N}^+$, $X_{N,M}^-$, $Y_M^+$, and $Y_N^-$, which will be derived momentarily (see Proposition \ref{modulation:prop-para}).

\begin{definition}[Para-controlled modulation equations]\label{modulation:def-para}
For all $1\leq a,m,n \leq \dimA$ and frequency scales $M$ and $N$, the para-controlled modulation equations are given by 
\begin{align}
X_{M,N,mn}^{+,a} &= - P_{\leq M^{1-\sigma}}^u \Big( \chi^+ \chi^- \mathcal{X}^{+,a}_{M,N,mn} \Big) \label{modulation:eq-X} 
\end{align}
and 
\begin{align}
Y_{M,m}^{+,a} &= \theta \delta^{a}_m - P_{\leq M^{1-\sigma}}^u \Com_{\chi,\parall,\I}^v(\mathcal{X}_{M,N,mn}^{+,a}, \phi^{-,n}_N ) \label{modulation:eq-Y1} \\
&- P_{\leq M^{1-\sigma}}^u \Big( \chi^+ \chi^- 
\Big( \I_v \Yc^{+,a}_{M,m}(u,v)- \I_v \Yc^{+,a}_{M,m}(u,u)\Big)  \Big) \label{modulation:eq-Y2}.
\end{align}
Here, $\Com_{\chi,\parall,\I}^v$ is as in Definition \ref{prep:def-integral-commutator}. Furthermore, $\mathcal{X}^{+,a}_{M,N,mn}$ is given by 
\begin{align}
\mathcal{X}^{+,a}_{M,N,mn} &= P_{\leq M^{1-\sigma}}^u \Big[ 
\Second_{ij}^a(\phi) A_{M,m}^{+,i} A_{N,n}^{-,j} \label{modulation:eq-Xc1} \\
&+ \sum_{\substack{L\colon \\ N^{\frac{1}{1-\delta}} \leq L \ll M}} \hspace{-1ex}
\SecondC{a}{ij}(\phi) A_{M,m}^{+,i} \phi_L^{+,\ell} X^{+,j}_{L,N,\ell n} \label{modulation:eq-Xc2} \\
&+ \sum_{\substack{K\colon \\ K\sim_\delta N \\ K\leq M^{1-\delta}}} \SecondC{a}{ij}(\phi)
A_{M,m}^{+,i} B_{K,N,kn}^j \phi^{+,k}_K \Big]. \label{modulation:eq-Xc3}
\end{align}
Finally, the driving force $\Yc^{+,a}_{M,m}$ is given by 
\begin{equation}
 \Yc^{+,a}_{M,m}= 
 \Yc^{(1),+,a}_{M,m}
 +\Yc^{(2),+,a}_{M,m}
 +\Yc^{(3),+,a}_{M,m}
 +\Yc^{(4),+,a}_{M,m},
\end{equation}
where the four summands are given by
\begin{align}
\Yc^{(1),+,a}_{M,m}&
= P_{\leq M^{1-\sigma}}^u \bigg[ 
\sum_{N\leq M^{1-\delta}} \Big( \SecondC{a}{ij}(\phi) A_{M,m}^{+,i} A_{N,n}^{-,j} \Big) \paragtrsimv \partial_v \phi^{-,n}_N \label{modulation:eq-yc-1} \allowdisplaybreaks[4]\\
&+ \sum_{L\ll M} \sum_{N \leq L^{1-\delta}} X^{+,j}_{L,N,\ell n} \phi^{+,\ell}_L \bigg(  \Big( \SecondC{a}{ij}(\phi) A_{M,m}^{+,i} \Big) \parasimv \partial_v \phi_N^{-,n} \bigg) \label{modulation:eq-yc-2} \allowdisplaybreaks[4] \\
&+ \sum_{\substack{K,N\colon \\ K,N\leq M^{1-\delta} \\ K \sim_\delta N}} \phi^{+,k}_K \Big( \SecondC{a}{ij}(\phi) A_{M,m}^{+,i} B_{K,N,kn}^j \Big) \paragtrsimv \partial_v \phi^{-,n}_N  \label{modulation:eq-yc-3} \allowdisplaybreaks[4] \\
&+ \sum_{\substack{N \colon \\ N\sim_\delta M}} \Big( \SecondC{a}{ij}(\phi) A^{+,i}_{M,m} A^{-,j}_{N,n} \Big) \parasimv \partial_v \phi^{-,n}_N   \label{modulation:eq-yc-4} \bigg], \allowdisplaybreaks[4] \\
\Yc^{(2),+,a}_{M,m}&= P_{\leq M^{1-\sigma}}^u \bigg[  \sum_{L \gtrsim M} \SecondC{a}{ij}(\phi) A_{M,m}^{+,i} \phi^{+,\ell}_L \partial_v A^{+,j}_{L,\ell}   \label{modulation:eq-yc-5} \allowdisplaybreaks[4] \\
&+\sum_{\substack{K,N\colon \\ \max(K,N)>M^{1-\delta}, \\ K\sim_\delta N }} 
\SecondC{a}{ij}(\phi) A^{+,i}_{M,m} \partial_v \Big( B^j_{K,N,kn} \phi^{+,k}_K \phi^{-,n}_N\Big) \bigg], \label{modulation:eq-yc-missed}\\
\Yc^{(3),+,a}_{M,m}&=P_{\leq M^{1-\sigma}}^u \bigg[ \sum_{L \ll M} \sum_{N \leq L^{1-\delta}} \phi^{+,\ell}_L 
\Com_{\parall}^v\Big( \SecondC{a}{ij}(\phi) A^{+,i}_{M,m}, X^{+,j}_{L,N,\ell n}, \partial_v \phi^{-,n}_N \Big) \bigg] \label{modulation:eq-yc-6} \allowdisplaybreaks[4], \\
\Yc^{(4),+,a}_{M,m}&= P_{\leq M^{1-\sigma}}^u \bigg[ \sum_{N\leq M^{1-\delta}} \SecondC{a}{ij}(\phi) A_{M,m}^{+,i} \partial_v A_{N,n}^{-,j} \phi^{-,n}_N  \label{modulation:eq-yc-7} \allowdisplaybreaks[4] \\
&+\sum_{L\ll M} \sum_{N \leq L^{1-\delta}} \SecondC{a}{ij}(\phi) A_{M,m}^{+,i} \phi_{L}^{+,\ell} \Big( \partial_v X^{+,j}_{L,N,\ell n} \parallv \phi_N^{-,n}\Big)  \label{modulation:eq-yc-8} \allowdisplaybreaks[4] \\
&+\sum_{L\ll M} \SecondC{a}{ij}(\phi) A^{+,i}_{M,m} \phi^{+,\ell}_L \partial_v Y^{+,j}_{L,\ell} 
 \label{modulation:eq-yc-9} \allowdisplaybreaks[4] \\
&+\sum_{\substack{K,N \colon \\ K,N \leq M^{1-\delta} \\ K \sim_\delta N}} \SecondC{a}{ij}(\phi) A^{+,i}_{M,m} \partial_v \big( B^j_{K,N,kn}\big) \phi^{+,k}_K \phi^{-,n}_N 
 \label{modulation:eq-yc-10} \allowdisplaybreaks[4] \\
&+ A_{M,m}^{+,i} \SecondC{a}{ij}(\phi) \partial_v \psi^j  \label{modulation:eq-yc-11} \allowdisplaybreaks[4] \\
&+ \sum_{\substack{N \colon \\ N\sim_\delta M}} \SecondC{a}{ij}(\phi) B^i_{M,N,mn} 
\big( \phi^{-,n}_N \paradownv \partial_v \psi^j \big) \bigg]. 
 \label{modulation:eq-yc-12}
\end{align}
The para-controlled modulation equations for $X^{-,a}_{N,M,nm}$ and $Y^{-,a}_{N,n}$ are similar to \eqref{modulation:eq-X}-\eqref{modulation:eq-yc-12} but with reversed roles of the $u$ and $v$-variables. 
\end{definition}

In the definition of $\Yc^{+,a}_{M,m}$, we have collected similar terms in four different groups. The individual components $\Yc^{(j)}$ contain the following terms:
\begin{enumerate}
    \item The terms in $\Yc^{(1)}$ all contain $\phi^{-,n}_N$ but involve a second factor at a comparable or higher $v$-frequency.
    \item The terms in  $\Yc^{(2)}$ contains at least two factors at $u$-frequencies $\gtrsim M^{1-\sigma}$. 
    \item The term in $\Yc^{(3)}$ has a commutator structure.
    \item The terms in $\Yc^{(4)}$ are ``easy". 
\end{enumerate}

Of course, the motivation behind \eqref{modulation:eq-Xc1}-\eqref{modulation:eq-yc-12} is not clear from the Definition, but can rather be seen from (the proof of) the following Proposition.

\begin{proposition}[Para-controlled modulation equations]\label{modulation:prop-para}
If $X^+$, $X^-$, $Y^+$, and $Y^-$ solve the system of para-controlled modulation equations (as in Definition \ref{modulation:def-para}), then $A^+$ and $A^-$ solve the modulation equations \eqref{modulation:eq-p} and \eqref{modulation:eq-m}. 
\end{proposition}

\begin{proof}
We only prove the result for $A^+$, since the argument for $A^-$ is similar. To this end, we fix a frequency-scale $M$. We now assign different parts of the right-hand side in \eqref{modulation:eq-p} to either $X^+_{M,N}$, where $N\leq M^{1-\delta}$, or $Y^+_M$. First, we recall from Definition \ref{ansatz:def-F} that $F^{+}_M$ consists of the four summands \eqref{nonlinear:eq-Fp-1}, \eqref{nonlinear:eq-Fp-2}, \eqref{nonlinear:eq-Fp-3}, and \eqref{nonlinear:eq-Fp-4}, which are treated separately. \newline

\emph{Contribution of \eqref{nonlinear:eq-Fp-1}:} We decompose
\begin{align}
&P_{\leq M^{1-\sigma}}^u \sum_{N\leq M^{1-\delta}} A_{M,m}^{+,i} \SecondC{a}{ij}(\phi) \partial_v \big( A^{-,j}_{N,n} \phi^{-,n}_N \big)  \notag \\
=& P_{\leq M^{1-\sigma}}^u \sum_{N \leq M^{1-\delta}}  \Big( \SecondC{a}{ij}(\phi) A^{+,i}_{M,m} A^{-,j}_{N,n} \Big) \parallv \partial_v \phi^{-,n}_N \label{modulation:eq-para-1-p1} \\
+& P_{\leq M^{1-\sigma}}^u \sum_{N\leq M^{1-\delta}}  \Big(\SecondC{a}{ij}(\phi) A^{+,i}_{M,m} A^{-,j}_{N,n} \Big) \paragtrsimv \partial_v \phi^{-,n}_N \label{modulation:eq-para-1-p2} \\
+& P_{\leq M^{1-\sigma}}^u \sum_{N\leq M^{1-\delta}} \SecondC{a}{ij}(\phi) A^{+,i}_{M,m} \partial_v A^{-,j}_{N,n} \phi^{-,n}_N \label{modulation:eq-para-1-p3}.
\end{align}
The contribution of \eqref{modulation:eq-para-1-p1} is included in the commutator in \eqref{modulation:eq-Y1} and \eqref{modulation:eq-Xc1}. The contributions of \eqref{modulation:eq-para-1-p2} and \eqref{modulation:eq-para-1-p3} are contained in \eqref{modulation:eq-yc-1} and \eqref{modulation:eq-yc-7}, respectively. \newline

\emph{Contribution of \eqref{nonlinear:eq-Fp-2}:} We distinguish between different types of $\zeta$. \\

\emph{Contribution of \eqref{nonlinear:eq-Fp-2}, $\Type \zeta=\tp$:} We first decompose
\begin{align}
\sum_L \SecondC{a}{ij}(\phi) A^{+,i}_{M,m} \partial_v \big( A^{+,j}_{L,\ell} \phi^{+,\ell}_L\big) 
&= \sum_{L\ll M}\SecondC{a}{ij}(\phi) A^{+,i}_{M,m} \partial_v  A^{+,j}_{L,\ell} \, \phi^{+,\ell}_L
\label{modulation:eq-para-2-p1} \\
&+  \sum_{L\gtrsim M}\SecondC{a}{ij}(\phi) A^{+,i}_{M,m} \partial_v  A^{+,j}_{L,\ell} \, \phi^{+,\ell}_L.
\label{modulation:eq-para-2-p2}
\end{align}
The contribution of \eqref{modulation:eq-para-2-p2} is included in \eqref{modulation:eq-yc-5}. Thus, we now continue decomposing \eqref{modulation:eq-para-2-p1}. Inserting the para-controlled Ansatz for $A_{L,\ell}^{+,j}$ from \eqref{modulation:eq-ansatz-p}, we have that 
\begin{align}
&\sum_{L\colon L\ll M} \SecondC{a}{ij}(\phi) A_{M,m}^{+,i} \phi^{+,\ell}_L \partial_v A^{+,j}_{L,\ell} \notag \\
=& \sum_{\substack{L,N \colon \\ N^{1/(1-\delta)} \leq   L\ll M}} 
\SecondC{a}{ij}(\phi) A_{M,m}^{+,i} \phi^{+,\ell}_L \partial_v \big( X^{+,j}_{L,N,\ell n} \parallv \phi^{-,n}_N \big) 
+ \sum_{L\colon L \ll M}
\SecondC{a}{ij}(\phi) A_{M,m}^{+,i} \phi^{+,\ell}_L \partial_v Y_{L,\ell}^{+,j}\label{modulation:eq-para-2-q1}.
\end{align}
The second summand in \eqref{modulation:eq-para-2-q1} corresponds to \eqref{modulation:eq-yc-9} and it remains to treat the first summand in \eqref{modulation:eq-para-2-q1}. Using the product rule, it holds that
\begin{align}
 &\sum_{\substack{L,N \colon \\ N^{1/(1-\delta)} \leq   L\ll M}} 
\SecondC{a}{ij}(\phi) A_{M,m}^{+,i} \phi^{+,\ell}_L \partial_v \big( X^{+,j}_{L,N,\ell n} \parallv \phi^{-,n}_N \big) \notag \\
=&  \hspace{-2ex} \sum_{\substack{L,N \colon \\ N^{1/(1-\delta)} \leq   L\ll M}}  \hspace{-2ex}
\SecondC{a}{ij}(\phi) A_{M,m}^{+,i} \phi^{+,\ell}_L  \big( X^{+,j}_{L,N,\ell n} \parallv \partial_v \phi^{-,n}_N \big) \label{modulation:eq-para-2-q2} \\
+& \hspace{-2ex}
\sum_{\substack{L,N \colon \\ N^{1/(1-\delta)} \leq   L\ll M}} \hspace{-2ex} 
\SecondC{a}{ij}(\phi) A_{M,m}^{+,i} \phi^{+,\ell}_L  \big( \partial_v X^{+,j}_{L,N,\ell n} \parallv  \phi^{-,n}_N \big).\label{modulation:eq-para-2-q3}
\end{align}
The second summand \eqref{modulation:eq-para-2-q3} corresponds to \eqref{modulation:eq-yc-8}. Therefore, it remains to treat the first summand \eqref{modulation:eq-para-2-q2}. Using Definition \ref{prep:def-commutator-ODE}, it holds that
\begin{align}
&\hspace{-2ex} \sum_{\substack{L,N \colon \\ N^{1/(1-\delta)} \leq   L\ll M}}  \hspace{-2ex}
\SecondC{a}{ij}(\phi) A_{M,m}^{+,i} \phi^{+,\ell}_L  \big( X^{+,j}_{L,N,\ell n} \parallv \partial_v \phi^{-,n}_N \big) \notag \allowdisplaybreaks[3]\\
=& \sum_{N \colon N\leq M^{1-\delta}} \sum_{\substack{L\colon N^{1/(1-\delta)} \leq L \ll M}} 
\big( \SecondC{a}{ij}(\phi) A_{M,m}^{+,i} \phi^{+,\ell}_L X^{+,j}_{L,N,\ell n} \big) \parallv \partial_v \phi^{-,n}_N \label{modulation:eq-para-2-p3} \allowdisplaybreaks[3] \\
+& \sum_{\substack{L,N \colon \\ N^{1/(1-\delta)} \leq   L\ll M}}  
X_{L,N,\ell n}^{+,j} \phi^{+,\ell}_L \bigg( \big( \SecondC{a}{ij}(\phi) A_{M,m}^{+,i}  \big) \parasimv \partial_v \phi^{-,n}_N \bigg) \label{modulation:eq-para-2-p4} \allowdisplaybreaks[3]\\
+& \sum_{\substack{L,N \colon \\ N^{1/(1-\delta)} \leq   L\ll M}} 
\phi^{+,\ell}_L \Com_{\parall}^v\Big(  \SecondC{a}{ij}(\phi) A_{M,m}^{+,i}, X^{+,j}_{L,N,\ell n}, \partial_v \phi^{-,n}_N\Big) \label{modulation:eq-para-2-p5}. 
\end{align}
The first summand  \eqref{modulation:eq-para-2-p3} is contributes to the commutator in  \eqref{modulation:eq-Y1} and \eqref{modulation:eq-Xc2}. The remaining summands \eqref{modulation:eq-para-2-p4} and \eqref{modulation:eq-para-2-p5} are responsible for \eqref{modulation:eq-yc-2} and \eqref{modulation:eq-yc-6}, respectively. \newline

\emph{Contribution of \eqref{nonlinear:eq-Fp-2}, $\Type \zeta=\tpm$:} 
The total contribution (without the $P_{\leq M^{1-\sigma}}^u$ operator) is given by
\begin{align}
&\sum_{\substack{K,N\colon \\ K\sim_\delta N}} 
\SecondC{a}{ij}(\phi) A^{+,i}_{M,m} \partial_v \big( B^j_{K,N,kn} \phi^{+,k}_K \phi^{-,n}_N \big)  \notag  \allowdisplaybreaks[3] \\
=& \sum_{\substack{K,N\colon \\ K,N \leq M^{1-\delta} \\ K\sim_\delta N}} 
\SecondC{a}{ij}(\phi) A^{+,i}_{M,m} \partial_v \big( B^j_{K,N,kn} \phi^{+,k}_K \phi^{-,n}_N \big) \label{nonlinear:eq-para-missed-1} \\
+&\sum_{\substack{K,N\colon \\ \max(K,N) > M^{1-\delta} \\ K\sim_\delta N}} 
\SecondC{a}{ij}(\phi) A^{+,i}_{M,m} \partial_v \big( B^j_{K,N,kn} \phi^{+,k}_K \phi^{-,n}_N \big). \label{nonlinear:eq-para-missed-2} 
\end{align}
The second summand \eqref{nonlinear:eq-para-missed-2}  is responsible for \eqref{modulation:eq-yc-missed}. 
Using the product rule, we decompose \eqref{nonlinear:eq-para-missed-1} further as
\begin{equation}\label{modulation:eq-para-q1}
\begin{aligned}
&\sum_{\substack{K,N\colon \\ K,N\leq M^{1-\delta}\\ K\sim_\delta N}} 
\SecondC{a}{ij}(\phi) A_{M,m}^{+,i} \partial_v \big( B_{K,N,kn}^j \phi^{+,k}_K \phi^{-,n}_N \big) \\
=& \sum_{\substack{K,N\colon \\ K,N\leq M^{1-\delta}\\ K\sim_\delta N}} 
\SecondC{a}{ij}(\phi) A_{M,m}^{+,i}   B_{K,N,kn}^j \phi^{+,k}_K \partial_v\phi^{-,n}_N 
+ \sum_{\substack{K,N\colon \\ K,N\leq M^{1-\delta}\\ K\sim_\delta N}} 
\SecondC{a}{ij}(\phi) A_{M,m}^{+,i}   \partial_vB_{K,N,kn}^j \phi^{+,k}_K \phi^{-,n}_N.
\end{aligned}
\end{equation}
The second summand in \eqref{modulation:eq-para-q1} is responsible for \eqref{modulation:eq-yc-10}. The first summand in \eqref{modulation:eq-para-q1} is decomposed further by writing
\begin{align*}
&\sum_{\substack{K,N\colon \\ K,N\leq M^{1-\delta}\\ K\sim_\delta N}} 
\SecondC{a}{ij}(\phi) A_{M,m}^{+,i}   B_{K,N,kn}^j \phi^{+,k}_K \partial_v\phi^{-,n}_N  \\
&= \hspace{-2ex} \sum_{\substack{K,N\colon \\ K,N\leq M^{1-\delta}\\ K\sim_\delta N}}  \hspace{-2ex} \Big(
\SecondC{a}{ij}(\phi) A_{M,m}^{+,i}   B_{K,N,kn}^j \phi^{+,k}_K  \Big) \parallv \partial_v\phi^{-,n}_N 
+ \hspace{-2ex}\sum_{\substack{K,N\colon \\ K,N\leq M^{1-\delta}\\ K\sim_\delta N}} 
\hspace{-2ex} \Big( \SecondC{a}{ij}(\phi) A_{M,m}^{+,i}   B_{K,N,kn}^j \phi^{+,k}_K \Big) \paragtrsimv \partial_v\phi^{-,n}_N .
\end{align*}
The $\parallv$-term contributes to the commutator in \eqref{modulation:eq-Y1} and \eqref{modulation:eq-Xc3}. The $\paragtrsimv$-term is included in included in \eqref{modulation:eq-yc-3}. \newline

\emph{Contribution of \eqref{nonlinear:eq-Fp-2}, $\Type \zeta=\ts$:} In this case, $\zeta=\psi$. The contribution is given by \eqref{key:eq-direction-11}. \newline

\emph{Cotribution of \eqref{nonlinear:eq-Fp-3}:} This term coincides with \eqref{modulation:eq-yc-4}. \newline

\emph{Contribution of \eqref{nonlinear:eq-Fp-4}:} This term coincides with \eqref{modulation:eq-yc-12}. 
\end{proof}

\subsection{Estimates for the modulation equations}\label{section:modulation-estimates}

In this subsection, we will derive the main estimates needed for the modulation equations. Similar as in Sections \ref{section:bilinear}-\ref{section:duhamel}, we capture the bounds used in our analysis using a single hypothesis (Hypothesis \ref{hypothesis:modulation}). This hypothesis will then form the basis of the contraction mapping argument in the proof of Proposition \ref{ansatz:prop-modulation}. 

\begin{hypothesis}[Modulation hypothesis]\label{hypothesis:modulation}
We assume that 
\begin{equation*}
\| (\phi^+,\phi^-) \|_{\Ds}\leq \theta.
\end{equation*}
Furthermore, we assume that the modulations satisfy
\begin{align*}
\max\bigg(  \sup_M \| A^+_M \|_{\Mod_M^+}, \,  \sup_N \| A_N^- \|_{\Mod_N^-},  \sup_{M\sim_\delta N} \| B_{M,N}\|_{\Cprod{s}{s}} \bigg) \leq 5\theta.
\end{align*}
Finally, we assume that the para-controlled modulations satisfy
\begin{align*}
\max\bigg(  
\sup_{\substack{M,N \colon \\ N \leq M^{1-\delta}}} \| X^+_{M,N} \|_{\Cprod{s}{s}},
\sup_{\substack{M,N \colon \\ M \leq N^{1-\delta}}} \| X_{N,M}^- \|_{\Cprod{s}{s}}, \,
\sup_{M} \| Y_M^+ \|_{\Cprod{s}{r}}, \,
\sup_{N} \| Y_N^- \|_{\Cprod{r}{s}} \bigg)
\leq 2 \theta. 
\end{align*}
\end{hypothesis}

\begin{lemma}[Resonant-estimate]\label{modulation:lem-resonant}
Let Hypothesis \ref{hypothesis:modulation} be satisfied, let $1\leq i,m,n \leq \dimA$, and let $M$ and $N$ be frequency scales.  Then, it holds that 
\begin{align}
\big\| A_{M,m}^{+,i} \parasimv \partial_v \phi^{-,n}_N \big\|_{\Cprod{s}{r-1}} &\lesssim N^{-\eta} \theta^2, \label{modulation:eq-resonant-1} \\
\big\| \SecondC{a}{ij}(\phi) \parasimv \partial_v \phi^{-,n}_N \big\|_{\Cprod{s}{r-1}} &\lesssim N^{-\eta} \theta. \label{modulation:eq-resonant-2}
\end{align}
\end{lemma}

\begin{proof}
We prove \eqref{modulation:eq-resonant-1} and \eqref{modulation:eq-resonant-2} separately. The proof of \eqref{modulation:eq-resonant-1} will be based on the para-controlled Ansatz \eqref{modulation:eq-ansatz-p}, which was not used in the PDE-analysis. Once the first estimate \eqref{modulation:eq-resonant-1} has been established, we prove the second estimate \eqref{modulation:eq-resonant-2} by using \eqref{modulation:eq-resonant-1} and our previous product estimate (Lemma \ref{bilinear:lemma-product}).  \newline

\emph{Proof of \eqref{modulation:eq-resonant-1}:} Recalling the para-controlled Ansatz \eqref{modulation:eq-ansatz-p}, we have that 
\begin{equation*}
A_{M,m}^{+,i}= \sum_{K\leq M^{1-\delta}} X_{M,K,mk}^{+,i} \parallv \phi^{-,k}_K + Y^{+,i}_{M,m}. 
\end{equation*}
We first estimate the contribution of the para-controlled terms $ X_{M,K,mk}^{+,i} \parallv \phi^{-,k}_K$. To this end, we write
\begin{equation}\label{modulation:eq-resonant-1-p1}
\begin{aligned}
&\big( X_{M,K,mk}^{+,i} \parallv \phi^{-,k}_K \big) \parasimv \partial_v \phi^{-,n}_N  \\
=& 1\{ K \sim N \} \Big( X_{M,K,mk}^{+,i} \big( \phi_K^{-,k} \parasimv \partial_v \phi_N^{-,n} \big) + \Com_{\parall,\parasim}^v \big( X_{M,K,mk}^{+,i}, \phi^{-,k}_K , \partial_v \phi^{-,n}_N \big) \Big), 
\end{aligned}
\end{equation}
where the commutator is as in Definition \ref{prep:def-commutator-ODE}. Using the multiplication estimate (Corollary \ref{prep:corollary-multiplication}) and Lemma \ref{ansatz:lemma-data}, the first summand in \eqref{modulation:eq-resonant-1-p1} is bounded by 
\begin{align*}
\big\| X_{M,K,mk}^{+,i} \big( \phi_K^{-,k} \parasimv \partial_v \phi_N^{-,n} \big) \big\|_{\Cprod{s}{r-1}} 
&\lesssim \big\|X_{M,K,mk}^{+,i} \big\|_{\Cprod{s}{s}} \big\| \phi^{-,k}_K \parasimv \partial_v \phi^{-,n}_N \big\|_{\C_v^{r-1}} \\
&\lesssim N^{r-2s} \theta^3. 
\end{align*}
Since $r-2s\approx -1/4$, this contribution is acceptable. Using Lemma \ref{prep:lem-commutator-ODE}, we have that 
\begin{align*}
\big\| \Com_{\parall,\parasim}^v \big( X_{M,K,mk}^{+,i}, \phi^{-,k}_K , \partial_v \phi^{-,n}_N \big) \big\|_{\Cprod{s}{r-1}} 
&\lesssim \big\| X^{+,i}_{M,K,mk}\big\|_{\Cprod{s}{s}} 
\big\| \phi^{-,k}_K \big\|_{\Cprod{s}{s}} \big\| \partial_v \phi^{-,n}_N \big\|_{\Cprod{s}{-1+2\eta}} \\
&\lesssim N^{-s+2\eta} \theta^3.
\end{align*}
Since $-s+2\eta \approx -1/2$, this term is acceptable. We now estimate the contribution of $Y^{+,i}_{M,m}$. Using the bilinear estimate (Proposition \ref{prep:prop-bilinear}.\ref{prep:item-res}), we have that 
\begin{align*}
\big\| Y_{M,m}^{+,i} \parasimv \partial_v \phi^{-,n}_N \big\|_{\Cprod{s}{r-1}} 
&\lesssim \big\|  Y_{M,m}^{+,i}\big\|_{\Cprod{s}{r}} 
\big\| \partial_v \phi^{-,n}_N \big\|_{\C_v^{-r^\prime}} \\
&\lesssim N^{1-s-r^\prime} \theta^2.
\end{align*}
Since $1-s-r^\prime\approx -1/4$, this is acceptable. \\

\emph{Proof of \eqref{modulation:eq-resonant-2}:}
Using Lemma \ref{prep:lem-para-localized} and the composition estimate (Lemma \ref{prep:lemma-bony}), we have that 
\begin{align*}
&\big\| \SecondC{a}{ij}(\phi) \parasimv \partial_v \phi^{-,n}_N - \widetilde{P}_N^v \big( \SecondC{a}{ij}(\phi) \big) \partial_v \phi^{-,n}_N \big\|_{\Cprod{s}{r-1}} \\
\lesssim& N^{r-1-s-(s-1)} \big\| \SecondC{a}{ij}(\phi) \big\|_{\Cprod{s}{s}} \big\| \partial_v \phi^{-,n}_N \big\|_{\C_v^{s-1}} \\
\lesssim& N^{r-2s} \theta.
\end{align*}
Since $r-2s\approx -1/4$, this is acceptable. 
As a result, it remains to prove 
\begin{equation*}
\big\| P_K^v \big( \SecondC{a}{ij}(\phi) \big) \partial_v \phi^{-,n}_N \big\|_{\Cprod{s}{r-1}} \lesssim N^{-\eta} \theta^2
\end{equation*}
for all $K\sim N$. Using Bony's para-linearization (Lemma \ref{prep:lemma-bony}) and the bilinear estimate (Proposition \ref{prep:prop-bilinear}), it further suffices to prove that 
\begin{equation*}
\big\| P_K^v \phi^k \partial_v \phi^{-,n}_N \big\|_{\Cprod{s}{r-1}} \lesssim N^{-\eta} \theta^2
\end{equation*}
for all $K\sim N$. By inserting the Ansatz for $\phi$, it then remains to prove the two estimates
\begin{align}
\big\| \sum_M P_K^v\big( A_{M,m}^{+,k} \phi^{+,m}_M \big) \, \partial_v \phi^{-,n}_N \big\|_{\Cprod{s}{r-1}} &\lesssim N^{-\eta} \theta^3 \label{modulation:eq-resonant-2-p1},\\
 \sum_{\substack{\Type \zeta= \\  \tm, \tpm, \ts}} \big\| P_K^v \zeta \, \partial_v \phi^{-,n}_N \big\|_{\Cprod{s}{r-1}} &\lesssim N^{-\eta} \theta^2 \label{modulation:eq-resonant-2-p2}.
\end{align}
We start with the first estimate \eqref{modulation:eq-resonant-2-p1}. Using Corollary \ref{prep:cor-bilinear}, we have that 
\begin{align*}
\big\| \sum_M P_K^v \big( A_{M,m}^{+,k} \phi^{+,m}_M \big) \partial_v \phi^{-,n}_N \big\|_{\Cprod{s}{r-1}}
&\lesssim \sup_M \big\| \phi^{+,m}_M \big\|_{\Cprod{s}{s}} \, 
\sup_M  \big\| P_K^v A_{M,m}^{+,k} \partial_v \phi^{-,n}_N \big\|_{\Cprod{s}{r-1}}. 
\end{align*}
The desired estimate then follows from \eqref{modulation:eq-resonant-1}. The second estimate \eqref{modulation:eq-resonant-2-p2} follows from the symmetry in the $u$ and $v$-variables and Proposition \ref{bilinear:lemma-product}, where $\HH_{M,m}^{+,i}=\delta^i_m$. 

\end{proof}

\begin{lemma}[Estimate of $\Yc^{(2)}$]\label{modulation:lem-yc2}
Let Hypothesis \ref{hypothesis:modulation} be satisfied, let $1\leq a,m\leq \dimA$, and let $M$ be a frequency scale. Then, it holds that 
\begin{equation*}
\| \Yc^{(2),a,+}_{M,m} \|_{\Cprod{s}{r-1}} \lesssim M^{-\eta} \theta^3.
\end{equation*}
\end{lemma}
\begin{proof} We estimate the two terms \eqref{modulation:eq-yc-5} and \eqref{modulation:eq-yc-missed} separately. For \eqref{modulation:eq-yc-5}, we have that 
\begin{align*}
 & \Big\| P^u_{\leq M^{1-\sigma}} 
\Big( \sum_{L\colon L\gtrsim M} \SecondC{a}{ij}(\phi) A_{M,m}^{+,i} \partial_v A_{L,\ell}^{+,j} \phi^{+,\ell}_L  \Big) \Big\|_{\Cprod{s}{r-1}} \\
\lesssim& M^{(1-\sigma) s} \sum_{L\colon L \gtrsim M} 
\Big\| P_{\gtrsim L}^u \Big( \SecondC{a}{ij}(\phi) A^{+,i}_{M,m} \partial_v A_{L,\ell}^{+,j} \Big) \phi^{+,\ell}_L \Big\|_{L_u^\infty \C_v^{r-1}} \\
\lesssim& M^{(1-\sigma)s} \sum_{L\colon L \gtrsim M} L^{-2s}
\big\| \SecondC{a}{ij}(\phi) \big\|_{\Cprod{s}{s}} 
\big\| A^{+,i}_{M,m} \big\|_{\Cprod{s}{s}}  
\big\| \partial_v A_{L,\ell}^{+,j}\big\|_{\Cprod{s}{r-1}} 
\big\| \phi^{+,\ell}_L \big\|_{\C_u^s} \\
\lesssim& M^{(1-\sigma)s } \sum_{L\colon L \gtrsim M} L^{r-3s} \theta^3 \\
\lesssim& M^{r-2s}\theta^3. 
\end{align*}
Since $r-2s\approx -1/4$, this term is acceptable. We now estimate the second term in $\Yc^{(2)}$, which is given by \eqref{modulation:eq-yc-missed}. To this end, we first note that $\max(K,N)>M^{1-\delta}$ and $K\sim_\delta N$ imply that $K,N\geq M^{1-2\delta}$. Then, we estimate
\begin{align*}
&\bigg\| P_{\leq M^{1-\sigma}}^u \bigg[ \sum_{\substack{K,N\colon \\ \max(K,N)>M^{1-\delta}, \\ K\sim_\delta N }} 
\SecondC{a}{ij}(\phi) A^{+,i}_{M,m} \partial_v \Big( B^j_{K,N,kn} \phi^{+,k}_K \phi^{-,n}_N\Big) \bigg] \bigg\|_{\Cprod{s}{r-1}} \\
&\lesssim  M^{(1-\sigma)s} 
\sum_{\substack{K,N\colon \\ K,N\geq M^{1-2\delta}, \\ K\sim_\delta N }} 
\Big\| P_{\gtrsim K}^u \Big( \SecondC{a}{ij}(\phi) A^{+,i}_{M,m} \partial_v \Big( B^j_{K,N,kn} \phi^{-,n}_N\Big) \Big) \Big\|_{L^\infty_u \C_v^{r-1}} \| \phi^{+,k}_K\|_{L^\infty_u} \\
&\lesssim M^{(1-\sigma)s} \sum_{\substack{K,N\colon \\ K,N\geq M^{1-2\delta}, \\ K\sim_\delta N }} K^{-2s} N^{r-s} \theta^4. 
\end{align*}
The dyadic sum can be estimated by
\begin{align*}
    M^{(1-\sigma)s} \sum_{\substack{K,N\colon \\ K,N\geq M^{1-2\delta}, \\ K\sim_\delta N }} K^{-2s} N^{r-s} 
    &\lesssim M^{(1-\sigma) s} \sum_{\substack{N\colon \\ N \geq M^{1-2\delta}}}
    N^{r-s-(1-\delta)2s} \\
    &\lesssim M^{(1-\sigma)s-(1-2\delta)(r-s-(1-\delta)2s)}.
\end{align*}
Since $r-2s \approx -1/4$, this is acceptable. 
\end{proof}

\subsection{Proof of Proposition \ref{ansatz:prop-modulation}}

Equipped with the estimates from Section \ref{section:modulation-estimates}, we can now prove the local well-posedness of the modulation equations.

\begin{proof}[Proof of Proposition \ref{ansatz:prop-modulation}]
We use a contraction mapping argument. For expository purposes, we split the proof into three steps. In the first step, we set up the solution space $\Sc$, which captures the norms in Hypothesis \ref{hypothesis:modulation}, and the mapping $\Gamma$, which encodes the modulation equations. In the second step, we show that $\Gamma$ maps a $\theta$-ball in $\Sc$ to itself. In the third step, we briefly discuss the contraction property and the continuous dependence on the data $(\phi_0(0),\phi^+,\phi^-,\psi)$, but omit the details. \\

\emph{Step 1: Setup}. We first define the norm
\begin{align*}
& \hspace{-7ex}\big\| (B,X^+,X^-,Y^+,Y^-) \big\|_{\Sc} \\
=\max\bigg(&
\, 5 \sup_{M\sim_\delta N} \| B_{M,N}\|_{\Cprod{s}{s}}, 
2 \hspace{-1ex} \sup_{\substack{M,N \colon \\ N \leq M^{1-\delta}}} 
\hspace{-1ex}  \| X^+_{M,N} \|_{\Cprod{s}{s}}, 
 2  \hspace{-1ex} \sup_{\substack{M,N \colon \\ M \leq N^{1-\delta}}} 
 \hspace{-1ex} \| X_{N,M}^- \|_{\Cprod{s}{s}}, \\
&2 \sup_{M} \| Y_M^+ \|_{\Cprod{s}{r}}, \,
2 \sup_{N} \| Y_N^- \|_{\Cprod{r}{s}} \bigg).
\end{align*}
The absolute constants in $\| \cdot \|_{\Sc}$ have been chosen to match Hypothesis \ref{hypothesis:modulation}. We define the corresponding solution space $\Sc$ as the set of functions that have finite $\Sc$-norm and satisfy the frequency-support conditions. More precisely, we define
\begin{align*}
\Sc := \Big \{ &(B,X^+,X^-,Y^+,Y^-)\colon \| (B,X^+,X^-,Y^+,Y^-) \|_{\Sc} <\infty \text{ and} \\
&\,  \eqref{ansatz:eq-condition-pm}, \, 
\eqref{modulation:eq-freq-cond-X+}, \,
\eqref{modulation:eq-freq-cond-Y+}, \,
\eqref{modulation:eq-freq-cond-X-}, \,
\text{and }
\eqref{modulation:eq-freq-cond-Y-}
\text{ hold}  \Big\}.
\end{align*}
Furthermore, we define the $\theta$-ball $\Sc_\theta$ by 
\begin{equation*}
\Sc_\theta := \Big \{ (B,X^+,X^-,Y^+,Y^-) \in \Sc\colon \| (B,X^+,X^-,Y^+,Y^-) \|_{\Sc} \leq \theta \Big\}. 
\end{equation*}
For any element $(B,X^+,X^-,Y^+,Y^-)\in \Sc_\theta$, we define the corresponding modulations $A^+=A^+[X^+,Y^+]$ and $A^-=A^-[X^-,Y^-]$ through the para-controlled Ansatz in \eqref{modulation:eq-ansatz-p} and \eqref{modulation:eq-ansatz-m}. Using the bilinear estimate for families (Corollary \ref{prep:cor-bilinear}), one easily obtains that
\begin{align*}
\sup_M \| A_M^+ \|_{\Mod_M^+} &\leq C \theta \sup_{M,N} \| X_{M,N}^+ \|_{\Cprod{s}{s}} +  2 \sup_{M} \| Y_M^+ \|_{\Cprod{s}{r}}, \\
\sup_N \| A_N^- \|_{\Mod_N^-} &\leq C \theta \sup_{M,N} \| X_{N,M}^- \|_{\Cprod{s}{s}} +  2 \sup_{N} \| Y_N^- \|_{\Cprod{r}{s}}. 
\end{align*}
In particular, the condition $(B,X^+,X^-,Y^+,Y^-)\in \Sc_\theta$ implies that $A^+$ and $A^-$ satisfy the conditions in Hypothesis \ref{hypothesis:modulation}. 
We now define a map $\Gamma=\Gamma[B,X^+,X^-,Y^+,Y^-]$  which encodes the para-controlled modulation equations. We define
\begin{align*}
(\Gamma_{B})^{b}_{M,N,mn} 
&:= \text{ RHS of } \eqref{ansatz:eq-modulation-pm}, \\
(\Gamma_X)^{+,a}_{M,N,mn} 
&:= \text{ RHS of } \eqref{modulation:eq-X}, \\
(\Gamma_Y)^{+,a}_{M,m} 
&:= \text{ RHS of } \eqref{modulation:eq-Y1} + 
\eqref{modulation:eq-Y2}. 
\end{align*}
The components $(\Gamma_X)^-$ and $(\Gamma_Y)^-$ are similar as $(\Gamma_X)^+$ and $(\Gamma_Y)^+$ but with reversed roles of the $u$ and $v$-variables. \\

\emph{Step 2: Self-mapping property.} 

In this step, we prove that for any $(B,X^+,X^-,Y^+,Y^-)\in \Sc_\theta$, it holds that 
\begin{equation*}
\Gamma[B,X^+,X^-,Y^+,Y^-] \in \Sc_\theta.
\end{equation*}
The frequency-support conditions follow directly from the definition of $\Gamma$. As a result, it remains to prove that 
\begin{equation}\label{modulation:eq-lwp-1}
\| \Gamma[B,X^+,X^-,Y^+,Y^-] \|_{\Sc}\leq \theta. 
\end{equation}
By symmetry in the $u$ and $v$-variables, it suffices to prove the required bounds in \eqref{modulation:eq-lwp-1} for the $\Gamma_B$, $(\Gamma_X)^+$, and $(\Gamma_Y)^+$-components. \\

\emph{Step 2.i: Self-mapping bound for the $\Gamma_B$-component.} Using the bilinear estimate (Proposition \ref{prep:prop-bilinear}), the composition estimate (Lemma \ref{prep:lemma-bony}), and Definition \ref{ansatz:def-F}, we have that
\begin{align*}
\| (\Gamma_B)_{M,N} \|_{\Cprod{s}{s}} \lesssim
\| \Second^\diamond(\phi) \|_{\Cprod{s}{s}} \| A^+_M \|_{\Cprod{s}{s}} \| A^-_N \|_{\Cprod{s}{s}} \lesssim \theta^2.
\end{align*}
Since $\theta>0$ is sufficiently small, we obtain that 
$\| (\Gamma_B)_{M,N} \|_{\Cprod{s}{s}}\leq 5\theta$, which is the desired estimate. \\

\emph{Step 2.ii: Self-mapping bound for the $(\Gamma_X)^+$-component.} It suffices to estimate $\Xc^{+,a}_{M,N,mn}$ in $\Cprod{s}{s}$. Using the bilinear estimate (Proposition \ref{prep:prop-bilinear}), the bilinear estimate for families (Corollary \ref{prep:cor-bilinear}), and the composition estimate (Lemma \ref{prep:lemma-bony}), we have that 
\begin{align*}
\| (\Gamma_X)_{M,m}^{+,a} \|_{\Cprod{s}{s}}  \leq& \| \Xc^{+,a}_{M,N,mn} \|_{\Cprod{s}{s}} \\
\leq & \| \eqref{modulation:eq-Xc1} \|_{\Cprod{s}{s}} 
+ \| \eqref{modulation:eq-Xc2} \|_{\Cprod{s}{s}} 
+\| \eqref{modulation:eq-Xc3} \|_{\Cprod{s}{s}} \\
\lesssim& \| \Second^\diamond(\phi) \|_{\Cprod{s}{s}} \| A^+_M \|_{\Cprod{s}{s}}
\| A^-_N \|_{\Cprod{s}{s}} \\
+&  \sup_{L} \| \Second^\diamond(\phi) \|_{\Cprod{s}{s}}  \| A^+_M \|_{\Cprod{s}{s}} \| \phi^+_L \|_{\C_u^s} \| X^+_{L,N}\|_{\Cprod{s}{s}} \\
+& \sup_K \| \Second^\diamond(\phi) \|_{\Cprod{s}{s}} \| A^+_M \|_{\Cprod{s}{s}} 
\| B_{K,N} \|_{\Cprod{s}{s}} \| \phi^+_K \|_{\C_u^s} \\
\lesssim& \theta^2. 
\end{align*}
Since $\theta>0$ is sufficiently small, this implies the desired estimate. \\

\emph{Step 2.iii: Self-mapping bound for the $(\Gamma_Y)^+$-component.} In this step, we prove that 
\begin{equation*}
\| (\Gamma_Y)^+ - \theta \operatorname{Id} \|_{\Cprod{s}{r}} \lesssim \theta^2,
\end{equation*}
which implies the desired estimate. Using the definition of $(\Gamma_Y)^+$, it suffices to prove that 
\begin{align}
\big\| \sum_{N\leq M^{1-\delta}} 
\Com_{\chi,\parall,\I}^v(\Xc_{M,N,mn}^{+,a}, \phi^{-,n}_N ) \big\|_{\Cprod{s}{r}} &\lesssim \theta^2, \label{modulation:eq-lwp-2} \\
\| \Yc^{+,a}_{M,m} \|_{\Cprod{s}{r-1}} &\lesssim \theta^3. \label{modulation:eq-lwp-3}
\end{align}
In Step 2.ii, we have already estimated $\Xc^{+}_{M,N}$ in $\Cprod{s}{s}$. Then, the estimate \eqref{modulation:eq-lwp-2} directly follows from Lemma \ref{prep:lem-commutator-ODE}. To prove the estimate \eqref{modulation:eq-lwp-3}, we further distinguish between the four components $\Yc^{(1)}$, $\Yc^{(2)}$, $\Yc^{(3)}$, and $\Yc^{(4)}$. \\

\emph{Estimate of $\Yc^{(1)}$:} Using Lemma \ref{prep:lem-commutator-ODE} for $\parasimv$, the bilinear estimate for families (Corollary \ref{prep:cor-bilinear}), and the multiplication estimate (Corollary \ref{prep:eq-multiplication}), it remains to control 
\begin{alignat}{3}
A_{M,m}^{+,i} &\parasimv \partial_v \phi^{-,n}_N,& \quad
\SecondC{a}{ij}(\phi)  &\parasimv \partial_v \phi^{-,n}_N, \label{modulation:eq-yc-e1} \\
A_{N,n}^{-,j} &\parasimv \partial_v \phi^{-,n}_N,& \quad \text{and} \quad
B_{K,N,kn}^j &\parasimv  \partial_v \phi^{-,n}_N \label{modulation:eq-yc-e2}
\end{alignat}
in $\Cprod{s}{r-1}$. The two terms in \eqref{modulation:eq-yc-e1} are controlled by Lemma \ref{modulation:lem-resonant}. The two terms in \eqref{modulation:eq-yc-e2} vanish due to the frequency-support conditions on $A^{-,j}_{N,n}$ and $B^j_{K,N,kn}$. 

\emph{Estimate of $\Yc^{(2)}$:} This term is the content of Lemma \ref{modulation:lem-yc2}.

\emph{Estimate of $\Yc^{(3)}$:} This term can be bounded directly through Lemma \ref{prep:lem-commutator-ODE}. 

\emph{Estimate of $\Yc^{(4)}$:} The remaining ``easy" terms can be bounded using the bilinear estimate (Proposition \ref{prep:prop-bilinear}) and Lemma \ref{key:lemma-unfortunate}. As a result, we omit the details. \\

This completes the proof of the second estimate \eqref{modulation:eq-lwp-3} and, therefore, the proof of the self-mapping estimate \eqref{modulation:eq-lwp-1}. \\

\emph{Step 3: The contraction estimate and continuous dependence on the data.} 
In Step 2, we proved that $\Gamma$ maps $\Sc_\theta$ back into itself. In order to utilize the contraction mapping theorem, it remains to prove that $\Gamma$ is a contraction on $\Sc_\theta$, which would follow from the estimate
\begin{equation}\label{modulation:eq-lwp-contraction}
\begin{aligned}
&\| \Gamma[B,X^+,X^-,Y^+,Y^-] - \Gamma[B^\prime,(X^\prime)^+,(X^\prime)^-,(Y^\prime)^+,(Y^\prime)^-] \|_{\Sc} \\
\lesssim& \theta \| (B,X^+,X^-,Y^+,Y^-) - (B^\prime,(X^\prime)^+,(X^\prime)^-,(Y^\prime)^+,(Y^\prime)^-) \|_{\Sc}.
\end{aligned}
\end{equation}
Compared to our previous estimate \eqref{modulation:eq-lwp-1}, the only new difficulties in proving \eqref{modulation:eq-lwp-contraction} are notational, since $\Gamma$ contains a large number of terms. The individual estimates, however, are the exact same estimates as used in the proof of \eqref{modulation:eq-lwp-1}. As a result, we omit the (extremely tedious but standard) details. \\
As is common for contraction mapping arguments, our estimates also yield the Lipschitz-continuous dependence on the data, which is given by $(\phi_0(0),\phi^+,\phi^-,\psi) \in \M \times  \Ds \times \Cprod{r}{r}$. This can be shown by further generalizing the contraction estimates \eqref{modulation:eq-lwp-contraction}. After reflecting the dependence on $(\phi_0(0),\phi^+,\phi^-,\psi)$ in our notation, the desired estimate reads
\begin{align}
&\| \Gamma[\phi_0(0),\phi^+,\phi^-,\psi,B,X^+,X^-,Y^+,Y^-] - \Gamma[\phi_0^\prime(0),(\phi^\prime)^+,(\phi^\prime)^-,\psi^\prime,B^\prime,(X^\prime)^-,(Y^\prime)^+,(Y^\prime)^-] \|_{\Sc} \notag  \\
\lesssim& \theta \Big( \| \phi_0(0) - \phi_0^\prime(0)\|_{\R^\dimA}+ \| (\phi^+,\phi^-); ((\phi^\prime)^+,(\phi^\prime)^-) \|_{\Ds} 
+ \| \psi - \psi^\prime \|_{\Cprod{r}{r}} \label{modulation:eq-lwp-cts} \\
&\hspace{3ex}+\| (B,X^+,X^-,Y^+,Y^-) - (B^\prime,(X^\prime)^+,(X^\prime)^-,(Y^\prime)^+,(Y^\prime)^-) \|_{\Sc} \Big). \notag
\end{align}
As for the contraction estimate \eqref{modulation:eq-lwp-contraction}, its generalization \eqref{modulation:eq-lwp-cts} can be proven using the estimates leading to \eqref{modulation:eq-lwp-1}. We therefore omit the (extremely tedious but standard) details.
\end{proof}

%%%%%%%%%%%%%%%%%%%%%%%% Local well-posedness %%%%%%%%%%%%%%%%%%%%
\section{Local well-posedness}\label{section:lwp}

In the final section of the article, we present the proof of Proposition \ref{ansatz:prop-forced-wm} and Theorem \ref{intro:thm-rigorous}. 

\begin{proof}[Proof of Proposition \ref{ansatz:prop-forced-wm}]
Throughout the proof, we assume that $\|(\phi^+,\phi^-)\|_{\Ds}\leq \theta$ and restrict our discussion to $\psi\in \Cprod{r}{r}$ satisfying $\| \psi \|_{\Cprod{r}{r}}\leq \theta$. We now split the argument into two steps. In the first step, we simplify the forced wave maps equation \eqref{ansatz:eq-forced-wm}. In the second step, we solve the simplified version using a contraction mapping argument. \\

\emph{Step 1: The simplified forced wave maps equation.}
Due to Proposition \ref{ansatz:prop-modulation}, there exist modulations $A^+[\psi],A^-[\psi]$, and $B[\psi]$ satisfying the modulation equations from Definition \ref{ansatz:def-modulation-eqs}. We recall from \eqref{ansatz:eq-phi} that 
\begin{equation}\label{lwp:eq-forced-p1}
\begin{aligned}
\phi^k[\psi](u,v) 
&= \sum_M A^{+,k}_{M,m}[\psi](u,v) \phi^{+,m}_M(u) 
+ \sum_N A^{-,k}_{N,n}[\psi](u,v) \phi^{-,n}_N(v) \\
&+ \sum_{\substack{M,N\colon \\ M\sim_\delta N}} B^{k}_{M,N,mn}[\psi](u,v) \phi^{+,m}_M(u) \phi^{-,n}_N(v) 
+ \psi^k(u,v). 
\end{aligned}
\end{equation}
Using the modulation equations and Proposition \ref{duhamel:prop-approx}, we have that \begin{equation}\label{lwp:eq-forced-p2}
\begin{aligned}
&\theta \big( \phi^{+,k} +  \phi^{-,k} \big) -  \chi^+ \chi^- \Duh\Big[ \Second^{k}_{ij}(\phi[\psi]) \partial_u \phi^i[\psi] \partial_v \phi^j[\psi]\Big] \\
=& \sum_{\substack{M,N\colon \\ M\sim_\delta N}} B^{k}_{M,N,mn}[\psi](u,v) \phi^{+,m}_M(u) \phi^{-,n}_N(v)  
+\sum_M A^{+,k}_{M,m}[\psi](u,v) \phi^{+,m}_M(u) \\
+& \sum_N A^{-,k}_{N,n}[\psi](u,v) \phi^{-,n}_N(v) 
+ \Rc^k\big[\phi_0(0),\phi^\pm,A^\pm[\psi],B[\psi],\psi\big].
\end{aligned}
\end{equation}
Here, the remainder $\Rc=\Rc\big[\phi_0(0),\phi^\pm,A^\pm,B,\psi\big]$ is as in the statement of Proposition \ref{duhamel:prop-approx}. By combining \eqref{lwp:eq-forced-p1} and \eqref{lwp:eq-forced-p2}, all terms but the remainder cancel, which yields
\begin{align*}
&\theta \big( \phi^{+,k} +  \phi^{-,k} \big) - \chi^+ \chi^- \Duh\Big[ \Second^k_{ij}(\phi[\psi]) \partial_u \phi^i[\psi] \partial_v \phi^j[\psi]\Big] - (\phi^k[\psi]-\psi^k) \\
=& \Rc^k\big[\phi_0(0),\phi^\pm,A^\pm[\psi],B[\psi],\psi\big]. 
\end{align*}
As a result, the forced wave maps equation \eqref{ansatz:eq-forced-wm} reads
\begin{equation}\label{lwp:eq-forced-p3}
\psi^k=\Rc^k\big[\phi_0(0),\phi^\pm,A^\pm[\psi],B[\psi],\psi\big]. 
\end{equation}
We call \eqref{lwp:eq-forced-p3} the simplified forced wave maps equation, which we now solve using a contraction mapping argument.\\

\emph{Step 2: Solving the simplified forced wave maps equation.}
We define the solution set 
\begin{equation*}
\mathcal{S}_\theta=\{ \psi \in \Cprod{r}{r}\colon \| \psi \|_{\Cprod{r}{r}} \leq \theta\},
\end{equation*}
which is equipped with the $\Cprod{r}{r}$-norm. To simplify the notation, we write
\begin{equation}
\Gamma[\phi_0(0),\phi^\pm,\psi] := \Rc[\phi_0(0),\phi^\pm,A^\pm[\psi],B[\psi],\psi]. 
\end{equation}
We now let $\psi,\widetilde{\psi}\in \mathcal{S}_\theta$ and assume that $(\phi^+,\phi^-),(\widetilde{\phi}^+,\widetilde{\phi}^-)\in \Ds$ satisfy
\begin{equation*}
\| (\phi^+,\phi^-)\|_{\Ds}\leq \theta \qquad \text{and} \qquad  \| (\widetilde{\phi}^+,\widetilde{\phi}^-)\|_{\Ds}\leq \theta.
\end{equation*}
In order to prove that $\Gamma$ is a contraction on $\mathcal{S}_\theta$  and that the resulting fixed-point depends continuously on the data, it suffices to prove the two estimates
\begin{equation}\label{lwp:eq-forced-p4} 
\| \Gamma[\phi_0(0),\phi^\pm,\psi]\|_{\Cprod{r}{r}} \leq \theta
\end{equation}
and
\begin{equation}\label{lwp:eq-forced-p5} 
\begin{aligned}
&\| \Gamma[\phi_0(0),\phi^\pm,\psi]-\Gamma[\widetilde{\phi}_0(0),\widetilde{\phi}^\pm,\widetilde{\psi}]\|_{\Cprod{r}{r}} \\
\lesssim& \theta \big( \| \phi_0(0) - \widetilde{\phi}_0(0)\|_{\R^\dimA}+\| (\phi^+,\phi^-);(\widetilde{\phi}^+,\widetilde{\phi}^-) \|_{\Ds} + \| \psi - \widetilde{\psi} \|_{\Cprod{r}{r}} \big). 
\end{aligned}
\end{equation}
The first inequality \eqref{lwp:eq-forced-p4} follows directly from Proposition \ref{ansatz:prop-modulation} and Proposition \ref{duhamel:prop-approx}. Indeed, we obtain from Proposition \ref{ansatz:prop-modulation} that $A^+[\psi]$, $A^-[\psi]$, and $B[\psi]$ satisfy Hypothesis \ref{hypothesis:smallness}. Then, it follows from Proposition \ref{duhamel:prop-approx} that 
\begin{equation*}
    \| \Gamma[\phi_0(0),\phi^\pm,\psi]\|_{\Cprod{r}{r}}= \| \Rc[\phi_0(0),\phi^\pm, A^\pm[\psi], B[\psi], \psi] \|_{\Cprod{r}{r}} \lesssim \theta^2 \leq \theta. 
\end{equation*}
Compared to the first inequality \eqref{lwp:eq-forced-p4}, the only new difficulties in proving the second inequality \eqref{lwp:eq-forced-p5} are notational. The reason is that, once we insert the expression for $\Rc$ implicit from \eqref{duhamel:eq-R}, the left-hand side of \eqref{lwp:eq-forced-p5} contains numerous terms. The individual estimates, however, are the exact same estimates as those leading to the proof of \eqref{lwp:eq-forced-p4}. As a result, we omit the (extremely tedious but standard) details.

\end{proof}

We now prove the local well-posedness of the wave maps equation \eqref{intro:eq-WM} for small data in $\Ds$. As we show below, the main theorem then follows from a scaling argument. 

\begin{proposition}[Local well-posedness for small data]\label{lwp:prop-small-data}
Let $(\phi_0^\varepsilon,\phi_1^\varepsilon)_\varepsilon \subseteq C^\infty(\R \rightarrow T\M)$ and assume that the following three conditions are satisfied:
\begin{enumerate}[label={(\roman*)},leftmargin=10mm]
\item The initial positions $\phi^\varepsilon_0(0)$ converges in $\M$.
\item The shifted linear waves $\phi^{\diamond,+,\varepsilon}$ and $\phi^{\diamond,-,\varepsilon}$ converge in $\Ds$.
\item The shifted linear waves $\phi^{\diamond,+,\varepsilon}$ and $\phi^{\diamond,-,\varepsilon}$ are small in $\Ds$, i.e.,
\begin{equation*}
\sup_{\varepsilon>0} \| (\phi^{\diamond,+,\varepsilon},\phi^{\diamond,-,\varepsilon})\|_{\Ds} \leq \theta.
\end{equation*}
\end{enumerate}
Then, the smooth global solutions $\phi^\varepsilon$ of \eqref{intro:eq-WM} with initial data $(\phi_0^\varepsilon,\phi_1^\varepsilon)$ converge in
\begin{equation}\label{lwp:eq-space}
(C_t^0 C_x^s \medcap C_t^1 C_x^{s-1})([-1,1]\times [-1,1] \rightarrow \M).
\end{equation}
\end{proposition}

\begin{proof}[Proof of Proposition \ref{lwp:prop-small-data}:]
We define the shifted wave map $\phi^{\diamond,\varepsilon}\colon \R_{u,v}^{1+1} \rightarrow \R^{\dimA}$ by 
\begin{equation*}
    \phi^{\diamond,\varepsilon}(u,v)= \phi^\varepsilon(u,v)-\phi_0^{\varepsilon}(0).
\end{equation*}
Since $\phi^\varepsilon$ is a global smooth solution of \eqref{intro:eq-WM}, the shifted wave map is a solution of
\begin{equation}\label{lwp:eq-shifted}
\phi^{\diamond,\varepsilon,k}= \phi^{\diamond,+,\varepsilon,k} + \phi^{\diamond,-,\varepsilon,k} - \Duh[ \SecondC{k}{ij}(\phi^{\diamond,\varepsilon}) \partial_u \phi^{\diamond,\varepsilon,i} \partial_v \phi^{\diamond,\varepsilon,j} ]. 
\end{equation}
In addition to \eqref{lwp:eq-shifted}, we consider the localized Duhamel integral formulation
\begin{equation}\label{lwp:eq-shifted-localized}
\widetilde{\phi}^{\diamond,\varepsilon,k}= \phi^{\diamond,+,\varepsilon,k} + \phi^{\diamond,-,\varepsilon,k} - \chi^+ \chi^- \Duh[ \SecondC{k}{ij}(\widetilde{\phi}^{\diamond,\varepsilon}) \partial_u \widetilde{\phi}^{\diamond,\varepsilon,i} \partial_v \widetilde{\phi}^{\diamond,\varepsilon,j} ]. 
\end{equation}
Using the assumptions, Proposition \ref{ansatz:prop-modulation}, and Proposition \ref{ansatz:prop-forced-wm}, we obtain the decomposition
\begin{equation}\label{lwp:eq-structure}
\begin{aligned}
\widetilde{\phi}^{\diamond,\varepsilon,k} 
&= \sum_M A^{+,\varepsilon,k}_{M,m} \phi^{\diamond,+,\varepsilon,m}_M 
+ \sum_N A^{-,\varepsilon,k}_{N,n} \phi^{\diamond,-,\varepsilon,n}_N
+\sum_{M\sim_\delta N} B^{\varepsilon}_{M,N,mn}  \phi^{\diamond,+,\varepsilon,m}_M
 \phi^{\diamond,-,\varepsilon,n}_N + \psi^{\varepsilon,k}, 
\end{aligned}
\end{equation}
where the modulations and nonlinear remainder satisfy the following properties: 
\begin{enumerate}[label={(\roman*)},leftmargin=12mm]
    \item The modulations $A^{+,\varepsilon}$, $A^{-,\varepsilon}$, and $B^\varepsilon$ satisfy Condition \ref{condition:frequency}. 
    \item The modulations $A^{+,\varepsilon}$, $A^{-,\varepsilon}$, and $B^\varepsilon$ converge with respect to $\| \cdot\|_{\Mod}$ from Definition \ref{ansatz:def-modulation-norms}. 
    \item The nonlinear remainder $\psi^\varepsilon$ converges in $\Cprod{r}{r}$. 
    \item The modulations and nonlinear remainder are small, i.e.,
    \begin{equation*}
        \sup_{\varepsilon>0} \Big( 
        \frac{1}{2} \sup_M \|A^{+,\varepsilon}_M \|_{\Mod_M^+} 
        + \frac{1}{2} \sup_N \| A^{-,\varepsilon}_N \|_{\Mod_N^-}
        +\sup_{M\sim_\delta N} \| B^{\varepsilon}_{M,N}\|_{\Cprod{s}{s}} 
      +    \| \psi^{\varepsilon} \|_{\Cprod{r}{r}}  \Big)\leq \theta. 
    \end{equation*}
\end{enumerate}
For every fixed $\varepsilon>0$, it follows from the smoothness of the shifted linear waves, the decomposition \eqref{lwp:eq-structure}, and the properties of $A^{+,\varepsilon}, A^{-,\varepsilon}, B^{\varepsilon}$, and $\psi^{\varepsilon}$ that $\phi^{\diamond,\varepsilon,k}\in \Cprod{r}{r}$. From the deterministic theory (Lemma \ref{prep:lem-deterministic-uniqueness}), it follows that 
\begin{equation}
    \widetilde{\phi}^{\diamond,\varepsilon}(u,v)= \phi^{\diamond,\varepsilon}(u,v) \qquad \text{for all } u,v \in [-2,2]. 
\end{equation}
In order to prove the proposition, it therefore suffices to prove the convergence of $\widetilde{\phi}^{\diamond,\varepsilon}$ in \eqref{lwp:eq-space}. However, this follows directly from Lemma  \ref{ansatz:lem-ctcx}. 
\end{proof}

Equipped with Proposition \ref{lwp:prop-small-data}, we are now ready to prove the main theorem. 

\begin{proof}[Proof of Theorem \ref{intro:thm-rigorous}] 
The proof consists of two steps: In the first step, we perform a series of changes of variables which lead to small initial data. In the second step, we apply Proposition \ref{lwp:prop-small-data} and conclude the theorem. \\

\emph{Step 1: Change of variables.} 
We first recall that $\phi^{\varepsilon}$ solves the wave maps equation with initial data $(B^{\varepsilon},V^\varepsilon)$, i.e., 
\begin{equation}\label{lwp:eq-wm}
\begin{cases}
\partial_\mu \partial^\mu \phi^{\varepsilon,k} 
= - \Second^{k}_{ij}(\phi^{\varepsilon})
\partial_\mu \phi^{\varepsilon} \partial^\mu \phi^{\varepsilon} 
 ~,\\
\phi^{\varepsilon}(0,x) = B^{\varepsilon}(x), \quad \partial_t \phi^{\varepsilon}(0,x) = V^\varepsilon(x).
\end{cases}
\end{equation}
We note that since $B^{\varepsilon}\colon \R \rightarrow \M \subseteq \R^\dimA$ and $V^\varepsilon \colon \R \rightarrow \R^\dimA$ are smooth and satisfy $V^\varepsilon(x)\in T_{B^\varepsilon(x)}\M$, \eqref{lwp:eq-wm} has a unique global smooth solution. We now perform a re-scaling, spatial translation, and localization.\\

\emph{Step 1.1: Re-scaling and spatial translation.}
Let $x_0 \in \R$ be arbitrary. We define the re-scaled and translated solution 
$\phi^{\varepsilon}_{\tau,x_0}$, Brownian path $B^{\varepsilon}_{\tau,x_0}$, and velocity $V^\varepsilon_{\tau,x_0}$ by 
\begin{equation*}
\phi^{\varepsilon}_{\tau,x_0}(t,x):= \phi^{\varepsilon}(\tau t, \tau x+x_0), \quad
B^{\varepsilon}_{\tau,x_0}(x) = B^{\varepsilon}(\tau x+x_0), \quad
\text{and} \quad V^\varepsilon_{\tau,x_0}(x)=\tau V^\varepsilon(\tau x+x_0).
\end{equation*}
Due to the scale and translation-invariance of the wave maps equation, $\phi^{\varepsilon}_{\tau,x_0}$ solves 
\begin{equation}\label{lwp:eq-wm-scaled}
\begin{cases}
\partial_\mu \partial^\mu \phi^{\varepsilon}_{\tau,x_0}
= - \Second^{k}_{ij}(\phi^{\varepsilon}_{\tau,x_0})
\partial_\mu \phi^{\varepsilon}_{\tau,x_0} \partial^\mu \phi^{\varepsilon}_{\tau,x_0} 
 ~,\\
\phi^{\varepsilon}_{\tau,x_0}(0,x) = B^{\varepsilon}_{\tau,x_0}(x), \quad \partial_t \phi^{\varepsilon}_{\tau,x_0}(0,x) = V^\varepsilon_{\tau,x_0}(x).
\end{cases}
\end{equation}

\emph{Step 1.2: Localization.}
We let $B^{\varepsilon}_{\tau,x_0,\loc}$ and $V^\varepsilon_{\tau,x_0,\loc}$ be the localized versions $B^{\varepsilon}_{\tau,x_0}$ and $V^\varepsilon_{\tau,x_0}$, which were defined in \eqref{prep:equ-Bepstaux0loc_def} and \eqref{prep:equ-Vepstaux0loc_def}. We then define $\phi^\varepsilon_{\tau,x_0,\loc}$ as the solution to the wave maps equation with initial data 
$(B^{\varepsilon}_{\tau,x_0,\loc},V^\varepsilon_{\tau,x_0,\loc})$. Due to finite speed of propagation, it holds that 
\begin{equation}\label{lwp:eq-finite-speed}
\phi^\varepsilon_{\tau,x_0,\loc}(t,x)=\phi^\varepsilon_{\tau,x_0}(t,x) \qquad \text{for all } (t,x)\in [-1,1]^2. 
\end{equation}

\emph{Step 1.3: Shifts.} Similar as in \eqref{ansatz:eq-shifted-linear}, we define the shifted linear waves by 
\begin{equation*}
\phi^{\diamond,\pm,\varepsilon}_{\tau,x_0,\loc}(x):= \frac{1}{2\theta} 
\Big( B^\varepsilon_{\tau,x_0,\loc}(x)-B^\varepsilon_{\tau,x_0,\loc}(0) \mp
\int_0^x \dy \, V^{\varepsilon}_{\tau,x_0,\loc}(y) \Big). 
\end{equation*}

By Proposition~\ref{prep:prop-brownian-path-approximation}, Proposition~\ref{prep:prop-velocity-approximation}, and Proposition \ref{prep:lem-mixed_hhtolow}, there exists an event $\mathcal{E}(\tau,x_0)$ satisfying
\begin{equation}\label{lwp:eq-high-prob-1}
\mathbb{P}\big( \Omega \backslash\mathcal{E}(\tau,x_0) \big) \leq C \exp(-c \tau^{-c})
\end{equation}
for some absolute constant $c>0$ and a constant $C=C(\M)$ depending on the embedding $\M \hookrightarrow \bbR^\dimA$,
and such that, on the event $\mathcal{E}(\tau,x_0)$, the following properties hold:
\begin{enumerate}[label={(\roman*)},leftmargin=12mm]
\item The initial position $B^\varepsilon_{\tau,x_0,\loc}(0)$ converges in $\M$.
\item The shifted linear waves $\phi^{\diamond,+,\varepsilon}_{\tau,x_0,\loc}(x)$ and $\phi^{\diamond,-,\varepsilon}_{\tau,x_0,\loc}(x)$ converge in $\Ds$.
\item The shifted linear waves $\phi^{\diamond,+,\varepsilon}_{\tau,x_0,\loc}(x)$ and $\phi^{\diamond,-,\varepsilon}_{\tau,x_0,\loc}(x)$ are small in $\Ds$, i.e., 
\begin{equation*}
\sup_{\varepsilon>0} \big\|( \phi^{\diamond,+,\varepsilon}_{\tau,x_0,\loc},
\phi^{\diamond,-,\varepsilon}_{\tau,x_0,\loc} )\big\|_{\Ds} \leq \theta. 
\end{equation*}
\end{enumerate}

\emph{Step 2: Conclusion.} We proceed on the event $\mathcal{E}(\tau,x_0)$. Using the properties in Step 1.3 and Proposition \ref{lwp:prop-small-data}, it follows that \begin{equation*}
 \phi^\varepsilon_{\tau,x_0,\loc} \quad \text{converges in} \quad
\big( C_t^0 C_x^s \medcap C_t^1 C_x^{s-1} \big)([-1,1]^2 \rightarrow \M). 
\end{equation*}
Due to finite speed of propagation \eqref{lwp:eq-finite-speed}, we obtain that  
\begin{equation*}
 \phi^\varepsilon_{\tau,x_0} \quad \text{converges in} \quad
\big( C_t^0 C_x^s \medcap C_t^1 C_x^{s-1} \big)([-1,1]^2 \rightarrow \M). 
\end{equation*}
By undoing the scaling and spatial translation, it follows that 
\begin{equation}\label{lwp:eq-convergence-local}
 \phi^\varepsilon \quad \text{converges in} \quad
\big( C_t^0 C_x^s \medcap C_t^1 C_x^{s-1} \big)([-\tau,\tau] \times [ x_0 - \tau, x_0 + \tau] \rightarrow \M). 
\end{equation}
We now define the event $\mathcal{E}(\tau,R)$ in the statement of the theorem by 
\begin{equation*}
    \mathcal{E}(\tau,R):= \bigcap_{n=-\lceil 4R/\tau \rceil}^{\lceil 4R/\tau \rceil} \mathcal{E}(\tau, \tau n),
\end{equation*}
where $\lceil \, \cdot \, \rceil$ is the ceiling function. The estimate on the probability of $\mathcal{E}(\tau,R)$ follows from \eqref{lwp:eq-high-prob-1} and a union bound. The convergence statement on $[-\tau,\tau]\times [-R,R]$ follows directly from \eqref{lwp:eq-convergence-local} and a partition of unity in space.
\end{proof}

%%%%%%%%%%%%%%%%%%%%%%%%%%% Ill-posedness %%%%%%%%%%%%%%%%%%%%%%
\begin{appendix}

\section{Deterministic ill-posedness} 

We now prove the mild ill-posedness statement in Theorem \ref{intro:thm-deterministic}, which concerns the unboundedness of the first Picard iterate.

\begin{proof}[Proof of Theorem \ref{intro:thm-deterministic}.\ref{intro:item-determininistic-ill}]
We only treat the case $\mathbb{S}^2$, since the argument easily generalizes to  $\mathbb{S}^{D-1}$ with $D\geq 3$. We split the argument into two steps. In the first step, we present several reductions which simplify the first Picard iterate. In the second step, we construct an explicit sequence of functions for which the first Picard iterate diverges. \\

\emph{Step 1: Reductions.} We let $\phi_0,\phi_1\in C^\infty_b(\R\rightarrow \R^3)$ satisfy 
\begin{equation}\label{appendix:eq-constraint}
\| \phi_0(x) \|_2^2=1  \quad\text{and}\quad \langle \phi_0(x), \phi_1(x) \rangle=0 \qquad \forall x \in \R. 
\end{equation}
Here, $\| \cdot \|_2$ refers to the Euclidean norm on $\R^3$. Then, the right and left-moving linear waves are given by 
\begin{equation*}
\phi^\pm(x) = \frac{1}{2} \Big( \phi_0(x) \mp \int_0^x \dy \, \phi_1(y) \Big). 
\end{equation*}
We recall that the second fundamental form of the sphere $\mathbb{S}^2$ is given by
\begin{equation*}
\Second^k_{ij}(\phi)= \delta_{ij} \phi^k.
\end{equation*}
As a result, the first Picard iterate of the wave maps equation \eqref{intro:eq-WM} is given by 
\begin{equation*}
\Pic(t,x)= - \int_{x-t}^{x+t}\dv^\prime \int_{x-t}^{v^\prime} \du^\prime \, \big( \phi^+(u^\prime)+\phi^{-}(v^\prime)\big) \langle (\partial_u \phi^+)(u^\prime),(\partial_v \phi^-)(v^\prime) \rangle.
\end{equation*}
We now fix any time $t>0$. Furthermore, we let $\chi$ be our previous nonnegative, smooth cut-off function, which satisfies $\operatorname{supp}(\chi) \subseteq (-3,3)$, and let $0<\epsilon \leq t/100$. For any $r\in \R$, it then holds that
\begin{equation*}
\Big\| \int_{-\infty}^{\infty} \chi(x/\epsilon) \Pic(t,x) \, \mathrm{d}x\Big\|_2 
\lesssim_{\chi,\epsilon} \big\| \Pic(t,\cdot) \big\|_{\C^{r}}. 
\end{equation*}
In order to prove the unboundedness of the first Picard iterate on $\C^r\times \C^{r-1}$ for any $r\leq 1/2$, it therefore suffices to prove that 
\begin{equation}\label{appendix:eq-unbd-1}
\sup_{  \| \phi_0 \|_{\C^{1/2}}\leq 1 }
\sup_{ \| \phi_1 \|_{\C^{-1/2}} \leq 1 }
 \Big\| \int_{-\infty}^{\infty} \chi(x/\epsilon) \Pic(t,x) \, \mathrm{d}x\Big\|_2 = \infty.
\end{equation}
We now choose $\phi_0(x)=e_3 \in\R^3$ and choose $\phi_1(x)=\psi^\prime(x)$ for $\psi(x) \in C^\infty_c((-1,1)\rightarrow \R^3)$. In particular, it holds that $\| \psi\|_{\C^{1/2}} \lesssim \| \phi_1 \|_{\C^{-1/2}}$. Due to the geometric constraint in \eqref{appendix:eq-constraint}, $\psi$ has to satisfy $\langle e_3 , \psi \rangle=0$. Using our assumptions, the first Picard iterate takes the form
\begin{equation*}
\Pic(t,x)= - \frac{1}{8} \int_{x-t}^{x+t} \dv^\prime \int_{x-t}^{v^\prime} \du^\prime \,  
(2e_3 - \psi(u^\prime) + \psi(v^\prime) ) \langle \psi^\prime(u^\prime), \psi^\prime(v^\prime) \rangle. 
\end{equation*}
In order to further simplify $\mathcal{P}(t,x)$, we assume that $\operatorname{supp}(\psi) \subseteq (-2\epsilon,2\epsilon)$. In particular, it holds that $\psi(x+t)=\psi(x-t)=0$ for all $x\in (-2\epsilon,2\epsilon)$. Then, a direct computation yields that 
\begin{align*}
\int_{x-t}^{x+t} \dv^\prime \int_{x-t}^{v^\prime}  \du^\prime \, 
2e_3 \, \langle \psi^\prime(u^\prime), \psi^\prime(v^\prime) \rangle &= 0, \\
-\int_{x-t}^{x+t} \dv^\prime \int_{x-t}^{v^\prime}  \du^\prime \, 
 \psi(u^\prime)  \langle \psi^\prime(u^\prime), \psi^\prime(v^\prime) \rangle
 &= -\frac{1}{2} \int_{x-t}^{x+t} \dy \, \psi^\prime(y) \| \psi(y) \|_2^2, \\
 \int_{x-t}^{x+t} \dv^\prime \int_{x-t}^{v^\prime}  \du^\prime \, 
 \psi(v^\prime)  \langle \psi^\prime(u^\prime), \psi^\prime(v^\prime) \rangle 
 &= -\frac{1}{2} \int_{x-t}^{x+t} \dy \, \psi^\prime(y) \| \psi(y) \|_2^2. 
\end{align*}
As a result, we obtain that 
\begin{equation*}
    \Pic(t,x)= \frac{1}{8} \int_{x-t}^{x+t} \dy \, \psi^\prime(y) \| \psi(y) \|_2^2. 
\end{equation*}
We now write $\psi(y)=\psi^1(y) e_1 +\psi^2(y) e_2$, which satisfies the constraint $\langle \psi,e_3 \rangle=0$. Then, the first coordinate of $\Pic(t,x)$ is given by 
\begin{align*}
\langle \Pic(t,x) , e_1 \rangle &= \frac{1}{8} \int_{x-t}^{x+t} \dy \, (\psi^1)^\prime(y) \big( (\psi^1)^2(y) + (\psi^2)^2(y) \big) \\
&= \frac{1}{24} (\psi^1)^3(y) \Big|_{y=x-t}^{x+t} + \frac{1}{8} \int_{x-t}^{x+t} \dy \, (\psi^1)^\prime(y)  (\psi^2)^2(y) \\
&= \frac{1}{8} \int_{x-t}^{x+t} \dy \, (\psi^1)^\prime(y)  (\psi^2)^2(y).
\end{align*}
Thus, it suffices to prove that 
\begin{equation}\label{appendix:eq-unbd-2}
\sup_{\substack{\psi^1,\psi^2 \colon \\ \operatorname{supp}(\psi^j)\subseteq (-2\epsilon,2\epsilon) \\ \| \psi^j \|_{\C^{1/2}}\leq 1}} \Big | \int_{-\infty}^\infty \mathrm{d}x \chi(x/\epsilon) \int_{x-t}^{x+t} \dy \, (\psi^1)^\prime(y)  (\psi^2)^2(y) \Big|=\infty.
\end{equation}

\emph{Step 2: Proof of \eqref{appendix:eq-unbd-2}.}
The main idea in the proof of \eqref{appendix:eq-unbd-2} is to create a severe high$\times$high$\times$low-interaction in the integrand. In order to cover the endpoint $\C^{1/2}$, however, we need to be careful and work at multiple scales. 

We let $\kappa,\kappa_0 \in \mathbb{N}$ be arbitrary and define the set of frequencies 
\begin{equation*}
F_{\kappa,\kappa_0} = \{ 2^{10k} \colon \kappa_0 \leq k \leq \kappa \}. 
\end{equation*}
Then, we define
\begin{align}
\psi^1(y) &= \chi(y/\epsilon) \sum_{n\in F_{\kappa,\kappa_0}} n^{-1/2} \sin(ny),\label{appendix:eq-psi-1} \\
\psi^2(y) &= \chi(y/\epsilon) \Big( \sin(y) + \sum_{n\in F_{\kappa,\kappa_0}} n^{-1/2} \sin((n-1)y) \Big).\label{appendix:eq-psi-2} 
\end{align}
Since the frequencies in $F_{\kappa,\kappa_0}$ are well-separated, it holds that 
\begin{equation*}
\| \psi^1 \|_{\C^{1/2}} , \| \psi^2 \|_{\C^{1/2}} \lesssim_\epsilon 1,
\end{equation*}
where the implicit constant is uniform in $\kappa$ and $\kappa_0$. We now claim that 
\begin{equation}\label{appendix:eq-claim}
\Big| \int_{x-t}^{x+t} \dy \, (\psi^1)^\prime(y)  (\psi^2)^2(y) + \, \frac{1}{2} (\kappa-\kappa_0) \int_{x-t}^{x+t} \dy \, \chi(y/\epsilon)^3 (1-\cos(2y)) \Big| \lesssim_\epsilon 2^{-5\kappa_0} (\kappa-\kappa_0) + 1. 
\end{equation}
Before proving \eqref{appendix:eq-claim}, we first show that \eqref{appendix:eq-claim} implies \eqref{appendix:eq-unbd-2}. By integrating \eqref{appendix:eq-claim} against $\chi(x/\epsilon)$, using that $1-\cos(2y)\geq 0$, and using that $0<\epsilon \leq t/10$, we obtain 
\begin{align*}
&\, \bigg| \int_{-\infty}^\infty \dx \, \chi(x/\epsilon) \int_{x-t}^{x+t} \dy \, (\psi^1)^\prime(y)  (\psi^2)^2(y) \bigg| \\
\geq&\, \frac{1}{2} (\kappa-\kappa_0)
\bigg| \int_{-\infty}^\infty \dx \,  \chi(x/\epsilon)  \int_{x-t}^{x+t} \dy \, \chi(y/\epsilon)^3 (1-\cos(2y)) \bigg| 
- C_\epsilon \big( 2^{-5\kappa_0} (\kappa-\kappa_0) + 1 \big) \\ 
\geq& \, c_\epsilon (\kappa-\kappa_0) -  C_\epsilon \big( 2^{-5\kappa_0} (\kappa-\kappa_0) + 1 \big),
\end{align*}
where $C_\epsilon>0$ and $c_\epsilon>0$ are sufficiently large and small constants, respectively. We now obtain the desired conclusion \eqref{appendix:eq-unbd-2} by first choosing a parameter $\kappa_0=\kappa_0(\epsilon)$ such that $C_\epsilon 2^{-5\kappa_0}$ is smaller than $ c_\epsilon$ and then letting $\kappa\rightarrow \infty$. 
Thus, it now only remains to prove the claim \eqref{appendix:eq-claim}.
By inserting \eqref{appendix:eq-psi-1} and \eqref{appendix:eq-psi-2} into the integrand, we obtain that 
\begin{align}
&\int_{x-t}^{x+t} \dy \,  (\psi^2)^2(y) (\psi^1)^\prime(y)  \notag \\
=& 2 \sum_{m,n\in F_{\kappa,\kappa_0}}  m^{-1/2}  n^{1/2} \int_{x-t}^{x+t} \dy \, \chi(y/\epsilon)^3 \sin(y) \sin((m-1)y) \cos(ny)  \label{appendix:eq-term-1} \\
+& \sum_{n \in F_{\kappa,\kappa_0}} n^{1/2} \int_{x-t}^{x+t} \dy \, \chi(y/\epsilon)^3  \sin^2(y) \cos(ny) \label{appendix:eq-term-2}\\
 +&\sum_{\ell,m,n\in F_{\kappa,\kappa_0}}  \ell^{-1/2}m^{-1/2} n^{1/2}
 \int_{x-t}^{x+t} \dy \, \chi(y/\epsilon)^3 \sin((\ell-1)y) \sin((m-1)y) \cos(ny) \label{appendix:eq-term-3} \\
 +& \mathcal{O}_\epsilon(1), \notag
\end{align}
where we have already estimated terms in which the derivative hits the cut-off $\chi(y/\epsilon)$. We start by analyzing the main term \eqref{appendix:eq-term-1}. First, we treat the contribution of the diagonal case $m=n$. Using trigonometric identities, we have that 
\begin{equation*}
\sin(y) \sin((n-1) y) \cos(ny) = \frac{1}{4} \Big( -1 + \cos(2y) + \cos((2n-2) y) - \cos(2ny) \Big). 
\end{equation*}
Using integration by parts, this yields 
\begin{align*}
&2 \sum_{\substack{m,n\in F_{\kappa,\kappa_0}\colon\\ m=n }}  m^{-1/2}  n^{1/2} \int_{x-t}^{x+t} \dy \, \chi(y/\epsilon)^3 \sin(y) \sin((m-1)y) \cos(ny) \\
=& -\frac{1}{2} \sum_{n\in F_{\kappa,\kappa_0}} \int_{x-t}^{x+t} \dy \, \chi(y/\epsilon)^3 (1-\cos(2y)) + \mathcal{O}_\epsilon\Big( \sum_{n\in F_{\kappa,\kappa_0}} n^{-1} \Big) \\
=& -\frac{1}{2} (\kappa-\kappa_0) \int_{x-t}^{x+t} \dy \, \chi(y/\epsilon)^3 (1-\cos(2y)) + \mathcal{O}_\epsilon\big( 1\big),
\end{align*}
which is the main term in \eqref{appendix:eq-claim}.
In the non-diagonal case $m\neq n$, we use that the frequencies in $F_{\kappa,\kappa_0}$ are well-separated. Together with integration by parts, this yields
\begin{align*}
&\Big| \sum_{\substack{m,n\in F_{\kappa,\kappa_0}\\ \colon m \neq n }}  m^{-1/2}  n^{1/2} \int_{x-t}^{x+t} \dy \, \chi(y/\epsilon)^3 \sin(y) \sin((m-1)y) \cos(ny) \Big| \\
&\lesssim_\epsilon \sum_{m,n\in F_{\kappa,\kappa_0}}  m^{-1/2}  n^{1/2} \max(m,n)^{-1}
\lesssim 1.
\end{align*}
We now estimate the first error term \eqref{appendix:eq-term-2}. Using integration by parts, it follows that  
\begin{equation*}
\Big| \sum_{n \in F_{\kappa,\kappa_0}} n^{1/2} \int_{x-t}^{x+t} \dy \, \chi(y/\epsilon)^3  \sin^2(y) \cos(ny) \Big|\lesssim_\epsilon \sum_{n \in F_{\kappa,\kappa_0}} n^{-\frac{1}{2}} \lesssim 1.
\end{equation*}
It remains to estimate the third error term \eqref{appendix:eq-term-3}. To this end, we distinguish two cases.  In the case $\max(\ell,m,n)>\med(\ell,m,n)$, we use that the frequencies in $F_{\kappa,\kappa_0}$ are well-separated, which implies that
\begin{align*}
&\Big| \sum_{\substack{\ell,m,n\in F_{\kappa,\kappa_0}\colon \\
\max(\ell,m,n)>\med(\ell,m,n)}} 
\ell^{-1/2}m^{-1/2} n^{1/2}
 \int_{x-t}^{x+t} \dy \, \chi(y/\epsilon)^3 \sin((\ell-1)y) \sin((m-1)y) \cos(ny)\Big| \\
 \lesssim_\epsilon&  \sum_{\ell,m,n\in F_{\kappa,\kappa_0}}  
\ell^{-1/2}m^{-1/2} n^{1/2} \max(\ell,m,n)^{-1}  
\lesssim 1.
\end{align*}
In the case $\max(\ell,m,n)=\med(\ell,m,n)$, we only use that $\sin(\cdot)$ and $\cos(\cdot)$ are bounded,  which implies that
\begin{align*}
&\Big| \sum_{\substack{\ell,m,n\in F_{\kappa,\kappa_0}\colon \\
\max(\ell,m,n)=\med(\ell,m,n)}} 
\ell^{-1/2}m^{-1/2} n^{1/2}
 \int_{x-t}^{x+t} \dy \, \chi(y/\epsilon)^3 \sin((\ell-1)y) \sin((m-1)y) \cos(ny)\Big| \\
 \lesssim_\epsilon& \sum_{\substack{\ell,m,n\in F_{\kappa,\kappa_0}\colon \\
\max(\ell,m,n)=\med(\ell,m,n)}}\ell^{-1/2}m^{-1/2} n^{1/2}
\lesssim \sum_{\ell,m\in F_{\kappa,\kappa_0}} \ell^{-1/2} \lesssim 2^{-5\kappa_0} (\kappa-\kappa_0).
 \end{align*}
 This completes the proof of \eqref{appendix:eq-claim}. 
\end{proof}
\end{appendix}

\bibliography{BLS_library}
\bibliographystyle{myalpha}

\end{document}